\documentclass[a4paper,12pt,twoside]{book}
\usepackage{amsfonts,amsmath,amsthm,amssymb}
\allowdisplaybreaks
\usepackage{latexsym}
\usepackage[T1]{fontenc}
\usepackage[latin2]{inputenc}
\usepackage{paralist}
\usepackage{fancyhdr}
\usepackage{url}
\usepackage[titletoc]{appendix}
\usepackage{blkarray} 

\usepackage{graphicx}

\theoremstyle{plain} \newtheorem{thm}{Theorem}[section]
\theoremstyle{plain} \newtheorem{thmA}{Theorem}

\theoremstyle{plain} \newtheorem{prop}[thm]{Proposition}
\theoremstyle{plain} \newtheorem{lem}[thm]{Lemma}
\theoremstyle{plain} \newtheorem{cor}[thm]{Corollary}
\theoremstyle{plain} \newtheorem{con}[thm]{Conjecture}
\theoremstyle{plain} 

\theoremstyle{definition} \newtheorem{ex}[thm]{Example}
\theoremstyle{definition} \newtheorem{pro}[thm]{Problem}
\theoremstyle{plain} \newtheorem{cons}[thm]{Construction}
\theoremstyle{definition} \newtheorem{defi}[thm]{Definition}
\theoremstyle{remark} \newtheorem{rem}[thm]{Remark}

\newcommand{\beql}[1]{\begin{equation}\label{#1}}
\newcommand{\eeq}{\end{equation}}

\usepackage[%dvips,
bookmarks, colorlinks=true, citecolor=black, filecolor=black, linkcolor=black, urlcolor=black, pdftitle={Complex Hadamard Matrices}, pdfauthor={Ferenc Szollosi}, pdfsubject={PhD thesis}, pdfkeywords={Hadamard matrix, conference matrix, cyclic n-root},plainpages=false,pdfpagelabels]{hyperref}

\begin{document}
\hypersetup{pageanchor=false}
\begin{titlepage}
\pagestyle{empty}
\vspace*{5cm}
{

\LARGE

\begin{center}
Construction, classification and parametrization of complex Hadamard matrices
\end{center}

}

\cleardoublepage
\pagestyle{empty}

\newpage
\pagestyle{empty}

\vspace*{0.5cm}
\Huge{
\begin{center}
\end{center}
}

\Huge{
\pagestyle{empty}
\begin{center}
Construction, classification and parametrization of\\ complex Hadamard matrices
\end{center}
}
\vspace{2mm}
\normalsize
\begin{center}
by
\end{center}
\vspace{-2mm}
\LARGE{
\begin{center}
Ferenc Sz\"oll\H{o}si
\end{center}
}

\vspace{1.9cm}
\normalsize\scshape
\begin{center}
A thesis\\
presented to Central European University\\
in partial fulfillment of the requirements for the degree of\\
Doctor of Philosophy\\
in\\
Mathematics and its Applications\\
\end{center}
\normalfont
\vspace{1.6cm}

\begin{center}
Supervisor: Professor M\'at\'e Matolcsi
\end{center}

\bigskip

\bigskip

\begin{center}
Budapest, Hungary
\end{center}
\begin{center}
2011
\end{center}

\bigskip

\bigskip

\bigskip

\vspace{1mm}

\begin{center}
\includegraphics[scale=1.5]{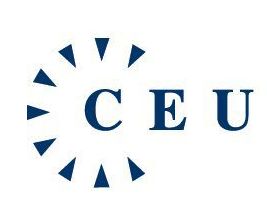}
\end{center}

\cleardoublepage
\pagestyle{empty}
\newpage
\pagestyle{empty}
\normalsize
\begin{center}
\textbf{\scshape{Dedication\\}}
To my high school teachers\\
Anna Pataki and G\'eza Magyar\\
for revealing me the beauty of mathematics.
\end{center}
\cleardoublepage
\pagestyle{empty}
\end{titlepage}
\hypersetup{pageanchor=true}

\frontmatter

\pagestyle{fancy}                       % Sets fancy header and footer 
\fancyfoot{}                            % Delete current footer settings 
\renewcommand{\chaptermark}[1]{         % Lower Case Chapter marker style 
  \markboth{\MakeUppercase{%\thechapter.\ \chaptername:\
  #1}}{}} % 
\renewcommand{\sectionmark}[1]{        % Lower case Section marker style 
  \markright{%\thesection.\
  #1}}         % 
\fancyhead[OR,EL]{\thepage}    % Page number (boldface) in left on even 
                                        % pages and right on odd pages 
\fancyhead[OL]{\leftmark}      % Chapter in the right on even pages 
\fancyhead[ER]{\leftmark}     % Section in the left on odd pages 
\renewcommand{\headrulewidth}{0.3pt}    % Width of head rule 

\pagenumbering{roman}

\thispagestyle{empty}
\setcounter{page}{7}%az elso szamozott oldal a tartalomjegyzek
\tableofcontents 
\addtocontents{toc}{\protect\thispagestyle{empty}}%kinyirja az elso toc oldalon az oldalszamot alulrol.
\newpage %Ez azert kell, hogy jobb oldalon kezdodjenek a chapterek.
\thispagestyle{empty} % ez is
\newpage
\chapter{Introduction}
\begin{flushright}
``Art is never finished, only abandoned.''\\[1mm]
\textsc{$\overline{\text{\hspace{86pt} Leonardo da Vinci}}$}
\end{flushright}
\thispagestyle{empty}
This thesis is based on our recent research contributions to the theory of complex Hadamard matrices \cite{SZF5}, \cite{SZF7}, \cite{SZF6}, \cite{SZF2}, \cite{SZF4}. Some earlier results are also touched upon such as \cite{JMS1}, \cite{MRS}, \cite{MSZ}, as well as several new results which will be the subject of forthcoming publications.

Complex Hadamard matrices form an important family of orthogonal arrays with the additional unimodularity constraint imposed on their entries. These matrices obey the algebraic identity $HH^\ast=n I_n$ where $\ast$ stands for the Hermitian transpose, and $I_n$ is the identity matrix of order $n$. It is easy to see that after proper scaling the Fourier matrices are well-known examples of complex Hadamard matrices.

Complex Hadamard matrices appear frequently in various branches of mathematics, including linear algebra \cite{GR1}, coding and operator theory \cite{UH1}, \cite{SP1} and harmonic analysis \cite{KM2}, 
\cite{TT1}. They play an important r\^ole in quantum optics, high-energy physics, and they are one of the key ingredients to quantum teleportation and dense coding schemes \cite{RW1} and mutually unbiased bases (MUBs) \cite{DEBZ}.

The intended purpose of this work is to provide the reader with a comprehensive, state-of-the art presentation of the theory of complex Hadamard matrices, or at least report on the very recent advances. This manuscript consists of three chapters, each describing one of three distinct faces of this field whose treatment require various mathematical tools ranging from combinatorics, functional analysis to symbolic computation. Although we firmly believe that these beautiful objects are interesting on their own and worth investigating from a purely mathematical perspective we make considerable efforts to highlight some of their applications we aware of. Industrial applications can be found in the textbooks \cite{SSA1} and \cite{KJH1} which are warmly recommended to the interested reader.

This thesis has three main ingredients: known results from the existing literature; our contribution to the literature; and new results which are subject to a series of forthcoming publications. Each of these
constitute about one third of this manuscript. We are certain that this dissertation is just the starting point of a long-term journey and in order to keep us (and hopefully others) entertained along the way we pose over $30$ research problems worth investigating in the future.

In the subsequent sections we outline, chapter by chapter, the contents of this thesis highlighting our main contributions.

\section*{Complex Hadamard matrices of composite orders}
In Chapter \ref{ch1} we lay down the foundations of this work by recalling the relevant notions and results from the existing literature. Original interest was in real Hadamard matrices \cite{JH1} whose existence is still a mystery. The Hadamard conjecture states that for every positive integer $m$ there is a real Hadamard matrix of order $4m$. This century-old problem is currently far out of reach, despite continuous efforts. It is natural to investigate the problem in more generality and study complex Hadamard matrices of order $n$ which are formed by some $q$th roots of unity. These matrices are called Butson-Hadamard matrices and are denoted by $BH(n,q)$. Researchers are interested in these more general objects because they can lead to real Hadamard matrices \cite{CCL} and hopefully one day to the solution of the Hadamard Conjecture.

One of the fundamental differences between real and complex Hadamard matrices is that while studying and classifying real Hadamard matrices is a discrete problem which can be handled by deep algebraic methods and sophisticated computer programs, in the complex case various infinite, parametric families appear (i.e.\ matrices with certain degree of freedom). Therefore the complex case is not finite any longer and one cannot hope for a finite list of complex Hadamard matrices.

We investigate the parametrization properties of $BH(n,q)$ matrices focusing on the cases $q=4$ and $6$. We give a comprehensive classification of $BH(n,4)$ matrices up to order $n=8$ and highlight some key facts regarding orders $n=10$ and $12$ from a joint work with Pekka Lampio and Patric \"Osterg\aa rd \cite{LSZO}. In particular, we mention the following two results.
\begin{thmA}
For $4\leq n\leq 12$ all $BH(n,4)$ matrices belong to some infinite, affine parametric family of complex Hadamard matrices.
\end{thmA}
On the other hand, there are $14\times 14$ matrices which cannot be parametrized at all. This is a fundamental difference from the real case.
\begin{thmA}
There exist isolated $BH(14,4)$ matrices.
\end{thmA}
The new result regarding $BH(n,6)$ matrices were obtained during a visit to Robert Craigen. We have obtained the following
\begin{thmA}
For every prime $p$ there exists a $BH(p^2,6)$ matrix.
\end{thmA}
To distinguish essentially different complex Hadamard matrices from each other we have introduced a powerful new invariant, the fingerprint \cite{SZF2}.

As an application of Butson-Hadamard matrices we give a further look at Fuglede's conjecture \cite{FUG1} and obtain new, previously unknown counterexamples.
\section*{Towards the classification of $6\times 6$ Hadamard matrices}
It is natural to ask what a ``typical'' complex Hadamard matrix of order $n$ looks like, and a satisfying answer to this question can be given provided we have at our disposal a complete characterization of Hadamard matrices of order $n$. These types of problems, however, are notoriously difficult even for small $n$.

There is a complete classification of real Hadamard matrices up to order $n=28$, and recent advances indicate that there is every reason to believe that at least an enumeration of the millions of matrices of order $32$ and $36$ is possible as well \cite{HK2}. In contrast, a complete classification of complex Hadamard matrices is only available up to order $n=5$. It is easy to see that the (rescaled) Fourier matrices are the unique examples in orders $n\leq 3$. The case $n=4$ is still elementary, and it was shown by Craigen that all complex Hadamard matrices of order $4$ belong to an infinite, continuous one-parameter family \cite{RC3}. In order $5$ we have uniqueness again, a result which already is absolutely non-trivial. In particular, Lov\'asz was the first who showed \cite{L} that the Fourier matrix is the only circulant complex Hadamard matrix of order $5$, and a decade later Haagerup managed to prove the uniqueness of the Fourier matrix by discovering an algebraic identity (see formula \eqref{HEQ}) relating the matrix entries in a surprising way \cite{UH1}.

In order $6$ various one \cite{BN1}, \cite{PD1}, \cite{MSZ} and \cite{GZ1}; two \cite{BK3}, \cite{SZF5}; and three-parameter families \cite{BK2} have been constructed recently and it was conjectured that these are part of a more general, four-parameter family of complex Hadamard matrices \cite{BBE}. This conjecture is supported by overwhelming numerical evidence \cite{SNS}, however so far only a fairly small subset of it was described by closed analytic formulae, including an isolated matrix $S_6^{(0)}$ and a three-parameter family of matrices $K_6^{(3)}$. One of the main reasons why the complex case is so difficult is the presence of infinite parametric families and the lack of understanding of vanishing sums of order $k$ for $k\geq 5$. In Chapter \ref{ch2} we make substantial progress towards the classification of $6\times 6$ complex Hadamard matrices by giving a construction with four degrees of freedom. We believe that our construction captures the typical features of complex Hadamard matrices of order $6$. The main result is essentially the following
\begin{thmA}
There is a four-parameter family of $6\times 6$ complex Hadamard matrices.
\end{thmA}
The reason why the $6\times 6$ case received significant attention in the past couple of years is the fact that complex Hadamard matrices are closely related to mutually unbiased bases. Recall that two orthonormal bases of $\mathbb{C}^n$, $\mathcal{B}_1$ and $\mathcal{B}_2$ are unbiased if for every $e\in \mathcal{B}_1$, $f\in \mathcal{B}_2$ we have $\left|\left\langle e,f\right\rangle\right|^2=1/n$. A family of orthonormal bases is said to be (pairwise) mutually unbiased if every two of them are unbiased. The famous MUB-$6$ problem asks for the maximal number of mutually unbiased bases in $\mathbb{C}^6$. On the one hand this number is at least $3$, as there exists various infinite families of triplets of MUBs in this order \cite{JMS1}, \cite{SZF5}, \cite{GZ1}. In particular, by utilizing an elegant construction of Zauner \cite{GZ1} we obtain the following
\begin{thmA}
There is a two-parameter family of triplet of $\mathrm{MUB}$s in $\mathbb{C}^6$.
\end{thmA}
On the other hand it is well-known that the number of pairwise mutually unbiased bases cannot be larger than $7$ (see the references of \cite{BBE}) in $\mathbb{C}^6$. In fact, it is conjectured that a triplet is the best one can come up with in dimension $6$ \cite{GZ1}. The connection between MUBs and Hadamard matrices of order $6$ has been exploited in our joint work \cite{JMS1} very recently, where a discretization scheme was offered to attack the problem and it was proved by means of computers, however, in a mathematically rigorous way, that the members of the two-parameter Fourier family $F_6^{(2)}(a,b)$ and its transpose cannot belong to a configuration of $7$ MUBs containing the standard basis in dimension $6$. One reasonable hope to finally settle the MUB-$6$ problem is to give a complete characterization of complex Hadamard matrices of order $6$ and apply the same technique to them.

During Chapter \ref{ch2} we propose a general framework towards the complete classification of complex Hadamard matrices of order $6$. In particular, by characterizing the orthogonal triplets of rows in complex Hadamard matrices we generalize an observation of Haagerup \cite{UH1} to obtain a new algebraic identity relating the matrix entries in an unexpected way. This is an essentially new tool to study complex Hadamard matrices of small orders, and one of the main achievements of this chapter. We apply this result to obtain complex Hadamard matrices, moreover we conjecture that the construction we present here reflects the true nature of complex Hadamard matrices of order $6$. It has the following three features: Firstly, it is general in contrast with the earlier attempts where always some additional extra structure was imposed on the matrices including self-adjointness \cite{BN1}, symmetry \cite{MSZ}, circulant block structure \cite{SZF5} or $H_2$-reducibility \cite{BK1}. Secondly, it has $4$ degrees of freedom and thirdly all the entries of the obtained matrices can be described by algebraic functions of roots of various sextic polynomials. This suggests on the one hand the existence of a four-parameter family of complex Hadamard matrices of order $6$ and reminds us on the other hand of the fact that the desired algebraical description where the entries are expressed by closed analytic formulae might not be possible at all. However, from the applicational point of view, and in particular, to utilize the computer-aided attack of \cite{JMS1} to the MUB-$6$ problem we anyway only need these matrices numerically. 

Finally, to illustrate the applications of MUBs we exhibit equiangular lines in real spaces. Our results slightly improve on a recent construction of de Caen \cite{DC1}.
\begin{thmA}
For each $n=3\cdot 2^{2t-1}+1$, with $t$ any positive integer, there exists an equiangular set of $\frac{2}{9}(n-1)(n+2)$ lines in $\mathbb{R}^n$.
\end{thmA}
As a consequence we obtain a new general quadratic lower bound on the number of real equiangular lines (see Corollary \ref{ch2constant}).
\section*{Complex Hadamard matrices of prime orders}
The final chapter is devoted to the discussion of complex Hadamard matrices of prime orders. Constructing complex Hadamard matrices of prime orders requires considerable efforts. One of the reasons for this is that design theory fundamentally relies on ``plug-in'' methods and block constructions resulting in objects of composite orders. Furthermore it is known that the Fourier matrices of prime orders are isolated thus we do not have natural examples of parametric families of complex Hadamard matrices in the prime order case. Throughout Chapter \ref{ch3} we recall Petrescu's construction \cite{MP1} and obtain new interesting examples of $BH(n,q)$ matrices. On the one hand we settle the existence of $BH(19,6)$ matrices which is listed as unresolved in \cite{CRC1}, on the other hand we obtain some non-existence results.
\begin{thmA}
There exists a $BH(19,6)$ matrix.
\end{thmA}
We list several new examples of $BH(n,q)$ matrices in Appendix \ref{APPA}. Among these there is a $BH(13,30)$ matrix leading to a four-parameter family of complex Hadamard matrices of order $13$.

In the subsequent sections we investigate circulant complex Hadamard matrices. Following the influential paper \cite{BH1} we construct new examples of index $4$ type circulant complex Hadamard matrices of prime orders. In particular, we prove the following
\begin{thmA}
For every prime $p\equiv 1$ $(\mathrm{mod}\ 8)$ there is a new, previously unknown complex Hadamard matrix of simple index $4$.
\end{thmA}
These matrices come from an object which is the analogue of the well-known Paley matrix, but instead of describing the quadratic residues in $\mathbb{Z}_p$ it encodes the quartic residues. In general the entries of these matrices are rather complicated. We believe that some further improvements on our work will eventually lead to the full classification of all cyclic $p$-roots of simple index $4$ when $p\equiv 1$ (mod $8$) is prime. The case $p=17$ is worked out in details in Appendix \ref{APPB}.

Finally we investigate the problem of finding all complex Hadamard matrices with circulant core and solve it for $n\leq 7$ yielding a new, previously unknown complex Hadamard matrix of order $7$. To deal with this problem we utilize the full machinery of the theory of Gr\"obner bases \cite{BB1} and we use it to investigate some special classes of higher order matrices as well. This computer experiment turned out to be fruitful as we have discovered two infinite classes of complex Hadamard matrices as follows:
\begin{thmA}
For every prime $p\equiv 1$ $(\mathrm{mod}\ 8)$ there are two new, previously unknown examples of complex Hadamard matrices of order $p+1$, having a circulant core.
\end{thmA}
The underlying object behind both of these matrices are, again, the quartic residues in $\mathbb{Z}_p$. These results are rather technical and they are contained in Appendix \ref{APPC}.

We conclude the thesis by investigating complex Hadamard matrices with few different entries, and we utilize them to construct new examples of equiangular tight frames \cite{BE1}.
\newpage %Ez azert kell, hogy jobb oldalon kezdodjenek a chapterek.
\thispagestyle{empty} % ez is
\chapter*{Acknowledgements}
\addcontentsline{toc}{chapter}{Acknowledgements}
\thispagestyle{empty}
I am greatly indebted to Professor M\'at\'e Matolcsi for his continuous guidance and support since my freshman year. I am grateful for his permanent influence on my mathematical outlook and for his patience and wisdom allowing me to explore the mathematics I am fascinated with. I feel fortunate that during these years I had the opportunity of learning from him and doing research with him.

My warmest thanks go to Pofessor Robert Craigen who supervised my work during my visit as a guest student at the University of Manitoba. The training he provided me with substantially improved my knowledge on the subject of this thesis.

I acknowledge financial support by Central European University grant no.\ DRSG-$131/2010$-$11$; Hungarian Scientific Research Fund (OTKA) grant no.\ K-$77748$; and National Sciences and Engineering Research Council of Canada (NSERC) grant no.\ $155855$-$09$.

\begin{flushright}
\begin{tabular}{c}
Ferenc Sz\"oll\H{o}si\\
\href{mailto:szoferi@gmail.com}{\nolinkurl{szoferi@gmail.com}}\\
Budapest, October 25, 2011.
\end{tabular}
\end{flushright}
\newpage %Ez azert kell, hogy jobb oldalon kezdodjenek a chapterek.
\thispagestyle{empty} % ez is
%----------------------------------CHAPTER BREAK-------------------------------------
%----------------------------------CHAPTER BREAK-------------------------------------
\mainmatter
\pagestyle{fancy}                       % Sets fancy header and footer 
\fancyfoot{}                            % Delete current footer settings 
\renewcommand{\chaptermark}[1]{         % Lower Case Chapter marker style 
  \markboth{\MakeUppercase{\thechapter\ #1}}{}} % 
\renewcommand{\sectionmark}[1]{        % Lower case Section marker style 
  \markright{\MakeUppercase{\thesection\ #1}}}         % 
\fancyhead[OR,EL]{\thepage}    % Page number (boldface) in left on even 
%                                        % pages and right on odd pages 
\fancyhead[ER]{\leftmark}      % Chapter in the right on even pages 
\fancyhead[OL]{\rightmark}     % Section in the left on odd pages 
\renewcommand{\headrulewidth}{0.3pt}    % Width of head rule 
%\setcounter{secnumdepth}{2}
%----------------------------------CHAPTER BREAK-------------------------------------
%----------------------------------CHAPTER BREAK-------------------------------------
\chapter{Complex Hadamard matrices of composite orders}\label{ch1}
\thispagestyle{empty}
The purpose of this introductory chapter is to describe the basic terminology and the fundamental properties of complex Hadamard matrices. We shall give various examples of infinite, parametric families as well as examples of isolated complex Hadamard matrices in the sequel. Hadamard matrices, composed from roots of unity shall be investigated in details. Our intentions were to include relevant structural results from the existing very recent literature giving this chapter a survey flavour. The reader interested in classical results concerning real and generalized Hadamard matrices should consult Agaian's \cite{SSA1} or Horadam's book \cite{KJH1}, or read the PhD thesis of Seberry \cite{JS1} or Craigen \cite{RC1}. Our contribution to this chapter is based on the papers \cite{MRS}, \cite{SZF2}, \cite{SZF4} and the preprint \cite{LSZO}. Additionally, we briefly mention the important achievements from \cite{SZF3} and \cite{SZF1}. The new results in Sections \ref{ch1bh6} and \ref{ch1tilespectral} are subject to a possible forthcoming publication.

We begin our journey with investigating the existence of complex Hadamard matrices.
\section{Existence of complex Hadamard matrices}
Throughout this thesis we denote by $\mathbb{N}$, $\mathbb{Z}$, $\mathbb{Q}$, $\mathbb{R}$, $\mathbb{C}$ and $\mathbb{T}$ the natural, integral, rational, real, complex and complex unimodular numbers, respectively. We assume that $0\in\mathbb{N}$. We denote the matrices by capital letters, e.g.\ $I$ or $I_n$ is the identity matrix of order $n$. We usually drop the subscript if it is irrelevant our clear from context. For a given matrix $A$ we denote by $\overline{A}, A^T$ and $A^\ast$ its entrywise conjugate, transpose and hermitean adjoint, respectively. The space of all $n\times n$ complex matrices is denoted by $\mathcal{M}_n\left(\mathbb{C}\right)$.

Our dissertation is entirely devoted to investigate the existence and the structure of the following concept.
\begin{defi}
A \emph{complex Hadamard matrix} $H$ is an $n\times n$ matrix with complex entries of modulus $1$ such that $HH^\ast=nI$.
\end{defi}
Therefore a complex Hadamard matrix is essentially a unitary matrix with unimodular entries. Of course, we obtain a unitary matrix after proper normalization. As we shall see, it is easy to exhibit a complex Hadamard matrix in every order $n$ but first let us fix some notations once and for all.
\begin{defi}\label{ch1defom}
Throughout this thesis we denote by $\Re[z]$ and $\Im[z]$ the real and imaginary part of a complex number $z$, respectively. We denote by $\mathbf{i}$ the principal fourth root of unity, that is $\mathbf{i}^2=-1$ with the convention that $\Im[\mathbf{i}]=1$. Finally, let us denote by $\omega$ the principal cubic root of unity, namely \[\omega=\cos\left(\frac{2\pi}{3}\right)+\mathbf{i}\sin\left(\frac{2\pi}{3}\right)=-\frac{1}{2}+\mathbf{i}\frac{\sqrt3}{2}.\]
\end{defi}
With these notations at our hand, we can offer our very first
\begin{lem}
For every $n\geq 1$ there exists a complex Hadamard matrix of order $n$.
\end{lem}
\begin{proof}
Indeed, as the Fourier matrix
\beql{ch1f}
[F_n]_{i,j}=\mathbf{e}^{2\pi\mathbf{i}(i-1)(j-1)/n},\ \ \ i,j=1,\hdots, n
\eeq
is a well-known example of complex Hadamard matrices.
\end{proof}
Note that throughout this manuscript the term ``Fourier matrix'' is used to describe the complex Hadamard matrix $F_n$ under \eqref{ch1f}, and not the unitary matrix $\frac{1}{\sqrt{n}}F_n$. We offer the following
\begin{ex}
The following are complex Hadamard matrices of orders $n=1,2$ and $3$, respectively:
\[F_1=\left[\begin{array}{c}
1\\
\end{array}\right],\ \ 
F_2=\left[\begin{array}{cc}
1 & 1\\
1 & -1\\
\end{array}\right],\ \ 
F_3=\left[\begin{array}{ccc}
1 & 1 & 1\\
1 & \omega & \omega^2\\
1 & \omega^2 & \omega\\
\end{array}\right].\]
\end{ex}
One immediately observes that the matrices displayed in the preceding example all have a bordering row and column of numbers $1$. This is a general feature of complex Hadamard matrices, up to the following equivalence.
\begin{defi}\label{ch1d1}
The complex Hadamard matrices $H$ and $K$ are called (permutation--phasing) \emph{equivalent}, if there are permutation matrices $P_1$ and $P_2$ and unitary diagonal matrices $D_1$ and $D_2$ such that $P_1D_1HD_2P_2=K$. This relation is denoted by \mbox{$H\sim K$}. If two Hadamard matrices are not equivalent then we say that they are inequivalent.
\end{defi}
Definition \ref{ch1d1} exploits the fact that the rearrangement of the rows and columns or the shift of the phase in any row or column of a complex Hadamard matrix maintains its fundamental properties: both the unitary and unimodular conditions still hold after these transformations. We recall the following
\begin{defi}
A complex Hadamard matrix $H$ of order $n$ is \emph{dephased} if the first row and column of it consists of entirely of numbers $1$. The lower right $(n-1)\times (n-1)$ submatrix is called the \emph{core} of $H$.
\end{defi}
A dephased real Hadamard matrix is usually called normalized. It is immediate, that the following is true.
\begin{lem}\label{ch1L117}
Every complex Hadamard matrix is equivalent to a dephased one.
\end{lem}
Henceforth it is enough to consider dephased complex Hadamard matrices. There are other ``trivial'' ways to obtain new complex Hadamard matrices from what is already at our hands. In particular, we have the following
\begin{lem}
If $H$ is a complex Hadamard matrix, then so is $H^\ast, \overline{H}$ and $H^T$.
\end{lem}
\begin{proof}
Let $H$ be a complex Hadamard matrix. Then conjugate the identity $HH^\ast=nI$ in order to see that $\overline{H}$ is a complex Hadamard matrix as well. It is also clear that $H$ is invertible and $H^{-1}=H^\ast/n$. Hence, it follows that $H^\ast H=nI$, which shows that both $H^\ast$ and, after conjugating again, $H^T$ are complex Hadamard matrices.
\end{proof}
Another natural approach to construct new matrices from old is to lift given objects to higher orders via the Kronecker product. We use the following concept repeatedly in our manuscript.
\begin{defi}
Let $H$ be an $n\times n$, $K$ be an $m\times m$ matrix. Then their \emph{Kronecker product} $H\otimes K$ is an $nm\times nm$ matrix with its $(i,j)$th block being given by $[H]_{i,j}K$, $i,j=1,\hdots, n$.
\end{defi}
The following is immediate.
\begin{lem}\label{ch1Syl}
If $H$ and $K$ are complex Hadamard matrices then so is $H\otimes K$.
\end{lem}
\begin{proof}
Clearly $H\otimes K$ is unimodular, moreover
\[(H\otimes K)(H\otimes K)^\ast=HH^\ast\otimes KK^\ast=nI_n\otimes mI_m,\]
which can be identified with $nmI_{nm}$.
\end{proof}
In $1867$ Sylvester constructed real Hadamard matrices of orders $n=2^k$ for every $k\geq 1$ via Lemma \ref{ch1Syl}, starting from the Fourier matrix $F_2$ \cite{JJS1}. For comparison, we give an additional
\begin{ex}\label{ch1E4}
\[H_4=F_2\otimes F_2=\left[\begin{array}{rr|rr}
1  &  1 & 1  & 1\\
1 & -1 & 1 & -1\\
\hline
1 & 1 & -1 & -1\\
1 & -1 & -1 & 1\\
\end{array}\right],\ \ \ 
F_4=\left[\begin{array}{rrrr}
1 & 1 & 1 & 1\\
1 & \mathbf{i} & -1 & -\mathbf{i}\\
1 & -1 & 1 & -1\\
1 & -\mathbf{i} & -1 & \mathbf{i}\\
\end{array}\right].\]
\end{ex}
It is not immediately clear that the matrices $H_4$ and $F_4$ are inequivalent. Later we shall see several invariants easily detecting inequivalence in this, and various other cases. Here we offer the general treatment of the Kronecker product of Fourier matrices.
\begin{thm}[Tadej, \cite{WT1}]\label{ch1T1}
Let $F=F_{n_1}\otimes F_{n_2}\otimes\cdots \otimes F_{n_r}$ be a Kronecker product of Fourier matrices. Then $F$ is equivalent to $F_{m_1}\otimes F_{m_2}\otimes \cdots\otimes F_{m_s}$ if and only if the sequence $(m_1,m_2,\hdots,m_s)$ is obtained from the sequence $(n_1,n_2,\hdots,n_r)$ using a series of operations from the list below:
\begin{enumerate}
\item permuting a sequence;
\item replacing a subsequence $n_a, n_b$ by $n_c=n_an_b$ if $n_a$ and $n_b$ are relatively prime;
\item replacing a sequence element $n_c$ by a subsequence $n_a, n_b$, if $n_c=n_an_b$ and $n_a, n_b$ are relatively prime.
\end{enumerate}
\end{thm}
It follows from Theorem \ref{ch1T1} that the matrices displayed in Example \ref{ch1E4} are inequivalent. Therefore we already have a handful of examples of complex Hadamard matrices coming from purely the Fourier matrices at our disposal. Another source of examples are the real Hadamard matrices. It is natural to ask the following
\begin{pro}
Decide for which orders $n$ exists a real Hadamard matrix.
\end{pro}
It is immediate from Lemma \ref{ch1L117} that the order of such a matrix must be $1$ or even. However, a little more careful analysis reveals the following
\begin{lem}\label{ch1lde}
If a real Hadamard matrix of order $n$ exists, then $n=1,2$ or $n\equiv 0$ $(\mathrm{mod}\ 4)$.
\end{lem}
It is widely believed that this is the only obstruction. In particular, we have the following long-standing
\begin{con}[The Hadamard conjecture, \cite{RP1}]\label{ch1Hc}
If $n\equiv 0$ $(\mathrm{mod}\ 4)$ then there is a real Hadamard matrix.
\end{con}
The truth of Conjecture \ref{ch1Hc} has been verified for orders $n\leq 664$; order $n=428$ has been constructed only very recently in \cite{HK1}. There are $13$ open cases below $n=2000$ for which the existence is undecided \cite{DJ1}, all of them are of the form $n=4p$ where $p\equiv 3$ (mod $4$) is prime.

From Lemma \ref{ch1lde} it follows that there is no real Hadamard matrix of orders $5, 6$ or $7$. For order $8$ we have a unique (up to equivalence) example given by $F_2\otimes F_2\otimes F_2$, while constructing a matrix of order $12$ requires some additional terminology (see also the intriguing Example \ref{ch11212}).
\begin{defi}
We say that an $n\times n$ matrix $C$ is \emph{circulant}, induced by its first row $x=(x_0,x_1,\hdots, x_{n-1})$, if $C_{i,j}=x_{j-i}$ where the indices are taken modulo $n$. We usually denote such a matrix by $C=\mathrm{Circ}(x)$. Similarly, we denote by $D=\mathrm{Diag}(x)$ a diagonal matrix of order $n$ with main diagonal $x$.
\end{defi}
Throughout this thesis we denote by $\mathbb{Z}_p$ the ring of integers modulo $p$. Similarly, $\mathbb{Z}_p^\ast$ stands for the multiplicative group of the nonzero elements of $\mathbb{Z}_p$. Now in order to construct a matrix of order $12$, we recall the following fundamental
\begin{thm}[Paley construction, \cite{RP1}]\label{ch1rpc1}
Let $p$ be a prime number, $p\equiv 3$ $(\mathrm{mod}\ 4)$, and consider a vector $x$ such that $x_0=0$ and $x_i=1$ if $i$ is a quadratic residue in $\mathbb{Z}_p^\ast$, otherwise $x_i=-1$. Let us define the Paley matrix $P:=\mathrm{Circ}(x)$. Then, the matrix $C:=P-I$, bordered by a row and column of numbers $1$ is a real Hadamard matrix of order $p+1$.
\end{thm}
The Paley construction can be generalized to prime power orders as well using finite fields. In particular, for every prime $p$ and $k\geq 1$ there are real Hadamard matrices of order $n=p^k+1$ or $n=2(p^k+1)$, depending on whether we have $p^k\equiv 3$ (mod $4$) or $p^k\equiv 1$ (mod $4$), respectively. It is worthwhile noting that Theorem \ref{ch1rpc1} leads to real Hadamard matrices with circulant core. We will return to these objects later in Appendix \ref{APPC}.

A Hadamard matrix of order $20$ was constructed by Hadamard in $1893$ \cite{JH1}, and it took a century to have a complete classification of real Hadamard matrices of order at most $28$ \cite{ES1}, while the enumeration of Hadamard matrices of order $32$ begun just very recently, yielding over $13$ million inequivalent matrices \cite{HK2}.

Let us try to understand how good we are doing in terms of Conjecture \ref{ch1Hc}. In order to measure this, let us denote by $S(x)$ the number of $n\leq x$ for which a Hadamard matrix of order $n$ exists. Now the Hadamard conjecture states that $S(x)$ is about $x/4$, and hence the set of Hadamard orders has density $1/4$. As yet, it is not even known if this set has positive density, as taking into account all major known construction methods, one obtains only the following
\begin{thm}[de Launey--Gordon, \cite{GdL1}]
For all $\varepsilon>0$ there is a natural number $x_\varepsilon$ such that for all $x>x_\varepsilon$
\[S(x)\geq \frac{x}{\log x}\mathrm{exp}\left(\left(C+\varepsilon\right)\left(\log\log\log x\right)^2\right),\]
for $C=0.8178\hdots$.
\end{thm}
However, from another point of view we are doing quite well, as Seberry proved that for every odd $m$ and for every large enough $t=t(m)$, there is a real Hadamard matrix of order $n=2^tm$ \cite{JS2}. This was subsequently improved by Craigen \cite{RC2}; currently the best known asymptotic result of this flavour is the following
\begin{thm}[see \cite{CRC1}]
A Hadamard matrix of order $2^tm$ for odd $m$ exists for all
\[t\geq 6\left\lfloor\frac{\log_2\frac{m-1}{2}}{16}\right\rfloor+2.\]
\end{thm}
In particular, for every odd $m$ there are only finitely many numbers $t$ such that the existence of a real Hadamard matrix of order $2^tm$ is undecided.

A different approach is to try to guarantee a large number of pairwise orthogonal rows of $\{\pm1\}$ entries in a given order $n$. Let us denote by $r(n)$ the largest number $r$ such that there are $r$ pairwise orthogonal $\{\pm1\}$ rows of length $n$. Then the following is true.
\begin{thm}[de Launey--Gordon, \cite{DLG2}]
Let $\varepsilon>0$. If the Extended Riemann Hypothesis is true, then for every sufficiently large $n\equiv 0$ $(\mathrm{mod}\ 4)$
\[r(n)\geq \frac{n}{2}-n^{\frac{17}{22}+\varepsilon}.\]
\end{thm}
Thus, for $n$ large enough we have about one half of a real Hadamard matrix of order $n$. It seems, however, that despite continuous efforts from researchers from a broad area, proving the truth of the Hadamard conjecture remains elusive, and fundamental new ideas are required for further progress.

\section{Parametrizing complex Hadamard matrices}
In this section we introduce the concept of parametrization, a powerful method obtaining new complex Hadamard matrices from old. The parametrization can be thought as a continuous analogue of the switching operation \cite{WPO1}, a well-known technique in design-theory. These ideas were used by Craigen \cite{RC3}, Di\c{t}\u{a} \cite{PD1}, Haagerup \cite{UH1} and Nicoar\u{a} \cite{RN1} (among others), who exhibited various parametric families in the last decade. It should be noted that Hadamard himself, who investigated $4\times 4$ matrices with maximum determinant, exhibited essentially the following
\begin{ex}[cf.\ \cite{RC3}, \cite{JH1}]\label{ch1ExF4i}
For every unimodular number $a$ the matrix
\[F_4^{(1)}(a)=\left[\begin{array}{rrrr}
1 & 1 & 1 & 1\\
1 & \mathbf{i}a & -1 & -\mathbf{i}a\\
1 & -1 & 1 & -1\\
1 & -\mathbf{i}a & -1 & \mathbf{i}a\\
\end{array}\right]\]
is complex Hadamard, forming a one-parameter family of complex Hadamard matrices, stemming from the starting point matrix $F_4=F_4^{(1)}(1)$.
\end{ex}
The preceding example describes a family of complex Hadamard matrices as it incorporates various inequivalent matrices through the indeterminate $a$. In particular, for $a=1$ we have the Fourier matrix $F_4=F_4^{(1)}(1)$, while for $a=\mathbf{i}$ we get the real Hadamard matrix $H_4=F_2\otimes F_2$, which are inequivalent. We can look at this phenomenon from a different point of view. One might start with the single matrix $F_4$, and then try to find a way to introduce some degree of freedom into it, say, by introducing the parameter $a$ as in Example \ref{ch1ExF4i} in order to exhibit new, previously unknown matrices, which are inequivalent from the starting point one. This approach, as we shall see in Section \ref{ch1tilespectral} where we discuss spectral sets and tiles in abelian groups, can lead to some very exciting discoveries.

Compare Example \ref{ch1ExF4i} with the following
\begin{ex}\label{ch1ExF4iC}
Consider the Fourier matrix $F_3$ and multiply its second row by a unimodular number $a$, and multiply its third column by a unimodular number $b$ to obtain the following
\[H^{(2)}_3(a,b)=\left[\begin{array}{rrr}
1 & 1 & b\\
a & \omega a & \omega^2ab\\
1 & \omega^2 & \omega b\\
\end{array}\right].\]
We immediately see that for any unimodular choice of $a$ and $b$ the following equivalence holds
\[\mathrm{Diag}(1,\overline{a},1)H^{(2)}_3(a,b)\mathrm{Diag}(1,1,\overline{b})=F_3.\]
Therefore the seemingly two degrees of freedom does not give us even a single new matrix, being inequivalent from the starting-point matrix $F_3$.
\end{ex}
The ``trivial'' parametrization coming from the multiplication by unitary diagonal matrices should be excluded, as it never leads to anything interesting. For this reason we will consider only families of dephased complex Hadamard matrices. Note here that for any complex Hadamard matrix $H$ there is only a finite number of dephased complex Hadamard matrices $K$ being equivalent to it (see Remark \ref{ch1haafin}). This means, loosely speaking, that if we have a $k$-parameter family of dephased complex Hadamard matrices then it genuinely leads to a $k$-parameter family of inequivalent matrices. Before formalizing these concepts mathematically rigorously, let us recall that $\circ$ denotes the entry-wise (or Hadamard) product of matrices. Furthermore, let us denote by $\mathrm{EXP}(.)$ the entry-wise exponential function acting on matrices. We recall the following
\begin{defi}[Tadej--\.Zyczkowski, \cite{TZ1}]\label{affinedef}
An \emph{affine Hadamard family} $H(V)$ stemming from a starting-point dephased $n\times n$ complex Hadamard matrix $H$, associated with a subspace $V$ of the space of all real $n\times n$ matrices with zeros in the first row and column is the set of matrices $\{H\circ \mathrm{EXP}(\mathbf{i}R) : R\in V\}$. We say that $H(V)$ is a $\mathrm{dim}(V)$-parameter affine family.
\end{defi}
Now let us see how to obtain interesting affine families in composite orders. We have the following fundamental
\begin{thm}[Hosoya--Suzuki, \cite{HS1}]\label{ch1HSc}
Let $M_1, M_2,\hdots, M_v$ be $k\times k$, $N_1,N_2,\hdots, N_k$ be $v\times v$ complex Hadamard matrices. Then the generalized tensor product matrix, denoted by $(M_1,M_2,\hdots,M_v)\otimes (N_1,N_2,\hdots,N_k)$, whose $(i,j)$th block is given by the matrix $\mathrm{Diag}([M_1]_{i,j},[M_2]_{i,j},\hdots,[M_v]_{i,j})N_j$ is a complex Hadamard matrix of order $vk$.
\end{thm}
\begin{cor}[Di\c{t}\u{a}'s construction, \cite{PD1}]\label{ch1ditaconst}
Let $M$ be a $k\times k$ and $N_1,N_2,\hdots, N_k$ be $v\times v$ dephased complex Hadamard matrices with $m$ and $n_1,n_2,\hdots, n_k$ free parameters, respectively. Then the block matrix $M\otimes (N_1,N_2,\hdots,N_k)$ with its $(i,j)$th block given by $[M]_{i,j}N_j$ is a complex Hadamard matrix of order $vk$ with $m+\sum_{i=1}^k n_i+(k-1)(v-1)$ free parameters.
\end{cor}
\begin{defi}[see \cite{MRS}]
We say that a complex Hadamard matrix is of \emph{Di\c{t}\u{a}-type}, if it is equivalent to a matrix arising with Corollary \ref{ch1ditaconst}.
\end{defi}
We offer the following
\begin{ex}\label{ch1exdhs}
Let $k=4, v=2$. Set the $4\times 4$ matrices to $M_1=F_4^{(1)}(a)$ and $M_2=F_4^{(1)}(b)$ while the $2\times 2$ matrices to $N_1=F_2, N_2=\mathrm{Diag}(1,c)F_2$, $N_3=\mathrm{Diag}(1,d)F_2$ and $N_4=\mathrm{Diag}(1,e)F_2$. Then the Hosoya--Suzuki construction leads to
\[H_8^{(5)}(a,b,c,d,e)=\left[
\begin{array}{rr|rr|rr|rr}
 1 & 1 & 1 & 1 & 1 & 1 & 1 & 1 \\
 1 & -1 & c & -c & d & -d & e & -e \\
 \hline
 1 & 1 & 1 & 1 & -1 & -1 & -1 & -1 \\
 1 & -1 & c & -c & -d & d & -e & e \\
 \hline
 1 & 1 & -1 & -1 & a & a & -a & -a \\
 1 & -1 & -c & c & b d & -b d & -b e & b e \\
 \hline
 1 & 1 & -1 & -1 & -a & -a & a & a \\
 1 & -1 & -c & c & -b d & b d & b e & -b e
\end{array}
\right].\]
Note that we obtained a five-parameter family whereas Di\c{t}\u{a}'s construction with $M_1=M_2$ would give us a matrix with $1+\sum_{i=1}^4 0+(2-1)(4-1)=4$ parameters only. However, we can set $k=2$ and $v=4$ and $M_1=M_2=M_3=M_4=F_2$ and $N_1=F_4^{(1)}(a)$, $N_2=\mathrm{Diag}(1,b,c,d)F_4^{(1)}(e)$ to obtain a five-parameter family via Di\c{t}\u{a}'s method as well.
\end{ex}
\begin{rem}
Di\c{t}\u{a}-type complex Hadamard matrices can be recognized algorithmically \cite{MRS}. See also \cite{AC1} for some other useful concepts.
\hfill$\square$\end{rem}
Not every parametric family is affine, and the treatment of non-affine families
is somewhat more technical. Fortunately, as long as they are given explicitly we can detect easily 
how many degrees of freedom they have.
\begin{defi}
Let $H\colon\mathbb{R}^m\to\mathbb{C}^{n^2}$ be a smooth function such that the domain of $H$ is an open neighbourhood of $(0,0,\hdots, 0)$. Assume that $H(x_1,x_2,\hdots,x_m)$ is a dephased complex Hadamard matrix for all choices of $(x_1,x_2,\hdots,x_m)$ in the domain of $H$. We say that $H(x_1,x_2,\hdots,x_m)$ is a {\it $k$-parameter family of complex Hadamard matrices}, stemming from $H(0,0,\hdots,0)$, if the rank of the Jacobian matrix $J_H(x_1, x_2, \hdots, x_m)$ is constant $k$ on the domain of $H$. Note that here we have identified $\mathcal{M}_n(\mathbb{C})$ with $\mathbb{C}^{n^2}$.
\end{defi}
It is clear that affine families are a special case of this more general concept. The reader might wish to jump ahead to Section \ref{ch2saf} to see an example of a non-affine family.

If we are given a complex Hadamard matrix $H$ we would like to know whether any smooth family is stemming from it. This is hard to decide in general, but we can give an upper bound on the dimension of any such potential family simply by differentiating the orthogonality relations between the rows. This leads to the following concept:
\begin{defi}[Tadej--\.Zyczkowski, \cite{TZ1}]
The \emph{defect} $\mathrm{d}(H)$ of an $n\times n$ complex Hadamard matrix $H$ reads $\mathrm{d}(H)=m-2n+1$ where $m$ is the dimension of the solution space of the real linear system with respect to the matrix variable $R\in \mathbb{R}^{n\times n}$:
\beql{ch1dss}
\sum_{k=1}^{n}H_{i,k}\overline{H}_{j,k}(R_{i,k}-R_{j,k})=0,\ \ \ 1\leq i<j\leq n.
\eeq
\end{defi}
Let us note here that the defect is well-defined in the following sense.
\begin{lem}[Tadej--\.Zyczkowski, \cite{TZ2}]
The defect is invariant under the usual equivalence. Moreover $\mathrm{d}\left(H\right)=\mathrm{d}\left(\overline{H}\right)=\mathrm{d}\left(H^T\right)=\mathrm{d}\left(H^\ast\right)$.
\end{lem}
We have the following fundamental
\begin{prop}[Tadej--\.Zyczkowski, \cite{TZ1}]\label{karoldefzero}
Let $H$ be a complex Hadamard matrix. Then $\mathrm{d}\left(H\right)$ gives an upper bound on the dimension of a smooth manifold of complex Hadamard matrices, stemming from $H$.
\end{prop}
It is easy to see that if a phasing matrix $R$ satisfies \eqref{ch1dss} then the matrix $K:=H\circ\mathrm{EXP}(\mathbf{i}R)$ is complex Hadamard to first order (i.e.\ we get a complex Hadamard matrix numerically if $R$ is very close to the zero matrix), \cite{DEBZ}.
\begin{ex}
For the real Hadamard matrix $H_4=F_2\otimes F_2$ we have $\mathrm{d}(H_4)=3$. In contrast, Example \ref{ch1ExF4i} describes a one-parameter family only. Later we shall see that this cannot be improved (see Proposition \ref{ch3pRC}), and hence the defect can be strictly larger than the dimension of any smooth manifold, stemming from $H_4$.
\end{ex}
The defect of the Fourier matrices is known explicitly, and is given by the following
\begin{prop}[S\l omczy\'nski's formula, \cite{TZ2}]
Let $n=p_1^{\alpha_1}p_2^{\alpha_2}\cdots p_r^{\alpha_r}\geq 2$ be a natural number with the factorization into distinct prime numbers $p_i,i=1,\hdots,r$. Then
\beql{ch1tadejf}
\mathrm{d}(F_n)=n\prod_{i=1}^{r}\left(1+\alpha_i-\frac{\alpha_i}{p_i}-2\right)+1.
\eeq
\end{prop}
While Theorem \ref{ch1HSc} provides access to parametric families of complex Hadamard matrices, it does not reveal how to introduce parameters into a given matrix. This question was investigated for the Fourier matrices in \cite{TZ2}, where an explicit construction was given to introduce $p^{k-1}\left(k(p-1)-p\right)+1$ affine parameters into $F_{p^k}$ thus demonstrating that the absolute upper bound in \eqref{ch1tadejf} is sharp.

In what follows we describe various parametrization schemes.

There is a fairly general method, called ``linear variation of phases'' \cite{TZ1} allowing one to find affine parametric families stemming from complex Hadamard matrices with certain additional symmetries; however the method becomes computationally expensive for higher order matrices, or for those lacking the required symmetries. Nevertheless the maximal affine families, stemming from the Fourier matrices has been obtained this way \cite{TZ2}, and the method can also be used to introduce additional degree of freedom into existing parametric families \cite{MRS}.

We investigated real Hadamard matrices in \cite{SZF1}, where the following simple way to introduce free parameters into complex Hadamard matrices of even orders $n\geq 4$ was mentioned.
\begin{lem}\label{ch1trivi}
Let $H$ be a dephased complex Hadamard matrix of even order $n\geq 4$ and suppose that there exists a pair of rows in $H$, say $u$ and $v$, such that for every $i=1,2,\hdots,n$ $u_i^2=v_i^2$. Then for all such $i$ for which $u_i+v_i=0$ replace $u_i$ with $\alpha u_i$ and $v_i$ with $\alpha v_i$ where $\alpha$ is a unimodular complex number to obtain a one-parameter family of complex Hadamard matrices $H(\alpha)$.
\end{lem}
\begin{proof}
Indeed, $H(\alpha)$ is unimodular and it is easy to see that the assumptions guarantee that the columns of it are pairwise orthogonal. Note that one might need to dephase the matrix in order to obtain an affine family, but this operation does not remove the parameters from it as $n\geq 4$.
\end{proof}
Let us denote by $\left\langle .,.\right\rangle$ the standard inner product in $\mathbb{C}^n$ with the usual convention that it is linear in the first, and conjugate-linear in the second variable. Additionally, let us denote by $1^n$ the all-$1$ vector of length $n$. The following is a more robust method obtaining parametric families.
\begin{thm}[see \cite{LSZO}]\label{ch1newp}
Let $H$ be a dephased complex Hadamard matrix with the following block structure
\[H=\left[\begin{array}{ccccc}
1 & 1 & 1 & 1^p & 1^q\\
1 & a & b & x & y\\
1 & b & a & x & -y\\
(1^r)^T & z^T & z^T & A & B\\
(1^s)^T & w^T & -w^T & C & D\\
\end{array}\right],\]
where $a$ and $b$ are arbitrary unimodular numbers. Then, after replacing the vectors $y$ with $\alpha y$ and $w$ with $\overline{\alpha} w$ we obtain a one-parameter family of complex Hadamard matrices $H(\alpha)$. If, in addition, $b=a$ then we can replace $w$ one more time with $\alpha\beta w$ to obtain a two-parameter family of complex Hadamard matrices $H(\alpha,\beta)$.
\end{thm}
\begin{proof}
We need to show that the rows of $H(\alpha,\beta)$ are pairwise orthogonal. It is immediate that
\[\left\langle y, 1^q\right\rangle=0,\]
and hence the first three rows of $H(\alpha,\beta)$ are pairwise orthogonal. Similarly, it is easily seen that the rest of the rows (beyond the first three) are pairwise orthogonal within themselves. Additionally, the first row is trivially orthogonal to all further rows, as $w_i$ cancels out from the orthogonality equations for $i=1,\hdots, s$. Therefore it remains to be seen that the second and third rows of $H(\alpha,\beta)$ are orthogonal to all additional one. We show first that they are orthogonal to rows which are type $[1, z_i, z_i, A_i, B_i]$, $i=1,\hdots, r$. In the original matrix $H$ (i.e.\ prior to parametrizing) we have
\begin{gather}\label{ch3pars1}
1+z_i(\overline{a}+\overline{b})+\left\langle A_i,x\right\rangle+\left\langle B_i, y\right\rangle=0,\\
\label{ch3pars2}
1+z_i(\overline{a}+\overline{b})+\left\langle A_i,x\right\rangle-\left\langle B_i, y\right\rangle=0,
\end{gather}
and hence $\left\langle B_i,y\right\rangle=0$ for every $i=1,\hdots,r$. It follows that after parametrization equations \eqref{ch3pars1}-\eqref{ch3pars2} remain valid. We proceed by proving that rows which are type $[1, w_i, -w_i, C_i, D_i]$, $i=1,\hdots, s$, after parametrization, are orthogonal to the second and third row of $H(\alpha,\beta)$. Again, in the original matrix $H$ we have
\begin{gather}\label{ch3pars3}
1+w_i\overline{a}-w_i\overline{b}+\left\langle C_i,x\right\rangle+\left\langle D_i, y\right\rangle=0,\\
\label{ch3pars4}
1-w_i\overline{a}+w_i\overline{b}+\left\langle C_i,x\right\rangle-\left\langle D_i, y\right\rangle=0,
\end{gather}
and hence $\left\langle C_i,x\right\rangle=-1$ for every $i=1,\hdots, s$. It follows, that \eqref{ch3pars3}-\eqref{ch3pars4} are valid, provided that
\beql{ch3pars5}
w_i\overline{a}-w_i\overline{b}+\left\langle D_i, y\right\rangle=0,
\eeq
and therefore \eqref{ch3pars3}-\eqref{ch3pars4} remains true, after parametrization. If, in addition, $b=a$, then $\left\langle D_i, y\right\rangle=0$ for every $i=1,\hdots, s$ and hence \eqref{ch3pars5} holds, independently of the scalar factor in $w$.
\end{proof}
\begin{rem}
If $a$ and $b$ are as in Theorem \ref{ch1newp}, then it is easy to see that the expression $2\Re[a\overline{b}]$ is an integral number and therefore $b\in \pm a\cdot\{1,\omega,\omega^2,\mathbf{i}\}$. It follows that the parametrizing scheme described is ``natural'' for complex Hadamard matrices with fourth and sixth roots of unity.
\hfill$\square$\end{rem}
\begin{rem}
Theorem \ref{ch1newp} describes a local property of the complex Hadamard matrix $H$. Its conditions can be fairly easily checked, even by hand, and the method can be implemented as a computer program to construct infinite families automatically. This has been done in \cite{LSZO}, where complex Hadamard matrices, composed of fourth roots of unity were investigated.
\hfill$\square$\end{rem}
Nicoar\u{a} in a series of papers \cite{RN1}, \cite{RN2} considered commuting squares of matrix algebras \cite{SP1}, and using operator theory obtained two interesting parametrization schemes. His results also offered a conceptual proof of the existence of certain sporadic examples of parametric families of complex Hadamard matrices. Before turning to his results let us introduce some notations first. We denote by $\mathbb{C}I_n$ and $\mathcal{D}_n$ the algebra of the scalar and diagonal matrices, respectively. Further we write $\mathrm{exp}(A)$ for the matrix exponential of $A$. Note that this is different from the entrywise exponential function $\mathrm{EXP(.)}$ we used earlier. Additionally, the symbol $\setminus$ denotes the set-theoretical exclusion, while $[A,B]:=AB-BA$ denotes the commutator of the matrices $A$ and $B$.
\begin{thm}[Nicoar\u{a}, \cite{RN1}]\label{ch1N1}
Let $H$ be a complex Hadamard matrix of order $n$, and assume that $A\in\mathcal{D}_n\setminus\mathbb{C}I_n$, $B\in H^\ast\mathcal{D}_nH\setminus\mathbb{C}I_n$ are self-adjoint, commuting matrices $($i.e.\ $[A,B]=AB-BA=0)$. Then, for any $t\in\mathbb{R}$, $U(t)=\mathrm{exp}(\mathbf{i}tAB)$ is a one-parameter family of unitary matrices, such that $K_n^{(1)}(t)=HU(t)$ is a one-parameter family of complex Hadamard matrices.
\end{thm}
\begin{ex}
Let $H=F_2\otimes F_2$, $A=\mathrm{Diag}(0,0,1,1)$, $B=H^\ast AH/4$. Then $A$ and $B$ are self-adjoint, commuting matrices. Therefore the conditions of Theorem \ref{ch1N1} are met, and as a result
\[U(t)=\left[
\begin{array}{cccc}
 1 & 0 & 0 & 0 \\
 0 & 1 & 0 & 0 \\
 0 & 0 & \frac{1+\mathbf{e}^{\mathbf{i} t}}{2} & \frac{1-\mathbf{e}^{\mathbf{i} t}}{2} \\
 0 & 0 & \frac{1-\mathbf{e}^{\mathbf{i} t}}{2} & \frac{1+\mathbf{e}^{\mathbf{i} t}}{2}
\end{array}
\right], \text{ such that } K_4^{(1)}(t)=HU(t)=\left[
\begin{array}{rrrr}
 1 & 1 & 1 & 1 \\
 1 & 1 & -1 & -1 \\
 1 & -1 & \mathbf{e}^{\mathbf{i} t} & -\mathbf{e}^{\mathbf{i} t} \\
 1 & -1 & -\mathbf{e}^{\mathbf{i} t} & \mathbf{e}^{\mathbf{i} t}
\end{array}
\right]\]
is a one-parameter family of complex Hadamard matrices.
\end{ex}
\begin{thm}[Nicoar\u{a}, \cite{RN1}]\label{ch1N2}
Let $H$ be a complex Hadamard matrix of order $n$, and assume that $A_1,A_2\in\mathcal{D}_n\setminus\mathbb{C}I_n$, $B_1,B_2\in H^\ast\mathcal{D}_nH\setminus\mathbb{C}I_n$ are orthogonal projections such that $[A_1,B_1]=[A_2,B_2]$. Then, for any $a\in\mathbb{T}$, $U(a)=I_n+(a-1)A_1B_1+(\overline{a}-1)A_2B_2$ is a one-parameter family of unitary matrices, such that $K_n^{(1)}(a)=HU^\ast(a)$ is a one-parameter family of complex Hadamard matrices.
\end{thm}
\begin{ex}
Let $H=F_2\otimes F_2\otimes F_2$, $A_1=\mathrm{Diag}(0,0,1,0,1,0,0,0)$, while $A_2=\mathrm{Diag}(0,0,0,1,0,1,0,0)$, $B_1=H^\ast A_1H/8$, $B_2=H^\ast A_2H/8$. It is clear that the matrices $A_1$, $A_2$, $B_1$, $B_2$ are self-adjoint projections, moreover $A_1A_2=B_1B_2=0$, $[A_1,B_1]=[A_2,B_2]$, and hence the conditions of Theorem \ref{ch1N2} are met resulting in the one-parameter family of complex Hadamard matrices:
\[K_8^{(1)}(a)=\left[
\begin{array}{rrrrrrrr}
 1 & 1 & 1 & 1 & 1 & 1 & 1 & 1 \\
 1 & -1 & 1 & -1 & 1 & -1 & 1 & -1 \\
 1 & 1 & \overline{a} & 1 & -\overline{a} & -1 & -1 & -1 \\
 1 & -1 & 1 & -a & -1 & a & -1 & 1 \\
 1 & 1 & -\overline{a} & -1 & \overline{a} & 1 & -1 & -1 \\
 1 & -1 & -1 & a & 1 & -a & -1 & 1 \\
 1 & 1 & -1 & -1 & -1 & -1 & 1 & 1 \\
 1 & -1 & -1 & 1 & -1 & 1 & 1 & -1
\end{array}
\right].\]
\end{ex}
Now we discuss some of the consequences of these results. From Lemma \ref{ch1trivi} it follows that every real Hadamard matrix of order $n\geq 4$ admits an $n/2-1$-parameter affine orbit. Actually, we have shown a little more in \cite{SZF1}:
\begin{prop}[see \cite{SZF1}]\label{probx1}
Every real Hadamard matrix of order $n\geq 12$ admits an $n/2+1$-parameter affine orbit.
\end{prop}
However, if there are real Hadamard matrices of orders $n_1$ and $n_2$, then one can do way better in orders $n_1n_2$ via Di\c{t}\u{a}'s construction, improving the coefficient from $1/2$ to $1$ as follows:
\begin{cor}\label{ch1c1}
Let $H$ and $K$ be real Hadamard matrices of order $n_1$ and $n_2$. Then there is a family of complex Hadamard matrices of order $n_1n_2$, with $(n_1-1)(n_2-1)$ affine parameters, stemming from a real Di\c{t}\u{a}-type Hadamard matrix.
\end{cor}
In light of this, it is natural to pose the following
\begin{pro}
Improve on Proposition \ref{probx1}.
\end{pro}
\begin{rem}
The parametrization highlighted in Proposition \ref{probx1} (essentially) introduces parameters into distinct pairs of rows of a real Hadamard matrix via Lemma \ref{ch1trivi}. If one could introduce parameters into the rows and columns of the matrix simultaneously, then one might reach the desired improvement and obtain similar results as described in Corollary \ref{ch1c1}.
\hfill$\square$\end{rem}
Finally, it is natural to ask if real Hadamard matrices can be obtained from parametric families of complex Hadamard matrices. For example, one might wonder if a real Hadamard matrix of order $668$ can be obtained from the Fourier matrix $F_{668}$, which is parametrized via Di\c{t}\u{a}'s construction. Unfortunately, real Hadamard matrices of order $4p$ are not of Di\c{t}\u{a}-type and therefore they cannot be obtained this way \cite{SZF1}. It is, however, unclear whether some clever parametrization scheme would yield such a matrix. We explicitly ask the following ``baby case'' of this
\begin{pro}\label{ch3pxx}
Investigate if the Fourier matrix $F_{12}$ can be parametrized in a way such that its orbit contains the real Hadamard matrix $H_{12}$.
\end{pro}
A positive, constructive resolution of Problem \ref{ch3pxx} might lead to the construction of new, previously intractable real Hadamard matrices of higher orders.

It would be nice to connect Theorem \ref{ch1newp} to Theorems \ref{ch1N1} and \ref{ch1N2} and check if it is meaningful in the more general operator algebraic context. We have the following
\begin{pro}
Investigate if Lemma \ref{ch1trivi} and Theorem \ref{ch1newp} can be reformulated in an operator algebraic fashion similarly to Theorems \ref{ch1N1} and \ref{ch1N2}.
\end{pro}
There are complex Hadamard matrices which cannot be parametrized at all. The relevant concept is what follows.
\begin{defi}[Tadej--\.Zyczkowski, \cite{TZ2}]\label{ch1defiso}
A complex Hadamard matrix $H$ of order $n$ is \emph{isolated}, if around a small neighborhood of $H$ all complex Hadamard matrices $K$ are equivalent to it.
\end{defi}
One might wonder why we allow other (yet equivalent) Hadamard matrices to exist around a neighborhood of $H$. The reason is quite simple: every complex Hadamard matrix can be transformed continuously to some other via multiplication by unitary diagonal matrices, which however is uninteresting as it never leads to new, inequivalent matrices (see Example \ref{ch1ExF4iC}). It turns out, that the defect provides a useful criterion detecting isolated matrices. In particular, from Proposition \ref{karoldefzero} it follows that no smooth families can be obtained from a matrix whose defect is zero. This, however does not say anything about the existence of non-smooth families (e.g.\ a sequence of inequivalent matrices converging to $H$). The following stronger result excludes this possibility.
\begin{prop}[Tadej--\.Zyczkowski, \cite{TZ2}]\label{ch1tzdefzeroi}
Suppose that the defect of a complex Hadamard matrix $H$ of order $n$ is zero. Then $H$ is isolated amongst all $n\times n$ complex Hadamard matrices.
\end{prop}
Nicoar\u{a} found an equivalent criterion (see \cite{TZ2}) what he calls the ``span condition'' \cite{RN1}. Here we recall his result in a form relevant to us. We denote by $[\mathcal{A},\mathcal{B}]$ the linear span of the commutator of the matrix algebras $\mathcal{A},\mathcal{B}\subset \mathcal{M}_{n}(\mathbb{C})$, namely
\[[\mathcal{A},\mathcal{B}]=\mathrm{span}\{AB-BA : A\in\mathcal{A},B\in\mathcal{B}\}.\]
\begin{prop}[Span condition, \cite{RN1}]
If $H$ is a complex Hadamard matrix of order $n$ such that the dimension of the vector space $[\mathcal{D}_n,H^\ast \mathcal{D}_nH]$ is $n^2-2n+1$ then $H$ is isolated among all complex Hadamard matrices of order $n$ $($up to equivalence$)$.
\end{prop}
It is easy to see that the matrices $F_1$ and $F_2$ satisfy the span condition (or equivalently, their defect is $0$), thus from Proposition \ref{probx1} it follows immediately that
\begin{cor}
A real Hadamard matrix is isolated if and only if it is equivalent to $F_1$ or $F_2$.
\end{cor}
For results pertaining the defect of unitary matrices see e.g.\ \cite{AK1}, \cite{WT2}, \cite{TZ2}.
\section{Equivalence of complex Hadamard matrices}
The existence of parametric families of complex Hadamard matrices in a given order $n$ implies the existence of infinitely many inequivalent matrices \cite{RC3}. This can be readily seen through the following invariant, introduced by Haagerup \cite{UH1}:
\begin{defi}
The \emph{Haagerup-set} of a complex Hadamard matrix $H$ of order $n$ is the set
\[\Lambda(H):=\left\{h_{ij}h_{kl}\overline{h}_{il}\overline{h}_{kj} : i,j,k,l=1,\hdots, n\right\}.\]
\end{defi}
The following is immediate.
\begin{lem}
The Haagerup-set is invariant under the usual equivalence.
\end{lem}
\begin{rem}\label{ch1haafin}
Informally speaking, if two Hadamard matrices $H$ and $K$ have ``essentially different'' set of entries, they cannot be equivalent. In particular, it follows from the Haagerup invariant that for any complex Hadamard matrix $H$ there is only a finite number of dephased complex Hadamard matrices $K$ being equivalent to it. Indeed, any entry of $K$ must be an element of $\Lambda(H)$.
\hfill$\square$\end{rem}
We illustrate an application of this this invariant in the following
\begin{lem}
The matrices $F_2\otimes F_2$ and $F_4$ are inequivalent $($cf.\ Example \ref{ch1E4}$)$.
\end{lem}
\begin{proof}
Indeed, as $\Lambda(F_2\otimes F_2)=\{\pm1\}$, and in particular, it does not contain the fourth root of unity $\mathbf{i}$, whereas $\Lambda(F_4)=\{\pm1,\pm\mathbf{i}\}$.
\end{proof}
However, it is known that there are $5$ inequivalent real Hadamard matrices of order $16$ \cite{MH1}, all of which sharing the same Haagerup set, therefore they cannot be distinguished by this concept. There are a number of invariants introduced to detect the inequivalence of real Hadamard matrices, including the $4$-profile \cite{CMW} or the Smith Normal Form \cite{MW}, however they cannot be translated directly to the complex case due to fundamental or practical reasons. We have introduced a new invariant in \cite{SZF2} to improve on this situation.
\begin{defi}[see \cite{SZF2}]
The \emph{fingerprint} of a complex Hadamard matrix $H$ of order $n$ is the following ordered set
\[\Phi(H):=\left\{\{(v_i(d),m_i(d)) : i\in I(d)\}_{d} : d=2,\hdots,\left\lfloor n/2\right\rfloor\right\},\]
where for every $2\leq d\leq\left\lfloor n/2\right\rfloor$ $I(d)$ is an index set, and $v_i(d)$ and $m_i(d)$ are the possible values of the moduli of the $d\times d$ minors and its multiplicities, respectively.
\end{defi}
We give here a quick
\begin{ex}
The fingerprint of the $8\times 8$ real Hadamard matrix $H:=F_2\otimes F_2\otimes F_2$ reads
\[\Phi(H)=\{\{(0, 336), (2, 448)\}_{2} , \{(0, 1344), (4, 1792)\}_{3}, \{(0, 1428), (8, 3136), (16, 336)\}_{4}\},\]
meaning that there are $336$ $2\times 2$ minors of value $0$ and $448$ minors of absolute value $2$ in $H$, etc.
\end{ex}
We remark here that the distribution of the absolute values of the minors of size greater than $\left\lfloor n/2\right\rfloor$ are not considered in the fingerprint because in unitary matrices any minor and its cofactor share the same absolute value \cite{SZF2}. This harmless observation lead to a surprising disproval of a natural conjecture of Koukovinos et al.\ on the distribution of minors of real Hadamard matrices \cite{KLMS}.

Compared to the Haagerup-set which considers $2\times 2$ submatrices only, the fingerprint aims to capture some of the global properties of a complex Hadamard matrix. One of the limitations of the fingerprint is its computational complexity, however in practice one is free to use some subset of $\Phi$, say the distribution of minors up to order $3$ or just the number of $4\times 4$ vanishing minors, etc., providing limited but hopefully enough information already. Another weak point of the concept is that it cannot distinguish a matrix from its transpose, conjugate or adjoint. Separating a matrix from its transpose, however, tends to be more important in the real case where the underlying design is studied. We point out a natural way to distinguish a complex Hadamard matrix from its transpose in the following
\begin{defi}
The \emph{rectangular rank profile} of a complex Hadamard matrix $H$ is the following ordered set
\[\mathcal{R}(H)=\{\{(r_i(j,k),m_i(j,k)) : i\in I(j,k)\}_{j\times k} : 2\leq j,k\leq n-2\}\},\]
where for every $2\leq j,k\leq n-2$ $I(j,k)$ is an index set, and $r_i(j,k)$ and $m_i(j,k)$ are the possible values of the rank of the $j\times k$ submatrices of the matrix $H$ and their multiplicities, respectively.
\end{defi}
\begin{lem}
The rectangular rank profile is invariant, up to equivalence.
\end{lem}
\begin{proof}
The rank of a rectangular matrix can be calculated by considering the size of its largest nonsingular minor, which however is preserved during the equivalence operations.
\end{proof}
\begin{ex}
The rectangular rank profile of the matrices $F_2\otimes F_2$ and $F_4$ (see Example \ref{ch1E4}) read:
\[\mathcal{R}(F_2\otimes F_2)=\{\{(1,12),(2,24)\}_{2\times2}\},\ \text{ and }\ \mathcal{R}(F_4)=\{\{(1,4),(2,32)\}_{2\times 2}\},\]
respectively, meaning that in the matrix $F_2\otimes F_2$ the number of $2\times 2$ submatrices with rank $1$ and $2$ are $12$ and $24$, respectively, while in $F_4$ the number of $2\times 2$ submatrices with rank $1$ is $4$ only.
\end{ex}
We tend to believe that in a ``typical'' complex Hadamard matrix there are no vanishing minors at all, and as a result their rectangular rank profile is trivial. Nevertheless, when we consider highly structured matrices, such as real Hadamard matrices, it can be a valuable asset at our disposal. We shall see further applications of these invariants in the subsequent sections.
\section{Butson-type complex Hadamard matrices}
This section is dedicated to investigate the existence and structure of the following concept.
\begin{defi}\label{ch1butsd}
We say that a complex Hadamard matrix $H$ of order $n$ is of \emph{Butson-type}, if $H$ is composed from some $q$th roots of unity. We denote this class of matrices by $BH(n,q)$.
\end{defi}
We advise the reader that the notation varies from author to author, some of them permuting the argument to $BH(q,n)$, or dropping the letter $B$ resulting in the notations $H(n,q)$ and $H(q,n)$, respectively.

We primarily focus on $BH(n,4)$ and $BH(n,6)$ matrices in this section, as the cases $q=4$ and $q=6$ are the two simplest generalization of the real Hadamard matrices, i.e.\ when $q=2$, namely they correspond to the prime-power and composite case, respectively. Clearly, Lemma \ref{ch1L117} applies to $BH(n,q)$ matrices as well, resulting in the following concept.
\begin{defi}
We say that a $k$-term sum $x_1+x_2+\hdots+x_k$ is a \emph{vanishing sum of order $k$} if its value is $0$.
\end{defi}
Now in order to find a pair of orthogonal rows of a $BH(n,q)$ matrix one needs to exhibit a vanishing sum of order $n$, composed from $q$th root of unity. It is natural to ask when this is possible. The following resolution of this problem is a necessary condition on the existence of $BH(n,q)$ matrices
\begin{thm}[Lam--Leung, \cite{LL}]\label{ch1thmcs}
For every $q=p_1^{\alpha_1}p_2^{\alpha_2}\cdots p_r^{\alpha_r}$ there is a vanishing sum of order $n$, composed from $q$th roots of unity if and only if $n$ is a natural linear combination of the numbers $p_i$, i.e.\ $n\in\mathbb{N}p_1+\mathbb{N}p_2+\hdots+\mathbb{N}p_r$.
\end{thm}
For a more general treatment of this question see \cite{CS}.

Theorem \ref{ch1thmcs} above, however, does not say anything about the existence of orthogonal triplets of rows in complex Hadamard matrices. For example, for numbers $n\equiv 2$ (mod $4$) one easily constructs a vanishing sum of order $n$ composed from the numbers $\pm1$, however, it follows from (the ideas behind) Lemma \ref{ch1lde} that it is impossible to exhibit orthogonal triplets of real rows in this case.

Let us denote by $\left(\frac{p}{q}\right)$ the Legendre symbol. We collect some well-known necessary conditions and the subsequent nonexistence results in the following
\begin{thm}[see \cite{AB1}, \cite{AW}]\label{ch1BN}
For a $BH(n,q)$ matrix the following hold.
\begin{enumerate}
\item If $p$ is prime and there is a $BH(n,p^a)$, then $n=pt$ for some positive integer $t$.
\item If $p$ is prime, there exists a $BH(2^jp^k,p)$ for all $0\leq j \leq k$.
\item Suppose $p\equiv 3$ $(\mathrm{mod}\ 4)$ is prime, $n=p^br^2s$ is odd, $s$ is square-free with $(s,p)=1$, and there exists a prime $q|s$ with quadratic character value $\left(\frac{q}{p}\right)=-1$. Then there is no $BH(n,p^a)$ and no $BH(n,2p^a)$.
\end{enumerate} 
\end{thm}
Butson's basic construction yields a $BH(2p,p)$ matrix for prime $p$; it is known that there are $BH(12,3)$ \cite{LO} and $BH(28,7)$ matrices \cite{EM}. It is natural to ask the following
\begin{pro}
Decide for which primes $p$ are there $BH(4p,p)$ matrices.
\end{pro}
Setting $p=3, n=q=s=5$ in Part $3$ above we arrive to the following
\begin{cor}\label{ch1CCC}
There is no $BH(5,6)$ matrix.
\end{cor}
Observe that we cannot apply Theorem \ref{ch1thmcs} directly, as $n=2+3$, and indeed, it is easy to exhibit a pair of orthogonal rows from sixth roots of unity. We shall recall a direct argument for Corollary \ref{ch1CCC}, but first let us see the following
\begin{lem}\label{ch1det}
For a complex Hadamard matrix $H$ of order $n$ we have $|\det(H)|=n^{n/2}$.
\end{lem}
\begin{proof}
Indeed, as the determinant is multiplicative, we have
\[n^n=\det(nI_n)=\det(HH^\ast)=|\det(H)|^2.\qedhere\]
\end{proof}
It is easy to see that complex Hadamard matrices have the largest determinant amongst all matrices whose entries are at most $1$ in absolute value \cite{SSA1}.

The following is attributed to de Launey (see \cite{KJH1}).
\begin{proof}[Proof of Corollary \ref{ch1CCC}]
We use Lemma \ref{ch1det} and try to reach a contradiction. Suppose to the contrary, that there is a $BH(5,6)$ matrix $H$. Then observe that its determinant is a sum of sixth roots of unity and hence there are integral numbers $A$ and $B$, such that $\det(H)=A+B\omega$. Therefore we find that
\[5^5=|\det(H)|^2=|A+B\omega|^2=A^2+B^2+AB(\omega+\omega^2)=A^2+B^2-AB,\]
the right hand side, however, cannot be $2$ modulo $3$, a contradiction.
\end{proof}
For a class of $BH(n,q)$ matrices some further non-existence results shall be obtained in Section \ref{ch3secp}.
Finally we emphasize that $BH(n,q)$ matrices can be quickly recognized.
\begin{lem}\label{ch1LRO}
Suppose that a complex Hadamard matrix $H$ is equivalent to a $BH(n,q)$ matrix. Then the dephased form of $H$ must be composed of purely complex $q$th roots of unity.
\end{lem}
\begin{proof}
Indeed, if the dephased form of $H$ would contain an entry which is not a root of unity, then the Haagerup invariant set would contain this entry as well. However, the Haagerup set of $BH(n,q)$ matrices can contain some subset of the $q$th roots of unity only.
\end{proof}
It follows that checking the number of inequivalent $BH(n,q)$ matrices in dephased parametric families is a finite process. We shall repeatedly use this fact in the next subsection.
\begin{defi}
The \emph{automorphism group} of a $BH(n,q)$ matrix $H$ is the group of pairs of monomial matrices $(P,Q)$, such that $H=PHQ$; a monomial matrix is an $n\times n$ matrix having a single nonzero entry in each row and column, these nonzero entries being complex $q$th roots of unity.
\end{defi}
Note that the automorphism group depends on the choice of $q$, i.e.\ on the ambient space in which the $BH(n,q)$ matrix is considered.
\subsection{\texorpdfstring{A course on $BH(n,4)$ matrices}{A course on BH(n,4) matrices}}
$BH(n,4)$ matrices are probably the most natural examples of complex Hadamard matrices. In fact, in various earlier literature the term ``complex Hadamard matrix'' was used to refer to them exclusively, and the term ``unit Hadamard'' was proposed for the more general objects we are concerned with \cite{RC3}. There are a vast number of papers available on $BH(n,4)$ matrices both in the mathematics and in the engineering literature, as they have many industrial applications, most notably in communications and coding theory \cite{HLT} through complex Golay sequences \cite{CHK}. For further applications we refer the reader to the manuals \cite{SSA1} and \cite{KJH1}.

In this brief section we investigate $BH(n,4)$ matrices of small orders, and give a full classification of them up to order $8$. We believe that this result is already known, but we were unable to find any relevant reference in this respect. Nevertheless, the presentation of these matrices in terms of infinite, affine families is indeed new, which is our contribution to this section.

The following is immediate from Theorem \ref{ch1thmcs}.
\begin{lem}
If a $BH(n,4)$ matrix exists then $n=1$ or $n$ is even.
\end{lem}
Similarly to the real case, it seems that this trivial restriction is the only one, as we have the following long-standing
\begin{con}[Seberry, cf.\ \cite{CRC1}]\label{ch1scon}
For every even $n$ there is a $BH(n,4)$ matrix.
\end{con}
It turns out that Conjecture \ref{ch1scon} implies the Hadamard conjecture thus the existence of $BH(n,4)$ matrices is deeply entwined with the existence of real Hadamard matrices (see \cite{KJH1}). We have the following folklore
\begin{lem}
If a $BH(n,4)$ matrix exists then so does a $BH(2n,2)$, that is a real Hadamard matrix of twice the size.
\end{lem}
\begin{proof}
We can lift the starting point $BH(n,4)$ matrix via the function $\varphi_B(.)$ which assigns to the entries $\pm1$ and $\pm\mathbf{i}$ the $2\times 2$ real Hadamard matrices
\[\varphi_B(\pm1)=\pm\left[\begin{array}{rr}
1 & 1\\
1 & -1\\
\end{array}\right],\ \ \ \ \ \varphi_B(\pm\mathbf{i})=\pm\left[\begin{array}{rr}
-1 & 1\\
1 & 1\\
\end{array}\right],\]
where the $\pm$ signs on the left and right hand side of the maps agree.
\end{proof}
\begin{ex}
We can lift the $BH(2,4)$ matrix $\mathrm{Diag}(1,\mathbf{i})F_2\mathrm{Diag}(1,\mathbf{i})$ as follows:
\[\left[\begin{array}{cc}
1 & \mathbf{i}\\
\mathbf{i} & 1\\
\end{array}\right]\leadsto
\left[\begin{array}{rr|rr}
1 & 1 & -1 & 1\\
1 & -1 & 1 & 1\\
\hline
-1 & 1 & 1 & 1\\
1 & 1 & 1 & -1\\
\end{array}\right].\]
\end{ex}
Let us stop for a moment and observe that both matrices featuring in the example above are equivalent to a circulant one. It is natural then to try to obtain circulant $BH(n,4)$ and $BH(n,2)$ matrices. However, we have the following discouraging
\begin{con}[Ryser, \cite{HJR}]\label{ch1ry}
There is no circulant real Hadamard matrix for $n>4$.
\end{con}
The fact that Conjecture \ref{ch1ry} is still open shows how little understanding we have about real Hadamard matrices. For the $BH(n,4)$ case we recall from \cite{ALM} the following
\begin{ex}[Arasu et al.\ \cite{ALM}]
The circulant matrices, induced by the first rows
\[(1,\mathbf{i});\ \ (1,-\mathbf{i},1,\mathbf{i});\ \ (1,1,\mathbf{i},1,1,-1,\mathbf{i},-1);\ \ (1,1,\mathbf{i},-\mathbf{i},\mathbf{i},1,1,\mathbf{i},-1,1,-\mathbf{i},-\mathbf{i},-\mathbf{i},1,-1,\mathbf{i})\]
are $BH(n,4)$ matrices for order $n=2,4,8$ and $16$, respectively.
\end{ex}
\begin{con}\label{circb}
There is no circulant $BH(n,4)$ matrix for $n>16$.
\end{con}
The following is a strong evidence, supporting the truth of Conjecture \ref{circb}:
\begin{thm}[Arasu et al.\ \cite{ALM}, cf.\ \cite{SCH1}]
Let $P$ be any finite set of primes. Then there are only finitely many circulant $BH(n,4)$ matrices of order $n$, where all prime divisors of $n$ are in $P$.
\end{thm}
So it seems that ``nice'' circulant complex Hadamard matrices are rare. This is not quite the case if one allows higher roots of unity or not necessarily root of unity entries to appear in the matrix. We will return to circulant matrices in Section \ref{ch3secC}.

It turns out, however, that $BH(n,4)$ matrices with circulant core exist for infinitely many orders $n$. The following is the complex counterpart of Theorem \ref{ch1rpc1}.
\begin{thm}[see e.g.\ \cite{KJH1}]\label{ch1pav2}
Let $p$ be a prime number, $p\equiv 1$ $(\mathrm{mod}\ 4)$, and consider a vector $x$ such that $x_0=0$ and $x_i=1$ if $i$ is a quadratic residue in $\mathbb{Z}_p^\ast$, otherwise $x_i=-1$. Let us define the Paley matrix $P:=\mathrm{Circ}(x)$. Then, the matrix $C:=\mathbf{i}P-I$, bordered by a row and column of numbers $1$ is a $BH(p+1,4)$ matrix.
\end{thm}

Now we turn to the parametrization of $BH(n,4)$ matrices, and investigate whether we have an analogue, more-or-less trivial way to introduce affine parameters into them similarly to the real case (see Lemma \ref{ch1trivi}) for large enough $n$. It is immediate, that Lemma \ref{ch1trivi} applies as long as a $BH(n,4)$ matrix features two purely real rows or columns. Not every $BH(n,4)$ matrix has this type of structure though, as Theorem \ref{ch1pav2} clearly shows; nevertheless members of that particular class can be parametrized, as was shown in our Master's thesis \cite{SZF3}. The following easily comes from Theorem \ref{ch1newp}.
\begin{cor}
Suppose that $H$ is a dephased $BH(n,4)$ matrix such that its core has main diagonal $-1$ and all other entries are $\pm\mathbf{i}$. Then $H$ admits an at least one-parameter affine orbit.
\end{cor}
The reader might wish to jump ahead to Example \ref{ch2d6ex} to see the one-parameter matrix $D_6^{(1)}(c)$ as an illustration.
So far we have accounted all $BH(n,4)$ matrices up to order $6$. Indeed, it is rather easy to see that they are equivalent to $F_1$, $F_2$, either $F_2\otimes F_2$ or $F_4$, and $D_6^{(1)}(1)$, respectively. The following is a short summary.
\begin{lem}
The number of inequivalent $BH(n,4)$ matrices reads $1,1,2,1$ for $n=1,2,4,6$. The matrices of orders $1$ and $2$ are isolated, whereas the matrices of orders $4$ and $6$ are part of an infinite, one-parameter affine family of complex Hadamard matrices.
\end{lem}
To further investigate the parametrizing capabilities of $BH(n,4)$ matrices, we have fully classified $BH(8,4)$ matrices. The results are quoted from \cite{SZF4} as follows.
\begin{thm}[see \cite{SZF4}]\label{ch18x8}
There are $15$ $BH(8,4)$ matrices, up to equivalence. All these matrices are members of $3$, partially overlapping family of complex Hadamard matrices. All these matrices can be distinguished by the number of $4\times 4$ vanishing minors they contain up to transposition, conjugation and adjoint.
\end{thm}
\begin{proof}
The proof is a straightforward exhaustive computer search.
\end{proof}
We remark here that the proof is ``straightforward'' because the size of the problem is rather moderate. However, as the parameters $n$ and $q$ are increased, the classification problem of $BH(n,q)$ matrices becomes computationally challenging. For some advanced techniques we refer the reader to \cite{KO1} and \cite{LO}.

It seems that all major features are shared by the matrices $H$, $H^\ast$, $\overline{H}$ and $H^T$, therefore we have introduced the following
\begin{defi}[see \cite{LSZO}]
The complex Hadamard matrices $H$ and $K$ are called \emph{$\mathrm{ACT}$-equivalent}, if $H$ is equivalent to at least one of $K$, $K^\ast$, $\overline{K}$ or $K^T$.
\end{defi}
The acronym $\mathrm{ACT}$ stands for adjoint, conjugation and transpose in alphabetical order. To utilize this concept, we state the following
\begin{prop}
There are $10$ $BH(8,4)$ matrices, up to $\mathrm{ACT}$-equivalence. All these matrices can be distinguished by the number of $4\times 4$ vanishing minors they contain.
\end{prop}
In what follows we describe the three parametric families mentioned in Theorem \ref{ch18x8}. Our first example is essentially the Fourier family $F_8^{(5)}$, reported first in \cite{TZ1}.
\[F_8^{(5)}(a,b,c,d,e)=\left[
\begin{array}{rrrrrrrr}
 1 & 1 & 1 & 1 & 1 & 1 & 1 & 1 \\
 1 & a & b & c & -1 & -a & -b & -c \\
 1 & d & -1 & -d & 1 & d & -1 & -d \\
 1 & e & -b & -\overline{a}c e & -1 & -e & b & \overline{a}c e \\
 1 & -1 & 1 & -1 & 1 & -1 & 1 & -1 \\
 1 & -a & b & -c & -1 & a & -b & c \\
 1 & -d & -1 & d & 1 & -d & -1 & d \\
 1 & -e & -b & \overline{a}c e & -1 & e & b & -\overline{a}c e
\end{array}
\right].\]
As $F_8^{(5)}(a,b,c,d,e)$ features two rows and columns with entries $\pm1$ it is an infinite family of jacket matrices (see \cite{LV}). Note that $F_8^{(5)}(a \tau,b \tau^2,c \tau^3,d \tau^2,e \tau^3)$ where $\tau=\mathbf{e}^{2\pi\mathbf{i}/8}$ is an eighth root of unity coincides with the matrix listed in \cite{TZ1}, and hence is of Di\c{t}\u{a}-type. The matrix corresponding to $F_8^{(5)}(-1/t^2,\mathbf{i},-\mathbf{i}/t^2,-\mathbf{i},-\mathbf{i}/t^2)$ is equivalent to an infinite family of circulant matrices (see Theorem \ref{ch3back}), containing in particular Horadam's matrix $K_4(\mathbf{i})$ (see \cite[p.\ $88$]{KJH1}).

After evaluating the matrix $F_8^{(5)}(a,b,c,d,e)$ at the $4^5=1024$ possible quintuples (see Lemma \ref{ch1LRO}), we obtain the following
\begin{prop}
The family $F_8^{(5)}(a,b,c,d,e)$ contains $8$ inequivalent $BH(8,4)$ matrices: the symmetric matrices $F_8^{(5)}(1,1,1,1,1)$, $F_8^{(5)}(1,\mathbf{i},\mathbf{i},\mathbf{i},\mathbf{i})$, $F_8^{(5)}(\mathbf{i},1,\mathbf{i},1,\mathbf{i})$ and $F_8^{(5)}(1,1,\mathbf{i},1,\mathbf{i})$; and further the matrices $F_8^{(5)}(1,1,1,1,\mathbf{i})$, $F_8^{(5)}(1,1,\mathbf{i},\mathbf{i},\mathbf{i})$, and their transpose.
\end{prop}
In \cite{MRS} another family of $8\times 8$ complex Hadamard matrices were obtained from translational tiles in abelian groups (see Section \ref{ch1tilespectral}):
\beql{ch1s8m}
S_8^{(4)}(a,b,c,d)=\left[
\begin{array}{rrrrrrrr}
 1 & 1 & 1 & 1 & 1 & 1 & 1 & 1 \\
 1 & d & -d & -d & -1 &  c d & - c d & d \\
 1 &  a\overline{d} &  b\overline{d} & - b\overline{d} & 1 & -1 & -1 & - a\overline{d} \\
 1 &  a & - b &  b & -1 & - c d &  c d & - a \\
 1 & -1 & - b\overline{d} &  b\overline{d} & 1 &  c & - c & -1 \\
 1 & -d &  b & - b & -1 & d & d & -d \\
 1 & - a\overline{d} & -1 & -1 & 1 & - c &  c &  a\overline{d} \\
 1 & - a & d & d & -1 & -d & -d &  a
\end{array}
\right].
\eeq
It was shown that the matrix $S_8^{(4)}(\mathbf{i},\mathbf{i},\mathbf{i},1)$ is not of Di\c{t}\u{a}-type, and as the set of Di\c{t}\u{a}-type matrices is closed it follows that a small neighborhood around it avoids the family $F_8^{(5)}(a,b,c,d,e)$ completely. The matrix corresponding to $S_8^{(4)}(1,1,\mathbf{i},\mathbf{i})$ is equivalent to the ``jacket conference matrix'' $J_8$ \cite{LV}. By evaluating the matrix $S_8^{(4)}(a,b,c,d)$ at the fourth roots of unity we find the following
\begin{prop}\label{myProp2}
The family $S_8^{(4)}(a,b,c,d)$ and its transpose contain $8$ inequivalent $BH(8,4)$ matrices: the real Hadamard matrix $S_8^{(4)}(1,1,1,1)$ and $S_8^{(4)}(1,1,\mathbf{i},\mathbf{i})$, which are equivalent to a symmetric matrix, and further the matrices $S_8^{(4)}(1,1,1,\mathbf{i})$, $S_8^{(4)}(1,\mathbf{i},1,\mathbf{i})$, $S_8^{(4)}(1,1,\mathbf{i},1)$, and their transpose.
\end{prop}
By analyzing the fingerprint of the obtained matrices, one can see that the following equivalences hold: $S_8^{(4)}(1,1,1,1)\sim F_8^{(5)}(1,1,1,1,1)$ and $\left(S_8^{(4)}(1,1,\mathbf{i},1)\right)^T \sim F_8^{(5)}(1,1,1,1,\mathbf{i})$, and therefore a further family is required to describe all of the $15$ $BH(8,4)$ matrices.
\begin{rem}
The family $S_8^{(4)}(a,b,c,d)$ is a maximal affine family (in the sense of \cite{TZ1}) and hence it cannot be part of larger, say, five-parameter affine family of complex Hadamard matrices.
\hfill$\square$\end{rem}
The following matrix was constructed from MUBs of order $4$ (see Section \ref{ch2MUB}) by Di\c{t}\u{a} very recently in \cite{PD2}:
\[D_{8B}^{(5)}(a,b,c,d,e)=\left[
\begin{array}{rrrrrrrr}
 1 & 1 & 1 & 1 & 1 & 1 & 1 & 1 \\
 1 & a & -a & d & -d & -a & a & -1 \\
 1 & b & b\overline{c}e & -d & d & -b\overline{c}e & -b & -1 \\
 1 & c & -e & -1 & -1 & e & -c & 1 \\
 1 & -c & e & -1 & -1 & -e & c & 1 \\
 1 & -b & -b\overline{c}e & -d & d & b\overline{c}e & b & -1 \\
 1 & -a & a & d & -d & a & -a & -1 \\
 1 & -1 & -1 & 1 & 1 & -1 & -1 & 1
\end{array}
\right].\]
We will see shortly that the family $D_{8B}^{(5)}(a,b,c,d,e)$ is essentially different from the families $F_8^{(5)}(a,b,c,d,e)$ and $S_8^{(4)}(a,b,c,d)$, as it accounts for most of the $BH(8,4)$ matrices. We remark here that the family $\left(D_{8B}^{(5)}(a,b,c,1,c)\right)^T$ is equivalent to the three-parameter family $P_8$ reported in \cite{BAS}. We have the following
\begin{prop}\label{myProp3}
The family $D_{8B}^{(5)}(a,b,c,d,e)$ together with its transpose contain $11$ inequivalent $BH(8,4)$ matrices, namely the real Hadamard matrix $D_{8B}^{(5)}(1,1,1,1,1)$, $D_{8B}^{(5)}(1,1,\mathbf{i},1,1)$ and the matrix $D_{8B}^{(5)}(1,\mathbf{i},\mathbf{i},1,\mathbf{i})$, which are all equivalent to symmetric matrices, and further $D_{8B}^{(5)}(1,1,1,1,\mathbf{i})$, $D_{8B}^{(5)}(1,1,1,\mathbf{i},\mathbf{i})$, $D_{8B}^{(5)}(1,1,\mathbf{i},\mathbf{i},1)$, $D_{8B}^{(5)}(1,\mathbf{i},\mathbf{i},\mathbf{i},\mathbf{i})$, and their transpose.
\end{prop}
In particular, the matrix $D_{8B}^{(5)}(1,1,\mathbf{i},\mathbf{i},1)$ is not a member of any of the families $F_8^{(5)}(a,b,c,d,e)$, $S_8^{(4)}(a,b,c,d)$ or $\left(S_8^{(4)}(a,b,c,d)\right)^T$.
\begin{rem}
It is possible to separate the matrices $F_8^{(5)}(1,1,1,1,\mathbf{i})$, $F_8^{(5)}(1,1,\mathbf{i},\mathbf{i},\mathbf{i})$, $S_8^{(4)}(1,1,1,\mathbf{i})$, $S_8^{(4)}(1,\mathbf{i},1,\mathbf{i})$ and $D_{8B}^{(5)}(1,1,\mathbf{i},\mathbf{i},1)$ from their transpose by considering their rectangular rank profile.
\hfill$\square$\end{rem}
The summary concerning $BH(8,4)$ matrices is available in Table \ref{chT1t} for which the legend is as follows:
\begin{table}[htbp] \caption{Summary of the $BH(8,4)$ matrices, up to $\mathrm{ACT}$-equivalence}\label{chT1t} % title name of the table 
\centering % centering table 
\begin{tabular}{cccccccccc} % creating 10 columns 
\hline %\hline% inserting double-line
No. & Family & Coordinates & ACT & HBS & Auto & Defect & Orbit & $\mathbb{Z}_4$ & Invariant\\
%\\ [0.5ex]
\hline % inserts single-line % Entering 1st row
1 & $F_8$ & $(1,1,1,1,1)$ & YYY & YYY & 43008 & 21 & 5 & 3 & 1428\\
2 & $F_8$ & $(1,\mathbf{i},\mathbf{i},\mathbf{i},\mathbf{i})$ & YYY & YYY & 1024 & 9 & 5 & 2 & 852\\
3 & $F_8$ & $(\mathbf{i},1,\mathbf{i},1,\mathbf{i})$ & YYY & NYY & 2048 & 13 & 5 & 2 & 1204\\
4 & $F_8$ & $(1,1,1,1,\mathbf{i})$ & NYN & NYN & 1536 & 15 & 5 & 3 & 948\\ 
5 & $F_8$ & $(1,1,\mathbf{i},\mathbf{i},\mathbf{i})$ & NYN & NYN & 512 & 7 & 5 & 2 & 836\\ 
6 & $F_8$ & $(1,1,\mathbf{i},1,\mathbf{i})$ & YYY & YYY & 256 & 11 & 5 & 3 & 596\\ 
7 & $S_8$ & $(1,1,\mathbf{i},\mathbf{i})$ & YYY & YYY & 768 & 11 & 4 & 4 & 504\\ 
8 & $S_8$ & $(1,1,1,\mathbf{i})$ & NYN & NYN & 192 & 5 & 4 & 3 & 360\\
9 & $S_8$ & $(1,\mathbf{i},1,\mathbf{i})$ & NYN & NYN & 256 & 9 & 4 & 3 & 652\\ 
10 & $D_{8B}$ & $(1,1,\mathbf{i},\mathbf{i},1)$ & NYN & NYN & 256 & 9 & 5 & 3 & 348\\
\hline
% [1ex] adds vertical space \hline % inserts single-line 
\end{tabular} \label{tab:PPer} 
\end{table}
column $\mathrm{ACT}$ describes if a matrix is equivalent to its adjoint, conjugate or transpose; column $\mathrm{HBS}$ indicates if a matrix is equivalent to a Hermitian matrix, contains a $4\times 4$ subhadamard matrix (i.e.\ comes from a trivial doubling construction) or equivalent to a symmetric matrix; column ``Auto'' describes the order of the automorphism group; while column ``Orbit'' describes the number of degree of freedom of the family from which the matrix is obtained via the respective values of the ``Coordinates'' column; the final column describes the number of $4\times 4$ vanishing minors in a matrix. The concept of the $\mathbb{Z}_q$-rank (which is the subject of column $\mathbb{Z}_4$ in Table \ref{chT1t}) is explained in the following
\begin{defi}\label{ch1Z4}
The \emph{$\mathbb{Z}_q$-rank of an integral matrix} $L$ of order $n$ if the smallest positive integer $r$, such that there are integral matrices $S$ and $T$ of orders $n\times r$ and $r\times n$, respectively, such that $ST\equiv L$ (mod $q$).
\end{defi}
\begin{defi}
The \emph{$\mathbb{Z}_q$-rank of a $BH(n,q)$ matrix} $H$ is the $\mathbb{Z}_q$-rank of the corresponding \mbox{$(0,\hdots, q-1)$}-matrix $L$ which can be obtained from $H$ by exchanging the $q$th roots of unity $\mathbf{e}^{2\pi\mathbf{i}t/q}$ with $t$ in it for every $t=0,\hdots,q-1$.
\end{defi}
We shall convince the reader later in Section \ref{ch1tilespectral} that some $BH(n,q)$ matrices with small $\mathbb{Z}_q$-rank has some interesting applications in harmonic analysis \cite{TT1}.

It does not seem obvious how to calculate the $\mathbb{Z}_4$-rank of a matrix efficiently, as $\mathbb{Z}_4$ is not a field (cf.\ \cite{CRC1} for the prime case). A modular Gaussian elimination gives an upper bound to the $\mathbb{Z}_q$-rank, but it is unclear if and when this bound is sharp. Therefore to be on the safe side we have conducted a brute-force search by utilizing the following easy
\begin{lem}
Suppose that the $\mathbb{Z}_4$-rank of an integral matrix $L$ is $r$. Then there are $r$ rows of $L$ forming an $r\times n$ matrix $T$ for which there is an integral matrix $S$ such that $ST\equiv L$ $(\mathrm{mod}\ 4)$.
\end{lem}
\begin{proof}
By Definition \ref{ch1Z4} there are integral matrices $S$ and $T$ of orders $n\times r$ and $r\times n$ such that $ST\equiv L$ (mod $4$). Our goal is to modify the decomposition $S$ and $T$ by eventually showing that $T$, after proper modifications, consists of some of the rows of $L$. First observe that one can suppose that there is a number $1$ in every column of $S$. Indeed: if there is no number $1$ present in a column, but there is a number $3$, then one can multiply that column and the corresponding row of $T$ by $3$ to leave the product matrix $ST$ unchanged. If there are neither numbers $1$ nor $3$ in a column, then one can change every $2$ to $1$ and then multiply the corresponding row of $T$ by $2$ to leave the product matrix $ST$ unchanged. Let us denote the $i$th row of $S$ and $T$ by $\alpha_i$, $i=1,\hdots,n$ and $v_i, i=1,\hdots, r$, respectively.

Now consider a row of $S$, say $\alpha_j$, of which the first coordinate is $1$. It follows that the $j$th row of $L$, $L_j$, can be written in the form $\alpha_j T\equiv L_j$ (mod $4$), and as the coefficient of the vector $v_1$ is $1$, one can express $v_1$ in terms of $L_j$ and $v_2, \hdots, v_r$. Now it follows trivially that we can use the row $L_j$ instead of $v_1$ in $T$ by appropriately changing $S$. In particular, the $j$th row of $S$ can be set to $[1,0,\hdots,0]$. Now the whole argument can be repeated, {\it mutatis mutandis} until $r$ rows of $L$ are featured in $T$.
\end{proof}

So it seemed that all $BH(n,4)$ matrices admit an affine orbit for $4\leq n\leq 8$, yet a general parametrization scheme remained to be found. In order to find more examples of $BH(n,4)$ matrices the author jointly with Pekka Lampio and Patric \"Osterg\aa rd proved the following result in \cite{LSZO}:
\begin{thm}[see \cite{LSZO}]\label{pekkathm}
Every $BH(n,4)$ matrix for $n=10, 12$ admit an affine orbit.
\end{thm}
This further supports the idea that all $BH(n,4)$ matrices can be parametrized in some ``magical'' way, perhaps through Theorem \ref{ch1newp}. As it is an ongoing work we do not give any further details regarding Theorem \ref{pekkathm}, its sole purpose is to give contrast to the following rather surprising
\begin{thm}[see \cite{LSZO}]\label{ch1defl14}
There are isolated $BH(14,4)$ matrices.
\end{thm}
\begin{ex}[Lampio]\label{ch1defl14ex}
The following is an isolated $BH(14,4)$ matrix:
\[L_{14A}^{(0)}=\left[
\begin{array}{rrrrrrrrrrrrrr}
 1 & 1 & 1 & 1 & 1 & 1 & 1 & 1 & 1 & 1 & 1 & 1 & 1 & 1 \\
 1 & 1 & 1 & -1 & -1 & 1 & -\mathbf{i} & -1 & -1 & \mathbf{i} & -\mathbf{i} & -1 & \mathbf{i} & 1 \\
 1 & 1 & \mathbf{i} & \mathbf{i} & \mathbf{i} & -\mathbf{i} & 1 & -\mathbf{i} & -1 & -\mathbf{i} & -1 & 1 & -1 & -1 \\
 1 & 1 & \mathbf{i} & -\mathbf{i} & -\mathbf{i} & -\mathbf{i} & -1 & \mathbf{i} & \mathbf{i} & -1 & 1 & \mathbf{i} & -\mathbf{i} & -1 \\
 1 & 1 & -\mathbf{i} & 1 & -1 & \mathbf{i} & -1 & -\mathbf{i} & 1 & 1 & \mathbf{i} & -1 & -1 & -1 \\
 1 & \mathbf{i} & -1 & -\mathbf{i} & -1 & -1 & -\mathbf{i} & \mathbf{i} & 1 & -\mathbf{i} & -1 & 1 & \mathbf{i} & 1 \\
 1 & \mathbf{i} & -\mathbf{i} & \mathbf{i} & 1 & -1 & -1 & -\mathbf{i} & -\mathbf{i} & -1 & -\mathbf{i} & \mathbf{i} & 1 & \mathbf{i} \\
 1 & -1 & 1 & 1 & -1 & \mathbf{i} & \mathbf{i} & \mathbf{i} & -1 & -\mathbf{i} & -\mathbf{i} & -\mathbf{i} & 1 & -1 \\
 1 & -1 & 1 & -\mathbf{i} & \mathbf{i} & -1 & 1 & -1 & -\mathbf{i} & \mathbf{i} & \mathbf{i} & \mathbf{i} & -\mathbf{i} & -\mathbf{i} \\
 1 & -1 & \mathbf{i} & -1 & 1 & 1 & -1 & -1 & 1 & 1 & -1 & -\mathbf{i} & -\mathbf{i} & \mathbf{i} \\
 1 & -1 & -1 & 1 & \mathbf{i} & -\mathbf{i} & -1 & 1 & -1 & \mathbf{i} & 1 & -\mathbf{i} & -1 & 1 \\
 1 & -1 & -\mathbf{i} & -1 & -\mathbf{i} & 1 & \mathbf{i} & -\mathbf{i} & \mathbf{i} & -1 & \mathbf{i} & 1 & \mathbf{i} & -\mathbf{i} \\
 1 & -\mathbf{i} & -1 & \mathbf{i} & 1 & \mathbf{i} & 1 & \mathbf{i} & 1 & -1 & -\mathbf{i} & -1 & -1 & -\mathbf{i} \\
 1 & -\mathbf{i} & -1 & -1 & -\mathbf{i} & -1 & 1 & 1 & -1 & 1 & \mathbf{i} & -1 & 1 & \mathbf{i}
\end{array}
\right].\]
\end{ex}
\begin{proof}[Proof of Theorem \ref{ch1defl14}]
The matrix $L_{14A}^{(0)}$, described in Example \ref{ch1defl14ex} above, has defect $0$ and hence the statement follows from Proposition \ref{ch1tzdefzeroi}.
\end{proof}
Therefore Lemma \ref{ch1trivi} does not have a complex analogue working for all $BH(n,4)$ matrices. It is natural though to ask the following
\begin{pro}
Identify some large class of $BH(n,4)$ matrices admitting an infinite, affine orbit.
\end{pro}
Once it is realized that there are isolated $BH(n,4)$ matrices one would like to understand how common they are. Isolated matrices are unique objects which cannot be obtained from parametric families due to their rigid structure. Therefore we pose the following
\begin{pro}
Find additional isolated $BH(n,q)$ matrices for $n\leq14$.
\end{pro}
A related, equally important problem besides the classification of $BH(n,4)$ matrices is finding new construction methods yielding $BH(n,4)$ matrices in those orders where the existence of these objects is unresolved. It seems (see \cite{KJH1}) that the smallest outstanding order is $n=70$; we pose this here as a relevant
\begin{pro}
Construct a $BH(70,4)$ matrix.
\end{pro}
\subsection{\texorpdfstring{A course on $BH(n,6)$ matrices}{A course on BH(n,6) matrices}}\label{ch1bh6}
In this section we briefly describe the simplest case of $BH(n,q)$ matrices when $q$ is composite, i.e.\ when $q=6$. There is a renewed interest in the existence of $BH(n,6)$ matrices, in part because of the following seminal result due to Compton, Craigen and de Launey. Recall that we have denoted by $\omega$ the principal cubic root of unity (see Definition \ref{ch1defom}).
\begin{thm}[see \cite{CCL}]\label{ch1ccl}
Let $H$ be a $BH(n,6)$ matrix with no $\pm1$ entries. Then $\varphi(H)$ is a real Hadamard matrix of order $4n$, where the map $\varphi\colon\mathbb{Q}\left(\omega\right)\to \mathcal{M}_4\left(\mathbb{Q}\right)$, given by
\[\varphi(a+b\omega+c\omega^2):=\left[
\begin{array}{cccc}
 2 a-b-c & b-c & b-c & b-c \\
 c-b & 2 a-b-c & b-c & c-b \\
 c-b & c-b & 2 a-b-c & b-c \\
 c-b & b-c & c-b & 2 a-b-c
\end{array}
\right],\]
acts entrywise on the matrix $H$.
\end{thm}
We recall the following
\begin{defi}[see \cite{CCL}]
A $BH(n,6)$ matrix with no $\pm1$ entries is called an \emph{unreal} $BH(n,6)$ matrix.
\end{defi}
\begin{ex}\label{ch11212}
The $3\times 3$ unreal circulant complex Hadamard matrix, equivalent to $F_3$, induces the (unique) real Hadamard matrix of order $12$ as follows
\[\left[\begin{array}{rrr}
\omega & \omega^2 & \omega^2\\
\omega^2 & \omega & \omega^2\\
\omega^2 & \omega^2 & \omega\\
\end{array}\right]\leadsto\left[\begin{array}{rrrr|rrrr|rrrr}
 -1 & 1 & 1 & 1 & -1 & -1 & -1 & -1 & -1 & -1 & -1 & -1 \\
 -1 & -1 & 1 & -1 & 1 & -1 & -1 & 1 & 1 & -1 & -1 & 1 \\
 -1 & -1 & -1 & 1 & 1 & 1 & -1 & -1 & 1 & 1 & -1 & -1 \\
 -1 & 1 & -1 & -1 & 1 & -1 & 1 & -1 & 1 & -1 & 1 & -1 \\
 \hline
 -1 & -1 & -1 & -1 & -1 & 1 & 1 & 1 & -1 & -1 & -1 & -1 \\
 1 & -1 & -1 & 1 & -1 & -1 & 1 & -1 & 1 & -1 & -1 & 1 \\
 1 & 1 & -1 & -1 & -1 & -1 & -1 & 1 & 1 & 1 & -1 & -1 \\
 1 & -1 & 1 & -1 & -1 & 1 & -1 & -1 & 1 & -1 & 1 & -1 \\
 \hline
 -1 & -1 & -1 & -1 & -1 & -1 & -1 & -1 & -1 & 1 & 1 & 1 \\
 1 & -1 & -1 & 1 & 1 & -1 & -1 & 1 & -1 & -1 & 1 & -1 \\
 1 & 1 & -1 & -1 & 1 & 1 & -1 & -1 & -1 & -1 & -1 & 1 \\
 1 & -1 & 1 & -1 & 1 & -1 & 1 & -1 & -1 & 1 & -1 & -1
\end{array}
\right].\]
\end{ex}
Theorem \ref{ch1ccl} establishes a strong connection between the existence of real and a class of $BH(n,6)$ matrices. However, by Corollary \ref{ch1CCC} there is no $BH(5,6)$ matrix and therefore it is impossible to obtain a real Hadamard matrix of order $20$ with this method. Another implication of Corollary \ref{ch1CCC} is that one does not have direct access to $BH(25,6)$ matrices via the Kronecker product. In what follows we use a related approach to construct such a matrix.

We plug some unknown, unimodular matrices $X$, $Y$ and $Z$ of order $5$ into the $5\times 5$ Paley matrix in place of its entries $\{0,1,-1\}$ to obtain
\beql{ch1pp}
W_{25}:=\left[\begin{array}{ccccc}
X & Y & Z & Z & Y\\
Y & X & Y & Z & Z\\
Z & Y & X & Y & Z\\
Z & Z & Y & X & Y\\
Y & Z & Z & Y & X\\
\end{array}\right],
\eeq
and try to solve the arising orthogonality equations. As the underlying matrix is circulant, we have only three equations to deal with, namely
\begin{gather}\label{ch2bu1}
XX^\ast+2YY^\ast+ZZ^\ast=25I_5,\\
\label{ch2bu2}
XY^\ast+YX^\ast+ZY^\ast+YZ^\ast+ZZ^\ast=0,\\
\label{ch2bu3}
XZ^\ast+ZX^\ast+YZ^\ast+ZY^\ast+YY^\ast=0.
\end{gather}
It remains to find some feasible matrices $X, Y$ and $Z$, composed from sixth roots of unity. It is straightforward to attack the problem with a computer search, however, to shrink the size of the search space considerably we impose the simplifying assumption that all $X, Y$ and $Z$ are circulant. In this way we have found the following solution to the system of equations \eqref{ch2bu1}-\eqref{ch2bu3}:
\begin{gather*}
X=\left[
\begin{array}{ccccc}
 1 & -\omega ^2 & -\omega ^2 & -\omega ^2 & -\omega ^2 \\
 -\omega ^2 & 1 & -\omega ^2 & -\omega ^2 & -\omega ^2 \\
 -\omega ^2 & -\omega ^2 & 1 & -\omega ^2 & -\omega ^2 \\
 -\omega ^2 & -\omega ^2 & -\omega ^2 & 1 & -\omega ^2 \\
 -\omega ^2 & -\omega ^2 & -\omega ^2 & -\omega ^2 & 1
\end{array}
\right],\ \ 
Y=\left[
\begin{array}{ccccc}
 -\omega  & -\omega  & \omega  & \omega  & -\omega  \\
 -\omega  & -\omega  & -\omega  & \omega  & \omega  \\
 \omega  & -\omega  & -\omega  & -\omega  & \omega  \\
 \omega  & \omega  & -\omega  & -\omega  & -\omega  \\
 -\omega  & \omega  & \omega  & -\omega  & -\omega 
\end{array}
\right],\\
Z=\left[
\begin{array}{ccccc}
 -\omega  & \omega  & -\omega  & -\omega  & \omega  \\
 \omega  & -\omega  & \omega  & -\omega  & -\omega  \\
 -\omega  & \omega  & -\omega  & \omega  & -\omega  \\
 -\omega  & -\omega  & \omega  & -\omega  & \omega  \\
 \omega  & -\omega  & -\omega  & \omega  & -\omega 
\end{array}
\right].
\end{gather*}
Next we describe our construction method in a more sophisticated way by using the Kronecker product. Let $P$ be the Paley matrix of order $p$ and decompose it into $(0,1)$-matrices $S$ and $N$ such that $P=S-N$. The matrices $S$ and $N$ are circulant, hence they commute and their first rows are the indicator vector of the squares and nonsquares modulo $p$. Based on the example we have found of order $25$ we expect that
\[H=I\otimes(I-\omega^2(J-I))-\omega S\otimes(P+I)+\omega N\otimes(P-I)\]
is a $BH(p^2,6)$ matrix. Note that it is unimodular, as required. Let us simplify $H$ as follows
\begin{equation*}
\begin{split}
H&=I\otimes I-\omega^2 I\otimes J+\omega^2 I\otimes I-\omega (S-N)\otimes P-\omega(S+N)\otimes I\\
&=-\omega P\otimes P-\omega(J-I)\otimes I-\omega^2 I\otimes J-\omega I\otimes I\\
&=-\omega P\otimes P-\omega J\otimes I-\omega^2 I\otimes J.
\end{split}
\end{equation*}
Now multiply $H$ by $-\omega^2$ to obtain
\[K=P\otimes P+J\otimes I+\omega I\otimes J\]
and check if $K$ is complex Hadamard. We have
\begin{equation*}
\begin{split}
KK^\ast&=PP^T\otimes PP^T+pJ\otimes I+\omega^2 J\otimes J+\omega J\otimes J+p I\otimes J\\
&=(pI-J)\otimes (pI-J)+pJ\otimes I+p I\otimes J -J\otimes J=p^2 I\otimes I.
\end{split}
\end{equation*}
This construction is a joint work with Professor Robert Craigen and will be included in a forthcoming publication in the near future. We summarize it in the following
\begin{thm}\label{ch1pp6}
Let $p$ be a prime number and $P$ be the Paley-matrix of order $p$. Then the matrix $H=P\otimes P+J\otimes I+\omega I\otimes J$ is a $BH(p^2,6)$ matrix.
\end{thm}
A construction, similar in spirit, can be found in the early literature (cf.\ \cite{VB1}, \cite{RT1}) resulting in various real orthogonal matrices.
\begin{cor}
For every natural numbers $n$ and $k$ there is a $BH(n^{2k},6)$ matrix.
\end{cor}
\begin{proof}
Use the binary expansion of $2k$ and the Kronecker product of square matrices from Theorem \ref{ch1pp6}.
\end{proof}
Now we construct some sporadic examples of bicirculant $BH(n,6)$ matrices.
\begin{ex}
Let
\begin{gather*}
A=\omega^2I+\omega P,\\
B=\omega I-\omega^2 N-\omega S,
\end{gather*}
where $P=S-N$ is the Paley-matrix of order $11$, such that $S$ and $N$ are the $(0,1)$-characteristic matrices of the squares and nonsquares modulo $11$.
Then, we find that
\begin{gather}\label{ch1in1}
AA^\ast=12I-J-\mathbf{i}\sqrt3 P,\\
\label{ch1in2}
BB^\ast=10I+J+\mathbf{i}\sqrt3 P,
\end{gather}
and it follows that $AA^\ast+BB^\ast=22I$. In particular the bicirculant matrix
\[W_{22}:=\left[\begin{array}{rr}\label{ch1w22}
A & B\\
B^\ast & -A^\ast\\
\end{array}\right]\]
is an unreal $BH(22,6)$ matrix.
\end{ex}
\begin{rem}
It is intriguing that the blocks $A$ and $B$ have already a nice structure as described by \eqref{ch1in1}-\eqref{ch1in2} and one begins to wonder if a similar construction might work in higher orders as well. However, experiments show that a similar phenomenon, where the product matrices $AA^\ast$ and $BB^\ast$ are already a linear combination of $I$, $J$ and $P$ only occurs in dimension $7$ and $11$.
\hfill$\square$\end{rem}
In the next two examples we utilize quartic residues.
\begin{ex}\label{ch1w34}
Let $J=I+Q+S+N$, where $Q,S$ and $N$ are the $(0,1)$-characteristic matrices of the quartic residues; quadratic, but not quartic residues; and the quadratic nonresidues modulo $17$, respectively. Then set
\begin{gather*}
A=I+Q-\omega^2 S+\omega N,\\
B=I-Q-\omega S-\omega^2 N.
\end{gather*}
One readily checks (by computer, if one wishes) that $AA^\ast+BB^\ast=34I$. Therefore the bicirculant matrix
\[W_{34}:=\left[\begin{array}{rr}
A & B\\
B^\ast & -A^\ast\\
\end{array}\right]\]
is a $BH(34,6)$ matrix.
\end{ex}
Similarly, one has
\begin{ex}\label{ch1w58}
Consider the generator $g=2$ of $\mathbb{Z}_{29}^\ast$ and let $J=I+Q+S+N_1+N_2$, where $Q,S$, and $N_1+N_2=N$ are the $(0,1)$-characteristic matrices of quartic residues; quadratic, but not quartic residues; and the quadratic nonresidues modulo $29$, respectively. Additionally, let the first row of $N_1$ be the characteristic vector of those elements $x\in\mathbb{Z}_{29}^\ast$ such that $x=g^{4i+1}$ for some $i$. Then set
\begin{gather*}
A=\omega I+\omega Q+\omega^2 S+\omega^2 N_1-\omega^2 N_2,\\
B=\omega^2I-\omega Q+\omega S-\omega N_1+\omega N_2.
\end{gather*}
One readily checks (by computer, if one wishes) that $AA^\ast+BB^\ast=58I$. Therefore the bicirculant matrix
\[W_{58}:=\left[\begin{array}{rr}
A & B\\
B^\ast & -A^\ast\\
\end{array}\right]\]
is an unreal $BH(58,6)$ matrix.
\end{ex}
Table \ref{ch1t2} summarizes the existence of $BH(n,6)$ matrices of small even orders (consult Table \ref{ch3tab2} for odd orders).
\begin{table}[htbp]
\caption{Existence of $BH(n,6)$ matrices for $n\leq 62$, $n\equiv 2$ (mod 4)}\label{ch1t2}
\centering
\begin{tabular}{rll|rll}
\hline
$n$ & \text{Existence} & \text{Remark} & $n$ & \text{Existence} & \text{Remark}\\
\hline
$2$ & $F_2$ &  & $34$ & $W_{34}$ & \text{Example \ref{ch1w34}}\\
$6$ & $S_6$ & \text{Example \ref{ch2s6}} & $38$ & $F_2\otimes W_{19}$ & \text{Example \ref{ch3w19}}\\
$10$ & $W'$ & \text{See \cite{CCL}} & $42$ & $F_2\otimes F_3\otimes P_7$ & \text{Example \ref{ch3pet7}}\\
$14$ & $F_2\otimes P_7$ & \text{Example \ref{ch3pet7}} & $46$ & ? & $46=2\cdot 23$\\
$18$ & $F_2\otimes F_3\otimes F_3$ & & $50$ & $F_2\otimes W_{25}$ & \text{Theorem \ref{ch1pp6}}\\
$22$ & $W_{22}$ & \text{Example \ref{ch1w22}} & $54$ & $F_2\otimes F_3\otimes F_3\otimes F_3$ & \\
$26$ & $F_2\otimes X_{13}$ & \text{See \cite{CCL}} & $58$ & $W_{58}$ & \text{Example \ref{ch1w58}}\\
$30$ & $F_3\otimes W'$ & \text{See \cite{CCL}} & $62$ & ? & $62=2\cdot 31$\\
\hline
\end{tabular}
\end{table}
These ad-hoc constructions of bicirculant $BH(n,6)$ matrices support a conjecture we formulated in our Master's thesis (see \cite[p.\ $74$]{SZF3}):
\begin{con}
For every prime $p$ there is a $BH(2p,6)$ matrix.
\end{con}
We shall revisit Butson-type complex Hadamard matrices once again in Section \ref{ch3secp}.
\subsection{An application of Butson matrices: Tiles and spectral sets}\label{ch1tilespectral}
A rather surprising application of complex Hadamard matrices has been found recently by Tao \cite{TT1}. To understand his achievement we need to introduce the concept of translational tiling and spectral sets. We remark that everything what is stated here has far reaching generalizations in harmonic analysis, and there are analogous statements available in locally compact abelian groups, but highlighting those connections are beyond the scope of this thesis. Instead, we recall only the most elementary concepts relevant to us, and to the theory of $BH(n,q)$ matrices. We refer the interested reader to the survey article \cite{KM1} and to the recent literature, including the papers \cite{FR}, \cite{IKT} and \cite{IL1}.
\begin{defi}
We say that a \emph{subset $T$ of $\mathbb{Z}_q^d$ tiles $\mathbb{Z}_q^d$ by translation} when there exists a subset $\Lambda\subset \mathbb{Z}_q^d$ such that every $g\in \mathbb{Z}_q^d$ has a unique representation
\beql{ch1tip}
g=t+\lambda,\ \ \ t\in T, \lambda\in \Lambda.
\eeq
The sets $T$ and $\Lambda$ are called tiling complements of each other.
\end{defi}
\begin{ex}
Let $q=3, d=2$, and $T=\{(0,0),(1,1),(2,1)\}\subset \mathbb{Z}_3^2$. Then $T$ tiles $\mathbb{Z}_3^2$ with the tiling complement $\Lambda=\{(0,0),(0,1),(0,2)\}$.
\end{ex}
There is an elegant characterization of the tiling property by means of harmonic analysis, but instead of recalling it we offer the following obvious condition which is enough for our purposes. We denote by $|S|$ the cardinality of the set $S$.
\begin{lem}\label{ch1tile}
If $T$ tiles $\mathbb{Z}_q^d$ by translation then $|T|$ divides $|\mathbb{Z}_q^d|=q^d$.
\end{lem}
\begin{proof}
Indeed, as by the tiling property \eqref{ch1tip} we have trivially $\left|\mathbb{Z}_q^d\right|=\left|T\right|\cdot\left|\Lambda\right|$.
\end{proof}
To introduce spectral sets we need some notation first. Let us identify the elements of $\mathbb{Z}_q^d$ with column vectors, each coordinate being in the range $\{0,1,\hdots,q-1\}$. Thus, a set $S\subset \mathbb{Z}_q^d$ with $n$ elements can be identified with a $d\times n$ matrix, the columns of which are the elements of $S$. We denote this matrix also by $S$. The dual group $\widehat{\mathbb{Z}_q^d}$ can be identified with row vectors, whose coordinates are, again, in $\{0,1,\hdots,q-1\}$. Recall that we have denoted by $\mathrm{EXP}(L)$ the entrywise exponential of the matrix $L$.
\begin{defi}
We say that an $n$-element subset $S\subset \mathbb{Z}_q^d$ is a \emph{spectral set} when there is an $n$-element subset $Q\subset \widehat{\mathbb{Z}_q^d}$ such that the matrix
\[\mathrm{EXP}\left(\frac{2\pi\mathbf{i}}{q}QS\right)\]
is a complex Hadamard, i.e.\ a $BH(n,q)$ matrix. We say that $Q$ is the spectrum of $S$ and that the two sets form a spectral pair.
\end{defi}
\begin{ex}
The whole group $\mathbb{Z}_q$ is spectral, and its spectrum is the whole dual group $\widehat{\mathbb{Z}_q}$ (i.e.\ we have $n=q$), giving rise to the Fourier matrix $F_n=\mathrm{EXP}\left(\frac{2\pi\mathbf{i}}{q}\widehat{\mathbb{Z}_q}\mathbb{Z}_q\right)$.
\end{ex}
While tiling seems to be a rather intuitive, geometric concept, spectral sets are more difficult to understand. The following connection between them, which is natural if one sees the underlying harmonic analytic concepts, was proposed (originally in $\mathbb{R}^d$) in $1974$:
\begin{con}[Fuglede, \cite{FUG1}]\label{ch1FUC}
A subset $S\subset \mathbb{Z}_q^d$ is spectral, if and only if it tiles $\mathbb{Z}_q^d$.
\end{con}
It turns out that via the tiling-spectral correspondence ``trivial'' tiling constructions indeed have a spectral analogue leading to families of Di\c{t}\u{a}-type complex Hadamard matrices \cite{MRS}. However, a non-standard tiling construction of Szab\'o \cite{SZS1} was used in \cite{MRS} to produce previously unknown families of complex Hadamard matrices in dimensions $8$, $12$ and $16$ (see the matrix $S_8^{(4)}(a,b,c,d)$ in \eqref{ch1s8m}).

Several positive partial results were obtained regarding some special cases of Fuglede's conjecture (see \cite{IKT}, \cite{IL1}), before Tao exhibited a non-tiling spectral set in $\mathbb{Z}_3^6$, which then he slightly modified to obtain a counterexample in $\mathbb{Z}_3^5$ as well.
\begin{ex}[Tao, \cite{TT1}]
Let $d=6$, $n=6$, $q=3$, $S=I_6$, and
\[Q=\left[\begin{array}{cccccc}
0 & 0 & 0 & 0 & 0 & 0\\
0 & 0 & 1 & 2 & 2 & 1\\
0 & 1 & 0 & 1 & 2 & 2\\
0 & 2 & 1 & 0 & 1 & 2\\
0 & 2 & 2 & 1 & 0 & 1\\
0 & 1 & 2 & 2 & 1 & 0\\
\end{array}\right].\]
Then the matrix $\mathrm{EXP}\left(\frac{2\pi\mathbf{i}}{3}QS\right)$ is a $BH(6,3)$ matrix, and hence $S$ is spectral in $\mathbb{Z}_3^6$ with spectrum $Q$. However, $S$ cannot tile $\mathbb{Z}_3^6$ by Lemma \ref{ch1tile}.
\end{ex}
Tao's arguments also implied that Fuglede's conjecture fails in $\mathbb{R}^d$ in any dimension higher than $5$ and subsequent counterexamples were found refuting both directions for $d\geq 3$, see \cite{FMM} and \cite{KM2}. In fact, much of the technical difficulties arise in transferring the counterexamples from the finite group setting to $\mathbb{R}^d$. However, for the purposes of this thesis we simply stay in the finite setting, where the complex Hadamard matrices play their essential r\^ole.

We recall the following improvement over Tao's result.
\begin{ex}[Kolountzakis--Matolcsi, \cite{KM2}]\label{ch1KMex}
Let $d=3$, $n=6$, $q=8$ and 
\beql{ch1mms}
S=\left[\begin{array}{cccccc}
0 & 2 & 4 & 1 & 5 & 6\\
0 & 6 & 3 & 4 & 2 & 7\\
0 & 6 & 7 & 2 & 4 & 3\\
\end{array}\right], \text{ and }
Q^T=\left[
\begin{array}{cccccc}
 0 & 0 & 1 & 0 & 0 & 7 \\
 0 & 1 & 0 & 1 & 0 & 1 \\
 0 & 1 & 0 & 0 & 1 & 1
\end{array}
\right],
\eeq
where we have given $Q^T$ instead of $Q$ for typographical reasons. We have
\beql{ch1mmm}
QS\equiv\left[\begin{array}{cccccc}
0 & 0 & 0 & 0 & 0 & 0\\
0 & 4 & 2 & 6 & 6 & 2\\
0 & 2 & 4 & 1 & 5 & 6\\
0 & 6 & 3 & 4 & 2 & 7\\
0 & 6 & 7 & 2 & 4 & 3\\
0 & 2 & 6 & 5 & 1 & 4\\
\end{array}\right]\ \ \ (\mathrm{mod}\ 8),
\eeq
and it is easy to see that $\mathrm{EXP}\left(\frac{2\pi\mathbf{i}}{8}QS\right)$ is a $BH(6,8)$ matrix. It follows that the set $S\subset \mathbb{Z}_8^3$ is spectral, but does not tile $\mathbb{Z}_8^3$ by Lemma \ref{ch1tile}, as $6$ does not divide $8^3$.
\end{ex}
It seemed that these small dimensional counterexamples (i.e.\ $d=6, 3$) surfaced because one has access to ``sporadic'', small dimensional $BH(n,q)$ matrices with small $\mathbb{Z}_q$-rank for $q=3$, or $q=8$. Of course, by considering larger sets, and hence, larger Butson Hadamard matrices it seems more and more difficult to control the $\mathbb{Z}_q$-rank, and therefore one does not really expect to find further $BH(n,q)$ matrices satisfying the divisibility conditions required to the failure of Lemma \ref{ch1tile}. Rather surprisingly, however, we have discovered an infinite sequence of counterexamples with $\mathbb{Z}_q$-rank $3$, which is our contribution to this section. We have the following
\begin{thm}\label{ch1tilingkr}
If $H$ is a $BH(n,q)$ matrix with $\mathbb{Z}_q$-rank $r$, then for every integral number $m\geq 1$ there is a Di\c{t}\u{a}-type $BH(mn,mq)$ matrix with $\mathbb{Z}_{mq}$-rank at most $r$.
\end{thm}
Let us denote by $C_1$ the matrix in which the first column is pure $1$ and all other entries are $0$. Similarly, let us denote by $R_1$ the matrix in which the first row is pure $1$ and all other entries are $0$. The size of the matrices shall be clear from context.
\begin{proof}
Let $Q$ and $S$ be $n\times r$ and $r\times n$ $(0,\hdots,q-1)$-matrices, respectively, such that $\mathrm{EXP}\left(\frac{2\pi\mathbf{i}}{q}QS\right)=H$. Then define the block matrix
\[L=\left[\begin{array}{c}
Q\\
Q+qC_1\\
\vdots\\
Q+(m-1)qC_1\\
\end{array}\right]\cdot
\left[\begin{array}{cccc}
mS & mS+R_1 & \hdots & mS+(m-1)R_1\\
\end{array}\right].\]
Hence, the $(i,j)$th block of $L$ reads
\[\left(Q+(i-1)qC_1\right)\left(mS+(j-1)R_1\right)\equiv mQS+(j-1)QR_1+(i-1)(j-1)qJ\]
modulo $mq$. Now it is clear that the matrix $K=\mathrm{EXP}\left(\frac{2\pi\mathbf{i}}{mq}L\right)$ is exactly the Di\c{t}\u{a}-type matrix
\[K=F_m\otimes \left(H, \hdots, \mathrm{Diag}\left(Q_{1,1}^{j-1},\hdots, Q_{n,1}^{j-1}\right)H,\hdots, \mathrm{Diag}\left(Q_{1,1}^{m-1},\hdots, Q_{n,1}^{m-1}\right)H\right)\]
coming from Corollary \ref{ch1ditaconst}, where $F_m$ is the Fourier matrix of order $m$.
\end{proof}
\begin{rem}
Note that Theorem \ref{ch1tilingkr} does not describe the Kronecker product construction.
\hfill$\square$\end{rem}
We offer a simple
\begin{ex}
Let $n=m=q=2$, $H=F_2$. Then
\[S=\left[\begin{array}{cc}
0 & 1\\
\end{array}\right], Q=\left[\begin{array}{c}
0\\ 1\\
\end{array}\right].\]
We have
\[L=\left[\begin{array}{c}
0\\
1\\
\hline
2\\
3\\
\end{array}\right]\cdot\left[\begin{array}{cc|cc}
0 & 2 & 1 & 3\\
\end{array}\right]\equiv\left[\begin{array}{cc|cc}
0 & 0 & 0 & 0\\
0 & 2 & 1 & 3\\
\hline
0 & 0 & 2 & 2\\
0 & 2 & 3 & 1\\
\end{array}\right]\ \ (\mathrm{mod}\ 4),\]
giving rise to $F_4=\mathrm{EXP}\left(\frac{2\pi\mathbf{i}}{4}L\right)$, which is not equivalent to $F_2\otimes F_2$.
\end{ex}
If one has some additional structure on the initial matrix $H$, one can prove a somewhat stronger result. We state the case $m=2$ relevant to us.
\begin{cor}\label{ch1tilingcor1}
Let $q$ be even. Suppose that $H=\mathrm{EXP}\left(\frac{2\pi\mathbf{i}}{q}L\right)$ is a $BH(n,q)$ matrix with $\mathbb{Z}_q$-rank $r$ such that there is a decomposition $QS\equiv L$ $(\mathrm{mod}\ q)$ where every entry in the first row of $S$ is even. Then there is a $BH(2n,q)$ matrix with $\mathbb{Z}_q$-rank $r$.
\end{cor}
\begin{proof}
Indeed, the matrix
\[\left[\begin{array}{c}
Q\\
Q+(q/2)C_1\\
\end{array}\right]\cdot\left[\begin{array}{cc}
S & S+R_1\\
\end{array}\right]=\]
\[\left[\begin{array}{ll}
QS & QS+QR_1\\
QS+q/2C_1S & QS+QR_1+(q/2)C_1S+q/2J\\
\end{array}\right]\equiv\left[\begin{array}{ll}
QS & QS+QR_1\\
QS & QS+QR_1+q/2J\\
\end{array}\right]\]
(mod $q$) will do the job. Note that resulting matrix contains $L$ as a submatrix and therefore its $\mathbb{Z}_q$-rank cannot fall below $r$.
\end{proof}
\begin{cor}
There is a $BH(12,8)$ matrix with $\mathbb{Z}_8$-rank $3$.
\end{cor}
\begin{proof}
Observe that in the second row of the matrix $QS$ in \eqref{ch1mmm} all coordinates are even, therefore it is natural to check if $S$ in \eqref{ch1mms} can be replaced with some $S'$ featuring the second row of $L$. This is exactly the case. We have
\[S'=\left[
\begin{array}{cccccc}
 0 & 4 & 2 & 6 & 6 & 2 \\
 0 & 2 & 4 & 1 & 5 & 6 \\
 0 & 6 & 3 & 4 & 2 & 7
\end{array}
\right], \left(Q'\right)^T=\left[
\begin{array}{cccccc}
 0 & 1 & 0 & 0 & 1 & 1 \\
 0 & 0 & 1 & 0 & 0 & 7 \\
 0 & 0 & 0 & 1 & 7 & 0
\end{array}
\right],\]
so an application of Corollary \ref{ch1tilingcor1} yields the result.
\end{proof}
Once we have established the existence of a $BH(12,8)$ matrix with $\mathbb{Z}_8$-rank $3$ it was natural to look for a matrix possessing the special structure required by Corollary \ref{ch1tilingcor1}. We found the following by investigating various Di\c{t}\u{a}-type complex Hadamard matrices of order $12$.
\begin{ex}\label{ch1ex12}
The matrix $\mathrm{EXP}\left(\frac{2\pi\mathbf{i}}{8}QS\right)$ is a $BH(12,8)$ matrix with $\mathbb{Z}_8$-rank $3$, where
\begin{gather*}
S=\left[
\begin{array}{cccccccccccc}
 0 & 2 & 4 & 2 & 6 & 6 & 0 & 2 & 4 & 2 & 6 & 6 \\
 0 & 4 & 2 & 6 & 7 & 3 & 1 & 5 & 3 & 7 & 0 & 4 \\
 0 & 5 & 6 & 1 & 4 & 2 & 0 & 5 & 6 & 1 & 4 & 2
\end{array}
\right],\\
Q^T=\left[
\begin{array}{cccccccccccc}
 0 & 0 & 1 & 1 & 0 & 1 & 2 & 2 & 3 & 3 & 2 & 3 \\
 0 & 1 & 0 & 3 & 0 & 0 & 4 & 5 & 4 & 7 & 4 & 4 \\
 0 & 0 & 0 & 0 & 1 & 3 & 4 & 4 & 4 & 4 & 5 & 7
\end{array}
\right].
\end{gather*}
Note that the first row of $S$ consists of even numbers.
\end{ex}
Combining Corollary \ref{ch1tilingcor1} and Example \ref{ch1ex12} we arrive at the following
\begin{cor}
There is a $BH(24,8)$ matrix with $\mathbb{Z}_8$-rank $3$.
\end{cor}
The other direction of Fuglede's conjecture could not be treated directly by Tao's ideas, but using a tricky duality argument Kolountzakis and Matolcsi exhibited a non-spectral tile in $\mathbb{Z}_3^6$ thus disproving this part of Conjecture \ref{ch1FUC} too \cite{KM3}. It would be very interesting to find additional $BH(n,q)$ matrices with small $\mathbb{Z}_q$-rank hopefully leading to the disproval of Fuglede's conjecture in dimension $2$ as well. We formulate this as a
\begin{pro}
Find a $BH(n,q)$ matrix with $\mathbb{Z}_q$-rank $2$, such that $n$ does not divide $q^2$, or prove that no such a matrix exists.
\end{pro}
Despite the negative results in dimensions $d\geq 3$, Fuglede's conjecture might be true for $d=1$. For some interesting number-theoretic consequences consult the paper \cite{CM1} and its references.
\begin{pro}
Settle both directions of Fuglede's conjecture in $d\leq 2$ dimensions.
\end{pro}
\newpage %Ez azert kell, hogy jobb oldalon kezdodjenek a chapterek.
\thispagestyle{empty} % ez is
%----------------------------------CHAPTER BREAK-------------------------------------
%----------------------------------CHAPTER BREAK-------------------------------------
\chapter[\texorpdfstring{Towards the classification of $6\times 6$ Hadamard matrices}{Towards the classification of 6x6 Hadamard matrices}]{\texorpdfstring{Towards the classification of $6\times 6$ complex Hadamard matrices}{Towards the classification of 6x6 complex Hadamard matrices}}\label{ch2}
\thispagestyle{empty}
This chapter is devoted to the classification of complex Hadamard matrices of small orders. During the $1990$s the classification up to order $5$ has been completed \cite{RC3}, \cite{UH1} and recently various constructions of $6\times 6$ complex Hadamard matrices appeared in the literature \cite{BN1}, \cite{BFx1}, \cite{PD1}. However, the full classification of complex Hadamard matrices of order $5$ required non-trivial ideas already, and the $6\times 6$ case seemed, despite various efforts from the mathematics and physics community, far out of reach.

Complex Hadamard matrices of order $n$ are related to Mutually Unbiased Bases (MUBs), a geometric configuration of vectors in $\mathbb{C}^n$; an object which has wide range of applications in quantum information theory due to a celebrated result of Werner \cite{RW1}. Informally speaking, the basic step towards understanding MUBs requires understanding complex Hadamard matrices and once a full classification of complex Hadamard matrices is available (in a given order) significant advances are expected on the theory of MUBs as well (see \cite{BWB}). We shall investigate them in details in Section \ref{ch2MUB}.

This chapter is based on the papers \cite{JMS1}, \cite{MSZ}, \cite{SZF5} and \cite{SZF6}. The main result is a construction, capturing the structure of a ``typical'' complex Hadamard matrix of order $6$. Moreover, we conjecture that the construction in fact describes all, previously unknown matrices of this order. This method, however, is by no means a comprehensive and irredundant classification of complex Hadamard matrices of order $6$ and further and deeper understanding of these objects is needed in order to achieve this.
\section{\texorpdfstring{Classification up to order $5$}{Classification up to order 5}}
In this section we go through the description of all complex Hadamard matrices up to order $5$. Recall from Lemma \ref{ch1L117} that every complex Hadamard matrix is equivalent to a dephased one, and hence we can suppose that the first row and column of the matrices we describe here is composed of numbers $1$. Therefore one is interested in finding those unimodular row vectors whose coordinates add up to $0$. We begin with a characterization of vanishing sums of order $k$ for $k\leq 4$.
\begin{lem}\label{ch1l2}
Suppose that $x,y,u$ and $v$ are complex unimodular numbers. Then
\begin{enumerate}[$($a$)$]
\item $x+y=0$ implies $y=-x$;
\item $x+y+u=0$ implies $y=\varepsilon x$, $u=\varepsilon^2 x$ such that $1+\varepsilon+\varepsilon^2=0$.
\item $x+y+u+v=0$ implies $x\in\{-y,-u,-v\}$.
\end{enumerate}
\end{lem}
\begin{proof}
Indeed, this is easily seen through geometry.
\end{proof}
Lemmata \ref{ch1L117} and \ref{ch1l2} allow us to give a full classification of complex Hadamard matrices up to orders $4$.
\begin{lem}
For orders $n=1,2,3$ all complex Hadamard matrices are equivalent to the Fourier matrix $F_n$.
\end{lem}
However, at order $4$ we have already seen a non-trivial, one-parameter family (see Example \ref{ch1ExF4i}). That particular construction is the only possibility we have.
\begin{prop}[Craigen, \cite{RC3}]\label{ch3pRC}
All complex Hadamard matrices of order $4$ are members of the one-parameter Fourier family, described in Example \ref{ch1ExF4i}, such that $a=\mathbf{e}^{\mathbf{i}t}$, $t\in [0,\pi/2)$.
\end{prop}
All these results are elementary, and we invite the reader to verify the statements. However, classification of $5\times 5$ complex Hadamard matrices turned out to be extremely difficult. One of the main reasons of this difficulty is the lack of understanding vanishing sums of order $5$. Also it might be possible that we have a vanishing sum of order $5$ which cannot be extended with a further one being orthogonal to it. By Corollary \ref{ch1CCC} we know that there are no $BH(5,6)$ matrices, and it is fairly easy to see that there are no triplets of pairwise orthogonal rows of length $5$, composed from $6$th roots of unity. Yet, we have the following intriguing
\begin{ex}
The following is a triplet of orthogonal rows in order $5$, composed from $12$th roots of unity:
\[\left[\begin{array}{rrrrr}
1 & 1 & 1 & 1 & 1\\
1 & \omega & \omega^2 & \mathbf{i} & -\mathbf{i}\\
1 & -\mathbf{i}\omega^2 & -\mathbf{i}\omega & -\mathbf{i} & -1\\
\end{array}\right].\]
\end{ex}
In light of the example above, one cannot but wonder if, by utilizing $q$th roots of unity for some large $q$, nontrivial complex Hadamard matrices of order $5$ can be obtained. The next result gives a necessary condition on the structure of orthogonal triplets of rows in complex Hadamard matrices.
\begin{lem}[Haagerup's trick, \cite{UH1}]\label{ch2HT}
Let $(1,1,1,1,1)$, $(1,a,b,s_1,s_2)$ and $(1,c,d,$ $s_3,s_4)$ be pairwise orthogonal rows of a complex Hadamard matrix of order $5$. Then
\[(1+a+b)(1+\overline{c}+\overline{d})(1+c\overline{a}+d\overline{b})\in\mathbb{R}.\]
\end{lem}
This identity establishes a non-trivial algebraic relation in-between four entries of the $5\times 5$ matrix, or, in other words, eliminates $4$ out of the $8$ unknown variables. This leads, after further non-trivial arguments, to the following
\begin{thm}[Haagerup, \cite{UH1}]\label{ch2haaf5}
All complex Hadamard matrices of order $5$ are equivalent to the Fourier matrix $F_5$.
\end{thm}
We remark here that Lemma \ref{ch2HT} applies for other orders as well (cf.\ Corollary \ref{HT}). Haagerup's trick turned out to be powerful enough to deal with the full classification of the inverse-orthogonal matrices of order $5$ as well \cite{KN1}; it also has been successfully utilized in higher orders resulting in new examples of complex Hadamard matrices of order $6$ \cite{BN1}, \cite{MSZ}.
\section{\texorpdfstring{The evolution of Hadamard matrices of order $6$}{The evolution of Hadamard matrices of order 6}}\label{ch2secn}
To fully appreciate our achievements described in this chapter we first recall all previous attempts which contributed towards the deeper understanding of the structure of complex Hadamard matrices of order $6$.
\subsection{Well-known matrices and their orbit}
The most natural source of complex Hadamard matrices is the class of $BH(n,q)$ matrices (see Definition \ref{ch1butsd}). The case $q=2$ is well-known to be vacuus for $n=6$ but the next possible value already gives us the following 
\begin{ex}[Butson, \cite{AB1}]\label{ch2s6}
\[S_6^{(0)}=\left[\begin{array}{rrrrrr}
1 & 1 & 1 & 1 & 1 & 1\\
1 & 1 & \omega & \omega^2 & \omega^2 & \omega\\
1 & \omega & 1 & \omega & \omega^2 & \omega^2\\
1 & \omega^2 & \omega & 1 & \omega & \omega^2\\
1 & \omega^2 & \omega^2 & \omega & 1 & \omega\\
1 & \omega & \omega^2 & \omega^2 & \omega & 1\\
\end{array}\right].\]
\end{ex}
The $BH(6,3)$ matrix $S_6^{(0)}$ is isolated amongst all complex Hadamard matrices of order $6$ (see Definition \ref{ch1defiso}). It can be obtained from Butson's doubling construction \cite{AB1} and was utilized by Tao \cite{TT1} to disprove Fuglede's conjecture (see Section \ref{ch1tilespectral}). It is also a sporadic example of $2$-transitive complex Hadamard matrices \cite{EM}.

The next example comes from the conference matrix construction, described in Theorem \ref{ch1pav2}. However, it was only realized very recently that this particular matrix admits an affine orbit.
\begin{ex}[Di\c{t}\u{a}, \cite{PD1}]\label{ch2d6ex}
\[D_6^{(1)}(c)=\left[\begin{array}{rrrrrr}
1 & 1 & 1 & 1 & 1 & 1\\
1 & -1 & \mathbf{i} & -c\mathbf{i} & -\mathbf{i} & c\mathbf{i}\\
1 & \mathbf{i} & -1 & c\mathbf{i} & -\mathbf{i} & -c\mathbf{i}\\
1 & -\overline{c}\mathbf{i} & \overline{c}\mathbf{i} & -1 & \mathbf{i} & -\mathbf{i}\\
1 & -\mathbf{i} & -\mathbf{i} & \mathbf{i} & -1 & \mathbf{i}\\
1 & \overline{c}\mathbf{i} & -\overline{c}\mathbf{i} & -\mathbf{i} & \mathbf{i} & -1\\
\end{array}\right].\]
\end{ex}
The matrix $D_6=D_6^{(1)}(1)$ plays a central r\^ole in various constructions. We highlight some key facts as follows.
\begin{rem}
The free parameter $c$ can be introduced into the matrix $D_6$ through Theorem \ref{ch1newp}.
\hfill$\square$\end{rem}
\begin{rem}
The matrices $S_6^{(0)}$ and $D_6^{(1)}(1)$ are deeply entwined, as they both come from a general construction of complex Hadamard matrices with circulant core (see Lemma \ref{ch3index2B}). Also, they are the only $6\times 6$ dephased symmetric complex Hadamard matrices with real diagonal, \cite{MSZ}.
\hfill$\square$\end{rem}
\begin{rem}
The $BH(6,8)$ matrix $D_6^{(1)}(\mathbf{e}^{2\pi\mathbf{i}/8})$ is equivalent to the matrix found by Kolountzakis and Matolcsi, described in Example \ref{ch1KMex}.
\hfill$\square$\end{rem}
Next we introduce the two-parameter Fourier families known already by Hadamard \cite{JH1} as noted in \cite{RC4} and rediscovered independently by Haagerup \cite{UH1}. Their existence also follows from Di\c{t}\u{a}'s general construction (see \cite{TZ1}).
\begin{ex}
\[F_6^{(2)}(a,b)=\left[\begin{array}{rrr|rrr}
1 & 1 & 1 & 1 & 1 & 1\\
1 & \omega & \omega^2 & a & a\omega & a\omega^2\\
1 & \omega^2 & \omega & b & b\omega^2 & b\omega\\
\hline
1 & 1 & 1 & -1 & -1 & -1\\
1 & \omega & \omega^2 & -a & -a\omega & -a\omega^2\\
1 & \omega^2 & \omega & -b & -b\omega^2 & -b\omega\\
\end{array}\right].\]
Note that the families $F_6^{(2)}(a,b)$ and $F_6^{(2)}(a,b)^T$ are inequivalent.
\end{ex}
It is easy to see (for example, via Haagerup's invariant) that for generic values of the parameters $a, b$ and $c$ the families $F_6^{(2)}(a,b), F_6^{(2)}(a,b)^T$ and $D_6^{(1)}(c)$ are inequivalent matrices.

Our next example is the Bj\"orck--Fr\"oberg cyclic $6$-root matrix \cite{BFx1}.
\begin{ex}[Bj\"orck--Fr\"oberg, \cite{BFx1}]
\[C_6=\left[
\begin{array}{rrrrrr}
 1 & \mathbf{i} d & -d & -\mathbf{i} & -\overline{d} & \mathbf{i} \overline{d} \\
 \mathbf{i} \overline{d} & 1 & \mathbf{i} d & -d & -\mathbf{i} & -\overline{d} \\
 -\overline{d} & \mathbf{i} \overline{d} & 1 & \mathbf{i} d & -d & -\mathbf{i} \\
 -\mathbf{i} & -\overline{d} & \mathbf{i} \overline{d} & 1 & \mathbf{i} d & -d \\
 -d & -\mathbf{i} & -\overline{d} & \mathbf{i} \overline{d} & 1 & \mathbf{i} d \\
 \mathbf{i} d & -d & -\mathbf{i} & -\overline{d} & \mathbf{i} \overline{d} & 1
\end{array}
\right],\]
where $d$ is Bj\"orck's ``magical'' number:\footnote{As Bengtsson et al.\ refer to it in \cite{BBE}.}
\[d=\frac{1-\sqrt3}{2}+\mathbf{i}\frac{\sqrt{2\sqrt{3}}}{2}.\]
\end{ex}
We will return to the problem of the classification of circulant complex Hadamard matrices in Section \ref{ch3secC}.

It is natural to ask how well are we doing in terms of the classification of \mbox{$6\times 6$} complex Hadamard matrices. It turns out that randomly chosen members of the families above, as well as the matrix $C_6$ has defect $4$ and hence all these matrices might be part of a larger, four-parameter family of complex Hadamard matrices, except for the single isolated point $S_6^{(0)}$. This observation led to the following
\begin{con}[Bengtsson et al., \cite{BBE}]\label{ch2bengtss}
There is a four-parameter family of complex Hadamard matrices of order $6$.
\end{con}
Numerical experiments confirm the truth of this conjecture in a neighborhood around the Fourier matrix $F_6$ \cite{SNS}. In the next couple of subsections we shall recall further families with one, two and three parameters from the existing literature.
\subsection{The self-adjoint family}\label{ch2saf}
So far we have listed ``trivial'' examples of complex Hadamard matrices. They are trivial for two reasons: first, the starting point matrices $F_6^{(2)}(1,1), D_6^{(1)}(1)$ and the matrices $S_6^{(0)}$ and $C_6$ can be found by a straightforward computer search. One simply knows in advance (or at least have an intelligent guess) what to look for. On the other hand, they are trivial because their orbit is affine. The first non-trivial, that is: non-affine family of complex Hadamard matrices of order $6$ was discovered by Beauchamp and Nicoar\u{a} in $2006$ \cite{BN1}, whose efforts were motivated by various operator algebraic reasons (see e.g.\ \cite{GHJ} and \cite{VJ1}). Their result amounts to a full classification of all self-adjoint complex Hadamard matrices of order $6$.
\begin{ex}[Beauchamp--Nicoar\u{a}, \cite{BN1}]
\[B_6^{(1)}(\theta)=\left[\begin{array}{rrrrrr}
1 & 1 & 1 & 1 & 1 & 1\\
1 & -1 & -\overline{x} & -y & y & \overline{x}\\
1 & -x & 1 & y & \overline{z} & -\overline{x y z}\\
1 & -\overline{y} & \overline{y} & -1 & -\overline{x y z} & \overline{x y z}\\
1 & \overline{y} & z & -x y z & 1 & -\overline{x}\\
1 & x & -x y z & x y z & -x & -1\\
\end{array}\right],\]
where
\[x=\frac{1+2y+y^2\pm\sqrt{2}\sqrt{1+2y+2y^3+y^4}}{1+2y-y^2},\qquad z=\frac{1+2y-y^2}{y(-1+2y+y^2)},\]
and $y=\mathbf{e}^{\mathbf{i}\theta}$ such that
\beql{ch2rest}
\theta\in\left[-\pi,-\arccos\left(\left(\sqrt{3}-1\right)/2\right)\right]\cup\left[\arccos\left(\left(\sqrt{3}-1\right)/2\right),\pi\right].\eeq
\end{ex}
Note that this family is non-affine, and it is a priori unclear whether its entries are unimodular. In particular, the unimodularity of $x$ forces the restriction \eqref{ch2rest} on the phase of $y$.

The main reason why this particular family can be described by relatively simple algebraic formulae is the fact that the entry $-1$ appears in its core multiple times. Indeed, the presence of these entries imply that many of the vanishing sums of order $6$ within this matrix trivially simplify to $6$-term sums composed of a vanishing sum of order $4$ and a vanishing sum of order $2$, which we understand through Lemma \ref{ch1l2}.

We remark here that self-adjoint matrices are relatively rare, as the following folklore spectral analysis shows.
\begin{lem}\label{ch1gowslemma}
Let $H$ be a self-adjoint complex Hadamard matrix of order $n$. Then either $n$ is square or $n$ is even and $\mathrm{Tr}(H)=0$.
\end{lem}
\begin{proof}
Consider a self-adjoint complex Hadamard matrix $H$ of order $n$ and suppose that the number $1$ appears exactly $k$ times in its main diagonal. Clearly, as $H$ is self-adjoint its eigenvalues are $\pm\sqrt{n}$ with multiplicities $s$ and $n-s$, respectively. Then, we have
\[\mathrm{Tr}(H)=s\sqrt{n}-(n-s)\sqrt{n}=k-(n-k),\]
from which the assertion follows.
\end{proof}
\subsection{A symmetric family}
Once it was realized that some preliminary assumptions on the structure of the matrix gives such a huge reduction in complexity M\'at\'e Matolcsi and the author managed to exhibit a new, previously unknown family of symmetric matrices \cite{MSZ}.
\begin{ex}[see \cite{MSZ}]\label{ch2MSZex}
\[M_6^{(1)}(x)=\left[\begin{array}{rrrrrr}
1 & 1 & 1 & 1 & 1 & 1\\
1 & -1 & x & x & -x & -x\\
1 & x & a & b & c & d\\
1 & x & b & a & d & c\\
1 & -x & c & d & e & f\\
1 & -x & d & c & f & e\\
\end{array}\right],\]
where
\begin{gather*}
a=\frac{x^2-2x-1}{4}+\mathbf{i}\frac{(x^2-2x-1)\sqrt{16-|x^2-2x-1|^2}}{4|x^2-2x-1|},\\
b=\frac{x^2-2x-1}{4}-\mathbf{i}\frac{(x^2-2x-1)\sqrt{16-|x^2-2x-1|^2}}{4|x^2-2x-1|},\\
c=-\frac{x^2+1}{4}+\mathbf{i}\frac{(x^2+1)\sqrt{16-|x^2+1|^2}}{4|x^2+1|},\\
d=-\frac{x^2+1}{4}-\mathbf{i}\frac{(x^2+1)\sqrt{16-|x^2+1|^2}}{4|x^2+1|},\\
e=\frac{x^2+2x-1}{4}+\mathbf{i}\frac{(x^2+2x-1)\sqrt{16-|x^2+2x-1|^2}}{4|x^2+2x-1|},\\
f=\frac{x^2+2x-1}{4}-\mathbf{i}\frac{(x^2+2x-1)\sqrt{16-|x^2+2x-1|^2}}{4|x^2+2x-1|},
\end{gather*}
and $x\neq \pm\mathbf{i}$ is any unimodular complex number.
\end{ex}
The family $M_6^{(1)}(x)$ was discovered by a careful analysis of dephased symmetric complex Hadamard matrices with real diagonal. We note that it connects the Fourier families $F_6^{(2)}(a,b)$ with the family $D_6^{(1)}(c)$. Clearly the family $M_6^{(1)}(x)$ does not describe all symmetric complex Hadamard matrices of order $6$ as it misses the isolated matrix $S_6^{(0)}$. It is rather disappointing that a comprehensive characterization of all $6\times 6$ symmetric complex Hadamard matrices is still unavailable. We pose this as a
\begin{pro}
Classify all symmetric $6\times 6$ complex Hadamard matrices.
\end{pro}
\subsection{Bicirculant matrices}
The first non-trivial two-dimensional family was constructed by the author during summer $2008$ \cite{SZF5}. It was found by assuming that the matrix has a bicirculant structure, as follows.
\begin{ex}[see \cite{SZF5}]\label{x6ex}
Let
\[H=
\left[\begin{array}{rr}
A & B\\
B^\ast & -A^\ast\\
\end{array}\right],
\]
where $A$ and $B$ are $3\times 3$ circulant matrices:
\[A=\left[\begin{array}{ccc}
1 & a & b\\
b & 1 & a\\
a & b & 1\\
\end{array}\right], B=\left[\begin{array}{ccc}
1 & c & d\\
d & 1 & c\\
c & d & 1\\
\end{array}\right],\]
where the assumption that $A$ and $B$ have real diagonal $1$ can be made due to equivalence. After dephasing the matrix $H$ and change of variables we get
\[X_6^{(2)}(\alpha)\equiv X_6^{(2)}(x,y,u,v)=\left[
\begin{array}{cccccc}
 1 & 1 & 1 & 1 & 1 & 1 \\
 1 & x^2 y & x y^2 & \frac{x y}{u v} & u x y & v x y \\
 1 & \frac{x}{y} & x^2 y & \frac{x}{u} & \frac{x}{v} & u v x \\
 1 & u v x & u x y & -1 & -u x y & -u v x \\
 1 & \frac{x}{u} & v x y & -\frac{x}{u} & -1 & -v x y \\
 1 & \frac{x}{v} & \frac{x y}{u v} & -\frac{x y}{u v} & -\frac{x}{v} & -1
\end{array}
\right],\]
where $\alpha:=x+y+\overline{xy}$. It turns out, that in order to fill the matrix $X_6^{(2)}(\alpha)$ in with unimodular entries one should restrict $\alpha$ to come from the fundamental domain $\mathbb{D}$ characterized by the property $D(\alpha)\leq 0$ and $D(-\alpha)\leq 0$ where the function $D$ reads
\[D(\alpha):=|\alpha|^4+18|\alpha|^2-8\Re[\alpha^3]-27.\]
We remark here that $\mathbb{D}$ is an umbrella-shaped region coming from the intersection of two hypocycloids.
For such numbers $\alpha$ one can solve the cubic equation
\beql{ch2alphat}
f_\alpha(x):=x^3-\alpha x^2+\overline{\alpha}x-1
\eeq
to obtain its roots $r_1, r_2$ and $r_3$ and set $x=r_1$, $y=r_2$. One should solve the equation $f_{-\alpha}(u)$ as well to obtain its roots $q_1, q_2$ and $q_3$, and similarly, set $u=q_1$, $v=q_2$. For $\alpha\in\mathrm{int}\mathbb{D}$ the roots of \eqref{ch2alphat} are distinct and of modulus $1$. In general for $\alpha$ $\in$ $\mathrm{int}\mathbb{D}$ one has $6\cdot 6=36$ choices for the ordered quadruple $(x,y,u,v)$. However, an easy automated calculation shows that all these matrices will be equivalent to either $X_6^{(2)}(r_1,r_2,q_1,q_2)$ or its transpose.
\end{ex}
We remark here that the family $X_6^{(2)}(\alpha)$ unified and extended some of the eariler families introduced. In particular, for $|\alpha|=1$ we get the family $D_6^{(1)}(c)$, while on the boundary of $\mathbb{D}$ we find the self-adjoint family $B_6^{(1)}(\theta)$. The reader should consult \cite{SZF5} for the details. We shall utilize the family $X_6^{(2)}(\alpha)$ once again in Section \ref{ch2MUB}.
\subsection{Karlsson's observation}
The next two-parameter family was discovered by Karlsson \cite{BK3}, who extended the family $M_6^{(1)}(x)$ (see Example \ref{ch2MSZex}) with an additional degree of freedom as follows.
\begin{ex}[Karlsson, \cite{BK3}]
\[K_6^{(2)}(x_1,x_2)=\left[
\begin{array}{rr|rr|rr}
1 & 1 & 1 & 1 & 1 & 1\\
1 & -1 & z_1 & -z_1 & z_1 & -z_1\\
\hline
1 & z_2 & -f_1 & -z_2f_2 & -\overline{f}_3 & -z_2\overline{f}_4\\
1 & -z_2 & -z_1\overline{f}_2 & z_1z_2\overline{f}_1 & -z_1f_4 & z_1z_2f_3\\
\hline
1 & z_2 & -\overline{f}_3 & -z_2\overline{f}_4 & -f_1 & -z_2f_2\\
1 & -z_2 & -z_1f_4 & z_1z_2f_3 & -z_1\overline{f}_2 & z_1z_2\overline{f}_1\\
\end{array}\right],\]
where $z_1=\mathbf{e}^{\mathbf{i}x_1}$, $z_2=\mathbf{e}^{\mathbf{i}x_2}$, $f_1=f(x_1,x_2)$, $f_2=f(x_1,-x_2)$, $f_3=f(-x_1,-x_2)$ and $f_4=f(-x_1,x_2)$, where
\begin{multline*}
f(x_1,x_2)=\mathbf{e}^{\mathbf{i}(x_1+x_2)/2}\\
\times\left(\cos\left(\frac{x_1-x_2}{2}\right)-\mathbf{i}\sin\left(\frac{x_1+x_2}{2}\right)\right)\left(\frac{1}{2}+\mathbf{i}\sqrt{\frac{1}{1+\sin(x_1)\sin(x_2)}-\frac{1}{4}}\right).
\end{multline*}
Karlsson also pointed out that all $9$ $2\times 2$ blocks are complex Hadamard. This discovery of his later lead to breakthrough results, which we shall discuss in more details soon.
\end{ex}
Karlsson's matrix contains the Di\c{t}\u{a}-family at the four limit points $(x_1,x_2)$ $\rightarrow$ $(\pm\pi/2,\pm\pi/2)$ (the one-parameter comes from the direction of approach); $K_6^{(2)}(x_1,0)$ coincides with a one-parameter subfamily of the Fourier family $F_6^{(2)}(a,b)$, similarly, the family $K_6^{(2)}(0,x_2)$ coincides with a one-parameter subfamily of the transposed Fourier family $F_6^{(2)}(a,b)^T$, and additionally for $x_1=x_2=:x$ it coincides with the family $M_6^{(1)}(x)$, up to equivalence. Therefore Karlsson's matrix provides a bridge in-between previously known families.
\subsection{\texorpdfstring{$H_2$-reducible complex Hadamard matrices}{H2-reducible complex Hadamard matrices}}
In a subsequent paper \cite{BK1} Karlsson fully described the following important class of $6\times 6$ complex Hadamard matrices.
\begin{defi}[Karlsson, \cite{BK1}]
We say that a $6\times 6$ complex Hadamard matrix is \emph{$H_2$-reducible}, if it is equivalent to a matrix in which all $9$ $2\times 2$ submatrices are complex Hadamard.
\end{defi}
It turns out, that the $H_2$-reducible matrices can be described by a single, three-parameter family $K_6^{(3)}(\theta,\varphi,z_1)$ in a very elegant way. Let us mention here the following breakthrough
\begin{thm}[Karlsson, \cite{BK1}, \cite{BK2}]\label{K6}
Let $H$ be a dephased complex Hadamard matrix of order $6$. Then the following are equivalent:
\begin{enumerate}[$($a$)$]
\item $H$ belongs to the three-parameter family $K_6^{(3)}(\theta,\varphi,z_1)$;
\item $H$ is $H_2$-reducible;
\item The core of $H$ contains a $-1$;
\item Some row or column of $H$ contains a vanishing sum of order $2$;
\item Some row or column of $H$ contains a vanishing sum of order $4$.
\end{enumerate}
\end{thm}
This is a very powerful theorem as Part (c) and (d) allow us to quickly recognize if a matrix belongs to the family $K_6^{(3)}$, and we shall heavily use these conditions later. As the family $K_6^{(3)}$ forms a three-parameter subset, the matrices it contains are atypical, hence we often use the adjective ``degenerate'' in connection with it (see Conjecture \ref{C2}).

As we can see from Part (c) of Theorem \ref{K6} Karlsson's family contains all previously discovered matrices, except from the isolated matrix $S_6^{(0)}$, giving them an elegant and comprehensive presentation.

Let us recall Karlsson's three-parameter family as follows.
\begin{ex}[Karlsson, \cite{BK2}]
\[K_6^{(3)}(\theta,\varphi,z_1)=\left[\begin{array}{ccc}
F_2 & Z_1 & Z_2\\
Z_3 & \frac{1}{2}Z_3AZ_1 & \frac{1}{2}Z_3BZ_2\\
Z_4 & \frac{1}{2}Z_4BZ_1 & \frac{1}{2}Z_4AZ_2\\
\end{array}\right],\ \ A=\left[\begin{array}{cc}
A_{11} & A_{12}\\
\overline{A}_{12} & -\overline{A}_{11}\\
\end{array}\right],\]
where
\[A_{11}=-\frac{1}{2}+\mathbf{i}\frac{\sqrt3}{2}\left(\cos\theta+\mathbf{e}^{-\mathbf{i}\varphi}\sin\theta\right),\qquad A_{12}=-\frac{1}{2}+\mathbf{i}\frac{\sqrt3}{2}\left(-\cos\theta+\mathbf{e}^{\mathbf{i}\varphi}\sin\theta\right)\]
and $B=-F_2-A$ for any $\theta\in[0,\pi)$ and $\varphi\in[0,\pi)$. In the matrices
\[Z_i=\left[\begin{array}{cc}
1 & 1\\
z_i & -z_i\\
\end{array}\right]\]
for $i=1,2$ and their transpose (for $i=3,4$) the elements $z_i$ are related through the M\"obius transformations as follows:
\[z_3^2=\mathcal{M}_A(z_1^2),\qquad z_3^2=\mathcal{M}_B(z_2^2),\qquad z_4^2=\mathcal{M}_A(z_2^2),\qquad z_4^2=\mathcal{M}_B(z_1^2),\]
where
\[\mathcal{M}(z)=\frac{\alpha z-\beta}{\overline{\beta}z-\overline{\alpha}},\]
such that $\alpha_A=A_{12}^2, \beta_A=A_{11}^2$ and $\alpha_B=B_{12}^2, \beta_B=B_{11}^2$.
In general, any of the parameters $z_i$ can be chosen freely, say $z_1$ in which case
\[z_2^2=\mathcal{M}_A^{-1}(\mathcal{M}_B(z_1^2))=\mathcal{M}_B^{-1}(\mathcal{M}_A(z_1^2)),\qquad z_3^2=\mathcal{M}_A(z_1^2),\qquad z_4^2=\mathcal{M}_B(z_1)^2.\]
Any sign combination of the entries $z_i$ will lead to equivalent matrices.
\end{ex}
\begin{rem}
A similar structure theorem was suggested by Craigen, who asked if every real Hadamard matrix of order $n>1$ can be partitioned into a canonical form in which all of the $n^2/4$ $2\times 2$ matrices are of rank $2$. Robert Lecnik confirmed this up to order $28$ by a computer search.\footnote{Private communication from R. Craigen; consult \cite{RCW} for further details.}
\hfill$\square$\end{rem}
\section{A four-parameter family}
Karlsson's results are very impressive, however they do not shed any light on the existence of a four-parameter family, as the introduction of any additional parameter into the matrix $K_6^{(3)}(\theta,\varphi,z_1)$ would immediately destroy its underlying $H_2$-reducible property. Hence it is absolutely unclear how to generalize further Karlsson's construction, and it seems that one should come up with essentially new ideas to achieve this. In what follows we are going to describe one reasonable way to attack Conjecture \ref{ch2bengtss}. In the next subsection we collect a handful of ideas leading to a construction with four degrees of freedom, which is one of the main contributions of this manuscript. The results are to appear in \cite{SZF6}.
\subsection{The main ingredients}
After seeing Karlsson's results we became extremely motivated to construct a four-parameter family, but in order to achieve this, some essentially new tools were needed. The breakthrough came when we managed to improve on Haagerup's result (cf. Lemma \ref{ch2HT}). Let us introduce the following crucial
\begin{defi}[Haagerup polynomial, \cite{SZF6}]
The \emph{Haagerup polynomial} $\mathcal{H}$ associated with the rows $(1,1,1,1,1,1), (1,a,b,e,\ast,\ast)$ and $(1,c,d,f,\ast,\ast)$ of a complex Hadamard matrix of order $6$ reads
\[\mathcal{H}(a,b,c,d,e,f):=(1+a+b+e)(1+\overline{c}+\overline{d}+\overline{f})(1+c\overline{a}+d\overline{b}+f\overline{e}).\]
\end{defi}
The next result is our improvement upon Lemma \ref{ch2HT}.
\begin{thm}[see \cite{SZF6}]\label{HTIO}
Suppose that we have the partial rows $(1,a,b,e,\ast,\ast)$ and $(1,c,d,f,\ast,\ast)$, composed from unimodular entries. Then one can specify some unimodular numbers $s_1,s_2,s_3$ and $s_4$ in place of the unknown numbers $\ast$ to make these rows together with $(1,1,1,1,1,1)$ mutually orthogonal if and only if
\beql{C11}
\mathcal{H}(a,b,c,d,e,f)=4-|1+a+b+e|^2-|1+c+d+f|^2-|1+c\overline{a}+d\overline{b}+f\overline{e}|^2
\eeq
and
\beql{NI}
|1+a+b+e|\leq 2.
\eeq
\end{thm}
Before proceeding to the Proof of Theorem \ref{HTIO} we shall need two easy results first.
\begin{lem}\label{L1}
Suppose that we have a partial row $(1,a,b,e,\ast,\ast)$ composed from unimodular entries. Then one can specify some unimodular numbers $s_1$ and $s_2$ in place of the unknown numbers $\ast$ to make this row orthogonal to $(1,1,1,1,1,1)$ if and only if $|1+a+b+e|\leq 2$.
\end{lem}
\begin{proof}
We need to have $1+a+b+e+s_1+s_2=0$ to ensure orthogonality, from which it follows that $|1+a+b+e|=|s_1+s_2|\leq 2$. It is easily seen through geometry that in this case we can define the unimodular numbers required.
\end{proof}
The missing coordinates featuring in Lemma \ref{L1}, $s_1$ and $s_2$, can be obtained algebraically through the well-known
\begin{lem}[Decomposition formula, \cite{MSZ}]\label{ch2decf}
Suppose that the rows $(1,1,1,1,1,1)$ and $(1,a,b,e,s_1,s_2)$ containing unimodular entries are orthogonal. Let us denote by $\Sigma:=1+a+b+e$, and suppose that $0<|\Sigma|\leq 2$. Then
\beql{s12}
s_{1,2}=-\frac{\Sigma}{2}\pm\mathbf{i}\frac{\Sigma}{|\Sigma|}\sqrt{1-\frac{|\Sigma|^2}{4}};
\eeq
otherwise, if $\Sigma=0$, then $s_1$ is independent from $a,b,e$ but $s_2=-s_1$.
\end{lem}
\begin{proof}
If $\Sigma=0$ then the result is trivial. Otherwise $s_1$ and $s_2$ are the (unique) unimodular numbers with $s_1+s_2=-\Sigma$, required by orthogonality.
\end{proof}
Observe that in \eqref{s12} we would divide by zero in case $\Sigma=0$. This, however, can never happen when we investigate complex Hadamard matrices which are not $H_2$-reducible, as seen by Part (e) of Theorem \ref{K6}. Therefore it follows that once we have four entries in a row or column of a matrix which lies outside the family $K_6^{(3)}$ the Decomposition formula readily derives the unique remaining two values through \eqref{s12}. Now we can turn to the
\begin{proof}[Proof of Theorem \ref{HTIO}]
First we start by proving that \eqref{C11} holds. To do this, we utilize Haagerup's idea \cite{UH1} as follows: by pairwise orthogonality, we find that
\begin{gather*}
1+a+b+e=-s_1-s_2,\\
1+\overline{c}+\overline{d}+\overline{f}=-\overline{s}_3-\overline{s}_4,\\
1+c\overline{a}+d\overline{b}+f\overline{e}=-s_3\overline{s}_1-s_4\overline{s}_2.
\end{gather*}
The first equation above ensures that \eqref{NI} holds. Further, the product of these equations read \begin{equation*}
\begin{split}
\mathcal{H}&=-(s_1+s_2)(\overline{s}_3+\overline{s}_4)(s_3\overline{s}_1+s_4\overline{s}_2)=-(s_1\overline{s}_3+s_2\overline{s}_4+s_1\overline{s}_4+s_2\overline{s}_3)(s_3\overline{s}_1+s_4\overline{s}_2)\\
&=-|s_1\overline{s}_3+s_2\overline{s}_4|^2-(s_1\overline{s}_4+s_2\overline{s}_3)(s_3\overline{s}_1+s_4\overline{s}_2)=-|s_1\overline{s}_3+s_2\overline{s}_4|^2-2\Re(s_1\overline{s}_2+s_3\overline{s}_4).
\end{split}
\end{equation*}
To conclude the proof, we need to show that
\[2\Re(s_1\overline{s}_2+s_3\overline{s}_4)=|1+a+b+e|^2+|1+c+d+f|^2-4\]
holds, however this follows directly from the Decomposition formula.

To see the converse direction, we need to show that \eqref{C11} essentially encodes orthogonality. Let us use the notations $\Sigma:=1+a+b+e$, $\Delta:=1+c+d+f$, $\Psi:=1+c\overline{a}+d\overline{b}+f\overline{e}$. With this notation we have $|\Sigma|\leq 2$ from \eqref{NI} while condition \eqref{C11} boils down to
\beql{SH}
\mathcal{H}=\Sigma\overline{\Delta}\Psi=4-|\Sigma|^2-|\Delta|^2-|\Psi|^2.
\eeq
If, in addition, $|\Delta|\leq 2$ holds, then by the Decomposition formula we can find $s_1, s_2$, $s_3$ and $s_4$ to ensure orthogonality to row $(1,1,1,1,1,1)$. Now observe, that the mutual orthogonality of rows $(1,a,b,e,s_1,s_2)$ and $(1,c,d,f,s_3,s_4)$ reads
\beql{ORT}
\Psi+s_3\overline{s}_1+s_4\overline{s}_2=0.
\eeq
Suppose first that we have the trivial case $\Sigma=\Delta=0$. Then, by the Decomposition formula we have $s_2=-s_1$ and $s_4=-s_3$, and \eqref{SH} implies that $|\Psi|=2$. Therefore, if we set the unimodular number $s_3:=-\Psi s_1/2$ the orthogonality equation \eqref{ORT} is fulfilled.

Suppose secondly, that we have $\Delta=0$, but $\Sigma\neq 0$. Then we have $s_4=-s_3$, and from \eqref{SH} it follows that
\beql{PSI}
|\Psi|=\sqrt{4-|\Sigma|^2}.
\eeq
Now we use the Decomposition formula to find the values of $s_1$ and $s_2$ and the orthogonality equation \eqref{ORT} becomes
\beql{ORT2}
\Psi+s_3\left(-2\mathbf{i}\frac{\overline{\Sigma}}{|\Sigma|}\sqrt{1-\frac{|\Sigma|^2}{4}}\right)=0.
\eeq
This holds, independently of $s_3$, if $|\Sigma|=2$, as by \eqref{PSI} $\Psi=0$ follows. Otherwise, set the unimodular number
\[s_3:=-\mathbf{i}\frac{\Sigma\Psi}{|\Sigma||\Psi|}\]
to ensure the orthogonality through \eqref{ORT2}. The case $\Sigma=0$, $\Delta\neq 0$ can be treated similarly by noting that $|\Delta|\leq 2$ follows from \eqref{SH}.

Finally, let us suppose that $\Sigma\neq 0$ and $\Delta\neq 0$. Now observe that in this case the value of $\Psi$ needed for formula \eqref{ORT} can be calculated through \eqref{SH}. The other ingredient, namely the value of $s_3\overline{s}_1+s_4\overline{s}_2$ can be established through the Decomposition formula, once we derive the required bound $|\Delta|\leq 2$. Depending on the value of $\mathcal{H}$, we treat several cases differently.

CASE 1: Suppose that $-|\Psi|^2\leq \mathcal{H}$. This implies, via \eqref{SH}, that $|\Sigma|^2+|\Delta|^2\leq 4$, and in particular $|\Delta|\leq 2$ follows. Next we calculate $|\Psi|$ from \eqref{SH}.

Suppose first that $\mathcal{H}\geq 0$. Hence, after taking absolute values, \eqref{SH} becomes
\[|\Psi|^2+|\Sigma||\Delta||\Psi|+|\Sigma|^2+|\Delta|^2-4=0,\]
and the only non-negative root is
\beql{ps1}
|\Psi|=\frac{-|\Sigma||\Delta|+\sqrt{(4-|\Sigma|^2)(4-|\Delta|^2)}}{2}.
\eeq

Now suppose that $-|\Psi|^2\leq \mathcal{H}<0$. Hence, after taking absolute values, \eqref{SH} becomes
\beql{NEG}
|\Psi|^2-|\Sigma||\Delta||\Psi|+|\Sigma|^2+|\Delta|^2-4=0,
\eeq
and we find that its only non-negative root under the assumption $|\Sigma|^2+|\Delta|^2\leq 4$ reads
\beql{ps2}
|\Psi|=\frac{|\Sigma||\Delta|+\sqrt{(4-|\Sigma|^2)(4-|\Delta|^2)}}{2}.
\eeq

CASE 2: Suppose now that $\mathcal{H}<-|\Psi|^2$. This implies that $|\Sigma|^2+|\Delta|^2>4$, and we do not have a priori the bound $|\Delta|\leq 2$. Nevertheless, we derive equation \eqref{NEG} again, and we find that the values of $|\Psi|$ can be any of 
\beql{ps3}
|\Psi|_{1,2}=\frac{|\Sigma||\Delta|\pm\sqrt{(4-|\Sigma|^2)(4-|\Delta|^2)}}{2},
\eeq
provided that the roots are real, namely we have either $|\Sigma|>2$ and $|\Delta|>2$ or $|\Sigma|\leq 2$ and $|\Delta|\leq 2$. The first case, however, contradicts the crucial condition \eqref{NI}.

Once we have established the bound $|\Delta|\leq 2$ and the value(s) of $|\Psi|$ has been found, we are free to use the Decomposition formula to obtain the values of $s_1$, $s_2$, $s_3$ and $s_4$. Clearly, we can set $s_1$ with the $+$ sign, while $s_2$ with the $-$ sign as in formula \eqref{s12}, up to equivalence. However, we do not know a priori how to distribute the signs amongst $s_3$ and $s_4$, and to simplify the notations we define
\beql{S34}
s_3=-\frac{\Delta}{2}\pm\mathbf{i}\frac{\Delta}{|\Delta|}\sqrt{1-\frac{|\Delta|^2}{4}},\ \ \ \ 
s_4=-\frac{\Delta}{2}\mp\mathbf{i}\frac{\Delta}{|\Delta|}\sqrt{1-\frac{|\Delta|^2}{4}}.
\eeq
In particular, by using \eqref{SH} and \eqref{S34} we find that the orthogonality equation \eqref{ORT} becomes
\beql{FIN}
\frac{4-|\Sigma|^2-|\Delta|^2-|\Psi|^2}{\Sigma\overline{\Delta}}+\frac{\overline{\Sigma}\Delta}{2}\pm2\frac{\overline{\Sigma}\Delta}{|\Sigma||\Delta|}\sqrt{1-\frac{|\Sigma|^2}{4}}\sqrt{1-\frac{|\Delta|^2}{4}}=0,
\eeq
where the $\pm$ sign agrees with the definition of $s_3$. To conclude the theorem plug in all of the possible values of $|\Psi|$ as described in \eqref{ps1}, \eqref{ps2}, \eqref{ps3} into \eqref{FIN} to verify that for some sign choice it holds identically.
\end{proof}
\begin{rem}
Both of the two possible signs described by formula \eqref{ps3} can be realized, as the following two examples demonstrate:
\[\left[
\begin{array}{cccccc}
 1 & 1 & 1 & 1 & 1 & 1 \\
 1 & -\omega^2 & \omega & -1 & \omega^2 & -\omega \\
 1 & -\omega & \omega^2 & -1 & \omega & -\omega^2
\end{array}
\right],\ \ \ \left[
\begin{array}{cccccc}
 1 & 1 & 1 & 1 & 1 & 1 \\
 1 & -\omega^2 & \omega & -1 & \omega^2 & -\omega \\
 1 & -\omega^2 & \omega^2 & -\omega & \omega & -1
\end{array}
\right].\]
In particular, there are two different orthogonal triplet of rows of order $6$, composed from sixth roots of unity, such that $|\Sigma|=|\Delta|=\sqrt{3}$ in both cases, however, in one of the cases $|\Psi|=1$ while $|\Psi|=2$ in the other.
\hfill$\square$\end{rem}
It is natural to ask if condition \eqref{C11} implies the inequality \eqref{NI} automatically. While we tend to believe that this is so, proving such a statement would require a more sophisticated understanding of the geometry of orthogonal vectors of modulus $1$. We pose this as a relevant
\begin{pro}
Investigate if condition \eqref{C11} implies the inequality \eqref{NI} automatically in Theorem \ref{HTIO}.
\end{pro}
We have already seen the following statement formulated for $5\times 5$ matrices (cf.\ Lemma \ref{ch2HT}). We state it again tailored especially for the $6\times 6$ case.
\begin{cor}[Haagerup's trick, \cite{UH1}]\label{HT}
Suppose that $(1,1,1,1,1,1)$, $(1,a,b,e,s_1,s_2)$ and $(1,c,d,f,s_3,s_4)$ are three pairwise orthogonal rows of a complex Hadamard matrix of order $6$. Then
\beql{HEQ}
\mathcal{H}(a,b,c,d,e,f)\in\mathbb{R}.
\eeq
\end{cor}
As we have pointed out earlier, Haagerup used the property \eqref{HEQ} to give a complete characterization of complex Hadamard matrices of order $5$, or equivalently, describe the orthogonal maximal abelian $\ast$-subalgebras of the $5\times 5$ matrices \cite{UH1}. Since then it was used in \cite{BN1} and \cite{MSZ} to construct new, previously unknown complex Hadamard matrices of order $6$ as well. However, to guarantee the mutual orthogonality of three rows the necessary condition \eqref{HEQ} should be replaced by the more informative identity \eqref{C11}. Nevertheless, \eqref{HEQ} will play an essential r\^ole in our construction too. These type of identities are extremely useful as they feature less variables than the standard orthogonality equations considerably simplifying the calculations required.

Thus, we have given an algebraic characterization of the orthogonality of triplets of rows in complex Hadamard matrices of order $6$. This, however, still give us six degrees of freedom instead of the desired $4$, and it is not entirely trivial how to get rid of two of the six parameters.

The other essential tool of our construction is motivated by the following well-known theorem in functional analysis (see \cite{VP1}). We recall that an operator $A$ is said to be a contraction, if its operator norm satisfies $\left\|A\right\|\leq1$.
\begin{thm}[Unitary dilation]
Every contraction $A$ has a unitary dilation.
\end{thm}
\begin{proof}
Indeed, let us define the so-called defect operator
\[D_A:=(I-A^\ast A)^{\frac{1}{2}},\]
which is positive where the square root is defined via the continuous functional calculus. Then, the desired unitary matrix reads
\[U=\left[\begin{array}{cc}
A & D_{A^\ast}\\
D_A & -A^\ast\\
\end{array}\right].\qedhere\]
\end{proof}
Now we combine the previous two ideas as follows. We start with a submatrix
\beql{E}
E(a,b,c,d):=\left[\begin{array}{ccc}
1 & 1 & 1\\
1 & a & b\\
1 & c & d\\
\end{array}\right]
\eeq
and attempt to embed it into a complex Hadamard matrix of order $6$
\beql{G6}
G_6^{(4)}(a,b,c,d):=\left[\begin{array}{ccc|ccc}
1 & 1 & 1 & 1 & 1 & 1\\
1 & a & b & e & s_1 & s_2\\
1 & c & d & f & s_3 & s_4\\
\hline
1 & g & h & \ast & \ast & \ast\\
1 & t_1 & t_3 & \ast & \ast & \ast\\
1 & t_2 & t_4 & \ast & \ast & \ast\\
\end{array}\right]\equiv\left[\begin{array}{cc}
E & B\\
C & D\\
\end{array}\right]
\eeq
with $3\times 3$ blocks $E,B,C$ and $D$ in two steps, as follows. First we construct the submatrices $B$ and $C$ featuring unimodular entries to obtain three orthogonal rows and columns of $G_6$. Secondly we find the unique lower right submatrix $D$ to get a unitary matrix. Should the entries of this matrix become unimodular, we have found a complex Hadamard matrix. We conjecture that the submatrix $E$ can be chosen, up to equivalence, in a way that there will be only finitely many candidates for the blocks $B$ and $C$ and therefore we can ultimately decide throughout a finite, case-by-case analysis whether the submatrix $E$ can be embedded into a complex Hadamard matrix. The resulting matrix $G_6$ can be thought as the ``Hadamard dilation'' of the operator $E$.

Now we turn to the details of our construction.
\subsection{Further preliminary results}
In this section we present various other results necessary to the construction of the family $G_6^{(4)}$, outlined previously.

Solving the system of equations \eqref{C11}--\eqref{HEQ} is the key step to obtain the submatrices $B$ and $C$ of $G_6$, and once we have three orthogonal rows and columns we readily fill out the remaining lower right submatrix  $D$. This is explained by the following two lemmata, the first of which being a special case of a more general matrix inversion
\begin{lem}\label{MIL}
If $U$ and $V$ are $n\times n$ matrices then \[\left(I_n+UV\right)^{-1}=I_n-U\left(I_n+VU\right)^{-1}V\]
provided that one of the matrices $I_n+UV$ or $I_n+VU$ is nonsingular.
\end{lem}
\begin{proof}
By symmetry, we can suppose that the matrix $I_n+VU$ is nonsingular. Then, we have
\[(I_n+UV)(I_n-U(I_n+VU)^{-1}V)=I_n+UV-U(I_n+VU)(I_n+VU)^{-1}V=I_n.\qedhere\]
\end{proof}
\begin{lem}\label{DDD}
Suppose that we have a $6\times 6$ partial complex Hadamard matrix consisting of three orthogonal rows and columns, containing no vanishing $3\times 3$ minor. Then there is a unique way to construct a $6\times 6$ unitary matrix containing these rows and columns as a submatrix.
\end{lem}
\begin{proof}
Let $U$ be a $6\times 6$ matrix with $3\times 3$ blocks $A,B,C$ and $D$, as the following:
\[U=\left[\begin{array}{cc}
A & B \\
C & D\\
\end{array}\right].\]
By the orthogonality of the first three rows and columns and using the fact that the entries are unimodular, we have
\beql{28}
AA^\ast+BB^\ast=6I_3
\eeq
\beql{29}
A^\ast A+C^\ast C=6I_3
\eeq
To ensure orthogonality in-between the first three and the last three rows we need to have $AC^\ast+BD^\ast=0$. As $B$ is nonsingular by our assumptions we can define
\beql{DD}
D:=-CA^\ast(B^{-1})^\ast.
\eeq
Now we need to show that the last three rows are mutually orthogonal as well. Indeed, by using \eqref{DD} and \eqref{28} we have
\[CC^\ast+DD^\ast=C\left(I_3+A^\ast(BB^\ast)^{-1}A\right)C^\ast=C\left(I_3+A^\ast\left(6I_3-AA^\ast\right)^{-1}A\right)C^\ast,\]
which, by Lemma \ref{MIL} and \eqref{29} equals to
\[C\left(I_3+\frac{1}{6}A^\ast\left(I_3-\frac{1}{6}AA^\ast\right)^{-1}A\right)C^\ast=C\left(I_3-\frac{1}{6}A^\ast A\right)^{-1}C^\ast=6C\left(C^\ast C\right)^{-1}C^\ast=6I_3.\qedhere\]
\end{proof}
We do not state that the obtained unitary matrix $U$ is Hadamard, which is not true in general. Recall that our goal is to embed the submatrix $E$ into the matrix $G_6$ (see \eqref{E}-\eqref{G6}). We have the following trivial
\begin{lem}
Suppose that a submatrix $E$ can be embedded into a complex Hadamard matrix $G_6$ of order $6$ in which the upper right submatrix $B$ is invertible. Then for some unimodular submatrix $C$ for which the first three columns of $G_6$ are orthogonal the lower right submatrix $D=-CE^\ast(B^{-1})^\ast$ is unimodular.
\end{lem}
\begin{proof}
Indeed, this is exactly what embedding means.
\end{proof}
In particular, if the submatrices $B$ and $C$ are chosen carefully, the unimodular property of $D$ follows for free.
\begin{cor}\label{embC}
Start from a submatrix $E$ and suppose that there are only finitely many invertible candidate matrices $B\in SOL_B$ and $C\in SOL_C$ such that the first three rows and columns of the matrix $G_6$ are orthogonal. Then $E$ can be embedded into a complex Hadamard matrix of order $6$ if and only if there is some $B\in SOL_B$ and $C\in SOL_C$ such that the matrix $D=-CE^\ast(B^{-1})^\ast$ is unimodular.
\end{cor}
Note that due to the finiteness condition in Corollary \ref{embC} once we have all (and only finitely many) candidate matrices $B$ and $C$ we can decide algorithmically whether the submatrix $E$ can be embedded into a complex Hadamard matrix.

The next step is to characterize $6\times 6$ complex Hadamard matrices with vanishing $3\times 3$ minors. To do this we need two auxiliary results first.
\begin{lem}\label{AUX1}
Suppose that in a dephased $6\times 6$ complex Hadamard matrix there exists a noninitial row $($or column$)$ containing three identical entries $x$. Then $x=\pm 1$ and this row $($or column$)$ reads $(1,1,1,-1,-1,-1)$, up to permutations.
\end{lem}
\begin{proof}
Suppose to the contrary, that there are three nonreal numbers $x$ in a row (column). Then the sum of these numbers $x$ together with the leading $1$ read $|1+3x|^2=10+6\Re[x]>4$, and hence, by Lemma \ref{L1} this row (column) cannot be orthogonal to the first row (column), a contradiction. To ensure orthogonality to the first row (column) we should specify the remaining three entries to $-x$ and hence the last part of the statement follows.
\end{proof}
Now by combining Lemma \ref{AUX1} with Theorem \ref{K6} we arrive at the following
\begin{cor}\label{C213}
Suppose that in a dephased $6\times 6$ complex Hadamard matrix $H$ there exists a noninitial row $($or column$)$ containing three identical entries. Then $H$ belongs to the family $K_6^{(3)}$.
\end{cor}
Now we can state the desired structural result.
\begin{prop}\label{Vanish}
Suppose that a $6\times 6$ complex Hadamard matrix $H$ has a vanishing $3\times 3$ minor. Then $H$ belongs to the family $K_6^{(3)}$.
\end{prop}
\begin{proof}
Suppose that $H$ has a vanishing $3\times 3$ minor, say the upper left submatrix $E(a,b,c,d)$, as in formula \eqref{E}. Such an assumption can be made, up to equivalence. As $\mathrm{det}(E)=b+c-a-d+ad-bc=0$, we find that if any of the indeterminates $a,b,c,d$ is equal to $1$ then $E$ contains a noninitial row (or column) containing three $1$s, and therefore by Corollary \ref{C213} we conclude that this matrix belongs to the family $K_6^{(3)}$. Otherwise, we can suppose that none of $a,b,c,d$ is equal to $1$ and hence we find that $d=(a+b c-b-c)/(a-1)$, which should be of modulus one. To ensure this, solve the equation $d\overline{d}-1=0$ to find that either $b=a$ or $c=a$ should hold, but then we have either $d=c$ or $d=b$ as well. In particular, we find that $E$ has two identical columns (or rows). Therefore $E$ can be transformed, by appropriate multiplication of the rows (or columns) by unimodular numbers, into a matrix $E'$ which has two column (or row) of entries $1$, and a reference to Corollary \ref{C213} concludes the lemma.
\end{proof}
Therefore to investigate those matrices which lie outside the family $K_6^{(3)}$ we can safely use Lemma \ref{DDD} and in particular the inversion formula \eqref{DD}.

It turns out, that the isolated matrix $S_6^{(0)}$ (see Example \ref{ch2s6}) requires a special treatment as well. It is featured in the following
\begin{lem}\label{cubic}
Suppose that in a $6\times 6$ dephased complex Hadamard matrix $H$ there is a noninitial row or column composed of cubic roots of unity. Then $H$ is either equivalent to $S_6^{(0)}$ or belongs to the family $K_6^{(3)}$.
\end{lem}
Recall that we have denoted by $\omega$ the principal cubic root of unity.
\begin{proof}
We can assume that the core of $H$ does not contain a $-1$, otherwise, by Theorem \ref{K6}, we are done. Let us suppose that $H$ contains a full row of cubic entries, and its second and third row reads $(1,1,\omega,\omega,\omega^2,\omega^2)$ and $(1,a,b,c,d,e)$, respectively. Then, by considering the orthogonality equations in-between the first three rows of $H$, we easily obtain that $1+a-b\omega-c\omega=0$ implying (as $a\neq -1$) that $b=\omega^2$ or $c=\omega^2$. Now notice that this means that in any further row of $H$ one of the third or fourth entries must be $\omega^2$. However, this implies that both the third and fourth column of $H$ is composed of purely cubic entries. Similarly, one can deduce the equation $1+a-d\omega^2-e\omega^2=0$ and apply the same argument \emph{mutatis mutandis} to the fifth and sixth columns of $H$ as well. To conclude observe that from the aforementioned it follows that the matrix $H$ is composed entirely of cubic roots of unity and hence it is equivalent to $S_6^{(0)}$ as desired. The case when $H$ contains a full column of cubic entries is analogous.
\end{proof}
Now we turn to the presetting of the submatrix $E(a,b,c,d)$ (see \eqref{E}). In order to avoid the case when the system of equations \eqref{C11}--\eqref{HEQ} is linearly dependent we need to exclude various input quadruples $(a,b,c,d)$. However, it shall turn out that we are free to do such restrictions, up to equivalence. To simplify the forthcoming terminology let us define the following two-variable function mapping $\mathbb{T}^2$ to $\mathbb{C}$ as follows:
\[\mathcal{E}(x,y):=x+y+x^2+y^2+xy^2+x^2y.\]
We say that $y$ is an elliptical pair of $x$, if $\mathcal{E}(x,y)=0$. Observe that for a given $x\neq -1$ the sum of its elliptical pairs read $y_1+y_2=-(1+x^2)/(1+x)$. The following is a strictly technical
\begin{prop}[Canonical transformation]\label{CTr}
Suppose that we have a complex Hadamard matrix $H$ inequivalent from $S_6^{(0)}$ and any of the members of the family $K_6^{(3)}$. Then $H$ has a $3\times 3$ submatrix $E(a,b,c,d)$ as in formula \eqref{E}, up to equivalence, satisfying
\beql{canCC}
(b-1)(c-1)(b-d^2)(c-d^2)(b-c)(bc-d)\mathcal{E}(b,d)\mathcal{E}(c,d)\neq 0.
\eeq
\end{prop}
\begin{proof}
The strategy of the proof is the following: first we pick a ``central element'' $d$ from the core of the matrix and then we show that $b$ and $c$ can be set satisfying \eqref{canCC}. Recall, that by Lemma \ref{cubic} there is no noninitial row or column composed from cubic roots of unity in the matrix.

First let us assume that there is a number $1$ in the core somewhere; set $d=1$, and choose a non-cubic $c$ from its row. Note that there cannot be a further noninitial $1$ in the row or column of $d$ by Corollary \ref{C213}. Now observe that as the elliptical pairs of $1$ are $\omega$ and $\omega^2$, we can choose a suitable $b$ from the column of $d$ unless all entries there are members of the set $\{\omega, \omega^2,c,\overline{c}\}$. Note that $\omega$ together with $\omega^2$ cannot be in the column of $d$ at the same time, and from this it is easily seen that we can define the value of $b$ to satisfy \eqref{canCC} unless the column of $d$ is one of the following four cases, up to permutations: $(1,1,\omega,c,c,\overline{c})$, $(1,1,\omega,c,\overline{c},\overline{c})$, $(1,1,\omega^2,c,c,\overline{c})$, $(1,1,\omega^2,c,\overline{c},\overline{c})$. However, by normalization and by orthogonality, the sum of the entries in this column should add up to $0$, and we find in all cases that the unimodular solution to $c$ is a cubic root of unity, contradicting the choice of $c$. Therefore one of the entries in the column of $d$ is different from $\omega,\omega^2,c,\overline{c}$ which will be chosen as $b$.

Secondly, let us suppose, that the number $1$ is not present in the core. In particular, all entries in the core are different rowwise and columnwise.

Pick an arbitrary number $d$ from the core of the matrix, and let us denote by $c_1$ and $c_2$ its the elliptical pairs (maybe $c_1=c_2$, or they are undefined). Now we have several cases depending on the appearance of these values in the row and column containing $d$.

CASE $1$: $c_1$ and $c_2$ present in both the row and column containing $d$. Hence, in this row and column there are the entries $1$, $d$, $c_1$ and $c_2$ already and the remaining two, $\alpha$ and $\beta$, are uniquely determined by the Decomposition formula. Note that $\alpha\neq d^2$, as otherwise we would have $\beta=-(2d+d^2+d^3)/(1+d)$ (as the sum of all entries in a noninitial row add up to $0$, and the sum of the elliptical pairs is known), which is unimodular if and only if $d=\pm\mathbf{i}$ or $d=\omega$ or $d=\omega^2$. In the first case we have $\beta=\mp\mathbf{i}$ and we are dealing with a member of the family $K_6^{(3)}$. The second and third case imply that we have a full row and a full column of cubics, a contradiction. If $\beta\neq d/\alpha$ then by picking $c=\alpha$ from the row we can set $b=\beta$ in the column. Otherwise, in case we have $\beta=d/\alpha$, then reset the central element $d$ to the number $\alpha$ which is in the same row as $d$. Now observe that after this exchange we should meet the requirements of Case $1$ again (otherwise we are done here), and hence the elliptical pairs of $\alpha$, $\alpha_1$ and $\alpha_2$, should present in the row and column of $\alpha$, which therefore contain exactly the same entries. However $\alpha_{1,2}\neq d$, as otherwise orthogonality, together with the condition $\mathcal{E}(\alpha,d)=0$ would imply $d=\pm1$, a contradiction; $\alpha_{1,2}\neq d/\alpha$, as otherwise $d=\omega$ or $d=\omega^2$ would follow from the same argument implying that the row containing $\alpha$ has a noninitial $1$, a contradiction. Therefore, the only option left is that $\alpha_{1,2}=c_{1,2}$. But then we can set $c=d\neq\alpha^2$ and $b=d/\alpha\neq \alpha^2$, and we are done. 

CASE $2$: $c_1$ and $c_2$ present in (to say) the row containing $d$, but only one of these values (say $c_1$) is present in the column of $d$. Let us denote by $\alpha$ and $\beta$ the two further entries in this row which, again, are different from $d^2$. In the column of $d$ there are already the numbers $1$, $d$ and $c_1$, and observe that the remaining triplet of entries cannot be $(d^2,\alpha,\beta)$ as this would imply $d^2=c_2$, contradicting our case-assumption. Therefore one of the three unspecified entries $\gamma$ is different from $d^2$, $\alpha$ and $\beta$. Now if $\gamma\neq d/\alpha$ then set $c=\alpha$, $b=\gamma$ otherwise set $c=\beta$, $b=\gamma$. We are done.

CASE $3$: Only the value $c_1$ is in (to say) the row containing $d$ and in the column of $d$ as well. This is a tricky case, as it might happen that the undetermined triplet in both the $d$th row and column is precisely $(d^2, \alpha, d/\alpha)$, and therefore we cannot ensure condition \eqref{canCC}. However, from the orthogonality equation $1+d+d^2+c_1+\alpha+d/\alpha=0$ and from its conjugate we can eliminate the variable $c_1$ in order to obtain the equation $(1+\alpha)(d+\alpha)(c_1-\overline{d})=0$ leading us either to the family $K_6^{(3)}$ or to $c_1=\overline{d}$. In the latter case, however, we can calculate the values of $d$ and $\alpha$ explicitly by invoking the elliptical condition $\mathcal{E}(c_1,d)=0$. In particular, we find that the values of $d$ and $\alpha$ are given by some of the unimodular roots of the following polynomials $1+2d+2d^3+d^4=0$ and $1+4\alpha-2\alpha^2-8\alpha^3-8\alpha^4-8\alpha^5-2\alpha^6+4\alpha^7+\alpha^8=0$.
Now reset the ``central'' entry $d$ to $\alpha$, and observe that the elliptical pairs of $\alpha$ are not present in its row, and therefore the conditions of this subcase are no longer met. Otherwise we can suppose that the undetermined triplet in the column of $d$ is not $(d^2,\alpha,d/\alpha)$. Set $c=\alpha\neq d^2$. Pick $\gamma$ from the column which is different from $d^2$, $\alpha$, $d/\alpha$, set $b=\gamma$ and we are done.

CASE $4$: Only the value $c_1$ is in (to say) the row containing $d$ and in the column of $d$ there is the other elliptical value $c_2\neq c_1$. We can suppose that two of the undetermined entries in the row of $d$ satisfy $\alpha\neq d^2$ and $\beta\neq d^2$ and set $c=\alpha$. Now if in the column the undefined triplet is precisely $(\alpha$, $d/\alpha$, $d^2)$, then observe that the same triplet cannot appear in the row, as otherwise $c_1=c_2$ would follow. Therefore we can reset $c$ to a value different from $\alpha, d/\alpha, d^2$, and set $b=\alpha$. Otherwise there is an entry in the column which can be set to $b$, we are done.

CASE $5$: In the column of $d$ there are no elliptical values at all. Pick any $c=\alpha\neq d^2$ from the row of $d$ which is different from both $c_1$ and $c_2$. As in the column of $d$ there are four unspecified entries, one of them, say $\gamma$ will be different from $d^2$, $\alpha$ and $d/\alpha$. Set $b=\gamma$. We are done.
\end{proof}
However, not every submatrix $E$ can be embedded into a complex Hadamard matrix of order $6$ as we shall see shortly. Let us denote by $\left\|.\right\|_2$ both the Euclidean norm on $\mathbb{C}^6$ and the induced operator norm on the space of $6\times 6$ matrices. We have the following necessary condition.
\begin{lem}\label{Eig}
Suppose that $A$ is a $3\times 3$ submatrix of a complex Hadamard matrix $H$ of order $6$. Then $A/\sqrt{6}$ is a contraction.
\end{lem}
\begin{proof}
Clearly, we can assume that this submatrix $A$ is the upper left of the matrix $H$, which we will write in block form, as follows:
\[H=\left[\begin{array}{cc}
A & B\\
C & D\\
\end{array}\right].\]
Now suppose, to the contrary that there is some vector $s$, such that $\left\|As\right\|_2>\sqrt6\left\|s\right\|_2$ and consider the block vector $s':=(s,0)^T\in\mathbb{C}^6$. We have
\[\left\|Hs'\right\|_2=\left\|(As,Cs)^T\right\|_2\geq\left\|(As,0)^T\right\|_2=\left\|As\right\|_2>\sqrt6\left\|s\right\|_2=\sqrt6\left\|s'\right\|_2=\left\|Hs'\right\|_2,\]
where in the last step we used that the matrix $H/\sqrt6$ is unitary.
\end{proof}
\begin{cor}\label{CE}
Suppose that the submatrix $E$ can be embedded into a complex Hadamard matrix of order $6$. Then every eigenvalue $\lambda$ of the matrix $E^\ast E$ satisfies $\lambda\leq 6$.
\end{cor}
Corollary \ref{CE} is a useful criterion to show that a given matrix $E$ cannot be embedded into a complex Hadamard matrix, however, it is unclear how to utilize it for our purposes. In particular, we do not know how to characterize those $3\times 3$ matrices which satisfy its conditions. Also it is natural to ask whether the presence of a large eigenvalue is the only obstruction forbidding the submatrix $E$ to be embedded. The answer to this question might depend on the dimension, as it is easily seen that while every $2\times 2$ matrix can be embedded into a complex Hadamard matrix of order $4$, only a handful of very special $2\times 2$ matrices can be embedded into a complex Hadamard matrix of order $5$ due to the finiteness result of Haagerup \cite{UH1}; compare Example \ref{ch1ExF4i} with Theorem \ref{ch2haaf5}.

A related, yet seemingly weaker restriction is offered by Lindsey's
\begin{lem}[Lindsey, \cite{ESx1}]
Suppose that an $s\times t$ matrix $A$ appears as a submatrix within an $n\times n$ complex Hadamard matrix $H$. Then
\[\left|\sum_{i=1}^s\sum_{j=1}^t A_{ij}\right|\leq \sqrt{stn}.\]
\end{lem}
\begin{proof}
We can suppose, up to equivalence, that $A$ is in the first $s$ rows and first $t$ columns of $H$. Let us denote by $u$ and $v$ the vectors $(1^s,0^{n-s})$ and $(1^t,0^{n-t})^T$, respectively. Then we find, after an application of the Cauchy--Schwarz inequality, that
\[\left|\sum_{i=1}^s\sum_{j=1}^t A_{ij}\right|=\left|\left\langle u,Hv\right\rangle\right|\leq\left\|u\right\|_2\left\|Hv\right\|_2\leq\left\|u\right\|_2\left\|v\right\|_2\left\|H\right\|_2= \sqrt{stn},\]
where we have used that the norm of $H$ reads $\left\|H\right\|_2=\sqrt{\lambda_{\max}\left(H^\ast H\right)}=\sqrt{\lambda_{\max}\left(nI_n\right)}$, where $\lambda_{\max}(.)$ denotes the largest eigenvalue.
\end{proof}
Now we are ready to present a new family of complex Hadamard matrices. The next section gives an overview of the results.
\subsection{The construction: A high-level perspective}\label{HLev}
Here we describe the generic family $G_6^{(4)}$ from a high-level perspective. In particular, we outline the main steps only, and do not discuss some degenerate cases which might come up during the construction. We shall investigate the process in details later in Section \ref{ch2tcid}. The main result of this chapter is the following
\begin{cons}[The Dilation Algorithm]\label{mC}
Do the following step by step to obtain complex Hadamard matrices of order $6$.\normalfont
\begin{enumerate}[\#1: ]
\item \ttfamily{INPUT}\normalfont: the quadruple $(a,b,c,d)$, forming the upper left $3\times 3$ submatrix $E(a,b,c,d)$ as in formula \eqref{E}.
\item Use Haagerup's trick to the first three rows of $G_6^{(4)}$ (see \eqref{G6}) to obtain a quadratic equation to $f$:
\beql{Flin}
\mathcal{F}_1+\mathcal{F}_2f+\mathcal{F}_3f^2=0,
\eeq
where the coefficients $\mathcal{F}_1$, $\mathcal{F}_2$ and $\mathcal{F}_3$ depend on the parameters $a,b,c,d$ and the indeterminate $e$, and derive the following linearization formula from it:
\beql{lin}
f^2=-\frac{\mathcal{F}_1}{\mathcal{F}_3}-\frac{\mathcal{F}_2}{\mathcal{F}_3}f.
\eeq
\item Use Theorem \ref{HTIO} to obtain another quadratic equation to $f$:
\beql{Glin}
\mathcal{G}_1+\mathcal{G}_2f+\mathcal{G}_3f^2=0,
\eeq
where, again, the coefficients $\mathcal{G}_1$, $\mathcal{G}_2$ and $\mathcal{G}_3$ depend on the parameters $a,b,c,d$ and the indeterminate $e$. Plug the linearization formula \eqref{lin} into \eqref{Glin} and rearrange to obtain the companion value of $e$, $f=F(e)$, where
\beql{formF}
F(e):=-\frac{\mathcal{F}_3\mathcal{G}_1-\mathcal{F}_1\mathcal{G}_3}{\mathcal{F}_3\mathcal{G}_2-\mathcal{F}_2\mathcal{G}_3}.
\eeq
\item As $|f|=1$ should hold, calculate the sextic polynomial $\mathcal{P}_{a,b,c,d}(e)$ coming from the equation $F(e)\overline{F(e)}-1= 0$ and solve it for $e$.
\item Amongst the roots of $\mathcal{P}_{a,b,c,d}(e)$ find all unimodular triplets $(e,s_1,s_2)$ satisfying $e+s_1+s_2=-1-a-b$, calculate the companion values $f=F(e)$, $s_3=F(s_1)$ and $s_4=F(s_2)$ through formula \eqref{formF} and store all sextuples $(e,s_1,s_2,f,s_3,s_4)$ in a solution set called $SOL_B$.
\item Repeat Steps $\#2$-$\#5$ to the transposed matrix (i.e.\ to the first three columns), \emph{mutatis mutandis} to obtain the solution set $SOL_C$.
\item For every pair of sextuples from $SOL_B$ and $SOL_C$ construct the submatrices $B$ and $C$, check if the first three rows and columns of $G_6^{(4)}$ are mutually orthogonal and finally use Lemma \ref{DDD} to compute the lower right submatrix $D$ through formula \eqref{DD}.
\item \ttfamily{OUTPUT:}\normalfont all unimodular matrices found in Step $\#7$.
\end{enumerate}
\end{cons}
Construction \ref{mC} gives the essence of the new family discovered, and in Section \ref{ch2tcid} we shall give it a mathematically rigorous, low-level look. Before doing so, we invite the reader to join us to a rapid course in symbolic computation.
\subsection{A primer on Gr\"obner basis techniques}
Solving polynomial systems is of fundamental interest in mathematics. However, unless we are facing with very simple academic problems in two or three variables, we cannot solve these systems by hand in a reasonable time, and therefore in real-life situations we heavily rely on computers. In $1965$ Buchberger in his PhD thesis introduced the concept of Gr\"obner basis, and proposed an algorithm for deciding ideal membership, which is widely used today \cite{BB1}.

Here we give an informal introduction to Gr\"obner basis and illustrate some of the standard tricks of symbolic computation which shall come handy during the subsequent sections. The interested reader is referred to Lazard's bulletin which we consider as the perfect starting point to the subject \cite{DL1}.

During this thesis we shall encounter polynomial systems $F=\{f_1,f_2,\hdots,f_n\}$ in $m$ variables $x_1,x_2,\hdots, x_m$ and our goal is to find all common solutions, i.e.\ all $m$-tuples $y=(y_1,y_2,\hdots, y_m)$ for which $f_1(y)=f_2(y)=\hdots=f_n(y)=0$. If there are infinitely many solutions, then some of the solution vectors $y$ shall depend on some free parameters. By computing a Gr\"obner basis for $F$ we essentially create a new system $G=\{g_1,g_2,\hdots, g_k\}$ such that the polynomial ideals, generated by $F$ and $G$ are the same, yet $G$ has some further useful properties. The resulting Gr\"obner basis and its ``complexity'' greatly depends on the monomial ordering used. We shall simply use the most informative, yet the most computational expensive pure lexicographic ordering. For example, if $F$ generates a zero dimensional ideal, i.e.\ the system $F$ has finitely many solutions only, then a pure lexicographic Gr\"obner basis $G$ will be (essentially) a triangular system where $g_i=g_i(x_1,x_2,\hdots, x_i)$, $i=1,2,\hdots, m$. As a result, all solutions $y=(y_1,y_2,\hdots, y_m)$ can be extracted by solving the univariate polynomial $g_1(x_1)$ first, and then iteratively $g_2(y_1,x_2)$, etc.\ up to $g_m(y_1,y_2,\hdots,y_{m-1},x_m)$, where $g_1(y_1)=g_2(y_1,y_2)=\hdots=g_m(y_1,y_2,\hdots,y_m)=0$.

Our primary goal is to construct new examples of complex Hadamard matrices by means of computer algebra, and hence we are interested in the unimodular solutions of $F$ only. Unfortunately, due to the lack of the notion of complex conjugation one cannot formalize the unimodular conditions $|x_i|^2-1=0$, $i=1,\hdots, m$ as polynomials and include them into $F$. This means that in general we need to separate the unimodular solutions from the large number of complex solutions by hand. It is possible to slightly improve on this situation by observing that the conjugate of unimodular numbers is their reciprocal, and hence the conjugate of a multivariate polynomial $f$ with real coefficients, depending on the unimodular indeterminates $x_1,x_2,\hdots, x_m$ is just the rational function $h=f(1/x_1,1/x_2,\hdots, 1/x_m)$, where both the numerator and the denominator of $h$, denoted by $s$ and $t$, are polynomials. As we are interested in the common solutions of these systems we can include $s$ into $F$, but should exclude those solutions by hand which would result in $t=0$. One can circumvent this inconvenience by introducing a new, ``dummy'' variable $u$, and include $ut-1$ to $F$ as well. Clearly in this way we could encode that $t$ is nonvanishing.

In summary, instead of solving $F=\{f_1,f_2,\hdots, f_n\}$ in the variables $x_1,x_2,\hdots, x_m$ we rather consider $F'=\{f_1,s_1,f_2,s_2,\hdots,f_n,s_n,ut_1t_2\cdots t_n-1\}$ in the variables $u,x_1$, $x_2,\hdots, x_m$, where $f_i(1/x_1,1/x_2,\hdots, 1/x_m)=s_i(x_1,x_2,\hdots,x_m)/t_i(x_1,x_2,\hdots, x_m)$, $i=1,2,\hdots,n$. Experiments confirm that the solution set of $F'$ is considerably smaller than the solution set of $F$.

The reader might have some doubts why is it useful to include the conjugate of the polynomials as well into the systems we consider. To illustrate these ideas more transparently, we offer him or her the opportunity to study the following
\begin{ex}\label{ch2rapid}
Here we solve the polynomial system $F_1=\{f_1\}$, containing a single polynomial $f_1=x+y+y^2$, where $x,y\in\mathbb{T}$ by means of Gr\"obner basis. Of course, as we are dealing with a single polynomial calculating a Gr\"obner basis results in the trivial system $G_1=\{x+y+y^2\}$, as expected. However, we can improve on this by including the conjugate of $f_1$ (or rather, the numerator of its conjugate) as well, by considering \[h_1(x,y)=f_1(1/x,1/y)=\frac{1}{x}+\frac{1}{y}+\frac{1}{y^2}=\frac{x+xy+y^2}{xy^2}=:\frac{s_1(x,y)}{t_1(x,y)}\]
and the corresponding system $F_2=\{f_1,f_2\}$, where $f_2=s_1$. Calculating a Gr\"obner basis for $F_2$ shall result in $G_2=\{g_1,g_2\}=\{y+y^2+y^3,x+y+y^2\}$ which is a triangular system. By considering $g_1$ we find that $y_1=0, y_2=\omega, y_3=\omega^2$, and then, by considering $g_2$ we find that $(x_1,y_1)=(0,0)$, $(x_2,y_2)=(1,\omega)$ and $(x_3,y_3)=(1,\omega^2)$, where $\omega$ is the complex cubic root of unity. As we would like to avoid the solutions where either of the variables is $0$, we include $f_3=uxy-1$ into $F_2$ (strictly speaking, $f_3=ut_1-1=uxy^2-1$ should have been included, but it is always a good idea to keep the degree of the polynomials to the absolute minimum, and hence we just ignore the extra power of $y$) and consider $F_3=\{f_1,f_2,f_3\}$ in the variables $u,x,y\in\mathbb{T}$. We find that $G_3=\{1+y+y^2,x-1,1+u+y\}$ and by solving this triangular system we find that $(x_1,y_1,u_1)=(1,\omega,\omega^2)$ and $(x_2,y_2,u_2)=(1,\omega^2,\omega)$.
\end{ex}
After realizing that considering the conjugate equations as well is certainly a rewarding idea, the reader is ready to advance to the next section where the Dilation Algorithm is discussed in details. During our construction we have used various Gr\"obner bases techniques \cite{BB1} similar in spirit to \cite{BFx1}, \cite{JF1} and \cite{MG1}, and performed the required calculations with the aid of MAPLE and {\it Mathematica}. The reader is advised to use a computer algebra system of his or her choice for bookkeeping purposes.
\subsection{The construction in details}\label{ch2tcid}
Here we investigate the steps of Construction \ref{mC} in details. During the construction we shall reject all those matrices which turn out to be equivalent to $S_6^{(0)}$ or belong to the family $K_6^{(3)}$.

\smallskip
\noindent STEP \#1: Choose a quadruple $(a,b,c,d)$ in compliance with the Canonical Transformation described by Proposition \ref{CTr} as the \ttfamily{INPUT}\normalfont, and form the submatrix $E$. Check if it meets the requirements of Corollary \ref{CE}; if so, then proceed, otherwise conclude that it cannot be embedded into a complex Hadamard matrix of order $6$. Experimental results show that the four parameters can be chosen independently of each other and lead to a complex Hadamard matrix with positive probability. This means, heuristically, that we indeed get a four-parameter family. Proving such a statement rigorously would require explicit formulas for the matrix entries, which is out of reach due to the appearance of sextic polynomials (see \eqref{ch2fundp}).

\smallskip
\noindent STEP \#2: Consider the equation \[abcdef\left(\mathcal{H}(a,b,c,d,e,f)-\mathcal{H}(1/a,1/b,1/c,1/d,1/e,1/f)\right)=0,\]
and collect the variable $f$ in order to obtain the polynomial \eqref{Flin} whose coefficients read
\begin{gather*}
\mathcal{F}_1=cd(a^2 b + a b^2 + b c + b^2 c + a d + a^2 d) e-cd(a + b + a c + a b c + b d + a b d) e^2;\\
\begin{split}
\mathcal{F}_2=-ab(b c + b c^2 + a d + c^2 d + a d^2 + c d^2)+(a b c + b c^2 + a b d + b c^2 d + a d^2 + a c d^2)e^2+\\
(a^2 b c - b^2 c + b c^2 - a b^2 c^2 - a^2 d + a b^2 d - a c^2 d + 
  b^2 c^2 d + a d^2 - a^2 b d^2 + a^2 c d^2 - b c d^2)e;
\end{split}
\end{gather*}
and $\mathcal{F}_3=-(abcde)^2\overline{\mathcal{F}_1}$.
To obtain formula \eqref{lin} we need to see that $\mathcal{F}_3\neq0$. First we show that $\mathcal{F}_3$, as a polynomial of $e$, is not identically $0$. Indeed, suppose otherwise, which means that the following system of equations (where the first two are the coefficients of $e$ and $e^2$ in $\mathcal{F}_1$, and the last two are the conjugates of the first two, up to some irrelevant constant factors)
\beql{ch2id1}\left\{\begin{array}{ccc}
a^2 b + a b^2 + b c + b^2 c + a d + a^2 d & \equiv & 0,\\
a + b + a c + a b c + b d + a b d & \equiv & 0,\\
b^2 c + a b^2 c + a^2 d + a^2 b d + a c d + b c d & \equiv & 0,\\
a b c + a^2 b c + a b d + a b^2 d + a^2 b c d + a b^2 c d & \equiv & 0,
\end{array}\right.
\eeq
are fulfilled. We compute a Gr\"obner basis and find that the polynomial
\beql{ch2facone}b c d(1 + c^2) (c - d) (1 + d^2) (c^2 + d^2) (1 + d + d^2)\eeq
is a member of the ideal, generated by \eqref{ch2id1}. One can consider each of the factors of \eqref{ch2facone} one by one, and substitute them back into the original equations \eqref{ch2id1} to find that there is either a vanishing sum of order $2$ in $E$ or $a=b=1$ and therefore the whole family is a member of $K_6^{(3)}$, or we have $E=F_3$ or $E=F_3^\ast$ (the Fourier matrix or its adjoint) but these matrices have $b=c$ which however is not allowed by the Canonical transformation. Thus we have shown that if $\mathcal{F}_3\equiv 0$ then we are dealing with the family $K_6^{(3)}$ and therefore we do not proceed any further with our algorithm.

It might happen that $\mathcal{F}_3\not\equiv 0$ but there is a unimodular $e$ making it vanish, which cannot be anything else, but
\beql{FE}
e_0:=\frac{a^2 b+a^2 d+a b^2+a d+b^2 c+b c}{a b c+a b d+a c+a+b d+b}=\frac{ab(c + a c + d + b d + a c d + b c d)}{a^2 b d+a^2 d+a b^2 c+a c d+b^2
   c+b c d}.
\eeq
Nevertheless, in one of the pairs $(e,f)$, $(s_1,s_3)$ and $(s_2,s_4)$ the first coordinate must be different from $e_0$ above, otherwise we would have $e=s_1=s_2$, and therefore by Lemma~\ref{AUX1} $e=s_1=s_2=-1$, resulting in some member of the family $K_6^{(3)}$ by Corollary~\ref{C213}. We may therefore suppose, up to equivalence, that $e\neq e_0$ and conclude that $\mathcal{F}_3\neq 0$.

\smallskip
\noindent STEP \#3:
Multiply \eqref{C11} by $abcdef$ in order to get \eqref{Glin} whose coefficients read
\begin{gather*}
\mathcal{G}_1=cd(2 a b + a^2 b + a b^2 + b c + 2 a b c + b^2 c + a d + a^2 d + 2 a b d)e+2cd(a b + b c + a d)e^2;\\
\begin{split}
\mathcal{G}_2&=2abcd(1 + a + b)+\!(a b c + b c^2 + a b d + 2 a c d + 2 b c d + 2 a b c d + b c^2 d + 
  a d^2 + a c d^2)e^2\\
  \quad&+(2 a b c + a^2 b c + 2 a b^2 c + b c^2 + 2 a b c^2 + 2 b^2 c^2 + 
  2 a b d + 2 a^2 b d+ a b^2 d\\
  \quad& + 2 a c d + 2 a^2 c d + 2 b c d+ 
  12 a b c d + 2 a^2 b c d + 2 b^2 c d + 2 a b^2 c d+ 2 b c^2 d\\
  \quad&  + 
  2 a b c^2 d + b^2 c^2 d + a d^2 + 2 a^2 d^2 + 2 a b d^2 + 
  2 a c d^2 + a^2 c d^2 + 2 a b c d^2)e;
\end{split}\\
\mathcal{G}_3=ab(c + a c + 2 b c + d + 2 a d + b d + 2 c d + a c d + b c d)+2ab(c+d+cd)e.
\end{gather*}
Clearly, one cannot expect to recover a unique $f$ from $e$ in general, as formula \eqref{formF} might suggest. Indeed, there are complex Hadamard matrices in which $s_1=e$, but $s_3\neq f$. The reason for this phenomenon is that formulas \eqref{Flin} and \eqref{Glin} might be linearly dependent. After plugging \eqref{lin} into \eqref{Glin} we obtain the expression
\[\mathcal{F}_3\mathcal{G}_1-\mathcal{F}_1\mathcal{G}_3+\left(\mathcal{F}_3\mathcal{G}_2-\mathcal{F}_2\mathcal{G}_3\right)f=0,\]
where we claim that the left hand side is not identically $0$. To see this, assume that
\beql{ch2id2} \mathcal{F}_3\mathcal{G}_1-\mathcal{F}_1\mathcal{G}_3\equiv\mathcal{F}_3\mathcal{G}_2-\mathcal{F}_2\mathcal{G}_3\equiv 0.
\eeq
Then, we consider the equations \eqref{ch2id2} and their conjugate as polynomials of $e$, and using computer algebra we eliminate the variable $a$ from their coefficients to find that
\[b^2c(b-1) (b - c)  (b c - d) (1 + c + d) (b - d^2) \mathcal{E}(c,d)=0\]
must hold, if \eqref{ch2id2} holds. Therefore we have either $c+d=-1$ or one of the degenerate cases described in the Canonical transformation. The case $c+d=-1$ implies that the set $\{c,d\}$ consists of nontrivial cubic roots only. By directly solving the corresponding equations, we find that the quadruples $(1,1,\omega,\omega^2)$, $(\omega^2,\omega,\omega,\omega^2)$ and $(1,1,\omega^2,\omega)$, $(\omega,\omega^2,\omega^2,\omega)$ can make both polynomials vanish, but these cases were excluded by the Canonical transformation.

We have thus shown that $\mathcal{F}_3\mathcal{G}_1-\mathcal{F}_1\mathcal{G}_3$ and $\mathcal{F}_3\mathcal{G}_2-\mathcal{F}_2\mathcal{G}_3$ are not identically zero at the same time. It is, however, possible that
\beql{efeqv}
\mathcal{F}_3\mathcal{G}_1-\mathcal{F}_1\mathcal{G}_3=\mathcal{F}_3\mathcal{G}_2-\mathcal{F}_2\mathcal{G}_3=0
\eeq
for some numbers $|e_1|=1$, $e_1\neq e_0$, meaning that we do not have another condition on $f$ and therefore formula \eqref{formF} does not hold. However, having these numbers $e_1$ at our disposal we can check if they meet \eqref{NI} and then we can recover the two possible values of $f$ from \eqref{lin}. Note that condition \eqref{efeqv} together with our assumption that $\mathcal{F}_3\neq 0$ imply that \eqref{Glin} holds as well. We check if the numbers $f$ are unimodular and once we have the candidate pairs $(e_1,f_1)$ and $(e_1,f_2)$ for a given $e_1$, we readily calculate the remaining pairs $(s_1,s_3)$ and $(s_2,s_4)$ through the Decomposition formula. We disregard those cases in which $s_1+s_2=0$ or $s_3+s_4=0$ hold as they would lead to matrices belonging to $K_6^{(3)}$, and store all the remaining sextuples $(e,s_1,s_2,f,s_3,s_4)$ in the solution set $SOL_B$.

Next we suppose that $e\neq e_1$, $e\neq e_0$, derive formula \eqref{formF} for $F(e)$ and proceed to Step $\#4$.

\smallskip
\noindent STEP \#4:
Now we need to ensure that $f$ is of modulus one. To do this, we calculate the fundamental polynomial
\beql{ch2fundp}
\mathcal{P}_{a,b,c,d}(e)\equiv a^4b^4c^3d^3e^3\left(\left|\mathcal{F}_3\mathcal{G}_1-\mathcal{F}_1\mathcal{G}_3\right|^2-\left|\mathcal{F}_3\mathcal{G}_2-\mathcal{F}_2\mathcal{G}_3\right|^2\right).
\eeq
After some calculations, it will be apparent that $\mathcal{P}$ has the following remarkable structure:
\[\mathcal{P}\equiv a^8b^8c^6d^6\overline{P}+2a^8b^8c^6d^6\overline{P}(1+\overline{a}+\overline{b})e+a^8b^8c^6d^6\overline{Q}e^2+Re^3+Qe^4+2P(1+a+b)e^5+Pe^6\]
where the coefficients $P, Q$ and $R$ depend on the quadruple $(a,b,c,d)$ only. If $\mathcal{P}\equiv 0$ then the construction fails. Otherwise we find all unimodular roots of $\mathcal{P}(e)$ satisfying $e\neq e_0$ and $e\neq e_1$. Note that the unimodular numbers $e_1$ (defined after \eqref{efeqv}) are roots automatically. We shall discuss the issue of unimodularity later in Theorem \ref{ch2thmcohn}.

\smallskip
\noindent STEP \#5:
If the number $e_0$ defined in \eqref{FE} is not of modulus one, then the r\^ole of the pairs $(e,f), (s_1,s_3)$ and $(s_2,s_4)$ is symmetric, and from all of the roots of $\mathcal{P}$ of modulus one (which are different from the numbers $e_1$) we select all possible triplets $(e,s_1,s_2)$ satisfying $e+s_1+s_2=-1-a-b$, which is needed to ensure orthogonality of the first two rows. Note that this fulfills the second requirement \eqref{NI} of Theorem \ref{HTIO}. From \eqref{formF} we compute the unique companion values $f=F(e), s_3=F(s_1)$ and $s_4=F(s_2)$.

Otherwise, should the number on the right hand side of \eqref{FE} be of modulus one, then for every root $e$ ($\neq e_0,e_1$) of $\mathcal{P}$ we check if condition \eqref{NI} holds, calculate the unique companion value $f=F(e)$, and finally we use the Decomposition formula to determine the pairs $(s_1,s_3)$ and $(s_2,s_4)$. Again, we disregard the cases when $s_1+s_2=0$ or $s_3+s_4=0$.

At the end of this step we add all obtained sextuples $(e,s_1,s_2,f,s_3,s_4)$ to the solution set $SOL_B$. Typically two distinct sextuples are found.

\smallskip
\noindent STEP \#6:
In this step one constructs the first three columns along the lines of Steps $\#2$-$\#5$ described above and obtain the set $SOL_C$ in a similar way.

\smallskip
\noindent STEP \#7:
For every solution from $SOL_B$ and $SOL_C$ we build up the candidate matrices $B$ and $C$ and disregard those cases in which the block $B$ is singular. Note that as we have fulfilled both requirements of Theorem \ref{HTIO} the first three rows and columns of $G_6^{(4)}$ are orthogonal. Therefore we can use Lemma \ref{DDD} to obtain the missing lower right submatrix $D$ and check if it is composed of unimodular entries. Recall that by Proposition \ref{Vanish} we have disregarded members of the family $K_6^{(3)}$ only.

\smallskip
\noindent STEP \#8:
Finally, we \ttfamily OUTPUT \normalfont all unimodular matrices found during the process. We remark here that by Corollary \ref{embC} if no unimodular matrices were found then the submatrix $E$ cannot be embedded into any complex Hadamard matrix of order $6$ which is different from $S_6^{(0)}$ and lies outside the degenerate family $K_6^{(3)}$. If unimodular matrices are found, then typically we find two matrices, as the solution set $SOL_B$ and $SOL_C$ contains two suitable sextuples each, however, experimental results show that for each sextuple in $SOL_B$ there is a unique one in $SOL_C$ making $D$ unimodular as required.

We have finished the discussion of Construction \ref{mC}. The results are summarized in the following
\begin{thm}[see \cite{SZF6}]\label{precthmA}
Start from a submatrix $E$ as in \eqref{E} and suppose that there are only finitely many invertible candidate submatrices $B$ and $C$ such that the first three rows and columns of the matrix $G_6$ $($see \eqref{G6}$)$ are orthogonal. Then all complex Hadamard matrices containing $E$ as a submatrix, which are inequivalent from $S_6^{(0)}$ and do not belong to the family $K_6^{(3)}$, can be obtained, up to equivalence, through Construction \ref{mC}.
\end{thm}
The interested reader might want to see an example of generic Hadamard matrices which can be described by closed analytic formulae, that is for which the fundamental polynomials $\mathcal{P}_{a,b,c,d}$ and $\mathcal{P}_{a,c,b,d}$ are both solvable. We sketch how to obtain such a matrix in the following
\begin{ex}\label{ch2genex}
Let $a$ be a unimodular complex number such that its real part is the unique real solution of $4\Re[a]^3-2\Re[a]+1=0$, and its imaginary part is positive. Further set $c=(-a^3+a^2+a+1)/(a^4+a^3+a^2-a)$
and consider the input submatrix $E(a,\overline{a},c,a)$. These initial settings imply that one of the roots of $\mathcal{P}_{a,\overline{a},c,a}$ will be $a$ and then two additional one can be obtained through the Decomposition formula, while the remaining cubic polynomial can be solved by radicals. From this, and from the delicate structure of this specially tailored matrix it follows that $F(a)=\overline{a}$ (see \eqref{formF}) will be a root of $\mathcal{P}_{a,c,\overline{a},a}$ and, again, two additional roots can be found easily. Clearly, the complex Hadamard matrices we obtain are inequivalent from $S_6^{(0)}$ and do not belong to the family $K_6^{(3)}$. This latter fact can be checked through Part (c) of Theorem \ref{K6}. We spare the reader the details.
\end{ex}
\begin{pro}
Find a textbook example of generic complex Hadamard matrices, i.e.\ one with as ``simple'' entries as just possible.
\end{pro}
Now we briefly discuss the results.
\begin{rem}
One might simplify some of the lengthy case by case analyses of Steps $\#2$-$\#3$ of the construction above by including the equation
\[uabcd(a+1)(b+1)(c+1)(d+1)(a+b)(a+c)(b+d)(c+d)-1=0,\]
featuring a ``dummy'' variable $u$ into the system of equations \eqref{ch2id1} and \eqref{ch2id2}, respectively. This would ensure that $abcd\neq 0$ and a vanishing sum of order $2$ does not appear in the submatrix $E$ (cf.\ Example \ref{ch2rapid}). The presence of this extra variable, however, significantly increases the complexity of the problems and consequently prolongs the time and memory consumption of the required calculations.
\hfill$\square$\end{rem}
\begin{rem}
When $\mathcal{P}\equiv 0$ then the main difficulty we are facing is that we have infinitely many candidate submatrices $B$. In this case we have the trivial restriction $\eqref{NI}$ on $e$, while the companion value $f$ coming from \eqref{formF} is unimodular unconditionally. Although in principle we can find three orthogonal rows through the Decomposition formula for every suitable $e$, we do not know which one to favor in order to obtain a unimodular submatrix $D$ via formula \eqref{DD}. Also, it might happen that the polynomial $\mathcal{P}_{a,c,b,d}$, coming from considering the first three columns of $G_6^{(4)}$ during Step $\#6$, shall vanish as well, bringing another free parameter into the game and making things even more complicated. In contrast, if both $\mathcal{P}_{a,b,c,d}\not\equiv0$ and $\mathcal{P}_{a,c,b,d}\not\equiv0$, then we have finitely many choices for the submatrices $B$ and $C$ and we can use Corollary \ref{embC} to conclude the construction.
\hfill$\square$\end{rem}
\begin{rem}
The polynomial $\mathcal{P}$ formally can vanish when we have $\mathcal{F}_3\mathcal{G}_1-\mathcal{F}_1\mathcal{G}_3\equiv\mathcal{F}_3\mathcal{G}_2-\mathcal{F}_2\mathcal{G}_3\equiv0$, however this is excluded by the Canonical Transformation and is explained in details in Step $\#3$. It might vanish for some other, non-trivial quadruples as well making the whole construction process fail. In theory, the common roots of the coefficients of $\mathcal{P}$ can be calculated by means of Gr\"obner bases, but as these coefficients are rather complicated obtaining such a basis turned out to be a task beyond our capabilities. Nevertheless, we conjecture that the case $\mathcal{P}\equiv 0$ can be excluded completely in a similar fashion as we disregarded various quadruples during the Canonical Transformation. This would mean that all complex Hadamard matrices of order $6$, except for $S_6^{(0)}$ and $K_6^{(3)}$, can be recovered from Construction \ref{mC}.
\hfill$\square$\end{rem}
So far we have completely avoided the issue of unimodularity. It turns out that there are tools available guaranteeing that all six roots of the fundamental polynomial are unimodular. We recall the following
\begin{defi}\label{ch2si}
A complex polynomial $f(x)=a_0+a_1x+\hdots+a_mx^m$ of degree $m$ is called \emph{self-inversive} if $a_{m-k}=\varepsilon \overline{a}_k$ for every $k=0,\hdots,m$, where $|\varepsilon|=1$.
\end{defi}
We observe that the polynomial $\mathcal{P}_{a,b,c,d}(e)/(a^4b^4c^3d^3)$ is self-inversive (with $\varepsilon=1$). In particular, its middle coefficient $R/(a^4b^4c^3d^3)$ is real. A classical theorem of Cohn \cite{ACx1} gives a characterization of polynomials having exclusively unimodular roots (see also \cite{LL1} for an elegant sufficient condition).
\begin{thm}[Cohn, \cite{ACx1}]\label{ch2thmcohn}
All roots of the polynomial $f(x)=a_0+a_1x+\hdots+a_mx^m$ of degree $m$ are unimodular if and only if
\begin{enumerate}[$($a$)$]
\item $f$ is self-inversive; and
\item all roots $r$ of the derivative $f'$ satisfy $|r|\leq 1$.
\end{enumerate}
\end{thm}
Therefore self-inversive polynomials are the relevant objects to study. Cohn's theorem is powerful enough to guarantee that the fundamental polynomials $\mathcal{P}$ has exclusively unimodular roots around a small neighborhood of a fixed quadruple $(a,b,c,d)$. In particular, from Example \ref{ch2genex} we have the following
\begin{cor}
There is a four-parameter family of unitary matrices of order $6$, such that the first three rows and columns of the matrices are unimodular.
\end{cor}
Yet, it is unclear how to ensure the unimodularity of the lower right submatrix $D$ of the matrix $G_6^{(4)}$ automatically, i.e.\ without checking it by hand in a particular example.
\begin{rem}
We outline an alternate way to obtain the lower right submatrix $D$, instead of using formula \eqref{DD}. Once the first three rows and columns of the matrix $G_6^{(4)}$ has been obtained, we have several new $3\times 3$ submatrices, from which the Dilation algorithm can be ``restarted'' again in order to obtain some of the missing entries of $G_6^{(4)}$. The advantage of this method is that by Cohn's theorem the unimodularity of the entries can be guaranteed, the pitfall is that by this method we cannot control the mutual orthogonality of more than three rows at the same time.
\hfill$\square$\end{rem}
It is reasonable to think that every complex Hadamard matrix of order $6$ has some $3\times 3$ submatrix $E$ leading to nonvanishing fundamental polynomials. In particular, we do not expect any complex Hadamard matrices of order $6$ (except maybe $S_6^{(0)}$ and $K_6^{(3)}$) which cannot be recovered from Construction \ref{mC}. Therefore we formulate the following
\begin{con}\label{C2}
The list of complex Hadamard matrices of order $6$ is as follows: the isolated matrix $S_6^{(0)}$, the three-parameter degenerate family $K_6^{(3)}$ and the four-parameter generic family $G_6^{(4)}$ as described above.
\end{con}
It would be nice to understand the structure of $G_6^{(4)}$ more thoroughly and express the entries of these matrices by some well-chosen trigonometric functions in a similar fashion as $K_6^{(3)}$ is described, however, as we have encountered rather delicate sextic polynomials already the appearance of compact and elegant formulas is somewhat unexpected. Also, it is natural to ask which matrices satisfy the conditions of Corollary \ref{CE}. An algebraic characterization of these matrices might lead to a deeper understanding of the generic family $G_6^{(4)}$ and hopefully to the desired full classification of complex Hadamard matrices of order $6$.
\begin{pro}
Classify all embeddable matrices $E$.
\end{pro}
The results presented here might be further improved in the future, when faster computation of Gr\"obner basis shall be available; either by utilizing a large-scale computing cluster to handle the equations, or by using improved algorithms.
\begin{pro}
Characterize the cases when the fundamental polynomial $\mathcal{P}_{a,b,c,d}(e)$ \eqref{ch2fundp} vanishes $($e.g.\ by means of computing a Gr\"obner basis with respect to the pure lexicographic monomial order$)$.
\end{pro}
\begin{pro}
Give a full characterization of $6\times 6$ complex Hadamard matrices.
\end{pro}
\section{\texorpdfstring{$\mathrm{MUB}$s, $\mathrm{MASA}$s and complex Hadamard matrices}{MUBs, MASAs and complex Hadamard matrices}}\label{ch2MUB}
In this section we give some applications of complex Hadamard matrices. 
\subsection{Connections to operator algebras}
The following is a fundamental concept in the theory of operator algebras, see e.g.\ \cite{BN1}, \cite{UH1} and \cite{PSZW}.
\begin{defi}
We say that an algebra $\mathcal{A}\subset\mathcal{M}_n\left(\mathbb{C}\right)$ is a \emph{maximal abelian $\ast$-subalgebra} (or $\mathrm{MASA}$) if it is commutative, closed under the adjoint operation and maximal with respect to the inclusion.
\end{defi}
\begin{ex}
For every $n$ the algebra of diagonal matrices $\mathcal{D}_n$ is a MASA in $\mathcal{M}_n(\mathbb{C})$.
\end{ex}
In fact, it is easy to see that essentially this is the only example.
\begin{lem}
Let $\mathcal{A}$ be a $\mathrm{MASA}$ in $\mathcal{M}_n(\mathbb{C})$. Then there exists a unitary matrix $U$ such that $U^\ast\mathcal{A}U=\mathcal{D}_n$.
\end{lem}
\begin{proof}
All elements of $\mathcal{A}$ are pairwise commuting normal operators and hence they can be simultaneously diagonalized.
\end{proof}
\begin{defi}
The $\mathrm{MASA}$s $\mathcal{A}$ and $\mathcal{B}$ are called \emph{orthogonal} if the orthogonal complements $\mathcal{A}\ominus\mathbb{C}I$ and $\mathcal{B}\ominus\mathbb{C}I$ are orthogonal with respect to the Hilbert--Schmidt inner product in $\mathcal{M}_n(\mathbb{C})$.
\end{defi}
We denote the Hilbert--Schmidt inner product of the matrices $A,B\in\mathcal{M}_n(\mathbb{C})$ by $\left\langle A,B\right\rangle_{\mathrm{HS}}=\mathrm{Tr}(A^\ast B)$. The following question arises naturally
\begin{pro}
Decide the largest number of pairwise orthogonal $\mathrm{MASA}$s in $\mathcal{M}_n(\mathbb{C})$.
\end{pro}
We recall the following upper bound as follows.
\begin{lem}\label{noofmasas}
There are at most $n+1$ pairwise orthogonal $\mathrm{MASA}$s in $\mathcal{M}_n(\mathbb{C})$.
\end{lem}
\begin{proof}
Let $\mathcal{A}_1, \mathcal{A}_2,\hdots,\mathcal{A}_m$ be pairwise orthogonal $\mathrm{MASA}$s in $\mathcal{M}_n(\mathbb{C})$. Each of the subspaces $\mathcal{A}_i\ominus\mathbb{C}I$ are of dimension $n-1$, and hence, by orthogonality, $m(n-1)+1\leq n^2$ from which the result follows.
\end{proof}
We have the following characterization of orthogonal $\mathrm{MASA}$s.
\begin{lem}
Suppose that $\mathcal{D}_n$ and $U^\ast\mathcal{D}_nU$ is a pair of orthogonal $\mathrm{MASA}$s in $\mathcal{M}_n(\mathbb{C})$. Then $\sqrt{n}U$ is a complex Hadamard matrix.
\end{lem}
\begin{proof}
Let us fix some $1\leq i,j\leq n$ and consider the matrices $D_j\in\mathcal{D}_n$ and $U^\ast D_iU\in U^\ast\mathcal{D}_nU$, where $D_i$ is the diagonal matrix with its $(i,i)$th entry being $1$ and all other entries being $0$. Then, by orthogonality, we have
\[\left\langle D_j-\frac{1}{n}I,U^\ast D_iU-\frac{1}{n}I\right\rangle_{\mathrm{HS}}=\mathrm{Tr}\left[\left(D_j-\frac{1}{n}I\right)^\ast\left(U^\ast D_i U-\frac{1}{n}I\right)\right]=0,\]
from which $\mathrm{Tr}(D_jU^\ast D_iU)=|U_{i,j}|^2=1/n$ follows.
\end{proof}
\begin{cor}\label{ch2masap}
Let $\mathcal{D}_n, \frac{1}{n}H_2^\ast\mathcal{D}_nH_2, \hdots, \frac{1}{n}H_m^\ast\mathcal{D}_nH_m$ be a collection of $m$ pairwise orthogonal $\mathrm{MASA}$s in $\mathcal{M}_n(\mathbb{C})$. Then $H_i$, as well as $\frac{1}{\sqrt{n}}H_iH_j^\ast$ are complex Hadamard matrices for every $2\leq i<j\leq m$.
\end{cor}
\begin{ex}
Let $n=2$ and consider in $\mathcal{M}_2(\mathbb{C})$ the following $\mathrm{MASA}$s:
\begin{gather*}
\mathcal{D}_2=\left\{\left[\begin{array}{cc}
1 & 0\\
0 & 1\\
\end{array}\right],\left[\begin{array}{cc}
1 & 0\\
0 & -1\\
\end{array}\right]\right\},\quad
\mathcal{C}_2=U^\ast\mathcal{D}_2U=\left\{\left[\begin{array}{cc}
1 & 0\\
0 & 1\\
\end{array}\right],\left[\begin{array}{cc}
0 & 1\\
1 & 0\\
\end{array}\right]\right\},\ \ \text{and}\\
\mathcal{B}_2=V^\ast\mathcal{D}_2V=\left\{\left[\begin{array}{cc}
1 & 0\\
0 & 1\\
\end{array}\right],\left[\begin{array}{cc}
0 & \mathbf{i}\\
-\mathbf{i} & 0\\
\end{array}\right],\right\},\ U=\frac{1}{\sqrt{2}}\left[\begin{array}{rr}
1 & 1\\
1 & -1\\
\end{array}\right],\ V=\frac{1}{\sqrt{2}}\left[\begin{array}{rr}
1 & \mathbf{i}\\
1 & -\mathbf{i}\\
\end{array}\right].
\end{gather*}
Here $\mathcal{D}_2$ and $\mathcal{C}_2$ are the algebra of the diagonal and circulant matrices, respectively and $U$ is the matrix describing the usual discrete Fourier transform.
\end{ex}
Finally, we mention the following neat result highlighting a gap in the maximum number of pairwise orthogonal $\mathrm{MASA}$s.
\begin{thm}[Weiner, \cite{MWx1}]
The maximal number of pairwise orthogonal $\mathrm{MASA}$s in $\mathcal{M}_n(\mathbb{C})$ is either $n+1$ or at most $n-1$.
\end{thm}
In other words, if $n$ pairwise orthogonal $\mathrm{MASA}$s are constructed in $\mathcal{M}_n(\mathbb{C})$, an additional one can be obtained as well. We shall use these results coming from operator theory in the subsequent sections.
\subsection{\texorpdfstring{The problem of Mutually Unbiased Bases in $\mathbb{C}^n$}{The problem of Mutually Unbiased Bases in Cn}}
We begin with recalling the following concept arising in quantum information theory naturally.
\begin{defi}
Two orthonormal bases of $\mathbb{C}^n$, $\mathcal{B}_1$ and $\mathcal{B}_2$ are called \emph{mutually unbiased} if for every $e\in\mathcal{B}_1$ and $f\in\mathcal{B}_2$ we have $|\left\langle e,f\right\rangle|=1/\sqrt{n}$. A collection of orthonormal bases $\mathcal{B}_1,\mathcal{B}_2,\hdots,\mathcal{B}_m$ are called (pairwise) mutually unbiased if every two of them are mutually unbiased.
\end{defi}
We refer to mutually unbiased bases as MUBs, and identify them with unitary matrices $B_1,B_2,\hdots, B_m$ of order $n$, whose row vectors amount to the basis vectors. It is immediate that if we fix $B_1=I$, which we can do without loss of generality, then all further matrices $B_i$, $i=2,\hdots, m$ are complex Hadamard, up to a scaling factor of $1/\sqrt{n}$. Thus understanding MUBs requires understanding complex Hadamard matrices. Conversely, any construction of MUBs gives a construction of complex Hadamard matrices. The standard reference to the subject is \cite{DEBZ}.

Observe that MUBs has the following remarkable property (cf.\ Corollary \ref{ch2masap}).
\begin{lem}
Let $\{I,\frac{1}{\sqrt{n}}H_2,\hdots,\frac{1}{\sqrt{n}}H_m\}$ be a collection of $m$ $\mathrm{MUB}$s in $\mathbb{C}^n$. Then $H_i$, as well as $\frac{1}{\sqrt{n}}H_iH_j^\ast$ are complex Hadamard matrices for every $2\leq i<j\leq m$.
\end{lem}
Therefore $\mathrm{MASA}$s and $\mathrm{MUB}$s are exactly the same concept, coming from different areas of mathematics.
\begin{ex}
Let $n=2$, and consider the following unitary matrices:
\[B_1=\left[\begin{array}{cc}
1 & 0\\
0 & 1\\
\end{array}\right],
B_2=\frac{1}{\sqrt2}\left[\begin{array}{rr}
1 & 1\\
1 & -1\\
\end{array}\right],B_3=\frac{1}{\sqrt2}\left[\begin{array}{rr}
1 & \mathbf{i}\\
1 & -\mathbf{i}\\
\end{array}\right].\]
It is easy to see that $\{B_1,B_2,B_3\}$ form a triplet of MUBs in $\mathbb{C}^2$. In particular, the product matrix
\[\sqrt{2}B_2B_3^\ast=\frac{1}{\sqrt2}\left[\begin{array}{cc}
1-\mathbf{i} & 1+\mathbf{i}\\
1+\mathbf{i} & 1-\mathbf{i}\\
\end{array}\right]\]
is a complex Hadamard matrix.
\end{ex}
The following question arises naturally.
\begin{pro}
Decide the maximal number of $\mathrm{MUB}$s in $\mathbb{C}^n$.
\end{pro}
In what follows we collect some well-known facts on MUBs. The first result easily comes from Lemma \ref{noofmasas}.
\begin{cor}\label{ch2noofmubs}
There are at most $n+1$ $\mathrm{MUB}$s in $\mathbb{C}^n$ and at most $n/2+1$ real $\mathrm{MUB}$s in $\mathbb{R}^n$, respectively.
\end{cor}
\begin{thm}[see e.g.\ \cite{II1}, \cite{WF1}]
For every prime power $p^k$ there is a complete set of $\mathrm{MUB}$s consisting of $p^k+1$ bases in $\mathbb{C}^{p^k}$.
\end{thm}
For a discussion of the equivalence of various MUB constructions see \cite{GR1} and \cite{WK1}.

The non prime power case remains elusive, nevertheless it is easy to obtain a lower bound on the number of MUBs here as follows.
\begin{lem}\label{ch2tensor}
Let us suppose that there is $k_1$ and $k_2$ $\mathrm{MUB}$s in $\mathbb{C}^{n_1}$ and $\mathbb{C}^{n_2}$, respectively. Then there is at least $\min\{k_1,k_2\}$ $\mathrm{MUB}$s in $\mathbb{C}^{n_1n_2}$.
\end{lem}
\begin{proof}
Let $m:=\min\{k_1,k_2\}$ and consider $m$ MUBs in $\mathbb{C}^{n_1}$: $\{I_{n_1},\frac{1}{\sqrt{n_1}}H_2,\hdots,\frac{1}{\sqrt{n_1}}H_m\}$; and similarly, $m$ MUBs in $\mathbb{C}^{n_2}$: $\{I_{n_2},\frac{1}{\sqrt{n_2}}K_2,\hdots,\frac{1}{\sqrt{n_2}}K_m\}$. Then the following set
$\{I_{n_1}\otimes I_{n_2},\frac{1}{\sqrt{n_1n_2}}H_2\otimes K_2,\hdots,\frac{1}{\sqrt{n_1n_2}}H_m\otimes K_m\}$ is a collection of $m$ MUBs in $\mathbb{C}^{n_1n_2}$.
\end{proof}
As the smallest prime-power factor of a composite number is at least $2$, it follows that
\begin{cor}
For every $n>1$ there is a triplet of $\mathrm{MUB}$s in $\mathbb{C}^n$.
\end{cor}
Currently this is the best known general lower bound.

Finally, we mention here the following result connecting MUBs to latin squares.
\begin{thm}[Wocjan--Beth, \cite{WB1}]\label{ch2wb1}
For square orders $n=s^2$ there are at least $N(s)+2$ $\mathrm{MUB}$s, where $N(s)$ is the number of mutually orthogonal latin squares of order $s$.
\end{thm}
Theorem \ref{ch2wb1} results in $6$ MUBs in dimension $n=26^2$ whereas from Lemma \ref{ch2tensor} only $5$ MUBs can be obtained.

In what follows we briefly mention an elegant construction of Zauner, yielding infinite families of triplets of MUBs, coming from bicirculant complex Hadamard matrices.
\begin{thm}[Zauner's construction, \cite{GZ1}]\label{ch2zcons}
If T is a complex Hadamard matrix of order $2m$ with $m\times m$ circulant blocks, then there exist complex Hadamard matrices $Z_1$ and $Z_2$ of order $2m$ such that $T=\frac{1}{\sqrt{2m}}Z_1Z_2^\ast$. In particular, $\{I_{2m}, \frac{1}{\sqrt{2m}}Z_1,\frac{1}{\sqrt{2m}}Z_2\}$ forms a triplet of $\mathrm{MUB}$s in $\mathbb{C}^{2m}$.
\end{thm}
Zauner's construction is based on the following remarkable representation of $2\times 2$ unitary matrices, as well on a clever use of the Fourier matrices which are well-known to diagonalize circulant matrices. The details can be found in \cite{GZ1} (or in \cite{JMS1} in English).
\begin{lem}[Zauner, \cite{GZ1}]
Suppose that $M$ is a $2\times 2$ unitary matrix. Then there are complex unimodular numbers $u,v,x$ and $y$, such that
\[M=\frac{1}{2}\left[\begin{array}{cc}
u+v & y(u-v)\\
(u-v)/x & y(u+v)/x\\
\end{array}\right].\]
\end{lem}
Zauner applied his result to the following one-parameter family of bicirculant matrices which is equivalent to the family $D_6^{(1)}(c)$ (cf.\ Example \ref{ch2d6ex}):
\[\left[\begin{array}{rrr|rrr}
1 & \mathbf{i}\overline{c} & \mathbf{i}c & 1 & \overline{c} & -c\\
\mathbf{i}c & 1 & \mathbf{i}\overline{c} & -c & 1 & \overline{c}\\
\mathbf{i}\overline{c} & \mathbf{i}c & 1 & \overline{c} & -c & 1\\
\hline
1 & -\overline{c} & c & -1 & \mathbf{i}\overline{c} & \mathbf{i}c\\
c & 1 & -\overline{c} & \mathbf{i}c & -1 & \mathbf{i}\overline{c}\\
-\overline{c} & c & 1 & \mathbf{i}\overline{c} & \mathbf{i}c & -1\\
\end{array}\right],\]
and exhibited an infinite family of MUB triplets in $\mathbb{C}^6$. We can improve on this result with our two-parameter bicirculant family $X_6^{(2)}(\alpha)$ as follows.
\begin{cor}[see \cite{SZF5}]
There is a two-parameter family of triplet of $\mathrm{MUB}$s in $\mathbb{C}^6$, emerging from the family $X_6^{(2)}(\alpha)$ $($see Example \ref{x6ex}$)$ via Zauner's construction.
\end{cor}
\begin{rem}
One might wonder if further interesting results can be obtained from the consideration of four-circulant complex Hadamard matrices. The four-circulant property requires that $n=4m$, however in such orders one can construct quartets of MUBs via Lemma \ref{ch2tensor} easily, and it seems that investigating and controlling the properties of such a quartet requires a rather delicate analysis.
\hfill$\square$\end{rem}
Nevertheless, we pose this as a
\begin{pro}
Investigate a possible ``four-circulant'' construction of quartet of MUBs in orders $n=4m$ which resembles to Zauner's construction $($Theorem \ref{ch2zcons}$)$.
\end{pro}
For additional results on MUBs consult \cite{RG1}.

Several authors considered the question of the existence of real MUBs. Clearly they can exist only in orders $4m$, moreover a triplet of real MUBs can be found only in square orders, as regular Hadamard matrices are required. Despite these strong restrictions the following result demonstrates that the general upper bound given by Corollary \ref{ch2noofmubs} can be met.
\begin{thm}[Cameron--Seidel, \cite{CS1}]\label{ch2realmub}
For each $n=2^{2i}$ with $i$ any positive integer, there are $n/2+1$ real $\mathrm{MUB}$s in $\mathbb{R}^n$.
\end{thm}
For further reading on real MUBs we refer the reader to \cite{HKO}, \cite{LMO}.
\subsection{\texorpdfstring{Towards the solution of the MUB-$6$ problem}{Towards the solution of the MUB-6 problem}}
As we have mentioned earlier in the non prime power case the maximal number of MUBs is unknown. Even the simplest case $n=6$ is undecided, and we have the following long-standing
\begin{con}[Zauner, \cite{GZ1}]
There are no more than three $\mathrm{MUB}$s in $\mathbb{C}^6$.
\end{con}
Extensive numerical searches confirm the truth of this conjecture, see e.g.\ \cite{BW2}, \cite{BW1} and \cite{RLE}, as well as a systematic search for quartet of MUBs, coming from Butson-type complex Hadamard matrices \cite{BBE}. However, none of these methods are conclusive.

One possible approach to the solution of the MUB problem in dimension $d=6$ has been proposed in our recent paper \cite{JMS1}. The main idea is a combination of a discretization and a computer-aided attack which we illustrate here as follows. We start with the hypothesis that there are four unitary matrices in $\mathbb{C}^6$ forming a set of MUBs $\{I,\frac{1}{\sqrt6}H_2,\frac{1}{\sqrt6}H_3,\frac{1}{\sqrt6}H_4\}$ and try to reach a contradiction. Clearly, if such a set of four MUBs exist, then there is a collection of ``approximate MUBs'' $\{I,\frac{1}{\sqrt6}H'_2,\frac{1}{\sqrt6}H'_3,\frac{1}{\sqrt6}H'_4\}$ as well, where the matrices $H'_i$, $2\leq i\leq 4$ come from a finite set (very roughly speaking they are composed of some $p$th roots of unity for some large $p$) and they satisfy the unitary and unbiasedness conditions up to some controllable error. If we find that no such set containing four ``approximate MUBs'' exists then we conclude that there is no four MUBs either. This is the core argument of the paper. This idea turned out to be powerful enough to deal with the generalized Fourier family $F_6^{(2)}(a,b)$. In particular, the following has been shown by means of an exhaustive computer search:
\begin{thm}[see \cite{JMS1}]
None of the pairs $\{I,\frac{1}{\sqrt6}F_6^{(2)}(a,b)\}$ can be extended to a $\mathrm{MUB}$-quartet $\{I,\frac{1}{\sqrt6}F_6^{(2)}(a,b),\frac{1}{\sqrt6}H_3,\frac{1}{\sqrt6}H_4\}$.
\end{thm}
However, this approach is currently infeasible for the general problem due to the large number of possible candidate matrices $\frac{1}{\sqrt6}H'_2$. One realistic way to overcome this difficulty is to gain some insight into the structure of the MUB pairs $\{I,\frac{1}{\sqrt6}H_2\}$, or equivalently, to understand complex Hadamard matrices of order $6$. In particular, a complete characterization of complex Hadamard matrices of order $6$ would hopefully reduce the number of candidate matrices considerably. Based on the full characterization of complex Hadamard matrices of orders up to $5$ a complete characterization of MUBs in $\mathbb{C}^d$ has been obtained for $d\leq 5$ very recently \cite{BWB}.
\subsection{\texorpdfstring{MUBs and equiangular lines in $\mathbb{C}^n$}{MUBs and equiangular lines in Cn}}\label{ch2MUBEQ}
The purpose of this brief section is to relate MUBs and hence complex Hadamard matrices to equiangular sets of lines. A construction is given yielding large set of equiangular lines in real spaces, slightly improving on a recent construction of de Caen \cite{DC1}.
\begin{defi}\label{ch2lastdef}
A set of lines in $\mathbb{C}^n$, spanned by the unit vectors $v_1,v_2,\hdots,v_r$ is \emph{equiangular} if there exists a constant $c$ such that $|\left\langle v_i,v_j\right\rangle|=c$ for every $1\leq i<j\leq r$.
\end{defi}
\begin{rem}
The term ``equiangular vectors'' usually refers to a configuration in which the condition $|\left\langle v_i,v_j\right\rangle|=c$ in Definition \ref{ch2lastdef} is replaced by the stronger one $\left\langle v_i,v_j\right\rangle=c$.
\hfill$\square$\end{rem}
It is easy to obtain upper bound on the number of equiangular lines. The following is well-known.
\begin{lem}\label{ch2noeqs}
There are at most $n^2$ equiangular lines in $\mathbb{C}^n$, and at most $n(n+1)/2$ equiangular lines in $\mathbb{R}^n$, respectively.
\end{lem}
\begin{proof}[Proof $($sketch$)$]
If the lines are represented by the unit vectors $v_1,v_2,\hdots, v_m$, then the self-adjoint (resp.\ symmetric) rank-$1$ matrices $v_1^\ast v_1,v_2^\ast v_2,\hdots,v_m^\ast v_m$ shall be linearly independent elements of the real vector space of the self-adjoint (resp.\ symmetric) matrices of order $n$. Therefore $m$ cannot be larger than the dimension of these vector spaces from which the required bounds follow.
\end{proof}
The following, however, is a long-standing open
\begin{pro}
Determine the maximum number of equiangular lines in $\mathbb{C}^n$ and $\mathbb{R}^n$, respectively.
\end{pro}
\begin{ex}\label{ch2eq4dim}
The following is a collection of four equiangular lines in $\mathbb{C}^2$, along with the corresponding Gram matrix:
\[L=\frac{1}{\sqrt{3}}\left[\begin{array}{cc}
\sqrt{3} & 0\\
1 & \sqrt{2}\\
1 & \sqrt{2}\omega\\
1 & \sqrt{2}\omega^2\\
\end{array}\right],\ \ \ LL^\ast=\frac{1}{\sqrt{3}}\left[
\begin{array}{rrrr}
 \sqrt{3} & 1 & 1 & 1 \\
 1 & \sqrt{3} & -\mathbf{i} & \mathbf{i} \\
 1 & \mathbf{i} & \sqrt{3} & -\mathbf{i} \\
 1 & -\mathbf{i} & \mathbf{i} & \sqrt{3}
\end{array}
\right].\]
\end{ex}
In this section we are primarily concerned with real equiangular lines. For a handful of small dimensional examples see \cite{JCT}.
\begin{ex}\label{ch2eq3dim}
Six diagonals of the icosahedron form equiangular lines in $\mathbb{R}^3$ as follows:
\[L=\frac{1}{\sqrt{10}}\left[\begin{array}{ccc}
0 & \sqrt{5+\sqrt{5}} & \sqrt{5-\sqrt{5}}\\
0 & \sqrt{5+\sqrt{5}} & -\sqrt{5-\sqrt{5}}\\
\sqrt{5-\sqrt{5}} & 0 & \sqrt{5+\sqrt{5}}\\
-\sqrt{5-\sqrt{5}} & 0 & \sqrt{5+\sqrt{5}}\\
\sqrt{5+\sqrt{5}} & \sqrt{5-\sqrt{5}} & 0\\
\sqrt{5+\sqrt{5}} & -\sqrt{5-\sqrt{5}} & 0
\end{array}\right].\]
\end{ex}
\begin{rem}
Examples \ref{ch2eq4dim} and \ref{ch2eq3dim} demonstrate that the upper bounds given by Lemma \ref{ch2noeqs} can be met.
\end{rem}
It is easy to construct equiangular lines from Hadamard matrices (cf.\ Theorem \ref{MT2}).
\begin{lem}
Let $H$ be a symmetric, real Hadamard matrix with constant diagonal $1$ of order $n^2$. Then there is an equiangular set of $n^2$ lines in $\mathbb{R}^{n(n-1)/2}$.
\end{lem}
\begin{proof}
The eigenvalues of $H$ are $\pm n$ with multiplicity $n(n\pm1)/2$, respectively. Hence $G:=(-H+nI)/(n-1)$ is a positive semidefinite matrix of rank $n(n-1)/2$. Therefore $G$ is the Gram matrix of the desired configuration.
\end{proof}
The first constructive quadratic lower bound on the number of real equiangular lines was obtained by de Caen, who proved the following
\begin{thm}[de Caen, \cite{DC1}]\label{ch2ddc1}
For each $n=3\cdot2^{2t-1}-1$, with $t$ any positive integer, there exists an
equiangular set of $\frac{2}{9}(n+1)^2$ lines in $\mathbb{R}^n$.
\end{thm}
It is clear that a collection of equiangular lines in $\mathbb{R}^n$ can also be thought of as an equiangular configuration in $\mathbb{R}^{n+1}$. This observation together with de Caen's result leads easily to the following general quadratic lower bound on the number of real equiangular lines.
\begin{cor}\label{ch2ddccor}
For every $n\geq 1$ there exists an equiangular set of $\frac{1}{72}(n+2)^2$ lines in $\mathbb{R}^n$.
\end{cor}
We remark that de Caen constructed the Gram matrix of the system, i.e.\ he exhibited a positive semidefinite matrix of order $\frac{2}{9}(n+1)^2$ with rank $n=3\cdot2^{2t-1}-1$. We present here a related result leading to a slight improvement upon Corollary \ref{ch2ddccor}.
\begin{thm}\label{ch2equiferi}
For each $n=3\cdot2^{2t-1}+1$, with $t$ any positive integer, there exists an equiangular set of $\frac{2}{9}(n-1)(n+2)$ lines in $\mathbb{R}^n$.
\end{thm}
\begin{proof}
Let $t$ be a positive integer, set $n=3\cdot2^{2t-1}+1$ and use Theorem \ref{ch2realmub} to exhibit a collection of $2^{2t-1}+1$ real MUBs in $\mathbb{R}^{2^{2t}}$. Now rescale all these vectors with a weight $A$, and extend each of them with $2^{2t-1}+1$ additional coordinates set to $0$ (i.e.\ embed the vectors into $\mathbb{R}^n$ in the natural way). Now for every $i=1,\hdots,2^{2t-1}+1$ consider the vectors from the $i$th basis, and set their $i$th new coordinate to $B$. With the choice
\[A=\sqrt{\frac{2^{t}}{2^{t}+1}},\ \ \ \ B=\frac{1}{\sqrt{2^{t}+1}}\]
we obtain $2^{2t}(2^{2t-1}+1)=\frac{2}{9}(n-1)(n+2)$ equiangular lines in $\mathbb{R}^n$ having a common angle (in absolute value) of $\arccos\left(1/(2^t+1)\right)$.
\end{proof}
\begin{ex}
Let $t=1$, set $n=7$ and consider the collection of $3$ real MUBs in $\mathbb{R}^4$ as follows:
\[\left\{\left[\begin{array}{cccc}
1 & 0 & 0 & 0\\
0 & 1 & 0 & 0\\
0 & 0 & 1 & 0\\
0 & 0 & 0 & 1\\
\end{array}\right],\frac{1}{2}\left[\begin{array}{cccc}
1 & 1 & 1 & 1\\
1 & 1 & -1 & -1\\
1 & -1 & 1 & -1\\
1 & -1 & -1 & 1\\
\end{array}\right],\frac{1}{2}\left[\begin{array}{cccc}
1 & -1 & -1 & -1\\
1 & -1 & 1 & 1\\
1 & 1 & -1 & 1\\
1 & 1 & 1 & -1\\
\end{array}\right]\right\}.\]
From these ingredients we build up the array $L$, which we display here in transposed layout for typographical reasons as follows:
\[L^T=\frac{1}{\sqrt6}\left[
\begin{array}{cccc|cccc|cccc}
 2 & 0 & 0 & 0 & 1 & 1 & 1 & 1 & 1 & 1 & 1 & 1 \\
 0 & 2 & 0 & 0 & 1 & 1 & -1 & -1 & -1 & -1 & 1 & 1 \\
 0 & 0 & 2 & 0 & 1 & -1 & 1 & -1 & -1 & 1 & -1 & 1 \\
 0 & 0 & 0 & 2 & 1 & -1 & -1 & 1 & -1 & 1 & 1 & -1 \\
 \hline
 \sqrt{2} & \sqrt{2} & \sqrt{2} & \sqrt{2} & 0 & 0 & 0 & 0 & 0 & 0 & 0 & 0 \\
 0 & 0 & 0 & 0 & \sqrt{2} & \sqrt{2} & \sqrt{2} & \sqrt{2} & 0 & 0 & 0 & 0 \\
 0 & 0 & 0 & 0 & 0 & 0 & 0 & 0 & \sqrt{2} & \sqrt{2} & \sqrt{2} & \sqrt{2}
\end{array}
\right].\]
The column vectors of $L^T$ give us the desired configuration of $\frac{2}{9}(7-1)(7+2)=12$ equiangular lines in $\mathbb{R}^7$.
\end{ex}
Contrary to de Caen's approach, we obtained the lines directly, and it is straightforward from their structure how to exhibit large set of equiangular lines in smaller dimensions as well. For example, by simply neglecting the last row and the last four columns of the matrix $L^T$ above, one immediately gets $8$ equiangular lines in $\mathbb{R}^6$, etc. A detailed analysis of this idea led to the following improvement upon Corollary \ref{ch2ddccor}.
\begin{cor}\label{ch2constant}
For every $n\geq 1$ there exists an equiangular set of $\frac{8}{1089}n (4n + 33)$ lines in $\mathbb{R}^n$.
\end{cor}
\begin{proof}
For $n\leq 25$ we have $8n(4n + 33)/1089\leq n$, and hence any orthonormal basis fulfills the claim. Therefore we can suppose that $n\geq 26$. Let $t\geq 2$ be the unique integral number such that that $d=3\cdot 2^{2t-1}+1<n\leq 3\cdot2^{2t+1}+1$. If $n\leq4^{t+1}+2^{2t-3}$, then use Theorem \ref{ch2equiferi} to obtain $2^{2t}\left(2^{2t-1}+1\right)$ equiangular lines in $\mathbb{R}^d$ and embed them into $\mathbb{R}^n$ in the natural way. Otherwise, if $4^{t+1}+2^{2t-3}+1\leq n\leq3\cdot2^{2t+1}+1$ then use $m=n-4^{t+1}\geq2^{2t-3}+1$ MUBs of order $4^{t+1}$, and construct along the lines of Theorem \ref{ch2equiferi} $4^{t+1}m\geq4^{t+1}\left(2^{2t-3}+1\right)$ equiangular lines in $\mathbb{R}^n$. One can conclude by fitting a quadratic polynomial to the set of points $\{(4^{t+1}+2^{2t-3},2^{2t}\left(2^{2t-1}+1\right))\colon t\geq 2\}$ $\subset \mathbb{R}\times\mathbb{R}$, and noting that the number of known equiangular lines in dimension $n'=4^{t+1}+2^{t-3}+1$ is larger than the claimed lower bound $8n'(4n'+33)/1089$.
\end{proof}
Our idea can be reformulated in the complex settings as well, but the results are nowhere near close to the following
\begin{thm}[see e.g.\ \cite{GR1}]
There is a configuration of $n^2-n+1$ equiangular lines in $\mathbb{C}^n$ when $n-1$ is a prime power.
\end{thm}
We shall revisit complex equiangular lines from a different perspective once again in Section \ref{ch3ETFs}.
\newpage %Ez azert kell, hogy jobb oldalon kezdodjenek a chapterek.
\thispagestyle{empty} % ez is
%----------------------------------CHAPTER BREAK-------------------------------------
%----------------------------------CHAPTER BREAK-------------------------------------
\chapter{Complex Hadamard matrices of prime orders}\label{ch3}
%\epigraph{Happiness is an unproductive condition.}%
%{from the music video \textit{Don't Give Up} by \textsc{Chicane} feat. \textsc{Bryan Adams}}
\thispagestyle{empty}
Throughout this chapter we investigate the existence of complex Hadamard matrices of prime orders. Constructing infinite families of complex Hadamard matrices of prime orders is currently out of reach, apart from some sporadic examples of relatively small order. There are three major construction methods yielding some examples of complex Hadamard matrices of prime orders: Petrescu's construction leading to infinite families \cite{MP1}; the general theory of circulant complex Hadamard matrices \cite{BFx1}, \cite{BH1}; and combinatorial constructions coming from design theory \cite{AC1}, \cite{CG1} and \cite{SZF2}. In what follows we investigate these approaches in details. As a supplement, we extend this list with a fourth method by considering complex Hadamard matrices with circulant core in Appendix \ref{APPC}.

In Section \ref{ch3secp} we give a further look at Petrescu's construction and we utilize his method to obtain a $BH(19,6)$; in addition to this we exhibit a four-parameter family of complex Hadamard matrices of order $13$ thus considerably extending the list of known complex Hadamard matrices of this order. In Section \ref{ch3secC} we recall circulant complex Hadamard matrices of index $k$ type and for every prime $p\equiv 1$ (mod $8$) we construct two new, previously unknown examples of circulant complex Hadamard matrices of order $p$. As a related result we give a full classification of cyclic $17$-roots of simple index $4$ in Appendix \ref{APPB}. In Appendix \ref{APPC} we investigate the existence of complex Hadamard matrices with circulant core and exhibit new examples of complex Hadamard matrices of order $7$ and $11$. We conclude this chapter with highlighting some connections between complex Hadamard matrices and equiangular tight frames.

Our contributions to this chapter are the relevant results from \cite{SZF7} and \cite{SZF2}, while the additional new results presented here are subject to a series of forthcoming publications.

\section{Preliminaries}
We begin with the following finiteness result (cf.\ Theorem \ref{ch1HSc}).
\begin{thm}[Petrescu, \cite{MP1}]\label{ch3pett}
For every prime $p$ the Fourier matrix $F_p$ is isolated amongst all $p\times p$ complex Hadamard matrices.
\end{thm}
Nicoar\u{a} gave a new and independent proof of this theorem \cite{RN1}. One of the consequences of Theorem \ref{ch3pett} is that there is no ``cheap'' way to obtain infinite, parametric families of complex Hadamard matrices stemming from $F_p$ in a similar spirit to Di\c{t}\u{a}'s construction (cf.\ Corollary \ref{ch1ditaconst}). In fact, once it was asked by Enflo (see \cite{UH1} and the review article \cite{BS1}), whether all complex Hadamard matrices of prime orders are isolated, which was refuted by a clever construction of Petrescu \cite{MP1}. In particular, he proved the following
\begin{thm}[Petrescu \cite{MP1}]
For prime $p=7,13,19$ and $31$ there are infinite, parametric families of complex Hadamard matrices of order $p$.
\end{thm}
Note that all mentioned orders are of the form $p=6s+1$. However, we do not have any understanding of the case $p=6s+5$ (apart from the finiteness result of Haagerup concerning $5\times 5$ matrices). In particular, the following is already an open
\begin{pro}
Decide if there are infinitely many inequivalent complex Hadamard matrices of order $11$.
\end{pro}
It requires considerable efforts to exhibit complex Hadamard matrices of order $11$ (but see \cite{BTZ} for a couple of explicit examples as well as Example \ref{ch3exf3}), and it seems even more difficult to construct an infinite family.
\section{Petrescu's construction revisited}\label{ch3secp}
In this section we recall the main ideas of Petrescu's construction, but we advise the reader in advance that here we are only scratching the surface of what he has actually obtained in his PhD thesis \cite{MP1}. We combine his ideas with an intelligent computer search to obtain a $BH(19,6)$ matrix the existence of which was listed undecided in \cite{CRC1}. The details are as follows.
\begin{defi}
We say that a complex Hadamard matrix $H$ of order $3s+1$ is of \emph{Petrescu-type}, if it has the following block form
\beql{ch3parray1}
H:=\left[\begin{array}{ccc}
X & Y & T\\
Y & X & T\\
T^\ast & T^\ast & D\\
\end{array}\right],
\eeq
where the matrices $X$ and $Y$ are of order $s$, $D$ is of order $s+1$ and finally $T$ consists of $s$ noninitial rows of a dephased complex Hadamard matrix of order $s+1$.
\end{defi}
It seems somewhat surprising already that one can embed four copies of an essentially complete complex Hadamard matrix into the array $H$, whose size are roughly $n/3$ each. However, Petrescu illustrated that if $T$ and $D$ are chosen carefully enough, namely if they possess some further structural properties (e.g.\ circulant blocks, etc.), then not only some complex Hadamard matrices, but parametric families of complex Hadamard matrices can be obtained. After recalling that $\omega$ denotes the principal cubic root of unity, we offer the following
\begin{ex}[Petrescu, \cite{MP1}]\label{ch3pet7}
The array
\beql{ch3p7}
P_7^{(1)}(a)=\left[
\begin{array}{rr|rr|rrr}
 1 & a \omega ^2 & \omega  & -a \omega ^2 & \omega  & \omega ^2 & 1 \\
 \overline{a} \omega ^2 & 1 & -\overline{a} \omega ^2 & \omega  & \omega ^2 & \omega  & 1 \\
 \hline
 \omega  & -a \omega ^2 & 1 & a \omega ^2 & \omega  & \omega ^2 & 1 \\
 -\overline{a} \omega ^2 & \omega  & \overline{a} \omega ^2 & 1 & \omega ^2 & \omega  & 1 \\
 \hline
 \omega ^2 & \omega  & \omega ^2 & \omega  & -\omega ^2 & 1 & 1 \\
 \omega  & \omega ^2 & \omega  & \omega ^2 & 1 & -\omega ^2 & 1 \\
 1 & 1 & 1 & 1 & 1 & 1 & -\omega ^2
\end{array}
\right]
\eeq
is a one-parameter family of Petrescu-type complex Hadamard matrices. The matrix $P_7=P_7^{(1)}(1)$ is equivalent to Brock's example \cite{BBx1}.
\end{ex}
However, from Petrescu's analysis it follows that in practice it is very difficult to meet the requirements of the parametrization property, and a general method, working for all primes $p\equiv 1$ (mod $6$) is yet to be discovered. We investigate Petrescu-type $BH(n,q)$ matrices for small $n$ and $q$, and seek for various parametric families of complex Hadamard matrices. The underlying structure of these matrices heavily restricts the search space and makes an intelligent computer search feasible even in those orders, where a straightforward brute force attack would be far out of reach.

We assume throughout this section, without any further comment, that the embedded (partial) complex Hadamard matrix $T$ satisfies the following properties:
\begin{gather}\label{pt1}
TT^\ast=(s+1)I_s,\\
\label{pt2}
T^\ast T=(s+1)I_{s+1}-J_{s+1},\\
\label{pt3}
TJ=0,
\end{gather}
where $J$ denotes the all $1$ matrix, whose size should be clear from context.
\begin{rem}
We remark here that the arrangement of the four $T$ blocks in the border of $H$ forces the structure in the upper left corner; that is, the presence of two blocks of $X$ and $Y$ is not a further simplifying assumption, but an easy consequence of the imposed structure. We omit the corresponding calculations.
\hfill$\square$\end{rem}
From the orthogonality relations the following can be deduced.
\begin{lem}\label{PHL}
Under the assumptions on $T$ the array $H$ is complex Hadamard if and only if $X,Y$ and $D$ are unimodular matrices, moreover
\begin{gather}\label{p1}
(X+Y)(X+Y)^\ast=(s-1)I_s,\\
\label{p4}
(X-Y)(X-Y)^\ast=(3s+1)I_s,\\
\label{p3}
DD^\ast=D^\ast D=(s-1)I_{s+1}+2J_{s+1},\\
\label{p5}
(X+Y)T+TD^\ast=0.
\end{gather}
\end{lem}
\begin{proof}
The first two equations are equivalent to the sum and difference of
\begin{gather}\label{p123}
XX^\ast+YY^\ast=2sI_s,\\
\label{p423}
XY^\ast+YX^\ast=-(s+1)I_s;
\end{gather}
equation \eqref{p5} is one of the orthogonality conditions; while \eqref{p3} follows from the fact that $H^\ast$ is Hadamard as well.
\end{proof}
In particular, by \eqref{p3} the operator $D$ is normal. Note also its following property.
\begin{lem}\label{ch3Pinv}
Let $H$ be a Petrescu-type complex Hadamard matrix of order $3s+1$. Then
\beql{ch3detC}
|\mathrm{det}D|^2=(s-1)^s(3s+1).
\eeq
In particular, $D$ is invertible for $s>1$.
\end{lem}
\begin{proof}
We use the folklore fact that $\mathrm{det}(aI_n+bJ_n)=a^n+na^{n-1}b$. Therefore the result follows from \eqref{p3}.
\end{proof}
Following Petrescu's ideas we attempt to reconstruct $H$ via Lemma \ref{PHL}.
\begin{lem}[Petrescu, \cite{MP1}]
Suppose that \eqref{p3} holds. Then \eqref{p5} holds if and only if
\beql{PS1}
X+Y=-\frac{1}{s+1}TD^\ast T^\ast
\eeq
and
\beql{PS2}
JD=DJ.
\eeq
\end{lem}
The proof is somewhat technical, yet we include it for the reader's convenience, because the original work of Petrescu \cite{MP1} is not easily accessible. We denote by $\mathrm{Ker}A$ and $\mathrm{Im}A$ the kernel and the range of the operator $A$, respectively. Recall that the orthogonal complement of $\mathrm{Ker}A$, which we denote by $\left(\mathrm{Ker}A\right)^\perp$, is $\mathrm{Im}A$.
\begin{proof}
Suppose that \eqref{PS1} and \eqref{PS2} holds. Write
\begin{align*}
(X+Y)T+TD^\ast&=-\frac{1}{s+1}TD^\ast T^\ast T+TD^\ast=-\frac{1}{s+1}TD^\ast((s+1)I-J)+TD^\ast\\
&=\frac{1}{s+1}TD^\ast J=\frac{1}{s+1}TJD^\ast=0.
\end{align*}
To see the other direction, first multiply \eqref{p5} by $T^\ast$ to the right to get \eqref{PS1}. Now substitute \eqref{PS1} back into \eqref{p5} to get
\[TD^\ast J=0\]
as before, and multiply it by $T^\ast$ and rearrange to get
\beql{ch3adj}
\frac{1}{q+1}D^\ast J=\frac{1}{(q+1)^2}JD^\ast J.
\eeq
Now let us introduce the notation $P:=\frac{1}{q+1}J$, and observe that $P=P^2=P^\ast$ is a self-adjoint projection. After taking adjoints in \eqref{ch3adj} we obtain
\beql{PDP}
PD=PDP.
\eeq
Our aim is to show that $PD=DP$ holds, and hence \eqref{PS2} as well. Let $x\in \mathrm{Ker}P$. Then, by \eqref{PDP}, $Dx\in\mathrm{Ker}P$ as well, and therefore $PDx=DPx=0$.

Now suppose that $x\in\left(\mathrm{Ker}P\right)^\perp=\mathrm{Im}P$. It follows that $Dx\in\mathrm{Im}P$ as well, as $Dx$ is orthogonal to any vector $y$ belonging to $\mathrm{Ker}P$. To see this, observe that by Lemma \ref{ch3Pinv} $D$ is invertible, and hence it is enough to show that $Dx$ is orthogonal to any $Dy$. Indeed:
\begin{align*}
\left\langle Dx,Dy\right\rangle&=\left\langle D^\ast Dx,y\right\rangle=\left\langle((s-1)I+2J)x,y\right\rangle=\left\langle((3s+1)I-2((s+1)I-J))x,y\right\rangle\\
&=(3s+1)\left\langle x,y\right\rangle-2\left\langle T^\ast TPx,y\right\rangle=0,
\end{align*}
where we have used \eqref{p3}, the identity $Px=x$, \eqref{pt3} and the fact that $x$ and $y$ are orthogonal. It follows that $PDx=DPx=Dx$ and we are done.
\end{proof}
\begin{cor}
The system of equations \eqref{p3}, \eqref{PS1}, \eqref{PS2} imply \eqref{p1}.
\end{cor}
\subsection{\texorpdfstring{Constructing a $BH(19,6)$ matrix}{Constructing a BH(19,6) matrix}}
Now we describe the construction of the anticipated $BH(19,6)$ matrix as follows. We begin by constructing a matrix $D$ satisfying \eqref{p3}. Of course, there is absolutely no guarantee that an initial matrix $D$ can be extended to a $BH(19,6)$ matrix $H$, and in order to avoid considering multiple instances of essentially the same matrices $D$ we reduce them by a number of equivalence operations. In particular, $D$ can be replaced by $\lambda PDP^T$ where $\lambda$ is unimodular and $P$ is any permutation matrix, by considering $\mathrm{Diag}\left(\overline{\lambda}I_s,\overline{\lambda}I_s,\lambda P\right)H\mathrm{Diag}\left(I_s,I_s,P^T\right)$ instead of the initial array $H$ if it is necessary. The following is a further useful
\begin{lem}\label{PL1234}
Let $H$ be a Petrescu-type complex Hadamard matrix of order $n$. Then
\[JD=DJ=cJ,\]
where $|c|=\sqrt{n}$.
\end{lem}
\begin{proof}
Multiply \eqref{p3} by $J$ from the left and right simultaneously, then use \eqref{PS2}.
\end{proof}
Once $D$ has been specified the second step is to obtain the complex Hadamard matrix $T$. We do this intelligently through \eqref{PS1} by constructing $T$ row by row. We check, using the partial matrix $T_r$, consisting of $r$ row vectors, if the resulting matrix
\[-\frac{1}{s+1}T_rD^\ast T_r^\ast\]
can be decomposed as a sum of two matrices $X_r$ and $Y_r$, containing sixth roots of unity. The final step is to ensure \eqref{p4}, which, again, can be checked row-by-row by utilizing the partial matrices $X_r$ and $Y_r$. After implementing a more-or-less naive algorithm in \emph{Mathematica} we obtained a solution, virtually in seconds.
\begin{thm}
There exists a $BH(19,6)$ matrix.
\end{thm}
\begin{ex}\label{ch3w19}
The matrix
\beql{ch3w19F}
W_{19}=\mathrm{EXP}\left(\frac{2\pi\mathbf{i}}{6}L\right),
\eeq
where
\[L=\left[
\begin{array}{cccccc|cccccc|ccccccc}
 3 & 0 & 1 & 1 & 0 & 0 & 5 & 4 & 3 & 5 & 3 & 2 & 1 & 1 & 3 & 5 & 4 & 3 & 0 \\
 0 & 0 & 1 & 3 & 3 & 1 & 4 & 2 & 4 & 5 & 1 & 5 & 1 & 4 & 3 & 3 & 1 & 5 & 0 \\
 0 & 0 & 1 & 4 & 2 & 4 & 2 & 4 & 3 & 2 & 4 & 1 & 3 & 3 & 1 & 4 & 5 & 1 & 0 \\
 1 & 2 & 4 & 2 & 1 & 2 & 4 & 4 & 2 & 4 & 5 & 0 & 3 & 5 & 1 & 1 & 3 & 4 & 0 \\
 2 & 5 & 4 & 3 & 2 & 0 & 4 & 2 & 0 & 1 & 4 & 2 & 4 & 1 & 5 & 3 & 1 & 3 & 0 \\
 0 & 3 & 5 & 4 & 5 & 0 & 4 & 5 & 3 & 1 & 3 & 4 & 5 & 3 & 4 & 1 & 3 & 1 & 0 \\
 \hline
 5 & 4 & 3 & 5 & 3 & 2 & 3 & 0 & 1 & 1 & 0 & 0 & 1 & 1 & 3 & 5 & 4 & 3 & 0 \\
 4 & 2 & 4 & 5 & 1 & 5 & 0 & 0 & 1 & 3 & 3 & 1 & 1 & 4 & 3 & 3 & 1 & 5 & 0 \\
 2 & 4 & 3 & 2 & 4 & 1 & 0 & 0 & 1 & 4 & 2 & 4 & 3 & 3 & 1 & 4 & 5 & 1 & 0 \\
 4 & 4 & 2 & 4 & 5 & 0 & 1 & 2 & 4 & 2 & 1 & 2 & 3 & 5 & 1 & 1 & 3 & 4 & 0 \\
 4 & 2 & 0 & 1 & 4 & 2 & 2 & 5 & 4 & 3 & 2 & 0 & 4 & 1 & 5 & 3 & 1 & 3 & 0 \\
 4 & 5 & 3 & 1 & 3 & 4 & 0 & 3 & 5 & 4 & 5 & 0 & 5 & 3 & 4 & 1 & 3 & 1 & 0 \\
\hline
 5 & 5 & 3 & 3 & 2 & 1 & 5 & 5 & 3 & 3 & 2 & 1 & 0 & 0 & 0 & 0 & 1 & 1 & 3 \\
 5 & 2 & 3 & 1 & 5 & 3 & 5 & 2 & 3 & 1 & 5 & 3 & 0 & 0 & 1 & 3 & 0 & 1 & 0 \\
 3 & 3 & 5 & 5 & 1 & 2 & 3 & 3 & 5 & 5 & 1 & 2 & 1 & 3 & 0 & 0 & 0 & 1 & 0 \\
 1 & 3 & 2 & 5 & 3 & 5 & 1 & 3 & 2 & 5 & 3 & 5 & 0 & 1 & 0 & 1 & 0 & 4 & 1 \\
 2 & 5 & 1 & 3 & 5 & 3 & 2 & 5 & 1 & 3 & 5 & 3 & 1 & 0 & 4 & 1 & 1 & 0 & 0 \\
 3 & 1 & 5 & 2 & 3 & 5 & 3 & 1 & 5 & 2 & 3 & 5 & 1 & 0 & 1 & 0 & 4 & 0 & 1 \\
 0 & 0 & 0 & 0 & 0 & 0 & 0 & 0 & 0 & 0 & 0 & 0 & 4 & 1 & 1 & 0 & 1 & 0 & 0
\end{array}
\right].\]
is a $BH(19,6)$ matrix.
\end{ex}
It is natural to ask if the matrix $W_{19}$ admits an affine orbit, similarly to $P_7^{(1)}(a)$ (see Example \ref{ch3pet7}), but we were unable to calculate its defect due to the presence of irrational entries. Thus the following remained an unresolved
\begin{pro}
Decide if there is a Petrescu-type $BH(19,6)$ leading to an infinite, parametric family of complex Hadamard matrices.
\end{pro}
Further Petrescu-type $BH(n,q)$ matrices can be found in Appendix \ref{APPA}; among them there is a four-parameter family of complex Hadamard matrices of order $13$, which is the largest known family in this order.
\subsection{Some non-existence results}
There are three obvious necessary conditions restricting the existence of Petrescu-type $BH(n,q)$ matrices: first, it is very well be possible that a $BH(n,q)$ matrices cannot exist at all due to some number theoretic reasons; secondly we might not have access to a submatrix $T$ with the required properties; and lastly the submatrix $D$ might not exist. Throughout this section we investigate the roots of these obstructions. We begin with recalling the relevant case of Part $3$ of Theorem \ref{ch1BN}.
\begin{thm}[Winterhof, \cite{AW}]\label{ch3thmW}
Let $n$ be odd. If the squarefree part of $n$ is divisible by a prime $p\equiv 5$ $(\mathrm{mod}\ 6)$ then there is no $BH(n,6)$ matrix.
\end{thm}
Winterhof's result is far more general than stated here, but this is enough for our purposes. The proof of this theorem relies on algebraic number theory, and exploits the fact that the square of the determinant of $BH(n,6)$ matrices must have a representation of the form $A^2+3B^2$. We shall use repeatedly the following number theoretic
\begin{lem}\label{ch3numL}
Let $A$ and $B$ be integral numbers such that $A^2+3B^2=n$. Then every prime divisor $p$ of the squarefree part of $n$ is either $p=3$ or $p\equiv 1$ $(\mathrm{mod}\ 6)$.
\end{lem}
\begin{proof}
It is well-known that every prime number $p\equiv 1$ (mod $6$) has the representation $p=x^2+3y^2$ with $x$ and $y$ being integers (see Theorem $5$ in \cite[Chapter IX]{WS1}), as well as $3=0^2+3\cdot1^2$ and any square number $q^2=q^2+3\cdot 0^2$. If two numbers, $n_1$ and $n_2$, have this form, say $n_i=x_i^2+3y_i^2$, $i=1,2$, then so does their product $n_1n_2=(x_1x_2+3y_1y_2)^2+3(x_1y_2-x_2y_1)^2$. Hence all of these numbers have the desired representation. Next we prove that the other numbers cannot have this form. Let us suppose that 
\beql{ch3st}
A^2+3B^2=q^2r,
\eeq
where $r$ is squarefree. We can assume that $A$, $B$ and $r$ are relatively prime, otherwise we can divide $A$, $B$ and $q$ by their greatest common divisor.

If $r$ is even, then both $A$ and $B$ are odd, and as a result the left hand side of \eqref{ch3st} is divisible by $4$, but not by $8$, a contradiction.

If $p>3$ and $p|r$, then $A^2\equiv -3B^2$ (mod $p$), and as $A$ and $B$ are nonzero modulo $p$ it follows that $-3$ is a quadratic residue. Then, should we have $p\equiv 5$ (mod $6$), we would find that
\[\left(\frac{-3}{p}\right)=(-1)^{\frac{p-1}{2}}\left(\frac{3}{p}\right)=(-1)^{p-1}\left(\frac{p}{3}\right)=\left(\frac{2}{3}\right)=-1,\]
a contradiction.
\end{proof}
Now we turn to the
\begin{proof}[Proof of Theorem \ref{ch3thmW}]
The determinant of a $BH(n,6)$ matrix is an integral linear combination of $1$ and $\omega$, and hence, by referring to Lemma \ref{ch1det} we have
\[\left|A+B\omega\right|=n^{n/2},\]
which, after squaring and by a multiplication of $4$ becomes
\[(2A-B)^2+3B^2=4n^n,\]
and hence Lemma \ref{ch3numL} applies. In particular, the squarefree part of $4n^n$, which is easily seen to be the same as the squarefree part of $n$ (as $n$ is odd) cannot have a prime divisor $p\equiv 5$ (mod $6$).
\end{proof}
Obviously, to obtain a Petrescu-type $BH(3s+1,q)$ matrix, one needs a $BH(s+1,q)$ matrix in place of $T$ as a main ingredient. This, again, is heavily constrained by Theorem \ref{ch3thmW} above. We state the following trivial
\begin{lem}\label{ch3pne1}
If there is no $BH(s+1,q)$ matrix, then there is no Petrescu-type $BH(3s+1,q)$ matrix either.
\end{lem}
Another restriction comes from the submatrix $D$ via Lemma \ref{ch3Pinv}.
\begin{thm}\label{ch3NE}
If the squarefree part of $(3s+1)(s-1)^s$ is even or divisible by a prime $p\equiv 5$ $(\mathrm{mod}\ 6)$ then there is no Petrescu-type $BH(3s+1,6)$ matrix.
\end{thm}
\begin{proof}
From Lemma \ref{ch3Pinv} we see that the determinant of $D$, which is an integral linear combination of $1$ and $\omega$, after taking absolute value, squaring and multiplying by $4$ reads
\[(2A-B)^2+3B^2=4(3s+1)(s-1)^s=4|\mathrm{det}D|^2,\]
hence we can apply Lemma \ref{ch3numL}.
\end{proof}
\begin{rem}
One might wonder if a similar restriction can be stated via Lemma \ref{PL1234} by exploiting the fact that $s+1$ sextic roots of unity are needed adding up to $\sqrt{3s+1}$ in absolute value, from which one can conclude that the number $3s+1$ alone should met the conditions of Lemma \ref{ch3numL}. This property, however is already captured by Theorem \ref{ch3NE}. Similarly, by considering the determinant of the matrices $X\pm Y$ via \eqref{p1} and \eqref{p4} we do not get any further restriction.
\hfill$\square$\end{rem}
\begin{table}[htbp]%
\caption{Existence of $BH(n,6)$ matrices for $n\leq 39$, $n$ odd.}\label{ch3tab2}
\centering
\begin{tabular}{rccc|rccc}
\hline
$n$ & \text{Existence} & \text{Petr.-type} & \text{Remark} & $n$ & \text{Existence} & \text{Petr.-type} & \text{Remark}\\
%& &type & & & &type & \\
\hline
$1$ & $F_1$ & \text{Yes} & & $21$ & $F_3\otimes P_7$ & - & \text{See \eqref{ch3p7}}\\
$3$ & $F_3$ & - & &          $23$ & - & - &\\
$5$ & - & - & &              $25$ & $W_{25}$ & \text{?} & \text{See \eqref{ch1pp}}\\
$7$ & $P_7$ & \text{Yes} & \text{See \eqref{ch3p7}}& $27$ & $F_3\otimes F_3\otimes F_3$ & -&\\
$9$ & $F_3\otimes F_3$ & - & & $29$ & - & -&\\
$11$ & - & - & &             $31$ & \text{?} & \text{No}&\\
$13$ & $X_{13}$ & \text{No} & \text{See \cite{CCL}} & $33$ & - & -&\\
$15$ & - & - & &             $35$ & - & -&\\
$17$ & - & - & &             $37$ & \text{?} & \text{?}&\\
$19$ & $W_{19}$ & \text{Yes} & \text{See \eqref{ch3w19F}} & $39$ & $F_3\otimes X_{13}$ & - & \text{See \cite{CCL}}\\
\hline
\end{tabular}
\end{table}
Table \ref{ch3tab2} summarizes the existence of $BH(n,6)$ matrices of small orders (cf.\ Table \ref{ch1t2}). One can see that $n=37$ is the next interesting odd order where a Petrescu-type matrix might exists. Order $25$ might be possible, but a $BH(25,6)$ matrix already exists by Theorem \ref{ch1pp6}, while order $31$ is not possible due to the lack of $BH(11,6)$ matrices (see Lemma \ref{ch3pne1}).

To conclude, Petrescu's array seems to be a reasonable way to construct interesting matrices of prime orders. It would be nice to find additional examples of $BH(n,6)$ matrices coming from this construction.
\begin{pro}
Decide if a Petrescu-type $BH(37,6)$ matrix exist.
\end{pro}
\begin{rem}
Unfortunately, Petrescu's construction cannot yield new examples of real Hadamard matrices as either $3s+1$ or $s+1$ is not divisible by $4$.
\hfill$\square$\end{rem}
It is worthwhile noting that the method might lead to new examples of $BH(n,4)$ matrices. Unfortunately, however, Petrescu-type $BH(70,4)$ matrices do not exist due to restrictions on the submatrix $D$ and hence the smallest open case of $BH(n,4)$ matrices cannot be constructed this way. Petrescu's construction might yield new examples of (generalized) weighing matrices as well \cite{KS1}. The following is an intriguing
\begin{ex}
The following is a three-parameter family of Petrescu-type weighing matrices of order $10$.
\[W_{10}^{(3)}(a,b,c)=\left[
\begin{array}{rrr|rrr|rrrr}
 0 & a & b & 1 & -a & -b & 1 & 1 & -1 & -1 \\
 \overline{a} & 0 & b\overline{a} & -\overline{a} & 1 & -b\overline{a} & 1 & -1 & c & -c \\
 \overline{b} & a\overline{b} & 0 & -\overline{b} & -a\overline{b} & 1 & 1 & -1 & -c & c \\
 \hline
 1 & -a & -b & 0 & a & b & 1 & 1 & -1 & -1 \\
 -\overline{a} & 1 & -b\overline{a} & \overline{a} & 0 & b\overline{a} & 1 & -1 & c & -c \\
 -\overline{b} & -a\overline{b} & 1 & \overline{b} & a\overline{b} & 0 & 1 & -1 & -c & c \\
 \hline
 1 & 1 & 1 & 1 & 1 & 1 & 0 & 1 & 1 & 1 \\
 1 & -1 & -1 & 1 & -1 & -1 & 1 & 0 & 1 & 1 \\
 -1 & \overline{c} & -\overline{c} & -1 & \overline{c} & -\overline{c} & 1 & 1 & 0 & 1 \\
 -1 & -\overline{c} & \overline{c} & -1 & -\overline{c} & \overline{c} & 1 & 1 & 1 & 0
\end{array}
\right].\]
Additionally, as $W_{10}^{(3)}(a,b,c)$ is a self-adjoint family, we find that
\[H_{10}^{(3)}(a,b,c):=W_{10}^{(3)}(a,b,c)+\mathbf{i}I_{10}\]
is a three-parameter family of complex Hadamard matrices, stemming from a $BH(10,4)$ matrix. The family is equivalent to $D_{10}^{(3)}(a,b,c)$ from \cite{SZF1}.
\end{ex}
We conclude this section with the following fundamental open
\begin{pro}
Investigate if there is a construction, analogous to Petrescu's method resulting in infinite families of complex Hadamard matrices of orders $6s+5$.
\end{pro}
\section{Circulant complex Hadamard matrices}\label{ch3secC}
Another source of complex Hadamard matrices of prime orders is the pool of circulant complex Hadamard matrices. These objects were heavily investigated during the early $1990$s by the pioneering work of Bj\"orck \cite{GB1} (see also \cite{BFx1}, \cite{HJ1} and \cite{MW1}). Let us consider a circulant complex Hadamard matrix of order $n$ whose first row is given by $x=(x_0,x_1,\hdots, x_{n-1})$. For convenience, let us fix $x_0=1$ (which we can do up to equivalence). Then, by orthogonality, we have
\beql{cycyc}\begin{cases}
x_0=1,\\
\displaystyle{\frac{x_1}{x_0}+\frac{x_2}{x_1}+\hdots+\frac{x_0}{x_{n-1}}}=0,\\
\displaystyle{\frac{x_2}{x_0}+\frac{x_3}{x_1}+\hdots+\frac{x_1}{x_{n-1}}}=0,\\
\ \ \ \ \ \ \ \ \ \ \ \ \ \ \ \vdots\\
\displaystyle{\frac{x_{n-1}}{x_0}+\frac{x_0}{x_1}+\hdots+\frac{x_{n-2}}{x_{n-1}}}=0.\\
\end{cases}\eeq
We have the following well-known
\begin{ex}
For $n\geq 2$ let $(\alpha,\beta)\in \mathbb{Z}_n^\ast\times \mathbb{Z}_n$ and consider the vector $x^{(\alpha,\beta)}$ with coordinates
\beql{classicalc}
x_i^{(\alpha,\beta)}=\begin{cases}
\mathrm{exp}\left(\frac{2\pi\mathbf{i}}{n}\left(\frac{\alpha i^2}{2}+\beta i\right)\right), & i\in\mathbb{Z}_n,\ \ n \text{ even,}\\
\mathrm{exp}\left(\frac{2\pi\mathbf{i}}{n}\left(\frac{\alpha i(i-1)}{2}+\beta i\right)\right), & i\in\mathbb{Z}_n,\ \ n \text{ odd.}\\
\end{cases}
\eeq
These solutions of \eqref{cycyc} are called the classical solutions leading to the Fourier matrices. 
\end{ex}
It is easy to see that if $x$ is a unimodular solution of \eqref{cycyc} then the Fourier-transformed vector $\widehat{x}$ is unimodular as well:
\beql{ch3enf}
|x_i|=|\widehat{x}_i|=1,\ \ i\in\mathbb{Z}_n,
\eeq
as
\[|\widehat{x}_k|^2=\frac{1}{n}\sum_{i=0}^{n-1}\sum_{j=0}^{n-1}\mathbf{e}^{\frac{2\pi\mathbf{i}}{n}(k-1)(i-j)}x_i\overline{x}_j=\frac{1}{n}\sum_{i=0}^{n-1}\mathbf{e}^{\frac{2\pi\mathbf{i}}{n}(k-1)i}\sum_{j=0}^{n-1}\frac{x_{i+j}}{x_j}=1.\]
In $1983$ Enflo asked \cite{BS1} whether the vectors 
\[cx^{(\alpha,\beta)},\ \ \ \ (c,\alpha,\beta)\in\mathbb{T}\times\mathbb{Z}_p^\ast\times \mathbb{Z}_p,\]
given by \eqref{classicalc} are the only solutions to the bi-unimodular problem \eqref{ch3enf} for prime $p$. It turns out, that this is only true for $p=2,3$ and $5$, as starting from $p\geq 7$ there are non-classical solutions to the problem, yielding non-classical circulant complex Hadamard matrices which are inequivalent from the Fourier matrix \cite{GB1}, \cite{HJ1}, \cite{MW1}.

Motivated by Enflo's question, Bj\"orck reformulated \eqref{cycyc} in \cite{GB1} by passing to the quotients $z_i:=x_{i+1}/x_i,\ \ \ i=0,1,\hdots, n-1$, where the indices are taken modulo $n$ and found that the system of equations \eqref{cycyc} boils down to
\beql{cycyc2}\begin{cases}
z_0+z_1+\hdots+z_{n-1}=0,\\
z_0z_1+z_1z_2+\hdots+z_{n-1}z_0=0,\\
\ \ \ \ \ \ \ \ \ \ \ \ \ \ \ \vdots\\
z_0z_1\hdots z_{n-2}+z_1z_2\hdots z_{n-1}+\hdots+z_{n-1}z_0\hdots z_{n-3}=0,\\
z_0z_1\hdots z_{n-1}=1.\\
\end{cases}\eeq
A solution of \eqref{cycyc2}, $z=(z_0,z_1,\hdots,z_{n-1})$, is called a cyclic $n$-root. Clearly, by the last equation $z_i\neq 0$, and hence any unimodular cyclic $n$-root $z$ can be transformed to a solution $x$ of \eqref{cycyc}, namely 
\beql{ch3cycSol}
x=(1,z_0,z_0z_1,\hdots,z_0z_1z_2\hdots z_{n-2}).
\eeq
The vector $x$ is called a cyclic $n$-root on the $x$-level. Let us recall some recent advances.
\begin{thm}[Haagerup, \cite{UH2}]
For every prime $p$ there are exactly $\binom{2p-2}{p-1}$ solutions to the cyclic $p$-root problem, counted with multiplicity. In particular, there are at most $\binom{2p-2}{p-1}$ circulant complex Hadamard matrices with diagonal $1$ of order $p$.
\end{thm}
\begin{rem}
The multiplicity of an isolated zero is an analytic property; the precise definition can be found in \cite{UH2} (which actually refers to the monograph \cite{AY1}). In triangular systems
\[\begin{cases}
f_1(x_1)=0,\\
f_2(x_1,x_2)=0,\\
\ \ \ \ \ \ \ \ \vdots\\
f_n(x_1,x_2,\hdots, x_n)=0,\\
\end{cases}\]
with finitely many solutions only (i.e.\ if all solutions are isolated) the multiplicity of a zero $y=(y_1,\hdots, y_n)$ is exactly $\prod_{i=1}^n m_i$, where $m_i$ is the multiplicity of the univariate polynomial $f_i(y_1,\hdots, y_{i-1},x_i)$ for every $i=1,\hdots, n$. For a proof see \cite{ZFX}.
\hfill$\square$\end{rem}
In contrast, square multiple $n$ always gives rise to an infinite family of solutions.
\begin{thm}[Backelin, \cite{JB1}]\label{ch3back}
If $m^2$ divides $n$, then there is an at least $(m-1)$-dimensional family of solutions to the cyclic $n$-root problem, namely
\[z=(z_0,z_1,\hdots,z_{m-1},\alpha z_0,\alpha z_1,\hdots,\alpha z_{m-1},\hdots,\alpha^{n/m-1}z_0,\hdots,\alpha^{n/m-1}z_{m-1}),\]
where $z_0,z_1,\hdots,z_{m-2}$ are free parameters, $z_0z_1\hdots z_{m-1}=1$ and $\alpha$ is any primitive $n/m$-th root of unity.
\end{thm}
The proof can be seen by substituting back to \eqref{cycyc2}. We omit the (lengthy) details.
\begin{ex}
Let $n=12$, $2^2|12$ so we expect to find a one-parameter family of circulant complex Hadamard matrices of order $12$. Let $\alpha=-\omega^2$, the primitive sixth root of unity, and consider the one-parameter family of cyclic $12$-roots \[z(a)=(a,\overline{a},-\omega^2a,-\omega^2\overline{a},\omega a,\omega \overline{a},-a,-\overline{a},\omega^2 a,\omega^2\overline{a},-\omega a,-\omega \overline{a}).\]
It is readily verified that $z(a)$ is a solution to \eqref{cycyc2} for any unimodular number $a$, and hence induces a circulant complex Hadamard matrix via \eqref{ch3cycSol} as
\[C_{12}^{(1)}(a)=\mathrm{Circ}(1,a,1,-\omega^2a,\omega,\omega^2 a,1,-a,1,\omega^2a,\omega,-\omega^2a).\]
It is easily seen, after dephasing the matrix $C_{12}^{(1)}(a)$, that the one degree of freedom does not vanish, and therefore $C_{12}^{(1)}(a)$ is a one-parameter family of circulant complex Hadamard matrices, as expected. It follows, that the triplet $\{I_{12},F_{12},C_{12}^{(1)}(a)\}$ is a one-parameter family of triplet of MUBs in $\mathbb{C}^{12}$ (see Section \ref{ch2MUB}).
\end{ex}
Finally, let us recall the following
\begin{thm}[Faug\`ere, \cite{JF2}, \cite{JF1}]
All cyclic $n$-roots are known up to order $n\leq 10$.
\end{thm}
We advise the reader that these solutions are rather complicated. Further, all cyclic $11$-roots are known numerically \cite{DKK} while starting from order $12$ only partial results are known \cite{RS1}. The cyclic $n$-root problem turned out to be a challenging problem in symbolic computation, mostly because of its large number of solutions and in practice it is used for various benchmarking purposes. As the problem of fully classifying cyclic $n$-roots and circulant complex Hadamard matrices seems far out of reach, one is interested in special, more structural solutions giving access to examples in infinitely many orders.

Following Bj\"orck and Haagerup \cite{GB1}, \cite{UH1} now we recall the concept of cyclic $p$-roots of simple index $k$ as follows. Let $p$ be a prime and let $k\in\mathbb{N}$ be a number that divides $p-1$. The group $\mathbb{Z}_p^\ast$ has a unique subgroup $G_0$ of index $k$, namely the subgroup of the $k$th residues as follows:
\[G_0=\{x^k,x\in \mathbb{Z}_p^\ast\}.\]
Moreover, if $g\in\mathbb{Z}_p^\ast$ is a generator of $\mathbb{Z}_p^\ast$ then the $k-1$ nontrivial cosets of $G_0$ in $\mathbb{Z}_p^\ast$ are given by
\beql{ch3coset}
G_i=g^iG_0,\ \ \ \ \ 1\leq i\leq k-1.
\eeq
Note that it follows that $g\in G_1$. Following the notation of \cite{BH1}, a cyclic $p$-root $z$ has simple index $k$ if the corresponding cyclic $p$-root on the $x$-level \eqref{ch3cycSol} is of the following form:
\beql{ch3formx}
\begin{cases}
x_i=c_k, \text{ if } i\in G_k, 1\leq i\leq p-1; \ \ \ \text{and}\\
x_0=1,\\
\end{cases}
\eeq
where $(c_0,c_1,\hdots, c_{k-1})\in\left(\mathbb{C}^\ast\right)^k$. These special cyclic $p$-roots were introduced by Bj\"orck in \cite{GB1}, who fully classified the case $k=2$ through the following useful reformulation of the problem.
\begin{prop}[Bj\"orck, \cite{GB1}]\label{ch3BJP}
If $x=(x_0,x_1,\hdots,x_{p-1})$ has the form \eqref{ch3formx}, then the equations \eqref{cycyc} can be reduced to the following $k$ rational equations in $c_0, c_1, \hdots, c_{k-1}$:
\beql{72haa}
c_a+\frac{1}{c_{a+m}}+\sum\limits_{i,j=0}^{k-1}n_{ij}\frac{c_{a+j}}{c_{a+i}}=0,\qquad 0\leq a\leq k-1,
\eeq
where all indices are calculated modulo $k$. In \eqref{72haa} the number $m$ is determined by $p-1\in G_m$ and $n_{ij}$ denote the transition numbers of order $k$, namely, the number of $b\in G_i$ for which $b+1\in G_{j}$, $0\leq i,j\leq k-1$, with respect to the generator $g$ of $\mathbb{Z}_p^\ast$.
\end{prop}
We remark that the system of equations \eqref{72haa} is independent of the choice of $g$ up to relabelling the variables. Curiously enough, the number of solutions is known.
\begin{thm}[Haagerup, \cite{UH2}]\label{ch3hindex4}
For every $k\in \mathbb{N}$ and for every prime number $p$ for which $k$ divides $p-1$ the number of solutions $(c_0,c_1,\hdots,c_{k-1}) \in \left(\mathbb{C}^\ast\right)^k$ to \eqref{72haa} counted with multiplicity is equal to $\binom{2k}{k}$.
\end{thm}
Throughout this chapter we are interested in the unimodular cyclic $n$-roots only, i.e.\ the ones that correspond to complex Hadamard matrices.
\subsection{\texorpdfstring{Cyclic $p$-roots of simple index $2$}{Cyclic p-roots of simple index 2}}
The following are the index $2$ type cyclic complex Hadamard matrices.
\begin{thm}[Bj\"orck, \cite{GB1}]\label{index2a}
Let $p\equiv 3$ $(\mathrm{mod}\ 4)$ be a prime and consider the decomposition $J=I+S+N$ where $S$ and $N$ are $(0,1)$ circulant matrices whose first rows are the characteristic vector of the quadratic residues and nonresidues modulo $p$, respectively. Then the matrix $H:=I+S+\alpha N$ is complex Hadamard, where
\[\alpha=-\frac{p-1}{p+1}\pm\mathbf{i}\frac{2\sqrt{p}}{p+1}.\]
\end{thm}
Similarly, we have
\begin{thm}[Bj\"orck, \cite{GB1}]\label{index2b}
Let $p\equiv 1$ $(\mathrm{mod}\ 4)$ be a prime and consider the decomposition $J=I+S+N$ where $S$ and $N$ are $(0,1)$ circulant matrices whose first row are the characteristic vector of the quadratic residues and nonresidues modulo $p$, respectively. Then the matrices $H:=I+\alpha S+\overline{\alpha}N$ and $\overline{H}$ are complex Hadamard, where
\[\alpha=\frac{-1\pm\sqrt{p}}{p-1}+\mathbf{i}\frac{\sqrt{p^2-3p\pm2\sqrt{p}}}{p-1},\]
where the $\pm$ signs agree.
\end{thm}
For $p\geq 7$ Theorem \ref{index2a} and Theorem \ref{index2b} give rise to nonclassical circulant complex Hadamard matrices.
These results were subsequently rediscovered in \cite{MW1} and \cite{HJ1}.
\begin{rem}
Theorems \ref{index2a} and \ref{index2b} essentially exploit the fact that the underlying objects, namely the Hadamard design (when $p\equiv 3$ (mod $4$)) and the core of a conference matrix (when $p\equiv 1$ (mod $4$)), have very rich combinatorial structure. We have generalized these constructions for non-prime orders as well \cite{SZF2} (cf.\ Theorems \ref{CGtwoe} and \ref{ch3f5gen}, respectively).
\hfill$\square$\end{rem}
\subsection{\texorpdfstring{Cyclic $p$-roots of simple index $3$}{Cyclic p-roots of simple index 3}}
In a recent paper Bj\"orck and Haagerup continued investigating the cyclic $n$-root problem by giving a full algebraic classification of all cyclic $p$-roots of index $3$ \cite{BH1}. We recall their result in the following
\begin{thm}[Bj\"orck--Haagerup, \cite{BH1}]
For every prime $p\equiv 1$ $(\mathrm{mod}\ 6)$ all cyclic $p$-roots of index $3$ can be described by closed analytic formulae.
\end{thm}
The resulting formulas they obtained are highly non trivial, and here we recall the unimodular solutions as follows.
\begin{thm}[Bj\"orck--Haagerup, \cite{BH1}]
For every prime $p\equiv 1$ $(\mathrm{mod}\ 6)$ and for $k=3$ there are exactly $12$ unimodular solutions of the system of equations \eqref{72haa}, given by $(c_0^{(j)},c_1^{(j)},c_2^{(j)}), j=1,2$ and their cyclic permutation and their conjugate, where
\[c_i^{(j)}=\alpha^{(j)}+\beta^{(j)}\cos\left(\theta-\frac{2\pi}{3}i\right)+\gamma^{(j)}\sin\left(\theta-\frac{2\pi}{3}i\right),\ i=0,1,2,\ j=1,2,\]
where
\[\theta=\frac{1}{3}\arccos\left(\frac{A\sqrt{p}}{2p}\right),\]
and
\begin{gather*}\alpha^{(1)}=\frac{pA-2p-2A}{2(p^2-3p-A)}+\mathbf{i}\frac{3\sqrt{3}\sqrt{p}\sqrt{p-4}B}{2(p^2-3p-A)},\\
\beta^{(1)}=-\frac{\sqrt{p}(p-4)(A+2)}{2(p^2-3p-A)}-\mathbf{i}\frac{3\sqrt{3}\sqrt{p-4}(p-2)B}{2(p^2-3p-A)},\\
\gamma^{(1)}=-\frac{3\sqrt{3}\sqrt{p}(p-4)B}{2(p^2-3p-A)}+\mathbf{i}\frac{\sqrt{p-4}(pA-2p-2A)}{2(p^2-3p-A)},
\end{gather*}
and
\begin{gather*}\alpha^{(2)}=-\frac{u^2-uv-4}{2(u^2+uv+2)}+\mathbf{i}\frac{u\sqrt{4+u-v}\sqrt{4-u+v}}{2(u^2+uv+2)},\\
\beta^{(2)}=\frac{u(A+2)}{u^2+uv+2}+\mathbf{i}\frac{(u^2+uv+4)\sqrt{4+u-v}\sqrt{4-u+v}}{4(u^2+uv+2)},\\
\gamma^{(2)}=\frac{3\sqrt{3}uB}{u^2+uv+2}+\mathbf{i}\frac{(u^2-uv-4)\sqrt{u+v+4}\sqrt{u+v-4}}{4(u^2+uv+2)},
\end{gather*}
where $u=\sqrt{p}$, $v=\sqrt{p+4A+16}$ and $A\equiv 1$ $(\mathrm{mod}\ 3)$, $B>0$ comes from the Gaussian decomposition $4p=27A^2+B^2$. Additionally, the generator $g$ of $\mathbb{Z}_p^\ast$ should be chosen in a way that the nontrivial cosets $($see \eqref{ch3coset}$)$ $G_1$ and $G_2$ are defined to meet $n_{01}<n_{02}$ with respect to $g$.
\end{thm}
We offer the following enlightening
\begin{ex}
Let $p=7$. Then the Gaussian decomposition $4p=A^2+27B^2$ implies $A=B=1$. Let us choose a generator of $\mathbb{Z}_7^\ast$ as $g=2$. Then we find that $G_0=\{1,6\}$, $G_1=\{2,5\}$ and $G_2=\{3,4\}$, but as the transition numbers read $n_{01}=1, n_{02}=0$ we should interchange the r\^ole of $G_1$ and $G_2$ (or equivalently, choose another generator, say $g=3$). Therefore, we define the first row of the index $3$ circulant complex Hadamard matrix $H$ as follows: $x=(1,c_0,c_2,c_1,c_1,c_2,c_0)$, where
\[c_i^{(j)}=\alpha^{(j)}+\beta^{(j)}\cos\left(\theta-\frac{2\pi}{3}i\right)+\gamma^{(j)}\sin\left(\theta-\frac{2\pi}{3}i\right),\ i=0,1,2,\ j=1,2,\]
where
\[\theta=\frac{1}{3}\arccos\left(\frac{\sqrt{7}}{14}\right),\]
and
\[\alpha^{(1)}=-\frac{1}{6}+\mathbf{i}\frac{\sqrt{7}}{6},\qquad \beta^{(1)}=-\frac{\sqrt7}{6}-\frac{5}{6}\mathbf{i},\qquad \gamma^{(1)}=-\frac{\sqrt{21}}{6}-\mathbf{i}\frac{\sqrt3}{6},\]
and
\begin{gather*}
\alpha^{(2)}=\frac{6-\sqrt{21}}{6}+\mathbf{i}\frac{\sqrt{-63+14\sqrt{21}}}{6},\qquad \beta^{(2)}=\frac{7\sqrt{3}-3\sqrt{7}}{12}+\mathbf{i}\frac{\sqrt{-18+26\sqrt{21}}}{12},\\
\gamma^{(2)}=\frac{7-\sqrt{21}}{4}-\mathbf{i}\frac{\sqrt{-54+14\sqrt{21}}}{4}.
\end{gather*}
We remark that for $j=1$ the first set of solutions gives $(w^4,w,w^2)$ yielding the Fourier matrix $F_7^{(0)}=\mathrm{Circ}(1,w^4,w^2,w,w,w^2,w^4)$ where $w=\mathbf{e}^{2\pi\mathbf{i}/7}$ is the principal $7$th root of unity; in contrast, the numbers $(c_0^{(2)},c_1^{(2)},c_2^{(2)})$ arising from the second set of solutions have the following minimal polynomial:
\[c^{12}-12 c^{11}+24 c^{10}+4 c^9-9 c^8+27 c^7-21
   c^6+27 c^5-9 c^4+4 c^3+24 c^2-12 c+1.\]
\end{ex}
\subsection{\texorpdfstring{Cyclic $p$-roots of simple index $4$}{Cyclic p-roots of simple index 4}}
In this section we consider cyclic $p$-roots of simple index $4$ focusing especially on the case $p\equiv 1$ (mod $8$). In this case $-1$ is a quartic residue and therefore the system of equations \eqref{72haa} is ``symmetric'' in the sense that $m=0$. We exhibit a new, previously unknown family of circulant complex Hadamard matrices, and outline a general method obtaining all cyclic $p$-roots of simplex index $4$ for any given $p\equiv 1$ (mod $8$). We give detailed results for $p=17$ in Appendix \ref{APPB}.

We have already seen in the simple index $3$ case that some number theoretic properties of the prime $p$ played important r\^ole in the determination of the cyclotomic numbers. Here we have an analogous phenomenon.
\begin{thm}[Katre--Rajwade, \cite{KR1}]\label{ch3cycEQ}
Let $p\equiv 1$ $(\mathrm{mod}\ 8)$ be a prime and let $g$ be a generator of $\mathbb{Z}_p^\ast$. Define $s\equiv 1$ $(\mathrm{mod}\ 4)$ uniquely by $p=s^2+t^2$, and then the sign of $t$ uniquely by $t\equiv sg^{3(p-1)/4}$ $(\mathrm{mod}\ p)$. Then the transition numbers of order $4$ for $\mathbb{Z}_p$, with respect to the generator $g$, are given by
\begin{eqnarray*}
n_{00} & = & \frac{p-11-6s}{16},\\
n_{01}=n_{10}=n_{33} & = & \frac{p-3+2s+4t}{16},\\
n_{02}=n_{20}=n_{22} & = & \frac{p-3+2s}{16},\\
n_{03}=n_{30}=n_{11} & = & \frac{p-3+2s-4t}{16},\\
n_{12}=n_{21}=n_{13}=n_{31}=n_{23}=n_{32} & = & \frac{p+1-2s}{16}.
\end{eqnarray*}
\end{thm}
\begin{rem}
We can choose a generator $g$ of $\mathbb{Z}_p^\ast$ in a way that
\beql{ch3arri}
n_{01}>n_{03},
\eeq
or equivalently, $t>0$ in the decomposition of $p$. This can be always achieved by passing to a new generator (if it is necessary) $g':=g^{4i+3}$ for some $i$ where $4i+3$ is relative prime to $p-1$. As $g\in G_1$ and $g'\in G_3$ this will interchange $G_1$ and $G_3$ and thus we have arrived at \eqref{ch3arri}.
\hfill$\square$\end{rem}
Some relations amongst the numbers $n_{ij}$ are easy to see, for example, for $p\equiv 1$ (mod $8$) $-1$ is a quartic residue, and therefore $n_{ij}=n_{ji}, i,j=0,1,2,3$. Further,
\[\sum_{j=0}^3n_{ij}=\sharp\{(G_i\setminus\{p-1\})\},\]
and therefore
\[\sum_{j=0}^3n_{0j}=\frac{p-1}{4}-1,\ \ \ \sum_{j=0}^3n_{1j}=\sum_{j=0}^3n_{2j}=\sum_{j=0}^3n_{3j}=\frac{p-1}{4}.\]
There are further nontrivial relations in-between these numbers and some ideas how to obtain them are highlighted in \cite{BH1}. Now by Proposition \ref{ch3BJP} the cyclic $p$-roots of simple index $4$ on the $x$-level can be described by the solutions $(c_0,c_1,c_2,c_3)$ of the following four equations:
\begin{equation}\label{ch3eq1}
\begin{split}c_0+\frac{1}{c_0}&= -\frac{p-5}{4}-n_{01}\left(\frac{c_0}{c_1}+\frac{c_1}{c_0}\right)-n_{02}\left(\frac{c_0}{c_2}+\frac{c_2}{c_0}\right)-n_{03}\left(\frac{c_0}{c_3}+\frac{c_3}{c_0}\right)\\
&\quad-n_{12}\left(\frac{c_1}{c_2}+\frac{c_2}{c_1}+\frac{c_1}{c_3}+\frac{c_3}{c_1}+\frac{c_2}{c_3}+\frac{c_3}{c_2}\right),
\end{split}
\end{equation}
\begin{equation}\label{ch3eq2}
\begin{split}c_1+\frac{1}{c_1}&=-\frac{p-5}{4}-n_{01}\left(\frac{c_1}{c_2}+\frac{c_2}{c_1}\right)-n_{02}\left(\frac{c_1}{c_3}+\frac{c_3}{c_1}\right)-n_{03}\left(\frac{c_1}{c_0}+\frac{c_0}{c_1}\right)\\
&\quad-n_{12}\left(\frac{c_2}{c_3}+\frac{c_3}{c_2}+\frac{c_2}{c_0}+\frac{c_0}{c_2}+\frac{c_3}{c_0}+\frac{c_0}{c_3}\right),
\end{split}
\end{equation}
\begin{equation}\label{ch3eq3}
\begin{split}c_2+\frac{1}{c_2}&=-\frac{p-5}{4}-n_{01}\left(\frac{c_2}{c_3}+\frac{c_3}{c_2}\right)-n_{02}\left(\frac{c_2}{c_0}+\frac{c_0}{c_2}\right)-n_{03}\left(\frac{c_2}{c_1}+\frac{c_1}{c_2}\right)\\
&\quad-n_{12}\left(\frac{c_3}{c_0}+\frac{c_0}{c_3}+\frac{c_3}{c_1}+\frac{c_1}{c_3}+\frac{c_0}{c_1}+\frac{c_1}{c_0}\right),
\end{split}
\end{equation}
\begin{equation}\label{ch3eq4}
\begin{split}c_3+\frac{1}{c_3}=&-\frac{p-5}{4}-n_{01}\left(\frac{c_3}{c_0}+\frac{c_0}{c_3}\right)-n_{02}\left(\frac{c_3}{c_1}+\frac{c_1}{c_3}\right)-n_{03}\left(\frac{c_3}{c_2}+\frac{c_2}{c_3}\right)\\
&\quad-n_{12}\left(\frac{c_0}{c_1}+\frac{c_1}{c_0}+\frac{c_0}{c_2}+\frac{c_2}{c_0}+\frac{c_1}{c_2}+\frac{c_2}{c_1}\right),
\end{split}
\end{equation}
with $n_{01}, n_{02}, n_{03}$ and $n_{12}$ being given by Theorem \ref{ch3cycEQ}.
\begin{rem}\label{ch3remall}
It is easy to see that if $(c_0,c_1,c_2,c_3)$ is a solution to the above system of equations, then so are $(\overline{c_0},\overline{c_1},\overline{c_2},\overline{c_3})$, $(1/c_0,1/c_1,1/c_2,1/c_3)$ and $(1/\overline{c_0},1/\overline{c_1},1/\overline{c_2},1/\overline{c_3})$ as well as any cyclic permutation of these four.
\hfill$\square$\end{rem}
One can solve this system of equations by a straightforward application of the theory of Gr\"obner basis. There are two arising problems with this method, however. The first is that one encounters high degree (i.e.\ larger than $8$) irreducible palindromic polynomials and it is unclear how to solve them by radicals. The second problem is that it is unclear how to describe the resulting Gr\"obner basis by explicit formulas depending on the prime $p$.

Currently we cannot solve this problem in full generality but we extract some solutions ``by hand'' in order to find the splitting field of the polynomials forming a pure lexicographic Gr\"obner basis of the system of equations \eqref{ch3eq1}-\eqref{ch3eq4}. In particular, we solve the case $(c_0,c_1,c_2,c_3)=(a,b,\overline{a},\overline{b})$, where $a$ and $b$ are unimodular, which is our main contribution to this section.
\begin{thm}\label{ch3cyc4M}
Let $p$ be a prime number, $p\equiv 1$ $(\mathrm{mod}\ 8)$, $p=s^2+t^2$ such that $s\equiv 1$ $(\mathrm{mod}\ 4)$ and $t>0$. Then there are two symmetric circulant complex Hadamard matrices of order $p$ of simple index $4$, corresponding to the two values of $\zeta_{\pm}=\frac{2(-1\pm\sqrt{p})}{p-1}$, where $(c_0,c_1,c_2,c_3)=(a,b,\overline{a},\overline{b})$, or any cyclic permutation of it, where
\begin{gather}\label{ch3reaf}
a=\frac{\zeta_{\pm}}{2}+\sqrt{A+B\sqrt{p}}-\sqrt{C+D\sqrt{p}}+\mathbf{i}\sqrt{1-\left(\frac{\zeta_{\pm}}{2}+\sqrt{A+B\sqrt{p}}-\sqrt{C+D\sqrt{p}}\right)^2},\\
\label{ch3reaf2}
b=\frac{\zeta_{\pm}}{2}-\sqrt{A+B\sqrt{p}}+\sqrt{C+D\sqrt{p}}+\mathbf{i}\sqrt{1-\left(\frac{\zeta_{\pm}}{2}-\sqrt{A-B\sqrt{p}}+\sqrt{C+D\sqrt{p}}\right)^2},
\end{gather}
where
\begin{gather}
\label{ch3AB}
A=\frac{2 p (p-2 s+1)}{t^2(p-1)^2}, \ \ B=\pm\frac{2 (p (s-2)+s)}{t^2(p-1)^2}\\
\label{ch3CD}
C=\frac{p \left(p \left(t^2+2\right)-4 s-3
   t^2+2\right)}{t^2(p-1)^2}, \ \ D=\pm\frac{2 \left(p (s-2)+s+t^2\right)}{t^2(p-1)^2},
\end{gather}
such that the $\pm$ signs in $B$ and $D$ agree with the sign in $\zeta_{\pm}$.
\end{thm}
\begin{rem}
Note that $a$ and $b$ under \eqref{ch3reaf} and \eqref{ch3reaf2} are well-defined, as
\begin{gather*}
t^2(p-1)^2(A+B\sqrt{p})=2(-1\pm\sqrt{p})^2\sqrt{p}(\sqrt{p}\pm s)>0,\text{ and}\\
t^2(p-1)^2(C+D\sqrt{p})=\left(-1\pm\sqrt{p}\right)^2 \sqrt{p} \left(t^2
   \left(\sqrt{p}\pm2\right)+2 \left(\sqrt{p}\pm s\right)\right)>0,
\end{gather*}
where the $\pm$ sign agrees with the sign in $\zeta_{\pm}$.
\hfill$\square$\end{rem}
\begin{proof}
We consider the system of equations \eqref{ch3eq1}-\eqref{ch3eq4} coming from Proposition \ref{ch3BJP} and impose the simplifying assumption $c_2=\overline{c_0}$ and $c_3=\overline{c_1}$ and relabel the variables $a:=c_0$ and $b:=c_1$ to avoid the excessive usage of the subscripts. As we are interested in constructing complex Hadamard matrices, we further assume that $a$ and $b$ are unimodular, and in particular $a+1/a=a+\overline{a}=2\Re[a]$. Therefore we find that the problem of solving the four equations \eqref{ch3eq1}-\eqref{ch3eq4} boils down to solving
\begin{gather}\label{ch3zeta1A}
2\Re[a]+\frac{p+2s-3}{4}\Re[a]^2+\frac{p-2s+1}{4}\Re[b]^2+\frac{p+2t-1}{4}\Re[a/b]+\frac{p-2t-1}{4}\Re[ab]=1,\\
\label{ch3zeta1}
2\Re[b]+\frac{p+2s-3}{4}\Re[b]^2+\frac{p-2s+1}{4}\Re[a]^2+\frac{p+2t-1}{4}\Re[ab]+\frac{p-2t-1}{4}\Re[a/b]=1.
\end{gather}
Sum up and multiply by $2$ \eqref{ch3zeta1A}-\eqref{ch3zeta1}, then expand $\Re[ab]$ and $\Re[a/b]$ to obtain
\[(p-1)(\Re[a]+\Re[b])^2+4(\Re[a]+\Re[b])-4=0.\]
In particular, we have arrived at a formula involving $\zeta_{\pm}$:
\[\Re[a]+\Re[b]=\frac{2(-1\pm\sqrt{p})}{p-1}=\zeta_{\pm}.\]
Now let us substitute $\Re[b]=\zeta_\pm-\Re[a]$ into \eqref{ch3zeta1} to get
\beql{ch3imsign}
4t\Im[a]\Im[b]=2\left(\zeta_{\pm}-2\Re[a]\right)\left(2+\zeta_{\pm}(s-1)\right).
\eeq
Square both sides and eliminate $\Im[a]^2\equiv1-\Re[a]^2$, rearrange to get
\begin{equation}\label{ch3eq4r}
\begin{split}
16t^2\Re[a]^4-32t^2\zeta_{\pm} \Re[a]^3+16\left(\zeta_{\pm}^2 \left(t^2-(s-1)^2\right)+4\zeta_{\pm}(1-s)-2 t^2-4\right)\Re[a]^2\\
+8\left(\zeta_{\pm}^3 \left( p s- p+2 s^2-5 s+3\right)+2\zeta_{\pm}^2 (p+6 s-7)+4\zeta_{\pm} \left(t^2-s+5\right)-8\right)\Re[a]\\
+16 t^2 \left(1-\zeta_{\pm}^2\right)-\left(\zeta_{\pm}^2 (p+2 s-3)+8 \zeta_{\pm}-4\right)^2=0
\end{split}
\end{equation}
Now it is easy to see that for a fixed $\zeta_{\pm}$ the four roots of \eqref{ch3eq4r} are exactly the numbers
\beql{ch3fourr}
\Re[a]_{1,2,3,4}=\frac{\zeta_{\pm}}{2}\pm\sqrt{A+B\sqrt{p}}\pm\sqrt{C+D\sqrt{p}},
\eeq
by factoring out $16t^2(\Re[a]-\Re[a]_1)(\Re[a]-\Re[a]_2)(\Re[a]-\Re[a]_3)(\Re[a]-\Re[a]_4)$ and reducing modulo the identity $p=s^2+t^2$ in order to obtain \eqref{ch3eq4r}.

We proceed by showing that for $\zeta_{\pm}$ fixed either of the sign choice $(+,-)$ or $(-,+)$ in \eqref{ch3fourr} implies that $|\Re[a]|\leq 1$. Note that
\beql{ch3ki}
C+D\sqrt{p}-(A+B\sqrt{p})=\frac{\sqrt{p}(p\sqrt{p}-3\sqrt{p}\pm2)}{(p-1)^2}\geq0
\eeq
for all relevant primes $p$, where the sign choice agrees with the sign choice of $\zeta_{\pm}$.

We begin by investigating the sign choice $(+,-)$. The case $\Re[a]\leq 1$ is trivial by \eqref{ch3ki}. To see the other bound rearrange to obtain, again, by \eqref{ch3ki}
\[0\leq \sqrt{C+D\sqrt{p}}-\sqrt{A+B\sqrt{p}}\leq 1+\frac{\zeta_{\pm}}{2}.\]
Now square in order to get, after a considerable amount of calculations,
\[\frac{2(-1\pm\sqrt{p})^2(2p-2t^2\pm\sqrt{p}(2s-t^2))}{t^2(p-1)^2}\leq2\sqrt{AC+BDp+(AD+BC)\sqrt{p}},\]
which becomes, after squaring again
\[\frac{(-1\pm\sqrt{p})^4(2p-2t^2\pm\sqrt{p}(2s-t^2))^2}{t^4(p-1)^4}\leq\frac{2(-1\pm\sqrt{p})^4p(\sqrt{p}\pm s)(\sqrt{p}(2+t^2)\pm2(s+t^2))}{t^4(p-1)^4}\]
which holds if and only if
\[0\leq (2\pm\sqrt{p})^2t^2(\sqrt{p}\pm s)^2.\]
The sign choice $(-,+)$ follows along similar lines. The case $-1\leq\Re[a]$ is again trivial, while the other direction leads to equation
\[0\leq \sqrt{C+D\sqrt{p}}-\sqrt{A+B\sqrt{p}}\leq 1-\frac{\zeta_{\pm}}{2}.\]
After squaring and rearrangement we get
\[\frac{2(-1\pm\sqrt{p})^2\sqrt{p}(2\sqrt{p}\pm(2s+t^2))}{t^2(p-1)^2}\leq2\sqrt{AC+BDp+(AD+BC)\sqrt{p}},\]
which, after an additional squaring amounts to
\[\frac{(-1\pm\sqrt{p})^4p(2\sqrt{p}\pm(2s+t^2))^2}{t^4(p-1)^4}\leq\frac{2(-1\pm\sqrt{p})^4p(\sqrt{p}\pm s)(\sqrt{p}(2+t^2)\pm2(s+t^2))}{t^4(p-1)^4}.\]
Rearrange and reduce modulo $s^2+t^2=p$ to finally obtain
\[0\leq t^2(2p\pm2\sqrt{p}s-t^2)=t^2(s\pm\sqrt{p})^2.\]
To conclude, observe that the sign of $\Im[b]$ is determined by the sign of $\zeta_{\pm}-2\Re[a]$, as the second factor of \eqref{ch3imsign} is positive, and hence $b$ is uniquely determined by $a$. It is easy to see that the solutions coming from any of the four possible values of $a$ (for a fixed $\zeta_{\pm}$) are exactly the four cyclic permutations of the initial solution $(c_0,c_1,c_2,c_3)=(a,b,\overline{a},\overline{b})$, where $a$ and $b$ are given by \eqref{ch3reaf} and \eqref{ch3reaf2}.
\end{proof}
\begin{rem}
The reader might amuse him or herself by proving that the two-two other roots corresponding to the sign choice $(+,+)$ and $(-,-)$ in \eqref{ch3fourr} will not lead to unimodular solutions. We do not include the details of this fruitless calculation.
\hfill$\square$\end{rem}
We do offer, however, the following enlightening
\begin{ex}\label{ch3cyc4ex}
Let $p=17$. Choose a generator $g=3$ of $\mathbb{Z}_{17}^\ast$. Then we find that $(s,t)=(1,4)$, $t>0$ as required by \eqref{ch3arri}, so we keep this generator and proceed. The four sets, $G_0, G_1, G_2$ and $G_3$ are defined as $G_i:=\{g^{4j+i} : j=0,\hdots,16\}$. We obtain $G_0=\{1,4,13,16\}$, $G_1=\{3,5,12,14\}$, $G_2=\{2,8,9,15\}$ and $G_3=\{6,7,10,11\}$. Then for $\zeta_+=\frac{1}{8}(-1+\sqrt{17})$ we find that $(A,B,C,D)=\frac{1}{128}(17,-1,17,0)$ and we define
\begin{gather*}
\Re[c_0]:=\frac{-1+\sqrt{17}}{16}+\sqrt{\frac{17}{128}-\frac{1}{128}\sqrt{17}}-\sqrt{\frac{17}{128}},\ \  \Im[c_0]:=\sqrt{1-\Re[c_0]^2},\\
\Re[c_1]:=\frac{-1+\sqrt{17}}{16}-\sqrt{\frac{17}{128}-\frac{1}{128}\sqrt{17}}+\sqrt{\frac{17}{128}},\ \  \Im[c_1]:=\sqrt{1-\Re[c_1]^2},\\
c_2=\overline{c_0},\ \ c_3=\overline{c_1}.
\end{gather*}
Then $H=\mathrm{Circ}(1,c_0,c_2,c_1,c_0,c_1,c_3,c_3,c_2,c_2,c_3,c_3,c_1,c_0,c_1,c_2,c_0)$ is a circulant complex Hadamard matrix.

Similarly, for $\zeta_-=\frac{1}{8}(-1-\sqrt{17})$ we find that $(A,B,C,D)=\frac{1}{128}(17,1,17,0)$ and we define
\begin{gather*}
\Re[c_0]:=\frac{-1-\sqrt{17}}{16}+\sqrt{\frac{17}{128}+\frac{1}{128}\sqrt{17}}-\sqrt{\frac{17}{128}},\ \ \Im[c_0]:=\sqrt{1-\Re[c_0]^2},\\
\Re[c_1]:=\frac{-1-\sqrt{17}}{16}-\sqrt{\frac{17}{128}+\frac{1}{128}\sqrt{17}}+\sqrt{\frac{17}{128}},\ \ \Im[c_1]:=\sqrt{1-\Re[c_1]^2},\\
c_2=\overline{c_0},\ \ c_3=\overline{c_1}.
\end{gather*}
Then $H=\mathrm{Circ}(1,c_0,c_2,c_1,c_0,c_1,c_3,c_3,c_2,c_2,c_3,c_3,c_1,c_0,c_1,c_2,c_0)$ is a circulant complex Hadamard matrix. The two matrices are inequivalent, as they have different fingerprint.
\end{ex}
A further useful contribution of Theorem \ref{ch3cyc4M} is that it reveals a field extension where additional cyclic $p$-roots of simple index $4$ live. The following is an intriguing
\begin{cor}\label{ch3i4fg}
All cyclic $17$-roots of simple index $4$ can be described by closed analytic formulae.
\end{cor}
\begin{proof}
Non-trivial use of computer algebra. Consult Appendix \ref{APPB} for the details.
\end{proof}

The case $p\equiv 5$ (mod $8$) seems more difficult as we lose the symmetry in equation \eqref{72haa} of Proposition \ref{ch3BJP} as $-1$ is not a quartic residue any longer. Despite the result highlighted in Corollary \ref{ch3i4fg} the simple index $4$ case remains open in general.
\begin{pro}
Give a detailed algebraic classification of all cyclic $p$-roots of simple index $4$.
\end{pro}

Finally, it is natural to ask if there are real circulant complex Hadamard matrices (cf.\ Conjecture \ref{ch1ry}). We offer the following possible attack to Ryser's conjecture.
\begin{pro}
Investigate if there are real cyclic $n$-roots by studying \eqref{cycyc} and its possible real solutions.
\end{pro}
Having investigated circulant complex Hadamard matrices it is also natural to study complex Hadamard matrices with circulant core. This topic is the subject of Appendix~\ref{APPC}, where we present some additional constructions along with some sporadic examples of complex Hadamard matrices. Some of the results are similar in spirit (cf.\ Propositions \ref{ch3core1} and \ref{ch3core2} with Theorem \ref{ch3cyc4M}), while others are rather technical (see Theorem \ref{ch3corethm7}), and rely heavily on computer algebra. The main result of the supplement is the discovery of a new complex Hadamard matrix of order $7$, which we summarize in the following
\begin{ex}\label{ch3core7}
Let
\beql{ch3core7fa}
\alpha=\frac{40169}{3}+\frac{50\sqrt{1993741}}{3}\cos\left(\frac{1}{3}\arccos\left(\frac{2731019453\sqrt{1993741}}{1993741^2}\right)-\frac{4\pi}{3}\right),
\eeq
\begin{equation}
\begin{split}\label{ch3core7f}
h(u)&=2054570000 u^6+4109140000
   u^5+\left(16 \alpha^2+9768064 \alpha-5227993936\right)
   u^4\\
&+\left(1956 \alpha^2+64132324 \alpha-12170223176\right)
   u^3+\left(11393 \alpha^2+427075897
   \alpha\right.\\
   &\left.+1676016222\right) u^2+\left(4644
   \alpha^2+2280446676 \alpha+8524444776\right) u\\
   &-17074   \alpha^2+3269963754 \alpha+2727593304.
\end{split}
\end{equation}
The six real roots of $h(u)$, $r_1<r_2<r_3<r_4<r_5<r_6$, describe twice the real part of six unimodular numbers $z_i=r_i/2+\mathbf{i}\sqrt{1-r_i^2/4}$, $i=1,2,\hdots,6$. In particular, $\mathrm{Circ}([z_1,z_3,\overline{z_2},z_4,\overline{z_5},\overline{z_6}])$, bordered by a row and column of numbers $1$ is a new example of complex Hadamard matrices of order $7$, which we call $Q_7$. We remark here that the entries of $Q_7$ can be expressed by radicals (see Proposition \ref{appcrad}).
\end{ex}
\section{Hadamard matrices with few different entries}
In the preceding section we found examples of complex Hadamard matrices containing only a ``few'' different entries. This restriction on the matrices allowed us to solve the otherwise extremely complicated orthogonality equations resulting in analytic examples of complex Hadamard matrices. It is natural to investigate this problem in more generality. As a trivial warm-up result we note that the Fourier matrix $F_1$ is the only complex Hadamard matrix in which all entries are equal.
\subsection{Two different entries}
Real Hadamard matrices contain two different entries only, and it is natural to ask what other complex Hadamard matrices have this property. Clearly, if $H$ is such a matrix, then so is $aH$ for all unimodular numbers $a$, and therefore we can suppose, up to equivalence, that one of the two entries in $H$ is $1$.
\begin{ex}\label{ch3ex22}
The following are complex Hadamard matrices with two different entries. Note that they are not scalar multiples of the real Hadamard matrix $F_2$.
\[\left[\begin{array}{rr}
1 & \mathbf{i}\\
\mathbf{i} & 1\\
\end{array}\right],\quad
\left[\begin{array}{rr}
1 & -\mathbf{i}\\
-\mathbf{i} & 1\\
\end{array}\right],\quad
\left[\begin{array}{rr}
\mathbf{i} & 1\\
1 & \mathbf{i}\\
\end{array}\right],\quad
\left[\begin{array}{rr}
-\mathbf{i} & 1\\
1 & -\mathbf{i}\\
\end{array}\right].\]
\end{ex}
We have a given a non-trivial construction of complex Hadamard matrices with two different entries in \cite{SZF2} (cf.\ \cite{SZF3}) by design theoretical means.
\begin{defi}
A \emph{$2$-$(v,k,\lambda)$ design} is a $v\times v$ $(0,1)$-matrix $X$ such that $XJ=JX=kJ$ and $XX^T=(k-\lambda)I+\lambda J$. We say that a $2$-$(4m-1,2m-1,m-1)$ design is a \emph{Hadamard design}.
\end{defi}
\begin{rem}
If $X$ is a Hadamard design, then $2X-J$ is the core of a real Hadamard matrix. Conversely, if $H$ is a dephased real Hadamard matrix, then its core becomes a Hadamard design after replacing its $-1$ entries with $0$.
\hfill$\square$\end{rem}
\begin{rem}
If $X$ is $2$-$(v,k,\lambda)$ design, then $J-X$ is $2$-$(v,v-k,v-2k+\lambda)$ design. We say that $X$ and $J-X$ are the complement of each other.
\hfill$\square$\end{rem}
\begin{thm}[see \cite{SZF2}, \cite{SZF3}]\label{CGtwoe}
Let $X$ be a Hadamard design, and replace every number $0$ in it with the complex unimodular number
\beql{CGuni}
a=-1+\frac{1}{2m}\pm\mathbf{i}\frac{\sqrt{4m-1}}{2m}.
\eeq
Then the obtained matrix is a complex Hadamard matrix with two different entries.
\end{thm}
\begin{rem}
We readily see that if $H=X+a(J-X)$ is a complex Hadamard matrix, coming from Theorem \ref{CGtwoe}, then $\overline{a}H=(J-X)+\overline{a}(J-(J-X))$. Hence the same construction works when $X$ is the complement of a Hadamard design.
\hfill$\square$\end{rem}
\begin{ex}
By considering the $2$-$(7,3,1)$ Hadamard design $X$, one can obtain a complex Hadamard matrix $H$ of order $7$, with two different entries as follows:
\[X=\left[
\begin{array}{ccccccc}
 0 & 1 & 0 & 1 & 0 & 1 & 0 \\
 1 & 0 & 0 & 1 & 1 & 0 & 0 \\
 0 & 0 & 1 & 1 & 0 & 0 & 1 \\
 1 & 1 & 1 & 0 & 0 & 0 & 0 \\
 0 & 1 & 0 & 0 & 1 & 0 & 1 \\
 1 & 0 & 0 & 0 & 0 & 1 & 1 \\
 0 & 0 & 1 & 0 & 1 & 1 & 0
\end{array}
\right]\leadsto
H=\left[
\begin{array}{ccccccc}
 a & 1 & a & 1 & a & 1 & a \\
 1 & a & a & 1 & 1 & a & a \\
 a & a & 1 & 1 & a & a & 1 \\
 1 & 1 & 1 & a & a & a & a \\
 a & 1 & a & a & 1 & a & 1 \\
 1 & a & a & a & a & 1 & 1 \\
 a & a & 1 & a & 1 & 1 & a
\end{array}
\right],\ \ a=-\frac{3}{4}+\mathbf{i}\frac{\sqrt7}{4}.\]
Note that the matrix $H$ is equivalent to a circulant one (cf.\ Theorem \ref{index2a}).
\end{ex}
Chan and Godsil investigated inverse-orthogonal (or type-$\mathrm{II}$) matrices with two different entries. From their analysis it follows that the aforementioned examples are the only possibilities.
\begin{thm}[Chan--Godsil, \cite{CG1}]
Let $X$ be a $(0,1)$-matrix. Then $H=X+a(J-X)$ is a complex Hadamard matrix if and only if
\begin{enumerate}[$($a$)$]
\item $H$ is a real Hadamard matrix; or
\item $H$ is one of the $2\times 2$ complex Hadamard matrices from Example \ref{ch3ex22}; or
\item $X$ or $J-X$ is a Hadamard design and hence $H$ comes from Theorem \ref{CGtwoe}.
\end{enumerate}
\end{thm}
Therefore complex Hadamard matrices with two different entries are scalar multiples of these examples. Note that for every prime $p\equiv 3$ (mod $4$) Part (c) leads to complex Hadamard matrices of order $p$ (cf.\ Theorem \ref{ch1rpc1}).
\subsection{Three different entries}
A full classification of complex Hadamard matrices with three different entries is currently not available, and as far as we can tell is far out of reach. In what follows we enlist several constructions yielding various examples, but it is unknown whether our list is exhaustive.

The principal examples are, of course, the $BH(n,3)$ matrices (see Definition \ref{ch1butsd}). From Part $1$ of Theorem \ref{ch1BN} it follows that they can exist only if $n\equiv 0$ (mod $3$). We mention here two constructions, coming from real Hadamard matrices.

Contrary to the two-entry case, there is no restriction on the entries here in general. The following was pointed out by R.\ Craigen:\footnote{Private communication.}
\begin{lem}
Let $H$ be a normalized real Hadamard matrix of order $n\geq 2$. Then, after multiplying the first row of $H$ by a nonreal unimodular complex number $a$ one obtains a complex Hadamard matrix with three different entries.
\end{lem}
Of course, we do not get new examples of complex Hadamard matrices this way as all these matrices are equivalent to the underlying real Hadamard matrix.
\begin{ex}
Let $a\neq\pm1$. Then the following is a complex Hadamard matrix with three different entries:
\[H=\left[\begin{array}{rrrr}
a & a & a & a\\
1 & 1 & -1 & -1\\
1 & -1 & 1 & -1\\
1 & -1 & -1 & 1\\
\end{array}\right].\]
\end{ex}
It turns out, however, that in square orders one can obtain inequivalent matrices as well. We recall the following powerful construction.
\begin{thm}[see e.g.\ \cite{HKO}]\label{MT2}
For every positive integer $n$ there is a self-adjoint complex Hadamard matrix of order $n^2$ with constant diagonal.
\end{thm}
\begin{proof}
Let $H$ be any complex Hadamard matrix of order $n$, and let us denote its rows by $h_1,h_2,\hdots, h_n$. Consider the following block matrix, where the $(i,j)$th entry of $K$ is the rank-$1$ block $h_j^\ast h_i$:
\beql{KX}
K=\left[
\begin{array}{cccc}
h_1^\ast h_1 & . & . & h_n^\ast h_1\\
.	& . & . & .\\
. & . & . & .\\
h_1^\ast h_n & . & . & h_n^\ast h_n\\
\end{array}
\right].
\eeq
We show that $K$ is Hadamard. Indeed: consider its $i$th row. We have
\beql{EQ1}
\sum\limits_{k=1}^{n}(h_k^\ast h_i)(h_i^\ast h_k)=\sum\limits_{k=1}^{n}h_k^\ast (h_ih_i^\ast) h_k=n\sum\limits_{k=1}^{n}h_k^\ast h_k=n^2I_n.
\eeq
Also, for rows $i\neq j$ recalling that the rows of $H$ are complex orthogonal
\beql{EQ2}
\sum\limits_{k=1}^{n}(h_k^\ast h_i)(h_j^\ast h_k)=\sum\limits_{k=1}^{n}h_k^\ast (h_ih_j^\ast) h_k=O_n,
\eeq
where $O_n$ stands for the all $0$ matrix of order $n$. Clearly, $K$ is self-adjoint by construction, and its diagonal is constant $1$.
\end{proof}
This construction naturally leads to parametric families of complex Hadamard matrices.
\begin{prop}[see \cite{SZF2}]\label{par}
Suppose that $H$ is a dephased complex Hadamard matrix of order $n$ with $m$ free parameters. Then there is a complex Hadamard matrix of order $n^2$ with $m+(n-1)^2$ free parameters, moreover there is a self-adjoint complex Hadamard matrix with constant diagonal featuring $m+(n-1)(n-2)/2$ free parameters.
\end{prop}
\begin{proof}
Indeed, as one is free to replace the blocks of $K$ $h_j^\ast h_i$ in \eqref{KX} with $x_{i,j}h_j^\ast h_i$ for $2\leq i,j\leq n$ in Theorem \ref{MT2}, as this operation does not affect the validity of equations \eqref{EQ1} and \eqref{EQ2}. Moreover, if the unimodular variables $x_{i,j}$ are chosen in a way that $x_{i,i}=1$ and $x_{i,j}=\overline{x}_{j,i}$ for every $2\leq i,j\leq n$, then the resulting matrix is a self-adjoint complex Hadamard matrix with constant diagonal, as required. Note that the first $m$ variables are featured in the first $n$ rows of $K$, and therefore they are independent from the rest.
\end{proof}
\begin{cor}
If a real Hadamard matrix of order $n$ exists, then so does a one-parameter family of complex Hadamard matrices with three different entries of order $n^2$.
\end{cor}
\begin{proof}
Consider a normalized real Hadamard matrix $H$ of order $n$, and use Theorem \ref{MT2} to obtain a normalized real Hadamard matrix of order $n^2$. Now observe that its upper left $n\times n$ submatrix is $J$, and hence replacing it with $aJ$, $a\neq\pm1$ will result in a one-parameter family of complex Hadamard matrices with three different entries, up to equivalence.
\end{proof}
Note that after dephasing the matrix the one-parameter remains in it, which however will contain four different entries: $\{1,-1,a,-a\}$.
\begin{ex}
Let $n=2$, and consider the real Hadamard matrix $F_2$ from which we obtain the $4\times 4$ real Hadamard matrix via Theorem \ref{MT2} and the corresponding parametric family it induces:
\[\left[\begin{array}{rr}
1 & 1\\
1 & -1\\
\end{array}\right]\leadsto
\left[\begin{array}{rr|rr}
1 & 1 & 1 & 1\\
1 & 1 & -1 & -1\\
\hline
1 & -1 & 1 & -1\\
1 & -1 & -1 & 1\\
\end{array}\right]\leadsto
\left[\begin{array}{rr|rr}
a & a & 1 & 1\\
a & a & -1 & -1\\
\hline
1 & -1 & 1 & -1\\
1 & -1 & -1 & 1\\
\end{array}\right].\]
\end{ex}
Recall that $C$ is a conference matrix of order $n$ if $C_{i,i}=0$, $i=1,2,\hdots, n$ and all off-diagonal entries are $\pm1$, such that $CC^T=(n-1)I$. We use symmetric conference matrices and show three additional methods resulting in complex Hadamard matrices with three different entries. The first is a folklore construction (cf.\ Theorem \ref{ch1pav2}).
\begin{lem}
Let $C$ be a symmetric conference matrix of order $n$. Then $H=C\pm\mathbf{i}I$ is a complex Hadamard matrix with three different entries.
\end{lem}
\begin{proof}
Clearly $H$ is unimodular, and
\[HH^\ast=(C\pm\mathbf{i}I)(C^T\mp\mathbf{i}I)=CC^T\pm\mathbf{i}(C^T-C)+I=nI.\qedhere\]
\end{proof}
We give the following variation of this result.
\begin{prop}\label{ch3s6gen}
Let $C$ be a normalized, symmetric conference matrix of order $n\geq 6$. Then consider the matrix $C+I$, and replace in its core the off diagonal entries $(1,-1)$ with $(a,\overline{a})$, respectively, where
\[a=-\frac{2}{n-2}\pm\mathbf{i}\frac{\sqrt{n(n-4)}}{n-2}.\]
The resulting matrix is a complex Hadamard matrix with three different entries $\{1,a,\overline{a}\}$.
\end{prop}
\begin{proof}
Consider the complex Hadamard matrix $H$, arising from the construction. Both of the numbers $a$ and $\overline{a}$ appear $(n-2)/2$ times in every noninitial rows of $H$, and hence the first row is orthogonal to all subsequent one. It remains to be seen that the noninitial rows are pairwise orthogonal. Consider the underlying conference matrix $C$. Clearly any two noninitial rows of it are permutation equivalent to either of the following $2\times n$ matrices:
\[\begin{blockarray}{rrrrrr}
\begin{block}{[rr|rrrr]}
0  & 1 & 1 & 1  & -1 & -1\\
1 & 0 & 1 & -1 & 1  & -1\\
\end{block}
& & n_1 & n_2 & n_3 & n_4\\
\end{blockarray}
,\qquad\qquad
\begin{blockarray}{rrrrrr}
\begin{block}{[rr|rrrr]}
0  & -1 & 1 & 1  & -1 & -1\\
-1 & 0 & 1 & -1 & 1  & -1\\
\end{block}
& & n_5 & n_6 & n_7 & n_8\\
\end{blockarray},\vspace{-0.7cm}\]
where the column labels $n_1,n_5, n_2,n_6, n_3,n_7$ and $n_4,n_8$ are nonnegative integral numbers describing how many columns of type $(1,1)^T, (1,-1)^T, (-1,1)^T$ and $(-1,-1)^T$ are present in the matrices, respectively. From the orthogonality of $C$ it is elementary to deduce that
\[n_1=n_2=n_3=n_4=\frac{n-2}{4},\qquad n_5=\frac{n+2}{4},\quad n_6=n_7=\frac{n-2}{4},\quad n_8=\frac{n-6}{4}\]
and as a result the orthogonality conditions between the noninitial rows of $H$ read
\begin{gather*}
1+a+\overline{a}+\frac{n-6}{4}+\frac{n-2}{4}\left(a^2+\overline{a}^2\right)+\frac{n-2}{4}=0,\\
1+a+\overline{a}+\frac{n-2}{4}+\frac{n-2}{4}\left(a^2+\overline{a}^2\right)+\frac{n-6}{4}=0.
\end{gather*}
These two conditions are the same, which hold identically by the choice of $a$.
\end{proof}
\begin{rem}
By setting $n=6$ in Proposition \ref{ch3s6gen} we obtain the $BH(6,3)$ matrix $S_6^{(0)}$ (see Example \ref{ch2s6}).
\hfill$\square$\end{rem}
Another construction from \cite{SZF2} is the following.
\begin{thm}[see \cite{SZF2}, \cite{SZF3}]\label{ch3f5gen}
Let $C$ be a normalized, symmetric conference matrix of order $n\geq 6$. Discard the first row and column of $C$ and replace in the remaining matrix the numbers $(0,1,-1)$ with $(1,a,\overline{a})$, respectively, where
\[a=\frac{-1\pm_A\sqrt{n-1}}{n-2}\pm_B\mathbf{i}\frac{\sqrt{n^2-5n+4\pm_A2\sqrt{n-1}}}{n-2},\]
where $\pm_A$ is the sign of $\sqrt{n-1}$ and $\pm_B$ is the sign of the nested radical, respectively. The resulting matrix is a complex Hadamard matrix of order $n-1$ with three different entries $\{1,a,\overline{a}\}$.
\end{thm}
\begin{rem}
By setting $n=6$ in Theorem \ref{ch3f5gen} we obtain the Fourier matrix $F_5$. For $n=10$ we get a $BH(9,3)$ matrix along with an additional example. In general the construction describes at least two inequivalent complex Hadamard matrices.
\hfill$\square$\end{rem}
The constructions given by Theorems \ref{CGtwoe} and \ref{ch3f5gen} contain Theorems \ref{index2a} and \ref{index2b} as special cases, and hence they describe various complex Hadamard matrices of prime orders. It would be particularly interesting to find analogous constructions in the simple index $3$ and $4$ cases as well.
\begin{pro}
Generalize the simple index $3$ and simple index $4$ type complex Hadamard matrices for non-prime orders as well. 
\end{pro}
Finally, we remark here that Chan has given an algebraic description of those complex Hadamard matrices $H=I+aX+b(J-I-X)$, where $X$ is the adjacency matrix of a strongly regular graph very recently \cite{AC1}, and hence obtained complex Hadamard matrices with at most three different entries. It would be very interesting to see if the following more general problem can be treated by purely algebraic methods as well.
\begin{pro}
Give a complete characterization of complex Hadamard matrices with three different entries.
\end{pro}
\section{Application: Equiangular tight frames}\label{ch3ETFs}
In this section we give another application of complex Hadamard matrices and relate them to complex equiangular tight frames. The concepts are similar to what was discussed in Section \ref{ch2MUBEQ} but in contrast here we are interested in equiangular lines having a prescribed angle. The main objects we are concerned with is described in the following
\begin{defi}
Let $\mathcal{H}$ be a complex Hilbert space, and let $F=\{f_i\}_{i\in I}\subset \mathcal{H}$ be a subset. We call $F$ a \emph{frame} for $\mathcal{H}$ provided that there are two constants $C,D>0$ such that the inequality
\[C\left\|x\right\|^2\leq \sum_{i\in I}\left|\left\langle x,f_i\right\rangle\right|^2\leq D\left\|x\right\|^2\]
holds for every $x\in\mathcal{H}$. When $C=D=1$ then the frame is called normalized, tight (also called a Parseval frame).
\end{defi}
Real and complex equiangular tight frames have been heavily investigated recently (see e.g.\ \cite{BE1}, \cite{CRT}, \cite{DHS}, \cite{FMT} and \cite{JR1}) as they have applications in coding \cite{HP1} and quantum information theory \cite{SG1}. Here we shall discuss only Parseval frames for the $k$ dimensional complex Hilbert space, equipped with the usual inner product. We use the term $(n,k)$ frame to refer to a Parseval frame consisting of $n$ vectors in $\mathbb{C}^k$. Every such Parseval frame induces an isometric embedding of $\mathbb{C}^k$ into $\mathbb{C}^n$ with the map
\[ V\colon\mathbb{C}^k\to \mathbb{C}^n, (Vx)_j=\left\langle x,f_j\right\rangle, \text{for all } 1\leq j\leq n\]
which is called the analysis operator of the frame. As $V$ is linear, we can identify it with an $n\times k$ matrix and the frame vectors $\{f_1,\hdots, f_n\}$ are the respective columns of $V^\ast$. A standard argument shows \cite{HP1} that an $(n,k)$ frame is determined up to a natural unitary equivalence by its Gram matrix $VV^\ast$, which is a self-adjoint projection of rank $k$, moreover, if the frame is uniform and equiangular (i.e.\ $\|f_i\|^2$ and $\left|\left\langle f_i,f_j\right\rangle\right|$ are constants for all $1\leq i\leq n$ and for all $i\neq j$, $1\leq i,j\leq n$, respectively), it follows that the Gram matrix of the frame reads
\[VV^\ast=\frac{k}{n}I_n+\sqrt{\frac{k(n-k)}{n^2(n-1)}}Q,\]
where $Q$ is a self-adjoint matrix with vanishing diagonal and unimodular off-diagonal entries. We refer to such frames as equiangular tight frames. The matrix $Q$ is called the Seidel matrix \cite{LS1}, signature matrix (or a generalized conference matrix, \cite{CG1}) associated with the $(n,k)$ frame. Recall from Lemma \ref{ch2noeqs} that $n\leq k^2$. The following provides an elegant characterization of complex equiangular frames (see \cite{LS1} and \cite{LS2} for the real case).
\begin{thm}[Holmes--Paulsen, \cite{HP1}]
Let $Q$ be a self-adjoint $n\times n$ matrix with $Q_{ii}=0$ and $\left|Q_{ij}\right|=1$ for all $i\neq j$. Then the following are equivalent:
\begin{enumerate}[$($a$)$]
\item $Q$ is the signature matrix of an equiangular $(n,k)$ frame for some $k$;
\item $Q^2=(n-1)I+\mu Q$ for some necessarily real $\mu$; and
\item $Q$ has exactly two eigenvalues.
\end{enumerate}
\end{thm}
Note that the value $k$ above depends on $n$ and $\mu$ only. In particular (see \cite{BE1}), we have
\beql{kval}
k=\frac{n}{2}-\frac{\mu n}{2\sqrt{4(n-1)+\mu^2}}.
\eeq
It follows that if $Q$ is a signature matrix of an $(n,k)$ frame, then $-Q$ is a signature matrix of an $(n,n-k)$ frame. Let us recall here a quick
\begin{ex}[Bodmann et al.\ \cite{BPT}]
The matrix 
\beql{Q9}
Q_9=\left[
\begin{array}{ccccccccc}
 0 & 1 & 1 & 1 & 1 & 1 & 1 & 1 & 1 \\
 1 & 0 & 1 & \omega  & \omega  & \omega  & \omega ^2 & \omega ^2 & \omega ^2 \\
 1 & 1 & 0 & \omega ^2 & \omega ^2 & \omega ^2 & \omega  & \omega  & \omega  \\
 1 & \omega ^2 & \omega  & 0 & \omega  & \omega ^2 & 1 & \omega  & \omega ^2 \\
 1 & \omega ^2 & \omega  & \omega ^2 & 0 & \omega  & \omega  & \omega ^2 & 1 \\
 1 & \omega ^2 & \omega  & \omega  & \omega ^2 & 0 & \omega ^2 & 1 & \omega  \\
 1 & \omega  & \omega ^2 & 1 & \omega ^2 & \omega  & 0 & \omega ^2 & \omega  \\
 1 & \omega  & \omega ^2 & \omega ^2 & \omega  & 1 & \omega  & 0 & \omega ^2 \\
 1 & \omega  & \omega ^2 & \omega  & 1 & \omega ^2 & \omega ^2 & \omega  & 0
\end{array}
\right],
\eeq
where $\omega$ is the principal cubic root of unity, is a $9\times 9$ nontrivial cube root signature matrix of an equiangular $(9, 6)$-frame.
\end{ex}
The matrix $Q_9$ was the original object which raised our interest in frame theory. One can observe immediately, that the matrix $Q_9$ is ``almost'' a complex Hadamard matrix: indeed, $H=Q_9+I$ is complex Hadamard. The following result explains this phenomenon.
\begin{thm}[see \cite{SZF7}, cf.\ \cite{AC1}]\label{TH}
Let $Q$ be a self-adjoint $n\times n$ matrix with $Q_{ii}=0$ and $\left|Q_{ij}\right|=1$ for all $i\neq j$. Then the following are equivalent:
\begin{enumerate}[$($a$)$]
\item $Q^2=(n-1)I+\mu Q$ for some necessarily real $-2\leq\mu\leq 2$; and
\item $H:=Q+\lambda I$ is a complex Hadamard matrix for $\lambda=-\mu/2\pm\mathbf{i}\sqrt{1-\left|\mu\right|^2/4}$.
\end{enumerate}
\end{thm}
\begin{proof}
Suppose that we have a signature matrix $Q$ satisfying (a). Then, as the number $\lambda$ defined in (b) is unimodular we find that $H=Q+\lambda I$ is a unimodular matrix, moreover we have
\[HH^\ast=(Q+\lambda I)(Q^\ast+\overline{\lambda}I)=nI+\mu Q+2\Re[\lambda]Q=nI,\]
as required. Conversely, let us suppose that we have a complex Hadamard matrix $H$ with constant diagonal $\lambda$, such that the matrix $Q:=H-\lambda I$ is self-adjoint. Then, we find that
\[Q^2=QQ^\ast=(H-\lambda I)(H^\ast-\overline{\lambda}I)=(n+1)I-\lambda(Q+\lambda I)^\ast-\overline{\lambda}(Q+\lambda I)=(n-1)I-2\Re[\lambda]Q,\]
and, as $\lambda$ is unimodular, the statement (a) follows.
\end{proof}
In \cite{BE1} $n$th root signature matrices were investigated extensively where the following result was obtained.
\begin{thm}[Bodmann--Elwood, \cite{BE1}]\label{cbde}
For every $n\geq 2$ there is a self-adjoint complex Hadamard matrix of order $n^2$ with constant diagonal composed of $n$th roots of unity. Consequently there is a nontrivial $n$th root signature matrix corresponding to an equiangular $\left(n^2,n(n+1)/2\right)$ frame.
\end{thm}
Bodmann and Elwood used a direct, elementary argument concluding that the Fourier matrices of order $n$ can be lifted to order $n^2$ yielding the result. This, however, still left unanswered an earlier question from \cite{BPT}, namely whether or not an equiangular $(36,21)$ frame, whose signature matrix is composed from cubic root of unity exists. Our observation is that Theorem \ref{cbde} follows from Theorem \ref{MT2} rather easily and in particular various Butson-Hadamard matrices (see Definition \ref{ch1butsd}) can be used instead of the Fourier matrices resulting in signature matrices, composed from roots of unity.
\begin{cor}\label{CCC}
For every prime $p$ there is a nontrivial $p$th root signature matrix of order $4^ap^{2b}$ for all $0\leq a\leq b$ corresponding to an equiangular $\left(4^ap^{2b},2^ap^b(2^ap^b+1)/2\right)$ frame.
\end{cor}
\begin{proof}
Combine Part $2$ of Theorem \ref{ch1BN} with Theorem \ref{MT2}.
\end{proof}
Setting $p=3$ and $a=b=1$ in Corollary \ref{CCC} above, we immediately get an answer to one of the questions raised in \cite{BPT}.
\begin{cor}\label{C}
There is a nontrivial cube root signature matrix of order $36$ leading to an equiangular $(36,21)$ frame.
\end{cor}
In the corollary above we used the matrix $S_6^{(0)}$ (see Example \ref{ch2s6}).

The case $|\mu|=2$ in Part (b) of Theorem \ref{TH} corresponds to self-adjoint complex Hadamard matrices with constant diagonal which, however, can exist in square orders only (see Lemma \ref{ch1gowslemma}).

Next we give examples of frames in which the corresponding complex Hadamard matrices are not self-adjoint. In particular, $|\mu|\neq 2$. We recall that a $2$-$(4m-1,2m-1,m-1)$ Hadamard-design $A$ is skew if $A+A^T+I=J$. We have the following
\begin{prop}\label{L2}
Suppose that we have a skew Hadamard design $U$ of order $n$. Then there exists a complex Hadamard matrix $H$ of order $n$ with diagonal entries $\lambda$, such that the matrix $Q:=H-\lambda I$ is self-adjoint.
\end{prop}
\begin{proof}
Use Theorem \ref{CGtwoe} to exhibit a complex Hadamard matrix $K=U+aU^T+aI$ where $a$ is the unimodular complex number defined in \eqref{CGuni}. Now multiply this matrix with the unimodular number $\overline{\sqrt{a}}$ to obtain the matrix $H=\overline{\sqrt{a}}K=\overline{\sqrt{a}}U+\sqrt{a}U^T+\sqrt{a}I$.
\end{proof}
It is well-known that the Paley-type Hadamard designs of prime orders, described by Theorem \ref{ch1rpc1} are skew, moreover it is conjectured that skew Hadamard designs exist for every order $n=4m-1$ \cite{KS2}. Therefore we have infinitely many equiangular tight frames coming from Proposition \ref{L2}. In particular, we have the following
\begin{cor}[Renes, \cite{JR1}]\label{Cxd}
Suppose that we have a skew Hadamard design of order $n\geq 3$. Then there are equiangular $\left(n,(n-1)/2 \right)$ and $\left(n,(n+1)/2\right)$ frames.
\end{cor}
Hence, if the conjecture regarding the existence of skew Hadamard matrices is true, then we have equiangular tight frames with parameters $(2k-(-1)^k,k)$ for every $k$.

An interesting question is to ask if it is possible to construct $(k^2,k)$ equiangular frames from complex Hadamard matrices. These type of frames are equivalent to SIC-POVMs \cite{MKx1}, \cite{SG1}, and have many interesting applications in quantum information theory. Unfortunately, equation \eqref{kval} leads to $\mu=\sqrt{k+1}(k-2)$ which, combined with the bound $|\mu|\leq 2$, implies that $k=2$ or $3$. The case $k=2$ easily leads to the signature matrix described in Example \ref{ch2eq4dim}, while the case $k=3$ implies that $\mu=2$ which corresponds to, for example, the one-parameter signature matrix $-Q_9(a)$, where $Q_9(a)+I$ is the family of complex Hadamard matrices coming from Proposition \ref{par}, applied to the starting-point complex Hadamard matrix $Q_9+I$ (see formula \eqref{Q9}).

So far the ideas mentioned in this section considered the Gram matrix of the configurations. However, in all known construction methods of $(k^2,k)$ equiangular tight frames the frame vectors are obtained directly via the following way. We start with a so-called fiducial vector $\phi\in\mathbb{C}^k$ and a group $G$ of $k\times k$ matrices. The group is fixed, and the idea is to choose $\phi$ so that $\{M\phi : M\in G\}$ spans a set of equiangular lines. We illustrate this approach throughout the following remarkable
\begin{ex}[Hoggar, \cite{SH1}]
Let $X$, $Y$ and $Z$ be the Pauli matrices, namely
\[X=\left[\begin{array}{rr}
0 & 1\\
1 & 0\\
\end{array}\right],\ \ \ 
Z=\left[\begin{array}{rr}
1 & 0\\
0 & -1\\
\end{array}\right],\ \ \
Y=\mathbf{i}XZ=\left[\begin{array}{rr}
0 & -\mathbf{i}\\
\mathbf{i} & 0\\
\end{array}\right],\]
and let $G=\{I,X,Y,Z\}^{\otimes3}$ and
\[\phi=\frac{\sqrt{3}}{12}\left[0,0,\tau,\overline{\tau},\tau,-\tau,0,\sqrt{2}\right]^T,\]
where $\tau=(1+\mathbf{i})/\sqrt{2}$ is the principal eighth root of unity. Then $V^\ast=\{A\phi : A\in G\}$ is an equiangular $(64,8)$ frame. Moreover, the Gram matrix of the configuration reads
\[VV^\ast=\frac{1}{8}I+\frac{1}{24}Q,\]
where the signature matrix $Q$ is composed of fourth roots of unity.
\end{ex}
One cannot but wonder what combinatorial objects correspond to the signature matrices of $(k^2,k)$ equiangular tight frames for $k\geq 4$.
\newpage %Ez azert kell, hogy jobb oldalon kezdodjenek a chapterek.
\thispagestyle{empty} % ez is
%\backmatter

\newpage %Ez azert kell, hogy jobb oldalon kezdodjenek a chapterek.
\thispagestyle{empty} % ez is
\pagestyle{fancy}                       % Sets fancy header and footer 
\fancyfoot{}                            % Delete current footer settings 
\renewcommand{\chaptermark}[1]{         % Lower Case Chapter marker style 
  \markboth{\MakeUppercase{\appendixname\ \thechapter\ #1}}{}} % 
\renewcommand{\sectionmark}[1]{        % Lower case Section marker style 
  \markright{\MakeUppercase{\thesection\ #1}}}         % 
\fancyhead[OR,EL]{\thepage}    % Page number (boldface) in left on even 
                                        % pages and right on odd pages 
\fancyhead[ER]{\leftmark}      % Chapter in the right on even pages 
\fancyhead[OL]{\rightmark}     % Section in the left on odd pages 
\renewcommand{\headrulewidth}{0.3pt}    % Width of head rule 
\begin{appendices}
%\noappendicestocpagenum
%\addappheadtotoc
\chapter{Petrescu-type Hadamard matrices}\label{APPA}
\thispagestyle{empty}
Here we list a handful of Petrescu-type $BH(n,q)$ matrices. Recall that $\omega=\mathbf{e}^{2\pi\mathbf{i}/3}$.
\section{Some small dimensional examples}
\begin{ex}
$F_2\otimes F_2$ is the only real, Petrescu-type Hadamard matrix:
\[P_4^{(1)}(a)=\left[
\begin{array}{r|r|rr}
 a & -a & -1 & 1 \\
 \hline
 -a & a & -1 & 1 \\
 \hline
 -1 & -1 & 1 & 1 \\
 1 & 1 & 1 & 1
\end{array}
\right].\]
\end{ex}
\begin{ex}
The following is four-parameter family of complex Hadamard matrices, stemming from a $BH(13,30)$ matrix (cf.\ \cite{BTZ}). We have $P_{13A}^{(4)}(a,b,c,d)=$
\[\left[
\begin{array}{rrrr|rrrr|rrrrr}
 \omega  & a & b & c & -1 & -a & -b & -c & w & w^2 & w^3 & w^4 & 1 \\
 \frac{d^2 \omega }{a} & \omega  & -\frac{b d e}{a} & -\frac{c d \omega }{a e} & -\frac{d^2 \omega }{a} & -1 & \frac{b d e}{a} & \frac{c d \omega }{a e} & w^2 & w^4 & w & w^3 & 1 \\
 \frac{\omega }{b} & -\frac{a \omega }{b d e} & -1 & \frac{c \omega }{b d e} & -\frac{\omega }{b} & \frac{a \omega }{b d e} & \omega  & -\frac{c \omega }{b d e} & w^3 & w & w^4 & w^2 & 1 \\
 \frac{\omega }{c} & -\frac{a e}{c d} & \frac{b e}{c d} & -1 & -\frac{\omega }{c} & \frac{a e}{c d} & -\frac{b e}{c d} & \omega  & w^4 & w^3 & w^2 & w & 1 \\
 \hline
 -1 & -a & -b & -c & \omega  & a & b & c & w & w^2 & w^3 & w^4 & 1 \\
 -\frac{d^2 \omega }{a} & -1 & \frac{b d e}{a} & \frac{c d \omega }{a e} & \frac{d^2 \omega }{a} & \omega  & -\frac{b d e}{a} & -\frac{c d \omega }{a e} & w^2 & w^4 & w & w^3 & 1 \\
 -\frac{\omega }{b} & \frac{a \omega }{b d e} & \omega  & -\frac{c \omega }{b d e} & \frac{\omega }{b} & -\frac{a \omega }{b d e} & -1 & \frac{c \omega }{b d e} & w^3 & w & w^4 & w^2 & 1 \\
 -\frac{\omega }{c} & \frac{a e}{c d} & -\frac{b e}{c d} & \omega  & \frac{\omega }{c} & -\frac{a e}{c d} & \frac{b e}{c d} & -1 & w^4 & w^3 & w^2 & w & 1 \\
 \hline
 w^4 & w^3 & w^2 & w & w^4 & w^3 & w^2 & w & 1 & \omega ^2 & \omega ^2 & \omega ^2 & \omega ^2 \\
 w^3 & w & w^4 & w^2 & w^3 & w & w^4 & w^2 & \omega ^2 & 1 & \omega ^2 & \omega ^2 & \omega ^2 \\
 w^2 & w^4 & w & w^3 & w^2 & w^4 & w & w^3 & \omega ^2 & \omega ^2 & 1 & \omega ^2 & \omega ^2 \\
 w & w^2 & w^3 & w^4 & w & w^2 & w^3 & w^4 & \omega ^2 & \omega ^2 & \omega ^2 & 1 & \omega ^2 \\
 1 & 1 & 1 & 1 & 1 & 1 & 1 & 1 & \omega ^2 & \omega ^2 & \omega ^2 & \omega ^2 & 1
\end{array}
\right],\]
where $w=\mathbf{e}^{2\pi\mathbf{i}/5}$ and $\Re[e\omega]=-\Re[d]/2$. The parameters $a,b$ and $c$ were introduced via Theorem \ref{ch1newp} resulting in the largest known family of order $13$.
\end{ex}
\begin{ex}
The following is a $BH(16,4)$ matrix of Petrescu-type:
\[\left[
\begin{array}{rrrrr|rrrrr|rrrrrr}
 1 & 1 & \mathbf{i} & 1 & 1 & -1 & -\mathbf{i} & -1 & -1 & -1 & 1 & -1 & -1 & -\mathbf{i} & \mathbf{i} & 1 \\
 \mathbf{i} & -1 & -1 & -\mathbf{i} & -\mathbf{i} & -\mathbf{i} & -\mathbf{i} & -\mathbf{i} & \mathbf{i} & \mathbf{i} & \mathbf{i} & \mathbf{i} & -\mathbf{i} & -1 & -\mathbf{i} & 1 \\
 1 & 1 & -1 & 1 & -1 & 1 & -1 & 1 & -1 & 1 & -1 & 1 & \mathbf{i} & -\mathbf{i} & -1 & 1 \\
 \mathbf{i} & 1 & -\mathbf{i} & -1 & \mathbf{i} & -\mathbf{i} & -1 & \mathbf{i} & \mathbf{i} & -1 & -1 & -\mathbf{i} & -1 & \mathbf{i} & 1 & 1 \\
 1 & -1 & -\mathbf{i} & -\mathbf{i} & \mathbf{i} & -1 & 1 & \mathbf{i} & -1 & 1 & -\mathbf{i} & -1 & 1 & \mathbf{i} & -1 & 1 \\
 \hline
 -1 & -\mathbf{i} & -1 & -1 & -1 & 1 & 1 & \mathbf{i} & 1 & 1 & 1 & -1 & -1 & -\mathbf{i} & \mathbf{i} & 1 \\
 -\mathbf{i} & -\mathbf{i} & -\mathbf{i} & \mathbf{i} & \mathbf{i} & \mathbf{i} & -1 & -1 & -\mathbf{i} & -\mathbf{i} & \mathbf{i} & \mathbf{i} & -\mathbf{i} & -1 & -\mathbf{i} & 1 \\
 1 & -1 & 1 & -1 & 1 & 1 & 1 & -1 & 1 & -1 & -1 & 1 & \mathbf{i} & -\mathbf{i} & -1 & 1 \\
 -\mathbf{i} & -1 & \mathbf{i} & \mathbf{i} & -1 & \mathbf{i} & 1 & -\mathbf{i} & -1 & \mathbf{i} & -1 & -\mathbf{i} & -1 & \mathbf{i} & 1 & 1 \\
 -1 & 1 & \mathbf{i} & -1 & 1 & 1 & -1 & -\mathbf{i} & -\mathbf{i} & \mathbf{i} & -\mathbf{i} & -1 & 1 & \mathbf{i} & -1 & 1 \\
 \hline
 1 & -\mathbf{i} & -1 & -1 & \mathbf{i} & 1 & -\mathbf{i} & -1 & -1 & \mathbf{i} & 1 & 1 & 1 & 1 & 1 & -1 \\
   -1 & -\mathbf{i} & 1 & \mathbf{i} & -1 & -1 & -\mathbf{i} & 1 & \mathbf{i} & -1 & 1 & 1 & 1 & 1 & -1 & 1 \\
 -1 & \mathbf{i} & -\mathbf{i} & -1 & 1 & -1 & \mathbf{i} & -\mathbf{i} & -1 & 1 & 1 & 1 & \mathbf{i} & -\mathbf{i} & 1 & 1 \\
 \mathbf{i} & -1 & \mathbf{i} & -\mathbf{i} & -\mathbf{i} & \mathbf{i} & -1 & \mathbf{i} & -\mathbf{i} & -\mathbf{i} & \mathbf{i} & -\mathbf{i} & 1 & 1 & 1 & 1 \\
 -\mathbf{i} & \mathbf{i} & -1 & 1 & -1 & -\mathbf{i} & \mathbf{i} & -1 & 1 & -1 & -\mathbf{i} & \mathbf{i} & 1 & 1 & 1 & 1 \\
 1 & 1 & 1 & 1 & 1 & 1 & 1 & 1 & 1 & 1 & 1 & 1 & -\mathbf{i} & \mathbf{i} & 1 & 1
\end{array}
\right].\]
\end{ex}
\begin{ex}
Our last example is a $BH(16,6)$ matrix, which can be obtained by considering $H=\mathrm{EXP}\left(\frac{2\pi\mathbf{i}}{6}L\right)$, where 
\[L=\left[
\begin{array}{ccccc|ccccc|cccccc}
 0 & 0 & 0 & 0 & 0 & 3 & 0 & 3 & 3 & 3 & 0 & 2 & 2 & 4 & 4 & 0 \\
 0 & 0 & 0 & 2 & 4 & 0 & 3 & 3 & 5 & 1 & 0 & 4 & 4 & 2 & 2 & 0 \\
 0 & 0 & 3 & 4 & 2 & 3 & 3 & 3 & 1 & 5 & 3 & 0 & 3 & 0 & 3 & 0 \\
 0 & 2 & 4 & 2 & 3 & 3 & 5 & 1 & 5 & 3 & 3 & 2 & 5 & 4 & 1 & 0 \\
 0 & 4 & 2 & 3 & 4 & 3 & 1 & 5 & 3 & 1 & 3 & 4 & 1 & 2 & 5 & 0 \\
 \hline
 3 & 0 & 3 & 3 & 3 & 0 & 0 & 0 & 0 & 0 & 0 & 2 & 2 & 4 & 4 & 0 \\
 0 & 3 & 3 & 5 & 1 & 0 & 0 & 0 & 2 & 4 & 0 & 4 & 4 & 2 & 2 & 0 \\
 3 & 3 & 3 & 1 & 5 & 0 & 0 & 3 & 4 & 2 & 3 & 0 & 3 & 0 & 3 & 0 \\
 3 & 5 & 1 & 5 & 3 & 0 & 2 & 4 & 2 & 3 & 3 & 2 & 5 & 4 & 1 & 0 \\
 3 & 1 & 5 & 3 & 1 & 0 & 4 & 2 & 3 & 4 & 3 & 4 & 1 & 2 & 5 & 0 \\
 \hline
 0 & 0 & 3 & 3 & 3 & 0 & 0 & 3 & 3 & 3 & 0 & 0 & 0 & 0 & 0 & 3 \\
 4 & 2 & 0 & 4 & 2 & 4 & 2 & 0 & 4 & 2 & 0 & 0 & 0 & 0 & 3 & 0 \\
 4 & 2 & 3 & 1 & 5 & 4 & 2 & 3 & 1 & 5 & 0 & 0 & 0 & 3 & 0 & 0 \\
 2 & 4 & 0 & 2 & 4 & 2 & 4 & 0 & 2 & 4 & 0 & 0 & 3 & 0 & 0 & 0 \\
 2 & 4 & 3 & 5 & 1 & 2 & 4 & 3 & 5 & 1 & 0 & 3 & 0 & 0 & 0 & 0 \\
 0 & 0 & 0 & 0 & 0 & 0 & 0 & 0 & 0 & 0 & 3 & 0 & 0 & 0 & 0 & 0
\end{array}
\right].\]
\end{ex}
\newpage %Ez azert kell, hogy jobb oldalon kezdodjenek a chapterek.
\thispagestyle{empty} % ez is
%%%%%%%%%%%%%%%%%%%%%%%%%%%%%%%%%%%%%% APPENDIX %%%%%%%%%%%%%%%%%%%%%%%%%%%%%%%%%%%%%%%%%%%%
\chapter{\texorpdfstring{All cyclic $17$-roots of simple index $4$}{All cyclic 17-roots of simple index 4}}\label{APPB}
\thispagestyle{empty}
Throughout this section we give a concise guide to cyclic $17$-roots of simple index $4$. In particular, we solve case $p=17$ and $k=4$ of \eqref{72haa} with Gr\"obner basis techniques and discover new examples of complex Hadamard matrices of order $17$.
\section{The polynomial describing the solutions}
After calculating a pure lexicographic Gr\"obner basis we found that the solutions of \eqref{72haa} for $p=17$, $k=4$, i.e.\ the cyclic $17$-roots of simple index $4$ on the $x$-level can be obtained by considering the roots of the following polynomial of degree $70$:
\begin{equation*}%\label{f17c}
\begin{split}
\mathcal{C}_{17}(x)&=\left(x^2+15 x+1\right) \left(4 x^4+x^3+7
   x^2+x+4\right)\\
&\quad\times \left(4 x^8+2 x^7-41 x^6-3 x^5+59
   x^4-3 x^3-41 x^2+2 x+4\right)\\
&\quad\times \left(4 x^8+2
   x^7-7 x^6-3 x^5-9 x^4-3 x^3-7 x^2+2 x+4\right)\\
&\quad\times
   \left(4 x^{16}-64 x^{15}+344 x^{14}-744
   x^{13}+1364 x^{12}+4288 x^{11}-6048 x^{10}+1772x^9\right.\\
&\quad\left.+21577 x^8+1772 x^7-6048 x^6+4288 x^5+1364
   x^4-744 x^3+344 x^2-64 x+4\right)\\
&\quad\times \left(8788
   x^{32}-143312 x^{31}-6737120 x^{30}+66098104
   x^{29}+317898992 x^{28}\right.\\
&\quad-1844667456
   x^{27}-5388129896 x^{26}+4288299156
   x^{25}+16518934433 x^{24}\\
&\quad+9515028816
   x^{23}+6237511690 x^{22}+5188763380
   x^{21}-7207401 x^{20}\\
&\quad-9833239192
   x^{19}-6468787876 x^{18}-4434709436
   x^{17}-10393530091 x^{16}\\
&\quad-4434709436
   x^{15}-6468787876 x^{14}-9833239192
   x^{13}-7207401 x^{12}\\
&\quad+5188763380
   x^{11}+6237511690 x^{10}+9515028816
   x^9+16518934433 x^8\\
&\quad+4288299156 x^7-5388129896
   x^6-1844667456 x^5+317898992 x^4\\
&\quad\left.+66098104
   x^3-6737120 x^2-143312 x+8788\right).
\end{split}
\end{equation*}
Our goal is to find, among the roots of $\mathcal{C}_{17}(x)$, all quadruples $(c_0,c_1,c_2,c_3)\in\left(\mathbb{C}^\ast\right)^4$ which solve $\eqref{72haa}$. It turns out that in almost all cases if one of the values $c_0, c_1, c_2$ or $c_3$ is a root of a given factor of $\mathcal{C}_{17}(x)$ then all additional one is a root of the same factor. As a result the problem naturally can be divided into subcases by considering these factors separately. In order to do this, let us denote by $V_1(x), V_2(x),\hdots, V_6(x)$ the six factors of $\mathcal{C}_{17}(x)$ in the same order as they appear above, respectively.

In what follows we account all $70$ solutions of the problem (see Theorem \ref{ch3hindex4}) and show how to describe them by radicals. The general strategy is to exploit the symmetry of \eqref{72haa} and introduce the new variables
\beql{appbst}
h_i:=c_i+\frac{1}{c_i},\ \ i=0,1,2,3.
\eeq
This standard trick effectively halves the degrees of the polynomials $V_i(x), i=1,2,\hdots,6$ whose solutions are hopefully more easily accessible by radicals. The following is the way to recover the values of $c_i$ once the values of $h_i$, $i=0,\hdots, 3$ are determined.
\begin{lem}[Lifting formula]\label{appbLF}
Suppose that $x+\frac{1}{x}=a+\mathbf{i}b$, where $a, b\in\mathbb{R}$. Then
\[x_{1,2}=\begin{cases}
\displaystyle{\frac{a}{2}\pm\mathbf{i}\frac{\sqrt{4-a^2}}{2}}, & \hspace{-1.6cm}\text{if $b=0$, $|a|\leq 2$;}\\[0.32cm]
\displaystyle{\frac{a}{2}\pm\frac{\sqrt{a^2-4}}{2}}, & \hspace{-1.6cm}\text{if $b=0$, $|a|>2$;}\\[0.32cm]
\displaystyle{\mathbf{i}\left(\frac{b}{2}\pm\frac{\sqrt{b^2+4}}{2}\right)}, & \hspace{-1.6cm}\text{if $b\neq0$,\hspace{7pt} $a=0$;}\\[0.32cm]
\begin{aligned}
\frac{a}{2}&\pm\frac{1}{4}\sqrt{2a^2-2b^2-8+2\sqrt{4a^2b^2+\left(a^2-b^2-4\right)^2}}\\
&+\mathbf{i}\left(\frac{b}{2}\pm\frac{1}{4}\mathrm{sgn}(ab)\sqrt{-2a^2+2b^2+8+2\sqrt{4a^2b^2+\left(a^2-b^2-4\right)^2}}\right)\end{aligned}, & \hspace{-0.4cm}\text{otherwise},
\end{cases}
\]
such that the $\pm$ signs in the last part agree and $\mathrm{sgn}(ab)$ is the sign of $ab$. Note that $x_1x_2=1$ and hence $x_2=1/x_1$.
\end{lem}
We remark that if $x$ is complex unimodular then the first case applies. With the notations above let us introduce the lifting operator, mapping complex numbers $z=a+\mathbf{i}b$, $a,b\in\mathbb{R}$ to $x_1$:
\[\mathcal{L}(a+\mathbf{i}b)=x_1.\]
By using this shorthand notation we can avoid the excessive usage of the square roots.
\section{Analysing the solutions}
Now we analyse each of the factors $V_i$, $i=1,\hdots, 6$ and count and classify the solutions they induce.
\subsection{\texorpdfstring{Solutions coming from $V_1(x)$}{Solutions coming from V1(x)}}
First we suppose that one of $c_0, c_1, c_2$ or $c_3$ is a root of $V_1(x)$. It turns out that in this case $c_0=c_1=c_2=c_3$ and we obtain the ``$\varepsilon$-solutions'' (in the sense of \cite{GB1}):
\beql{appbep}(c_0,c_1,c_2,c_3)=\left(\frac{-15+\sqrt{221}}{2},\frac{-15+\sqrt{221}}{2},\frac{-15+\sqrt{221}}{2},\frac{-15+\sqrt{221}}{2}\right),
\eeq
and its reciprocal $(1/c_0,1/c_1,1/c_2,1/c_3)$. These two real solutions are invariant up to cyclic permutations and conjugation.
\subsection{\texorpdfstring{Solutions coming from $V_2(x)$}{Solutions coming from V2(x)}}
Next we suppose that one of $c_0,c_1,c_2$ or $c_3$ is a root of $V_2(x)$. It turns out that this case corresponds to the cyclic $17$-roots of simple index $2$, yielding the well-known solutions 
\beql{appbi2}
(c_0,c_1,c_2,c_3)=(c,\overline{c},c,\overline{c}),
\eeq
described by Theorem \ref{index2b}, where
\[c=-\frac{1\pm\sqrt{17}}{16}+\mathbf{i}\frac{\sqrt{238\pm2\sqrt{17}}}{16},\]
where the $\pm$ signs agree. These two solutions are unimodular; two additional one can be obtained by complex conjugation.
\subsection{\texorpdfstring{Solutions coming from $V_3(x)V_4(x)$}{Solutions coming from V3(x)V4(x)}}
Starting from the factor $V_3(x)$ the degree of these polynomials are larger than $4$ and hence there is no straightforward way to obtain the solutions by radicals (but see \cite{LM1}). Therefore we use the substitution \eqref{appbst} (or more precisely Lemma \ref{ch3T}), applied to the factor $V_3(x)V_4(x)$ to obtain, e.g.
\begin{gather*}
h_0(\pm_A,\pm_B)=-\frac{1\pm_A5\sqrt{17}}{8}\pm_B\frac{\sqrt{34\pm_A2\sqrt{17}}}{8},\\
h_1(\pm_A,\pm_B)=\frac{-1\pm_A3\sqrt{17}}{8}\pm_B\frac{\sqrt{34\pm_A2\sqrt{17}}}{8},
\end{gather*}
where the signs $\pm_A$ and $\pm_B$ describe the sign of $\sqrt{17}$ and the sign of the nested radical, respectively. We remark here that in general this case corresponds to an irreducible polynomial over $\mathbb{Q}$ of degree $16$; it could have been factored in this case into the product $V_3(x)V_4(x)$ because in the canonical decomposition $17=s^2+t^2$ we have $s=1$ and as a result $D=0$ in \eqref{ch3CD}.

The $16$ solutions, in which $c_2=1/c_0$, $c_3=1/c_1$ are:
\begin{gather}\label{appbu1}(c_0,c_1,c_2,c_3)=(\mathcal{L}(h_0(+,+)),\mathcal{L}(h_1(+,-)),1/\mathcal{L}(h_0(+,+)),1/\mathcal{L}(h_1(+,-))),\\
\label{appbu2}(c_0,c_1,c_2,c_3)=(\mathcal{L}(h_0(-,-)),1/\mathcal{L}(h_1(-,+)),1/\mathcal{L}(h_0(-,-)),\mathcal{L}(h_1(-,+))),\\
\label{appbr1}(c_0,c_1,c_2,c_3)=(\mathcal{L}(h_0(+,-)),1/\mathcal{L}(h_1(+,+)),1/\mathcal{L}(h_0(+,-)),\mathcal{L}(h_1(+,+))),\\
\label{appbr2}(c_0,c_1,c_2,c_3)=(\mathcal{L}(h_0(-,+)),\mathcal{L}(h_1(-,-)),1/\mathcal{L}(h_0(-,+)),1/\mathcal{L}(h_1(-,-)))
\end{gather}
and their cyclic permutations. The first two set of solutions are unimodular and coincide with the solutions coming from Theorem \ref{ch3cyc4M} (cf.\ Example \ref{ch3cyc4ex}), whereas the last two are real.
\subsection{\texorpdfstring{Solutions coming from $V_5(x)$}{Solutions coming from V5(x)}}
After factoring the transformed polynomial $T_{V_{5}}(u)$, coming from Lemma \ref{ch3T}, in the field extension $\mathbb{Q}\left(\mathbf{i},\sqrt{17}\right)$ we find that four values of $h_0$ can be obtained from
\begin{multline*}h_0(\pm_A,\pm_B)=\\
2\pm_A\frac{\sqrt{34+18\sqrt{17}}}{4}\pm_B\frac{\sqrt{34\pm_A4\sqrt{17\sqrt{17}}+2\sqrt{289+136\sqrt{17}\pm_A68\sqrt{17\sqrt{17}}}}}{4}\pm_A\mathbf{i}\\
\times\left(-\frac{\sqrt{-34+18\sqrt{17}}}{4}\mp_B\frac{\sqrt{-34\mp_A4\sqrt{17\sqrt{17}}+2\sqrt{289+136\sqrt{17}\pm_A68\sqrt{17\sqrt{17}}}}}{4}\right).
\end{multline*}
With this notation the solution read
\beql{appbcplx}(c_0,c_1,c_2,c_3)=\left(\mathcal{L}(h_0(+,+)),1/\mathcal{L}(h_0(-,-)),\mathcal{L}(h_0(+,-)),\mathcal{L}(h_0(-,+))\right),\eeq
and its conjugate, reciprocal, conjugate-reciprocal and the cyclic permutations of these four. In total $16$ complex (non real, non unimodular) solutions are found.
\subsection{\texorpdfstring{Solutions coming from $V_6(x)$}{Solutions coming from V6(x)}}
As $V_6(x)$ is of degree $36$ we, again, investigate the roots of the transformed polynomial $T_{V_6}(u)$. Unfortunately it turns out that the roots are cannot be described by nested square roots any longer, as cubic roots appear during the solution of some degree $4$ polynomials. Therefore instead of describing the values of $h_0, h_1, h_2$ and $h_3$ we rather describe the degree four polynomials, encoding the essentially different solutions (i.e.\ the unordered set $\{h_0,h_1,h_2,h_3\}$) coming from the factor $V_6(x)$. The aforementioned polynomials can be obtained via their coefficients which however can be expressed by the elementary symmetric polynomials of $h_i$, $i=0,1,2,3$, coming from \eqref{appbst}. Let us use the following notations:
\begin{gather*}
\sigma_1=h_0+h_1+h_2+h_3,\\
\sigma_2=h_0h_1+h_0h_2+h_0h_3+h_1h_2+h_1h_3+h_2h_3,\\
\sigma_3=h_0h_1h_2+h_0h_1h_3+h_0h_2h_3+h_1h_2h_3,\\
\sigma_4=h_0h_1h_2h_3.
\end{gather*}
Then, by calculating various Gr\"obner bases, we find that the following hold:
\begin{gather*}
13 \sigma_1^4-212 \sigma_1^3-17584 \sigma_1^2-512 \sigma_1+1024=0,\\
169 \sigma_2^4-96328 \sigma_2^3+8248259 \sigma_2^2+64168790 \sigma_2+119444818=0,\\
2197 \sigma_3^4-1603714 \sigma_3^3+232154357 \sigma_3^2-48071520
   \sigma_3-857293956=0,\\
8788 \sigma_4^4+47418572 \sigma_4^3+31284060733 \sigma_4^2-68872842666
   \sigma_4+34518025937=0.
\end{gather*}
Now we build up, by trial and error, four polynomials $\mathcal{I}_4^{(i)}(x)$, $i=0,1,2,3$, each describing a set $\{h_0^{(i)}, h_1^{(i)}, h_2^{(i)}, h_3^{(i)}\}$ corresponding to a solution coming from $V_6(x)$. We have
\[\mathcal{I}_4^{(i)}(x)=x^4-\sigma_1^{(i)}x^3+\sigma_2^{(i)}x^2-\sigma_3^{(i)}x+\sigma_4^{(i)}, \ \ \ \ i=0,1,2,3, \]
where
\begin{gather*}
\sigma_1^{(1)}=\sigma_1(+,-),\ \ \sigma_2^{(1)}=\sigma_2(+,-), \ \ \sigma_3^{(1)}=\sigma_3(+,-), \ \ \sigma_4^{(1)}=\sigma_4(+,+),\\
\sigma_1^{(2)}=\sigma_1(-,+),\ \ \sigma_2^{(2)}=\sigma_2(-,-), \ \ \sigma_3^{(2)}=\sigma_3(-,-), \ \ \sigma_4^{(2)}=\sigma_4(-,+),\\
\sigma_1^{(3)}=\sigma_1(-,-),\ \ \sigma_2^{(3)}=\sigma_2(-,+), \ \ \sigma_3^{(3)}=\sigma_3(-,+), \ \ \sigma_4^{(3)}=\sigma_4(-,-),\\
\sigma_1^{(4)}=\sigma_1(+,+),\ \ \sigma_2^{(4)}=\sigma_2(+,+), \ \ \sigma_3^{(4)}=\sigma_3(+,+), \ \ \sigma_4^{(4)}=\sigma_4(+,-),
\end{gather*}
where
\begin{gather*}
\sigma_1(\pm_A,\pm_B)=\frac{53\pm_A59\sqrt{17}}{13}\pm_B\frac{\sqrt{63546\pm_A6358\sqrt{17}}}{13},\\
\sigma_2(\pm_A,\pm_B)=\frac{48164\pm_A7130\sqrt{17}}{338}\pm_B\frac{\sqrt{3307173846 \pm_A 705523222\sqrt{17}}}{338},\\
\sigma_3(\pm_A,\pm_B)=\frac{801857 \pm_A 89585\sqrt{17}}{4394}\pm_B\frac{4\sqrt{48275292054 \pm_A 8767889417\sqrt{17}}}{4394},\\
\sigma_4(\pm_A,\pm_B)=\\
-\frac{11854643 \pm_A 2053159\sqrt{17}}{8788}\pm_B\frac{2\sqrt{53118166790942 \pm_A 12183750689646\sqrt{17}}}{8788},
\end{gather*}
where $\pm_A$ and $\pm_B$ describe the sign of $\sqrt{17}$ and the sign of the nested radical, respectively. Once the four roots of $\mathcal{I}_4^{(i)}(x)$ are found we can lift them via Lemma \ref{appbLF} to obtain the values of $(c_0,c_1,c_2,c_3)$ up to their order and reciprocal. Now if the four real roots of $\mathcal{I}_4^{(i)}(x)$, denoted by $r_j^{(i)}$, $i,j=1,2,3,4$ satisfy $r_1^{(i)}<r_2^{(i)}<r_3^{(i)}<r_4^{(i)}$, then we find that the values of $c_0,c_1,c_2$ and $c_3$ in the correct order are
\begin{gather}\label{appbu3}\left(c_0,c_1,c_2,c_3\right)=\left(\mathcal{L}\left(r_1^{\left(1\right)}\right),\mathcal{L}\left(r_3^{\left(1\right)}\right),\mathcal{L}\left(r_4^{\left(1\right)}\right),1/\mathcal{L}\left(r_2^{\left(1\right)}\right)\right),\\
\label{appbu4}\left(c_0,c_1,c_2,c_3\right)=\left(\mathcal{L}\left(r_1^{\left(2\right)}\right),\mathcal{L}\left(r_2^{\left(2\right)}\right),\mathcal{L}\left(r_4^{\left(2\right)}\right),1/\mathcal{L}\left(r_3^{\left(2\right)}\right)\right),\\
\label{appbr3}\left(c_0,c_1,c_2,c_3\right)=\left(\mathcal{L}\left(r_1^{\left(3\right)}\right),\mathcal{L}\left(r_2^{\left(3\right)}\right),\mathcal{L}\left(r_3^{\left(3\right)}\right),1/\mathcal{L}\left(r_4^{\left(3\right)}\right)\right),\\
\label{appbr4}\left(c_0,c_1,c_2,c_3\right)=\left(\mathcal{L}\left(r_1^{\left(4\right)}\right),1/\mathcal{L}\left(r_2^{\left(4\right)}\right),1/\mathcal{L}\left(r_3^{\left(4\right)}\right),1/\mathcal{L}\left(r_4^{\left(4\right)}\right)\right),
\end{gather}
and their reciprocal and their cyclic permutations. The first two set of solutions are complex unimodular whereas the last two are real. In particular, we obtain two new, previously unknown complex Hadamard matrices of order $17$.
\section{Concluding remarks}
We summarize the results of this section in the following
\begin{thm}
The set of equations \eqref{72haa} has exactly $70$ solutions in $\mathbb{C}^4$ for $p=17$ and $k=4$ as follows:
\begin{enumerate}
\item The real $\varepsilon$-solution \eqref{appbep} and its reciprocal;
\item Two unimodular simple index $2$ type solutions \eqref{appbi2} and their reciprocal;
\item The simple index $4$ type solutions:
\begin{enumerate}[$($a$)$]
\item Two unimodular \eqref{appbu1}-\eqref{appbu2} and two real solutions \eqref{appbr1}-\eqref{appbr2}, and their cyclic permutations;
\item Two unimodular \eqref{appbu3}-\eqref{appbu4} and two real solutions \eqref{appbr3}-\eqref{appbr4}, their cyclic permutations and reciprocal;
\item A complex solution \eqref{appbcplx}, its cyclic permutations, conjugate, reciprocal and conjugate-reciprocal.
\end{enumerate}
\end{enumerate}
Among these there are $26$ real; $28$ complex unimodular; and $16$ addititional solutions. Each of the solutions can be described by radicals.
\end{thm}
It remains an open problem to describe the solutions coming from $V_6(x)$ in an elegant way by radicals. To solve the simple index $4$ case in full generality one needs to find a way to control the cyclic order of the roots which we avoided as we could compare our analytic formulas with precalculated numerical solutions.
%%%%%%%%%%%%%%%%%%%%%%%%%%%%%%%%%%%%%%%%%%APPC%%%%%%%5
\chapter[Hadamard matrices with circulant core]{Complex Hadamard matrices with circulant core}\label{APPC}
In Section \ref{ch3secC} we have seen that there are various general constructions yielding circulant complex Hadamard matrices of prime orders. It is natural to consider complex Hadamard matrices with circulant core as well. Through this section we work out an analogous theory to the circulant case with appropriate modifications, and give a full classification of complex Hadamard matrices with circulant core up to order $7$. This classification results in the discovery of a new $7\times 7$ complex Hadamard matrix (see Example \ref{ch3core7}). We exhibit two additional examples of order $11$ and for every prime $p\equiv 1$ (mod $8$) new examples of order $p+1$.
\section{The general theory}\label{ch3secQ}
Let us consider a dephased complex Hadamard matrix of order $n+1$ with circulant core, the first row of which is given by $x=(x_0,x_1,\hdots, x_{n-1})$. Then, by orthogonality, we have
\beql{qcycyc}\begin{cases}
\displaystyle{x_0=-\frac{1}{1+\frac{x_1}{x_0}+\frac{x_1}{x_0}\frac{x_2}{x_1}+\hdots+\frac{x_1}{x_0}\frac{x_2}{x_1}\cdots\frac{x_{n-1}}{x_{n-2}}}},\\
\displaystyle{\frac{x_1}{x_0}+\frac{x_2}{x_1}+\hdots+\frac{x_0}{x_{n-1}}=-1},\\
\displaystyle{\frac{x_2}{x_0}+\frac{x_3}{x_1}+\hdots+\frac{x_1}{x_{n-1}}}=-1,\\
\ \ \ \ \ \ \ \ \ \ \ \ \ \ \ \ \ \vdots\\
\displaystyle{\frac{x_{n-1}}{x_0}+\frac{x_0}{x_1}+\hdots+\frac{x_{n-2}}{x_{n-1}}}=-1.\\
\end{cases}\eeq
Note that the first equation is just a useful reformulation of the trivial condition coming from orthogonality to the first row:
\[x_0=-1-x_1-\hdots-x_{n-1},\]
and therefore it follows that $x_0$ is uniquely determined by the quotients $x_{i+1}/x_i$, where $i=0, \hdots, n-1$. Note also that it can be further simplified if one assumes that the variables $x_i$ are unimodular.

We offer the following well-known solution to the problem \eqref{qcycyc} when the size of the matrix $n+1$ is prime.
\begin{ex}
Let $g$ be a generator of the multiplicative group $\mathbb{Z}_p^\ast$. Then the vector $x^{(g)}$, with coordinates
\[x_i^{(g)}=\mathrm{exp}\left(\frac{2\pi\mathbf{i}}{p}g^i\right), i=0,\hdots, p-2\]
will give us the core of a $BH(p,p)$ matrix (see Definition \ref{ch1butsd}).
\end{ex}
Other well-known solutions yielding complex Hadamard matrices of order $n+1$ with circulant core come from the Paley matrix when the order of the core is prime (see Theorems \ref{ch1rpc1} and \ref{ch1pav2}).

Next we work out a general framework to study these objects in every order $n$ by slightly modifying the cyclic $n$-root problem. Let us introduce the quotients $z_i:=x_{i+1}/x_i$ for $i=0,1,\hdots,n-1$, where indices are taken modulo $n$. We have the following
\begin{thm}\label{ch3qprop}
Any unimodular solution $x$ of \eqref{qcycyc} gives rise to a unimodular solution $z$ of the following system of equations
\beql{qcycyc2}\begin{cases}
z_0+z_1+\hdots+z_{n-1}=-1,\\
z_0z_1+z_1z_2+\hdots+z_{n-1}z_0=-1,\\
\ \ \ \ \ \ \ \ \ \vdots\\
z_0z_1\cdots z_{n-2}+z_1z_2\cdots z_{n-1}+\hdots+z_{n-1}z_0\cdots z_{n-3}=-1,\\
z_0z_1\cdots z_{n-1}=1,\\
\end{cases}
\eeq
and \it{vice versa}: any unimodular solution $z$ of \eqref{qcycyc2} can be lifted to a unimodular solution $x$ of \eqref{qcycyc}.
\end{thm}
Note that the first equation of \eqref{qcycyc} does not appear amongst the equations \eqref{qcycyc2}. In particular, if all the quotients $z_i$ are unimodular, then so is $x_0$.
\begin{proof}
The first direction is immediate: any unimodular solution $x$ of \eqref{qcycyc} corresponds to a unimodular solution of \eqref{qcycyc2} after passing to the quotients $z_i:=x_{i+1}/x_i$. To see the other direction, suppose that we have a unimodular solution $z$ of \eqref{qcycyc2}, and let us construct the corresponding solution $x$ of \eqref{qcycyc}. For simplicity, let us introduce the following new quantity
\beql{ch3sdef}
S:=z_0+z_0z_1+z_0z_1z_2+\hdots+z_0z_1\cdots z_{n-1}.
\eeq
Now observe that if we sum up the first $n-1$ equations plus $n$ times the last one of \eqref{qcycyc2} we get
\beql{ch3Qc1}
S\left(1+\frac{1}{z_0}+\frac{1}{z_0z_1}+\hdots+\frac{1}{z_0z_1\cdots z_{n-2}}\right)=1,
\eeq
however, note that the second factor of \eqref{ch3Qc1} can be rewritten by the last equation of \eqref{qcycyc2} and as a result we have
\[S\overline{S}=|S|^2=1.\]
From a given unimodular solution $z$ one can reconstruct $x$ up to an unknown phase, say $x_0$, giving $x=x_0(1,z_0,z_0z_1,\hdots,z_0z_1\cdots z_{n-2})$. However, the first coordinate of $x$, which is $x_0$, is constrained by the first equation of \eqref{qcycyc}. In particular, we need to have
\[x_0=-\frac{1}{S}=-\overline{S}.\]
To conclude, we found that the solution on the $x$-level coming from $z$ reads \[x=-\overline{S}\left(1,z_0,z_0z_1,\hdots,z_0z_1\cdot\hdots\cdot z_{n-2}\right),\]
where $S$ is given by \eqref{ch3sdef}. Clearly, the identification $x\leftrightarrow z$ is one-to-one.
\end{proof}
We have the following natural
\begin{pro}\label{qprob}
For a given $n\in\mathbb{N}$ classify all solutions of \eqref{qcycyc2}. In particular, classify all complex Hadamard matrices of order $n$ possessing a circulant core.
\end{pro}
We shall return to Problem \ref{qprob} in Section \ref{ch3seccl}.
\subsection{\texorpdfstring{Circulant core of index $k$ type}{Circulant core of index k type}}
Here we introduce complex Hadamard matrices with index $k$ type circulant core in a completely analogous way as simple index $k$ type circulant complex Hadamard matrices were defined in Section \ref{ch3secC}. Note that the matrices we obtain here are not of prime order, but have a prime order core instead.

Let $p$ be a prime and let $k\in\mathbb{N}$ a number that divides $p-1$. Consider $G_0$, the unique subgroup of $\mathbb{Z}_p^\ast$ of index $k$ and its nontrivial cosets $G_1,\hdots,G_{k-1}$. We say that a solution $z$ of \eqref{qcycyc2} has index $k$, if the corresponding solution on the $x$ level \eqref{qcycyc} is constant on the cosets $G_j$, $j=0,\hdots,k-1$, namely
\beql{ch3formc}
\begin{cases}
x_i=c_j,\ \ \text{ if } i\in G_j,\ \ 1\leq i\leq p-1;&\ \text{ and}\\
x_0=\displaystyle{-\frac{p-1}{k}}\left(c_0+c_1+\hdots+c_{k-1}\right)-1,&
\end{cases}
\eeq
where $(c_0,c_1,\hdots,c_{k-1})\in(\mathbb{C}^\ast)^k$. The following is the relevant analogue of Proposition~\ref{ch3BJP}.
\begin{prop}\label{ch3BJP2}
If $x=(x_0,x_1,\hdots,x_{p-1})$ has the form \eqref{ch3formc}, then the equations \eqref{qcycyc} can be reduced to the following $k$ rational equations in $c_0, c_1, \hdots, c_{k-1}$:
\beql{72haaN}
1+\frac{c_a}{x_0}+\frac{x_0}{c_{a+m}}+\sum\limits_{i,j=0}^{k-1}n_{ij}\frac{c_{a+j}}{c_{a+i}}=0,\ 0\leq a\leq k-1,
\eeq
where all indices are calculated modulo $k$. In \eqref{72haaN} the number $m$ is determined by $p-1\in G_m$ and $n_{ij}$ denote the transition numbers of order $k$, namely, the number of $b\in G_i$ for which $b+1\in G_{j}$, $0\leq i,j\leq k-1$, with respect to the generator $g$ of $\mathbb{Z}_p^\ast$.
\end{prop}
Again, the system of equations \eqref{72haaN} is independent of the choice of $g$ up to relabelling the variables.
Here we are interested in the unimodular solutions only, i.e.\ the ones that correspond to complex Hadamard matrices.

The following two results describe complex Hadamard matrices with index $2$ type circulant core coming from Proposition \ref{ch3BJP2}. The first one is the analogue of Theorem \ref{index2a} (cf.\ Theorem \ref{ch1rpc1}).
\begin{lem}
Let $p\geq 7$ be a prime, $p\equiv 3$ $(\mathrm{mod}\ 4)$, and suppose that $H$ is a complex Hadamard matrix of order $p+1$ with circulant core of index $2$ type. Then $H$ is equivalent to the real Paley-Hadamard matrix.
\end{lem}
The second result is the analogue of Theorem \ref{index2b} (cf.\ Theorem \ref{ch1pav2}).
\begin{lem}\label{ch3index2B}
Let $p$ be a prime, $p\equiv 1$ $(\mathrm{mod}\ 4)$, and suppose that $H$ is a complex Hadamard matrix of order $p+1$ with circulant core of index $2$ type. Then $(c_0,c_1)=(a,\overline{a})$, where
\[a=\pm\mathbf{i}; \text{ or } a=-\frac{2}{p-1}\pm\mathbf{i}\frac{\sqrt{(p-1)^2-4}}{p-1}.\]
\end{lem}
\begin{rem}
For $p=5$ we get the well-known complex Hadamard matrices $D_6^{(1)}(1)$ and $S_6^{(0)}$ (see Examples \ref{ch2d6ex} and \ref{ch2s6}), respectively. The $BH(p+1,4)$ matrices are well-known examples of complex Hadamard matrices (see Theorem \ref{ch1pav2}) while the other set of examples are just a special case of Proposition \ref{ch3s6gen}.
\hfill$\square$\end{rem}
We did not find any examples of complex Hadamard matrices with index $3$ type circulant core, however we can harvest some additional fruits from the theory of index $4$ type circulant matrices as follows. Interestingly both constructions yield complex Hadamard matrices with real diagonal.
\begin{prop}\label{ch3core1}
Let $p$ be a prime number, $p\equiv 1$ $(\mathrm{mod}\ 8)$, $p=s^2+t^2$ such that $s\equiv 1$ $(\mathrm{mod}\ 4)$ and $t>0$. Then there is a complex Hadamard matrix $H$ of order $p+1$ with circulant core of index $4$ type, where $(c_0,c_1,c_2,c_3)=(a,-a,\overline{a},-\overline{a})$, or any cyclic permutation of it, where
\[a=\frac{1-\sqrt{1+t^2}}{t}+\mathbf{i}\frac{\sqrt{-2+2\sqrt{1+t^2}}}{t}.\]
\end{prop}
\begin{prop}\label{ch3core2}
Let $p$ be a prime number, $p\equiv 1$ $(\mathrm{mod}\ 8)$, $p=s^2+t^2$ such that $s\equiv 1$ $(\mathrm{mod}\ 4)$ and $t>0$. Then there is a complex Hadamard matrix $H$ of order $p+1$ with circulant core of index $4$ type, where $(c_0,c_1,c_2,c_3)=(a,b,\overline{a},\overline{b})$, or any cyclic permutation of it, where
\begin{gather}\label{ch3ap}
a=-\frac{2}{p-1}+\sqrt{A-\sqrt{B}}+\mathbf{i}\sqrt{1-\left(\frac{2}{p-1}-\sqrt{A-\sqrt{B}}\right)^2},\\
\label{ch3bp}
b=-\frac{2}{p-1}-\sqrt{A-\sqrt{B}}-\mathbf{i}\sqrt{1-\left(\frac{2}{p-1}+\sqrt{A-\sqrt{B}}\right)^2},
\end{gather}
where
\begin{align*}
A&=\frac{t^2 \left(p^2+2   (s-2)^2+1\right)+2 (s-1)^4}{t^2(p-1)^2},\\
\begin{split}
B&=\frac{4\left(t^2(p+(s -2 )^2+2)+(s-1)^4\right)}{t^4(p-1)^4}\\
&\quad\times   \left(t^2 \left(t^2 \left(2   s^2+t^2-1\right)+\left((s+1)^2+2\right)   (s-1)^2\right)+(s-1)^4\right).
\end{split}
\end{align*}
\end{prop}
\begin{rem}
Formulas \eqref{ch3ap}-\eqref{ch3bp} are well defined, as clearly $A>0$, $B>0$ and
\[(p-1)^4(A^2-B)=(p+1)^2(p-3)^2>0\]
for all relevant primes $p$.
\hfill$\square$\end{rem}
The proof of Propositions \ref{ch3core1} and \ref{ch3core2} is completely analogous to the proof of Theorem \ref{ch3cyc4M}, the only difference being that we rely on Proposition \ref{ch3BJP2} instead of Proposition \ref{ch3BJP}.
\section{\texorpdfstring{Classification up to order $7$}{Classification up to order 7}}\label{ch3seccl}
In what follows we launch a brute-force attack against Problem \ref{qprob} via Gr\"obner basis techniques and classify all complex Hadamard matrices with circulant core up to order $7$. We report new complex Hadamard matrices of order $7$ and $11$ which are the main contribution of this section.

We shall encounter various self-inversive polynomials with real coefficients (see Definition \ref{ch2si}) during the sequel and therefore we recall the following
\begin{defi}
A complex polynomial $f(x)=a_0+a_1x+\hdots+a_dx^d$ of degree $d$ with real coefficients is called \emph{palindromic} if $a_{d-k}=a_k$ for every $k=0,\hdots,d$.
\end{defi}
If $f(x)$ is a palindromic polynomial of even degree $d$, then it is a standard trick to consider $f(x)/x^{d/2}$ instead which can be rewritten as a polynomial of degree $d/2$ in the new variable $u=x+1/x$. This operation is investigated in the following
\begin{lem}[cf.\ \cite{KMx1}]\label{ch3T}
Let $f(x)$ be a palindromic polynomial of even degree $d$. Then $f(x)$ has a unimodular root $x_0$ if and only if the transformed polynomial 
\[T_f(u)\equiv f\left(\frac{u+\sqrt{u^2-4}}{2}\right)\cdot\left(\frac{u-\sqrt{u^2-4}}{2}\right)^{d/2}\]
has a real root $|u_0|\leq 2$.
\end{lem}
\begin{proof}
Suppose that $f(x)$ has a unimodular root $x_0$. Then $u_0:=x_0+\frac{1}{x_0}=x_0+\overline{x_0}=2\Re[x_0]$ and hence $u_0$ is real root of $T_f(u)$ such that $|u_0|\leq 2$. To see the converse direction, suppose that the polynomial $T_f(u)$ has a real root $u_0$ such that $|u_0|\leq 2$. Then the numbers $x_0$ defined via the roots of the quadratic equation $x^2-u_0x+1=0$ are easily seen to be unimodular roots of $f(x)$.
\end{proof}
Lemma \ref{ch3T} gives a characterization of unimodularity via an inequality whose validity can be more easily checked than the unimodularity condition itself. However, it is not entirely obvious how to separate the real roots from those who have some negligible (but still nonzero) imaginary part.

We shall apply repeatedly the following useful ``averaging''
\begin{lem}\label{ch3aver1}
Let $n,m$ be positive integers, $S\subset\{0,1,\hdots,n-1\}$ and consider the system of $m+1$ complex equations in $n+1$ variables $(u,x_0,x_1,\hdots, x_{n-1})\in\mathbb{C}^{n+1}:$
\beql{ch3soll}
\begin{cases}
f_1(x_0,x_1,\hdots, x_{n-1})&=0,\\
f_2(x_0,x_1,\hdots, x_{n-1})&=0,\\
\ \ \ \ \ \ \ \ \ \ \ \ \vdots&\\
f_m(x_0,x_1,\hdots, x_{n-1})&=0,\\
u\prod\limits_{i\in S}x_i&=1.\\
\end{cases}
\eeq
Let us define the projected solution set with respect to the index set $S$ by
\[U_S:=\{u\in\mathbb{T} : (u,x_0,x_1,\hdots, x_{n-1})\text{ is a solution of } \eqref{ch3soll}\text{ for some } x_0,x_1,\hdots, x_{n-1}\}.\]
If $U_S$ is empty, then the system of equations does not have any unimodular solution.
\end{lem}
\begin{proof}
The Lemma is trivial. If $(u,x_0,x_1,\hdots,x_{n-1})\in\mathbb{T}^{n+1}$ is a unimodular solution of \eqref{ch3soll}, then $u\in U_S$.
\end{proof}
Lemma \ref{ch3aver1} exploits the fact that by looking at the possible values of the products of some of the variables $x_i$ (indexed by the set $S$), or equivalently, the values of the ``dummy'' variable $u$ only, one can conclude immediately that \eqref{ch3soll} has no unimodular solutions. Also, calculating the possible values of $u$ might be computationally cheaper than calculating the individual values of $x_0, x_1,\hdots, x_{n-1}$ directly.

Based on the full classification of complex Hadamard matrices up to order $5$ the following is easy to see (and in fact can be calculated by hand):
\begin{lem}
The complex Hadamard matrices with circulant core of orders $2\leq n+1\leq 5$ are equivalent to $F_2, F_3, F_2\otimes F_2$ and $F_5$.
\end{lem}
Note that the Fourier matrix $F_4$ is not equivalent to a complex Hadamard matrix with circulant core. The next result, however, is an entirely not obvious
\begin{prop}\label{ch3s6d6}
The complex Hadamard matrices with circulant core of order $n+1=6$ are equivalent to $S_6^{(0)}$ or $D_6^{(1)}(1)$ $($see Examples \ref{ch2s6} and \ref{ch2d6ex}, respectively$)$.
\end{prop}
\begin{proof}
By calculating a pure lexicographic Gr\"obner basis for the system of equations \eqref{qcycyc}, which we extended with an additional equation
\beql{ch3ext}
ux_0x_1x_2x_3x_4-1=0
\eeq
to avoid zero solutions, as usual, we found that there are finitely many solutions only. In particular, the polynomial, whose roots describe all possible values of the variable $x_0$ (and due to cyclic symmetry account for all possible values of any of the variables $x_i, i=0,1,\hdots, 4$) reads
\begin{equation*}
\begin{split}
\mathcal{X}_5(x)&=(x-1) (x+1) \left(x^2+1\right) \left(x^2-14
   x+1\right) \left(x^2+x+1\right)\left(x^2+4
   x+1\right)\left(x^{10}+2 x^9\right.\\
   &\quad\left.+15 x^8+108 x^7+420
   x^6+672 x^5+420 x^4+108 x^3+15 x^2+2 x+1\right).
\end{split}
\end{equation*}
Now it is easy to recover the (known) solutions $S_6^{(0)}$ and $D_6^{(1)}(1)$ as they correspond to the factors $(x-1)(x^2+x+1)$ and $(x+1)(x^2+1)$, respectively. The factor $(x^2-14x+1)(x^2+4x+1)$ does not have any unimodular roots. It remains to show that the degree $10$ factor, which we denote by $g(x)$, does not lead to any unimodular solutions $(x_0,x_1,x_2,x_3,x_4)\in \mathbb{T}^5$. By using Lemma \ref{ch3T} we see that $g(x)$ has unimodular roots, so the question is whether a unimodular root of it can be extended to a unimodular solution (i.e.\ if we can find additional unimodular numbers $x_1,\hdots, x_4$, satisfying \eqref{qcycyc}). Now we use Lemma \ref{ch3aver1} by adding $g(x_0)$ to the initial system of equations \eqref{qcycyc} accompanied again by \eqref{ch3ext}, and by calculating a Gr\"obner basis we find that the polynomial
\beql{ch3uthis}
u^2+322u+1
\eeq
is a member of the ideal, generated by the aforementioned polynomials. From this we see that as \eqref{ch3uthis} does not have any unimodular roots the projected solution set $U_S$, with respect to the index set $S=\{0,1,2,3,4\}$ is empty. Therefore the factor $g(x)$, although has some unimodular roots, does not lead to unimodular solutions.
\end{proof}
So it turns out that all complex Hadamard matrices with circulant core up to order $6$ are well-known. This is not the case for order $7$, as we have discovered a new, previously unknown complex Hadamard matrix, $Q_7$. The following is true.
\begin{thm}\label{ch3corethm7}
The complex Hadamard matrices with circulant core of order $n+1=7$ are equivalent to one of $F_7$, $P_7:=P_7^{(1)}(1)$, $P_7^\ast$ $($see Example \ref{ch3pet7}$)$, $Q_7$ or $Q_7^\ast$ $($see Example \ref{ch3core7}$)$.
\end{thm}
\begin{rem}
The list above might be redundant as it is unknown whether or not the matrices $Q_7$ and $Q_7^\ast$ are equivalent.
\hfill$\square$\end{rem}
\begin{proof}
Similarly to the $n+1=6$ case we managed to compute a pure lexicographic Gr\"obner basis for the system of equations \eqref{qcycyc}, which we extended with
\beql{ch3dum7}
ux_0x_1x_2x_3x_4x_5=1
\eeq
as usual. Again, we found that there are finitely many solutions only and the polynomial $\mathcal{X}_6(x)$, describing all possible roots of the system of equations \eqref{qcycyc} is of degree $69$. It is easy to recover the solutions corresponding to the known matrices $F_7$ and $P_7$ and additionally eliminate the factors without any unimodular roots. We are left with two polynomials, dividing $\mathcal{X}_6(x)$, namely
\begin{gather*}
g_{7A}(x)=2x^6+2x^5+2x^4-5x^3+2x^2+2x+2,\ \ \text{ and}\\
g_{7B}(x)=16x^{36}+96x^{35}+3584x^{34}+54800x^{33}+516276x^{32}+3534368x^{31}+\hdots,
\end{gather*}
of degree $6$ and $36$, respectively, having unimodular roots. The polynomial $g_{7A}(x)$ does not lead to any unimodular solutions, which can be seen in a similar way as we have excluded the degree $10$ factor of order $6$ during the proof of Proposition \ref{ch3s6d6}. In particular, we should add $g_{7A}(x_0)$ along with the dummy equations \eqref{ch3dum7} and $vx_0x_1x_2x_3-1=0$ to \eqref{qcycyc} and compute a Gr\"obner basis to find that the polynomial
\beql{ch3vvv}
v^6+v^5-6 v^4-13 v^3-6 v^2+v+1
\eeq
is member of the ideal, generated by the aforementioned polynomials. By an application of Lemma \ref{ch3T} we see that \eqref{ch3vvv} does not have any unimodular roots and hence by Lemma \ref{ch3aver1} (with $S=\{0,1,2,3\}$) there are no unimodular solutions in this case either.

It remains to deal with the degree $36$ factor $g_{7B}(x)$. We start by transforming it through Lemma \ref{ch3T} to obtain
\beql{ch3Tuf}
T_{g_{7B}}(u)=16 u^{18}+96 u^{17}+3296 u^{16}+53168 u^{15}+461092u^{14}+2723792 u^{13}+\hdots
\eeq
By evaluating $T_{g_{7B}}(u)$ at various rational points within the interval $[-3,3]$, one can detect at least $6$ real roots within the interval $[-2,2]$. Therefore $g_{7B}(x)$ has at least $12$ distinct unimodular roots. By calculating various Gr\"obner basis with respect to various monomial ordering, one can see that if any of $x_0,x_1,\hdots,x_5$ comes from $g_{7B}(x)$, then the values of all additional variables are roots of $g_{7B}(x)$ as well. Therefore, after solving (e.g.\ numerically) one of the obtained Gr\"obner basis containing $g_{7B}(x_0)$, we see that the unimodular solutions are being given by 
\begin{equation*}%\label{ch3x0sol}
\begin{split}
(x_0,x_1,x_2,x_3,x_4,x_5)\approx(-0.843295+0.537451\mathbf{i},-0.627425+0.778677\mathbf{i},\\
-0.68299-0.730428\mathbf{i},0.0470571 +0.998892\mathbf{i},\\
0.342548 -0.9395\mathbf{i},0.764105 -0.645092 \mathbf{i}),
\end{split}
\end{equation*}
its reversed form, i.e.\ $(x_5,x_4,x_3,x_2,x_1,x_0)$, the conjugate of both and the cyclic permutations of all four, featuring the unimodular roots we detected through $T_{g_{7B}}(u)$. Hence, they correspond to a complex Hadamard matrix and its adjoint which we call $Q_7$.

Finally one should see that the matrices $F_7$, $P_7$, $P_7^\ast$, $Q_7$ and $Q_7^\ast$ are inequivalent. It is trivial to distinguish the $BH(7,7)$ and $BH(7,6)$ matrices via Haagerup's invariant, while the inequivalence of $P_7$ and $P_7^\ast$ can be checked directly with a computer search. Lemma \ref{ch1LRO} guarantees that $Q_7$ cannot be a Butson-Hadamard matrix as in its dephased form there are entries which are not roots of unity. Finally, $Q_7$ and $Q_7^T$ are trivially equivalent, but we do not know if the same holds for $Q_7$ and $Q_7^\ast$.
\end{proof}
We are left with the following unsolved
\begin{pro}
Investigate if the matrices $Q_7$ and $Q_7^\ast$ are equivalent.
\end{pro}
More generally, we need to find new methods to distinguish a complex Hadamard matrix from its adjoint, conjugate and transpose in an efficient way.
\begin{pro}
Find new criteria distinguishing a complex Hadamard matrix from its adjoint, conjugate and transpose, respectively.
\end{pro}
There is some interest in describing roots of polynomial systems by radicals. A notable example is the problem of the existence of $d^2$ equiangular lines in $\mathbb{C}^d$ (see Section \ref{ch2MUBEQ}), where exact solutions are known only for $d\leq 16$ and for $d\in\{19,24,35,48\}$. These examples along with their closed analytic descriptions were found by various symmetry considerations and clever use of Gr\"obner-basis techniques, see e.g.\ \cite{ABBGGL}, \cite{SG1} and the references therein.

It seems that Landau and Miller \cite{LM1} found a polynomial time algorithm to construct the splitting field of a polynomial whose roots can be described by radicals. Both their method and our elementary (but rather technical) considerations are far beyond the scope of this thesis, and we just mention the following without any further comment.
\begin{prop}\label{appcrad}
The matrix $Q_7$ can be described by radicals
\end{prop}
\begin{proof}
Let $\alpha$ as in \eqref{ch3core7fa} and observe that $\alpha^3-40169 \alpha^2+122486812 \alpha+124134308=0$ and hence $\alpha$ can be described by radicals. Define further
\[r=\frac{1}{2} \sqrt{\frac{-1394249317 \alpha
   ^2+4934112167907 \alpha
   +5004997560582}{3595497500}+2},\]
and $\beta=343\sqrt{7}\left(r+\mathbf{i}\sqrt{1-r^2}\right)$. Then the degree $18$ polynomial $T_{g_{7B}}(u)$ \eqref{ch3Tuf} can be factored in the field extension $\mathbb{Q}\left(\alpha\right)$ into a product of polynomials of degree $6$ and $12$, where the degree $6$ factor $h(u)$ encodes (twice) the real part of the entries of $Q_7$ (see \eqref{ch3core7f}). Additionally, $h(u)$ can be factored in the field extension $\mathbb{Q}\left(\alpha,\beta\right)$, and hence all of its roots can be described by radicals, as stated.
\end{proof}
We were unable to proceed any further and classify higher order complex Hadamard matrices with circulant core due to insufficient computing power. However, in higher orders one might try conducting various numerical experiments in order to discover at least some examples of complex Hadamard matrices. One reasonable approach is attempting to minimize the function
\[F(x_0,\hdots,x_{(n-1)^2})=\left\|H(x_0,\hdots,x_{(n-1)^2})H^\ast(x_0,\hdots,x_{(n-1)^2})-nI\right\|,\]
where $H$ is a dephased matrix of order $n$ with an otherwise unspecified core.
These type of searches are (in general) inconclusive due to the large number of variables, because the minima values found are not necessarily capture the global minimum of the function. However, if one significantly reduces the number of variables (as in the case when circulant matrices or Hadamard matrices with circulant core are considered) then there might be some hope to extract some interesting examples.

Such a numerical method was used in \cite{BN1} and in our Master's thesis \cite{SZF3}, where new examples of complex Hadamard matrices were obtained. In particular, we exhibited a $9\times 9$ complex Hadamard matrix $Q_9$ with circulant core, whose core was induced by a vector $x=(a,b,c,d,\overline{a},\overline{b},\overline{c},\overline{d})$. The entries $a,b,c$ and $d$ are rather complicated, nevertheless they can be described by closed analytic formulae. Here we report on new examples of order $11$ which were, again, first discovered by the numerical method indicated above, and then their existence was confirmed by Gr\"obner basis techniques.
\begin{ex}\label{ch3exf3}
Here we describe two symmetric complex Hadamard matrices of order $11$ with circulant core. The core is induced by the vector $x=(a,b,\overline{b},c,\overline{c},\overline{a},\overline{c},c,\overline{b},b)$, where the values of $a, b$ and $c$ can be obtained as follows. Let
\[\gamma=\frac{1}{3}\sqrt{396+18 \sqrt[3]{8468+12 \sqrt{12309}}+3\sqrt[3]{1829088-2592
   \sqrt{12309}}},\]
and let $a_1, a_2\neq \overline{a_1}$ and their conjugate be the four roots of the polynomial
\[1008 a^4-(336 \gamma-1344)a^3- \left(\gamma^5\!-\!132 \gamma^3-56
   \gamma^2+1088 \gamma\!+\!1232\right)a^2-(336 \gamma-1344)a+1008.\]
It is easy to see (either by transforming the polynomial via Lemma \ref{ch3T} or by an application of Theorem \ref{ch2thmcohn}) that both $a_1$ and $a_2$ are unimodular. Further, let
\[b+c=-\frac{1}{48}\left(30 a^7+53 a^6-83 a^5-56 a^4+117a^3-176 a^2+103 a+59\right).\]
From this last equation $b$ and $c$ can be obtained via the Decomposition formula (see Lemma \ref{ch2decf}). For a fixed value of $a$ the order of $b$ and $c$ is irrelevant, as both choices lead to a complex Hadamard matrix or its conjugate. It is unknown whether or not these matrices are equivalent. However, the two essentially different matrices, arising from $a_1$ and $a_2$, have different fingerprints, and as a result they correspond to inequivalent complex Hadamard matrices of order $11$. We call these matrices $Q_{11A}$ and $Q_{11B}$, respectively.
\end{ex}
Despite the efforts of this chapter the existence of complex Hadamard matrices with circulant core remains an open problem in general. We conclude with the following
\begin{pro}
For every $n\geq 1$ construct a complex Hadamard matrix with circulant core of order $n$, or show that no such a matrix exists.
\end{pro}
\end{appendices}
\end{document}